\documentclass[10pt]{amsart}
\usepackage[foot]{amsaddr}
\usepackage[a4paper,margin=2.5cm]{geometry}
\usepackage[english]{babel}
\usepackage{csquotes}
\usepackage[T1]{fontenc}
\usepackage{textcomp}
\usepackage{amsmath,amssymb,amsthm,mathtools,mathrsfs}
\usepackage{newtxtext}
\usepackage{newtxmath}
\usepackage[bb=boondox,cal=boondoxo,scr=boondoxo]{mathalfa}
\usepackage{enumitem}
\usepackage{booktabs}
\usepackage{array}
\usepackage{tikz}
\usetikzlibrary{arrows.meta}
\definecolor{diagramHahn}{HTML}{5B4B8A}
\definecolor{diagramK}{HTML}{7A8F32}
\definecolor{diagramMII}{HTML}{2563A6}
\definecolor{diagramMI}{HTML}{C56A18}
\definecolor{diagramCII}{HTML}{16847B}
\definecolor{diagramCI}{HTML}{B84A5B}
\usepackage[hidelinks]{hyperref}
\hypersetup{
  pdftitle={An Askey-Type Confluence Scheme for Hahn-Like Multiple Orthogonality: Bernstein Discretization and Explicit Hypergeometric Formulas},
  pdfauthor={Manuel Ma\~nas},
  pdfsubject={Bernstein discretization of Wolfs' Jacobi-like multiple orthogonality and its explicit hypergeometric Askey-type confluence scheme},
  pdfkeywords={multiple orthogonal polynomials; Hahn-like weights; Kravchuk-like weights; Krawtchouk polynomials; Jacobi-like weights; Meixner-like weights; Charlier-like weights; Laguerre-like weights; Hermite-like systems; Askey-type confluence scheme; Bernstein transform; hypergeometric functions; Kampé de Fériet functions}
}
\usepackage{microtype}
\usepackage[
  backend=biber,
  style=numeric,
  sorting=nyt,
  giveninits=true,
  maxnames=99,
  doi=true,
  url=true,
  isbn=false
]{biblatex}
\DeclareFieldFormat[article]{title}{\mkbibemph{#1}}
\DeclareFieldFormat[article]{journaltitle}{#1}
\DeclareFieldFormat[article]{volume}{\textbf{#1}}
\DeclareFieldFormat{pages}{#1}
\renewbibmacro{in:}{%
  \ifentrytype{article}{}{\printtext{\bibstring{in}\intitlepunct}}}
\renewbibmacro*{journal+issuetitle}{%
  \usebibmacro{journal}%
  \setunit*{\addspace}%
  \printfield{volume}%
  \iffieldundef{number}{}{
    \setunit{\addcomma\space}%
    \printtext{no.\addspace}\printfield{number}}%
  \iffieldsequal{volume}{year}{}{
    \setunit{\addspace}%
    \printtext[parens]{\printfield{year}}}%
  \iffieldundef{eid}{}{
    \setunit{\addcomma\space}%
    \printfield{eid}}%
  \iffieldundef{pages}{}{
    \iffieldundef{eid}
      {\setunit{\addcomma\space}\printfield{pages}}
      {\iffieldsequal{eid}{pages}{}
        {\setunit{\addcomma\space}\printfield{pages}}}}%
  \newunit}
\renewbibmacro*{note+pages}{%
  \ifentrytype{article}{}{\printfield{note}%
    \setunit{\bibpagespunct}\printfield{pages}}}

\allowdisplaybreaks
\mathtoolsset{showonlyrefs}
\setlist[itemize]{leftmargin=*,itemsep=2pt,topsep=4pt}
\setlist[enumerate]{leftmargin=*,itemsep=2pt,topsep=4pt}

\newcommand{\N}{\mathbb N}
\newcommand{\Nzero}{\mathbb N_0}
\newcommand{\R}{\mathbb R}
\newcommand{\e}{\mathrm e}
\newcommand{\dd}{\,\mathrm d}
\newcommand{\one}{\boldsymbol 1}
\newcommand{\A}{\boldsymbol A}
\newcommand{\bbeta}{\boldsymbol\beta}
\newcommand{\mm}{\boldsymbol m}
\newcommand{\ee}{\boldsymbol e}
\newcommand{\abs}[1]{\lvert#1\rvert}
\newcommand{\fall}[2]{#1^{\underline{#2}}}
\newcommand{\pFq}[5]{{}_{#1}F_{#2}\!\left(\begin{matrix}#3\\#4\end{matrix};#5\right)}
\DeclareMathOperator*{\Res}{Res}

\theoremstyle{plain}
\newtheorem{theorem}{Theorem}[section]
\newtheorem{proposition}[theorem]{Proposition}
\newtheorem{lemma}[theorem]{Lemma}
\newtheorem{corollary}[theorem]{Corollary}

\theoremstyle{definition}
\newtheorem{definition}[theorem]{Definition}

\theoremstyle{remark}
\newtheorem{remark}[theorem]{Remark}

\title[An Askey-Type Hahn-Like Confluence Scheme]
{An Askey-Type Confluence Scheme for Hahn-Like Multiple Orthogonality\\[0.4ex]
	\mdseries\small Bernstein Discretization and Explicit Hypergeometric Formulas}
\author{Manuel Ma\~nas}
\address{Department of Theoretical Physics, Faculty of Physical Sciences,
Complutense University of Madrid, 28040 Madrid, Spain}
\email{manuel.manas@ucm.es}
\date{September 8, 2026}

\begin{document}

\begin{abstract}
We place the Jacobi-like and Laguerre-like systems for ordinary
type-I/type-II multiple orthogonality considered by Wolfs in a single
Askey-type confluence scheme. The construction proceeds in two linked
stages. First, applying the Bernstein transform simultaneously to the
\(q\) Jacobi-like weights produces a positive finite-lattice Hahn-like
ancestor. Its weights admit positive multiple beta-integral
representations, reduce to the classical Hahn weight when \(q=1\), and
converge to the original continuous weights. Second, from this common
ancestor, the Bernstein limit and further parameter and scaling limits yield a
Kravchuk-like system, two Meixner-like systems, and two Charlier-like
systems, together with the Jacobi-like, two Laguerre-like, and
Hermite-like continuous families.

We realize this diagram at the level of explicit hypergeometric
orthogonality data. For every near-diagonal multi-index with
\(1\le\abs\mm\le N\), the Hahn-like type-II polynomial is a terminating
\({}_{q+2}F_{q+1}(1)\) series and satisfies an exact inverse Bernstein
identity. We also construct the normalized type-I form and, under
explicit separation and nonvanishing conditions, recover each of its
polynomial components by finite sums of terminating hypergeometric
functions and prove uniqueness. The finite-pole contributions regroup
into terminating Kamp\'e de F\'eriet blocks. On the unreflected
Kravchuk-like and Meixner-II-like branches, the moving Hahn block is
isolated as a coefficient extraction from its terminating Kamp\'e de
F\'eriet factor and reconstructed at infinity; on the reflected branches,
the grouped blocks have ordinary coefficientwise limits given by finite
Lauricella--Horn sector sums. The
Kravchuk-like-to-Charlier-II-like,
reflected Kravchuk-like-to-Charlier-I-like,
Meixner-II-like-to-Charlier-II-like,
Meixner-II-like-to-Laguerre-I-like,
Meixner-I-like-to-Laguerre-II-like, and both
Charlier-like--Hermite-like confluences hold sectorwise.

Along every directed confluence arrow we determine the type-II polynomial, the
normalized type-I form, and, under the stated hypotheses, the individual
type-I components and their componentwise confluence. For one weight the
formulas reduce to the corresponding classical families. For several
weights, no permutation of the normalized rows identifies the two
Meixner-like systems, and the same holds for the two Charlier-like
systems.
\end{abstract}

\maketitle
\enlargethispage{4pt}
\tableofcontents

\section{Introduction}
\label{sec:intro}

The Hahn family is the finite-lattice member of the classical Askey
scheme, with limits to Jacobi, Meixner, and Kravchuk polynomials. In
multiple orthogonality, the Meixner and Laguerre families split into two
families that coincide only in the scalar case. Multiple Hahn systems and
their limits have been studied in
\cite{ArvesuCoussementVanAssche2003,BranquinhoDiazFoulquieManas2023Hahn,
BranquinhoDiazFoulquieManasWolfs2025,BranquinhoDiazFoulquieManas2025Classical};
standard scalar formulas and limits can be found in
\cite{KoekoekLeskySwarttouw2010}; the classical continuous multiple
families are surveyed in \cite{VanAsscheCoussement2001}.

For \(q=1\), the relation between Hahn polynomials and the Bernstein
coefficients of Jacobi polynomials is classical
\cite{Ciesielski1987,Waldron2006}.  In the multiple setting,
\cite[Proposition~3.18]{BranquinhoDiazFoulquieManasWolfs2025} gives a coefficient
identity between factorial and monomial expansions. We derive its
Bernstein form below. Classical multiple Hahn, Meixner-I, Meixner-II,
Kravchuk, and Charlier systems and their limits have likewise been studied in
\cite{ArvesuCoussementVanAssche2003,
BranquinhoDiazFoulquieManasWolfs2024Discrete}.

Our aim is to place the Jacobi-like and Laguerre-like systems for ordinary
type-I/type-II multiple orthogonality considered by Wolfs
\cite{Wolfs2024} inside a single Askey-type confluence scheme. In this
interpretation the continuous
families are not isolated constructions: they are connected through a
common finite-lattice ancestor and its discrete and continuous
confluences. The construction is carried out in two linked stages that
together form the central contribution of the paper.

The first stage constructs the common discrete ancestor. We start from
Wolfs' positive Jacobi-like system. Its weights have Mellin transforms
given by quotients of gamma products and admit multiple beta-integral
representations; explicit type-I and type-II formulas and the
corresponding AT conditions are known there. Applying the
Bernstein transform to all \(q\) weights produces a positive
finite-lattice system, which we call Hahn-like. We compute its factorial
moments and its type-II polynomial, prove that the Bernstein transform of
this polynomial is exactly Wolfs' Jacobi-like type-II polynomial, and
construct the normalized type-I form. Under explicit separation and
nonvanishing hypotheses, every type-I component is given by a finite
hypergeometric formula and the corresponding multiple-orthogonality
problem is normal.

The second stage develops the same scheme by confluence from the
Hahn-like node. Two endpoint scalings produce Meixner-like families that,
for \(q>1\), cannot be identified by permuting their normalized rows;
their further limits give two Charlier-like and two Laguerre-like
families. A partial parameter limit at fixed lattice size produces a
Kravchuk-like family. In this last limit only one parameter pair is sent
to infinity: sending every pair to infinity would collapse all normalized
rows to the same binomial weight and destroy the multiple system. The two
Charlier-like branches subsequently reach reflected realizations of the
Hermite-like family, while the Bernstein limit returns the original
Jacobi-like node.

The explicit hypergeometric formulas are the mechanism by which the nodes
and arrows of the diagram are realized. Besides the fully expanded
residue formulas, we retain the bivariate hypergeometric structure of the
Hahn-like type-I components throughout the confluences. At regular finite
poles this gives finite sums of terminating Kamp\'e de F\'eriet blocks. On
the unreflected Kravchuk- and Meixner-II-like branches, the moving block is
obtained as a single coefficient extraction from its terminating Hahn
Kamp\'e de F\'eriet factor; the subsequent Charlier-II-like and
Laguerre-I-like limits preserve the finite-pole sectors and the
reconstructed sector at infinity separately. On the reflected branches,
separating the triangular inversion by its terminal pole yields ordinary
blockwise limits and explicit finite Lauricella--Horn formulas. The
Meixner-I-like-to-Laguerre-II-like arrow is governed by an exact finite
rational deformation whose pole sectors converge separately.

The suffix ``-like'' distinguishes these systems from the standard
multiple Hahn, Meixner, Kravchuk, and Charlier families. Every
construction reduces to the corresponding classical scalar family when
\(q=1\). For \(q>1\), the two Meixner-like systems are mutually distinct
under every permutation of their normalized rows, as are the two
Charlier-like systems; see
Propositions~\ref{prop:Meixner-families-distinct} and
\ref{prop:Charlier-families-distinct}. The present paper establishes the
weights, type-II polynomials, normalized type-I forms, and explicit
hypergeometric component formulas throughout the scheme. Recurrence
coefficients and their factorizations form a separate layer and are not
treated here. We consider the usual type-I/type-II problem for \(q\)
weights; the mixed-type Hahn-like setting with two independent numbers of
weights is deliberately kept separate.

Figure~\ref{fig:intro-diagram} summarizes the Askey-type scheme
established here. Solid boxes denote discrete systems and dashed boxes
their continuous limits. Reflection of the lattice is indicated
explicitly.

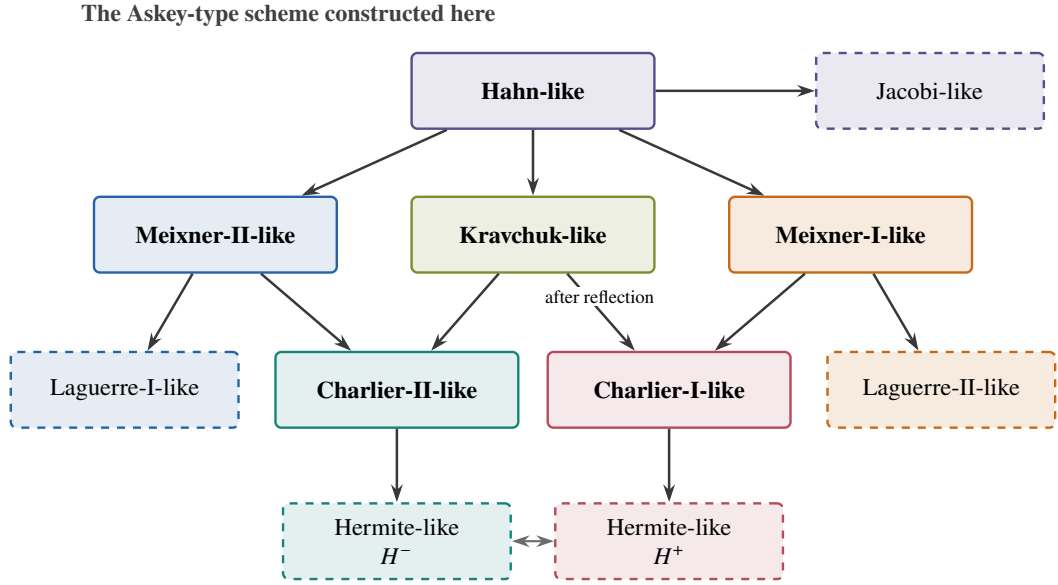
\begin{figure}[!htbp]
\centering
\begin{tikzpicture}[
  >={Stealth[length=2.35mm,width=1.55mm]},
  discrete/.style={
    line width=.9pt, rounded corners=2.2pt,
    align=center, font=\small\bfseries,
    text width=30mm, minimum height=10mm, inner sep=3pt
  },
  continuous/.style={
    line width=.9pt, dashed, rounded corners=2.2pt,
    align=center, font=\small,
    text width=28mm, minimum height=10mm, inner sep=3pt
  },
  hahn family/.style={draw=diagramHahn,fill=diagramHahn!13},
  k family/.style={draw=diagramK,fill=diagramK!14},
  mii family/.style={draw=diagramMII,fill=diagramMII!12},
  mi family/.style={draw=diagramMI,fill=diagramMI!14},
  cii family/.style={draw=diagramCII,fill=diagramCII!12},
  ci family/.style={draw=diagramCI,fill=diagramCI!12},
  discrete flow/.style={
    ->,line width=.95pt,draw=black!78
  },
  reflection/.style={<->,line width=.8pt,draw=black!58},
  panel title/.style={font=\small\bfseries,anchor=west,text=black!75}
]
\node[panel title] at (-6.1,6.75) {The Askey-type scheme constructed here};
\node[discrete,hahn family] (H) at (0,5.75) {Hahn-like};
\node[continuous,hahn family] (J) at (5.25,5.75) {Jacobi-like};
\node[discrete,mii family] (MII) at (-4.2,3.85)
  {Meixner-II-like};
\node[discrete,k family] (K) at (0,3.85)
  {Kravchuk-like};
\node[discrete,mi family] (MI) at (4.2,3.85)
  {Meixner-I-like};
\node[continuous,mii family] (LI) at (-5.4,1.8)
  {Laguerre-I-like};
\node[discrete,cii family] (CII) at (-1.8,1.8)
  {Charlier-II-like};
\node[discrete,ci family] (CI) at (1.8,1.8)
  {Charlier-I-like};
\node[continuous,mi family] (LII) at (5.4,1.8)
  {Laguerre-II-like};
\node[continuous,cii family] (Hm) at (-1.8,-.2)
  {Hermite-like\\\(H^{-}\)};
\node[continuous,ci family] (Hp) at (1.8,-.2)
  {Hermite-like\\\(H^{+}\)};

\draw[discrete flow] (H) -- (MII);
\draw[discrete flow] (H) -- (K);
\draw[discrete flow] (H) -- (MI);
\draw[discrete flow] (H) -- (J);
\draw[discrete flow] (MII) -- (LI);
\draw[discrete flow] (MII) -- (CII);
\draw[discrete flow] (K) -- (CII);
\draw[discrete flow] (K) --
 node[pos=.48,above=1.5pt,fill=white,inner sep=1pt,font=\scriptsize]
 {after reflection} (CI);
\draw[discrete flow] (MI) -- (CI);
\draw[discrete flow] (MI) -- (LII);
\draw[discrete flow] (CII) -- (Hm);
\draw[discrete flow] (CI) -- (Hp);
\draw[reflection] (Hm) -- (Hp);
\end{tikzpicture}
\caption{The Askey-type confluence scheme constructed in this paper. Solid
boxes denote discrete systems, dashed boxes continuous systems, and the
double arrow the reflection \(t\mapsto-t\). No interchangeability of
limits is asserted.}
\label{fig:intro-diagram}
\end{figure}
The names Charlier-I-like and Charlier-II-like refer to the limits from the
Meixner-I-like and Meixner-II-like families, respectively. The first is
described by convolutions of elementary sequences, whereas the second has
a multiple beta-integral representation. They coincide when \(q=1\), as
the two classical Meixner families do
\cite{ArvesuCoussementVanAssche2003}, but are different when \(q>1\); see
Proposition~\ref{prop:Charlier-families-distinct}.

The organization mirrors this construction.
Section~\ref{sec:preliminaries} fixes the notation and recalls Wolfs'
Jacobi-like weights. Sections~\ref{sec:Hahn-weights} and
\ref{sec:Hahn-forms} construct the Hahn-like ancestor and derive its
type-II and type-I formulas.
Section~\ref{sec:explicit-unreflected-B} supplies the common finite
reconstruction used to transport the type-I components through the
scheme. Section~\ref{sec:Kravchuk} then treats the fixed-lattice parameter
limit, Sections~\ref{sec:MII} and \ref{sec:MI} treat the two Meixner-like
limits and their Laguerre-like limits, and
Section~\ref{sec:Charlier-Hermite} treats the two Charlier-like systems
and their Hermite limits. Each family section states its own type-II
polynomial, type-I form, component specialization, confluence, and
uniqueness hypotheses.

\section{Notation and Jacobi-like weights}
\label{sec:preliminaries}

This section fixes the notation used throughout the paper and recalls the
normalized Jacobi-like weights that will be discretized. We also specify the
normalization used in the Hahn-to-Jacobi limit.

If \(F(z)=\sum_{k\ge0}f_kz^k\), then \([z^k]F(z)\coloneq f_k\).

\begin{definition}[Near-diagonal index]
A multi-index \(\mm\in\Nzero^q\) is near the diagonal if
\(\max_{h\in\{1,\ldots,q\}}m_h-
\min_{h\in\{1,\ldots,q\}}m_h\le1\).
Equivalently, \(\abs{m_h-m_j}\le1\) for all
\(h,j\in\{1,\ldots,q\}\).
\end{definition}

When \(\abs\mm\ge1\), this condition is also equivalent to
\begin{equation}
 m_h-\delta_{h,j}\ge r
 \quad\text{for every }h,
 \quad\text{every active }j,
 \quad\text{and every }r\in\Nzero\text{ with }0\le r<m_j.
 \label{eq:near-diagonal-proof-condition}
\end{equation}
Indeed, near-diagonality gives
\(m_h-\delta_{h,j}\ge m_j-1\ge r\). Conversely, if
\eqref{eq:near-diagonal-proof-condition} holds, choose \(j_*\) with
\(m_{j_*}=\max_hm_h\); it is active because \(\abs\mm\ge1\). Taking
\(r=m_{j_*}-1\) yields \(m_h\ge m_{j_*}-1\) for every \(h\), and hence
\(\max_hm_h-\min_hm_h\le1\).
Thus near-diagonality is exactly the combinatorial condition used in the
degree and divisibility arguments below, rather than an auxiliary
regularity assumption.

Whenever the type-I moment conditions are considered, we assume
\(\abs\mm\ge1\) and write \(n=\abs\mm-1\). For weights
\(w_1,\ldots,w_q\), a type-I linear form has the representation
\[
 \mathcal Q_{\mm}(x)=
 \sum_{\substack{j\in\{1,\ldots,q\}\\m_j>0}}
 B_{\mm,j}(x)w_j(x).
\]
Here \(\deg B_{\mm,j}<m_j\) for every
\(j\in\{1,\ldots,q\}\) with \(m_j>0\), and the moments of every order
\(r\in\Nzero\) with \(r<n\) vanish. In several sections we
first construct an explicit signed sequence or function satisfying these
moment conditions. We identify it with the type-I linear form only after
proving that it has the displayed representation. The representation is
unique precisely when the corresponding type-I moment matrix is
nonsingular.

Let \(\mathcal L_j\) denote the moment functional associated with
\(w_j\), put \(d=\abs\mm\), and let
\((\phi_r)_{r\ge0}\) be any monic degree-graded polynomial basis. The
type-I moment matrix at \(\mm\) is the \(d\times d\) matrix
\begin{equation}
 M_{\mm}\coloneq
 \bigl(\mathcal L_j[\phi_r\phi_s]\bigr)_{
  0\le r<d,\ (j,s):\,m_j>0,\ 0\le s<m_j}.
 \label{eq:type-I-moment-matrix-definition}
\end{equation}

\begin{definition}[Normal multi-index]
\label{def:normal-multi-index}
We call \(\mm\) \emph{normal} for
\((\mathcal L_1,\ldots,\mathcal L_q)\) if \(M_{\mm}\) is nonsingular.
This definition is independent of the chosen monic degree-graded basis.
Equivalently, the normalized type-I problem has a unique solution and the
monic type-II problem has a unique solution of degree \(d\), because their
coefficient matrices are \(M_{\mm}\) and \(M_{\mm}^{\mathsf T}\),
respectively. Thus normality is a simultaneous type-I and type-II
property. Multiplication of any row functional by a nonzero constant does
not affect it.
\end{definition}

For a system supported on \(\{0,\ldots,N\}\), we allow
\(d=\abs\mm=N+1\) in Definition~\ref{def:normal-multi-index}. At this
endpoint the monic nodal polynomial
\begin{equation}
 \Pi_{N+1}(k)=\prod_{t=0}^{N}(k-t)
 \label{eq:finite-lattice-nodal-polynomial}
\end{equation}
is a type-II solution and, whenever \(M_{\mm}\) is nonsingular, it is the
unique monic type-II solution. Since \(\Pi_{N+1}(0)=0\), the convention
\(A(0)=1\) used for the explicit finite-lattice type-II formulas is
reserved for \(d\le N\). Type-I formulas and normality statements may
still include \(d=N+1\).

Each theorem states its normalization explicitly. For the Hahn-like,
Kravchuk-like, Meixner-II-like, and Charlier-II-like families
we initially use the normalization inherited from the factor
\((r+\beta_j)^{-1}\). For the reflected Meixner-I-like and
Charlier-I-like families, the factorial moment of order \(n\) is
normalized to one. The additional rescaling used in the
Charlier-II-like-to-Hermite limit is stated where it is introduced.

\subsection{Factorials and hypergeometric series}

We use rising and falling factorials throughout. We also specify the vector
conventions and the meaning of the terminating hypergeometric series used
below.

For \(r\in\Nzero\), we write \((a)_r=\Gamma(a+r)/\Gamma(a)\) and
\(\fall{x}{r}=x(x-1)\cdots(x-r+1)=(-1)^r(-x)_r\),
with either expression equal to one for \(r=0\). Products over empty
parameter strings are one. For vectors, Pochhammer symbols are multiplied
componentwise. Throughout, \(\A=(A_1,\ldots,A_q)\),
\(\bbeta=(\beta_1,\ldots,\beta_q)\), \(\mm=(m_1,\ldots,m_q)\), and
\(\abs{\mm}=\sum_{h=1}^qm_h\).
We also write
\(\boldsymbol\delta=(\delta_1,\ldots,\delta_q)\) whenever the
\(\delta_h\) parameters are present.
For each \(r\ge1\), \(\one_r=(1,\ldots,1)\in\mathbb R^r\), and
\(\ee_j\) is the \(j\)-th coordinate vector.  When the dimension of
\(\one\) or \(\ee_j\) is omitted, it is determined by the surrounding
vector expression. For example, \((\A)_r=\prod_{h=1}^q(A_h)_r\) and
\((\bbeta+\mm)_r=\prod_{h=1}^q(\beta_h+m_h)_r\). We also use the
truncated strings \(\A_{<q}=(A_1,\ldots,A_{q-1})\),
\(\bbeta_{<q}=(\beta_1,\ldots,\beta_{q-1})\), and
\(\mm_{<q}=(m_1,\ldots,m_{q-1})\), together with the componentwise
conventions
\((\bbeta)_{\mm+n}\coloneq\prod_{h=1}^q(\beta_h)_{m_h+n}\) and
\((\bbeta_{<q})_{\mm_{<q}+n}
\coloneq\prod_{h=1}^{q-1}(\beta_h)_{m_h+n}\).
The generalized hypergeometric notation follows \cite{DLMF}:
\[
 \pFq{p}{s}{\boldsymbol a}{\boldsymbol b}{z}
 =\sum_{r=0}^{\infty}
 \frac{(\boldsymbol a)_r}{(\boldsymbol b)_r}
 \frac{z^r}{r!}.
\]
All finite-lattice hypergeometric series in this paper terminate before a
lower parameter can create a convergence issue. A nonterminating
\({}_pF_{p-1}\) is first defined by its series for \(\abs z<1\) and then
by its principal analytic continuation to
\(\mathbb C\setminus[1,\infty)\). On the nonpositive real axis this is the
continuation selected by the Euler-type integrals used below. Thus the
Meixner-II-like generating-function identities are first
obtained near \(z=1\) and then continued throughout the unit disk, while
the Laguerre-I Laplace formulas are understood at their nonpositive
hypergeometric argument. In a terminating expression, the finite summation
range is fixed by its distinguished nonpositive-integer upper parameter
before common upper and lower parameters are simplified. If the resulting
finite expression has a removable singularity as a function of the
parameters, its value is understood by continuation of the complete
expression.

We use twice the following elementary estimate for alternating differences.

\begin{lemma}[Alternating-difference estimate]
\label{lem:alternating-Pochhammer-estimate}
Let \(m\in\Nzero\) and \(p\in\N\), and let \(a_i,b_i\in\mathbb C\) for
every \(i\in\{1,\ldots,p\}\). For \(s\in\{0,\ldots,m\}\), set
\(F_s(z)=\prod_{i=1}^p(z+a_i)_s/\prod_{i=1}^p(z+b_i)_s\). Then
\begin{equation}
 \sum_{s=0}^m(-1)^s\binom msF_s(z)=\mathrm O(z^{-m})
 \qquad\text{as }\abs z\to\infty.
 \label{eq:alternating-Pochhammer-estimate}
\end{equation}
\end{lemma}

\begin{proof}
Choose \(R>0\) so large that the finitely many functions \(F_s\),
\(0\le s\le m\), are holomorphic and nonzero on \(\abs z>R\), and
there take the branch \(\log F_s\to0\) at infinity. Expanding each
\(\log(1+c/z)\) then gives, absolutely and locally uniformly,
\begin{equation}
 \log F_s(z)=
 \sum_{v\ge1}\frac{(-1)^{v+1}}{v z^v}
 \sum_{\substack{t\in\Nzero\\t<s}}
 \left[\sum_{i=1}^p(a_i+t)^v-\sum_{i=1}^p(b_i+t)^v\right].
 \label{eq:log-Pochhammer-quotient}
\end{equation}
The terms of degree \(v\) in \(t\) cancel. Hence the coefficient of
\(z^{-v}\) in \(\log F_s(z)\) is a polynomial in \(s\) of degree at most
\(v\). The coefficient of \(z^{-u}\) after exponentiation is a sum of
products whose total weighted order is \(u\), and therefore still has degree
at most \(u\) in \(s\). Its \(m\)-th alternating difference vanishes for
every \(u<m\). Since the Laurent series converges for \(\abs z>R\), this
coefficient cancellation proves \eqref{eq:alternating-Pochhammer-estimate}.
\end{proof}

We also use the classical bivariate Kamp\'e de F\'eriet series; see
\cite[Chapter~1]{SrivastavaKarlsson1985} and
\cite{BranquinhoDiazFoulquieManas2025Classical}. If the six parameter
strings have respective lengths \(p,q,s,\ell,m,u\), write
\begin{equation}
 F_{\ell:m;u}^{p:q;s}
 \left(
 \begin{matrix}
  \boldsymbol a:\boldsymbol b;\boldsymbol c\\
  \boldsymbol d:\boldsymbol e;\boldsymbol f
 \end{matrix}
 \middle|x,y\right)
 \coloneq\sum_{r,t\ge0}
 \frac{(\boldsymbol a)_{r+t}
       (\boldsymbol b)_r(\boldsymbol c)_t}
      {(\boldsymbol d)_{r+t}
       (\boldsymbol e)_r(\boldsymbol f)_t}
 \frac{x^r}{r!}\frac{y^t}{t!}.
 \label{eq:KdF-definition}
\end{equation}
In the convention fixed by \eqref{eq:KdF-definition}, the first pair of
parameter strings is coupled to \(r+t\); the second
depends only on \(r\), and the third only on \(t\). Empty strings are
allowed. Every Kamp\'e de F\'eriet series used below terminates because one
coupled upper parameter is a nonpositive integer. Removable exceptional
parameter values are understood by continuation of the complete
polynomial.

We need one particularly simple multivariable extension. For \(R\in\N\),
set
\begin{equation}
 \mathfrak K_R
 \left(A;\boldsymbol b;\boldsymbol x\right)
 \coloneq\sum_{r_1,\ldots,r_R\ge0}
 (A)_{r_1+\cdots+r_R}
 \prod_{v=1}^R\frac{(b_v)_{r_v}}{r_v!}x_v^{r_v}.
 \label{eq:multiple-KdF-simple}
\end{equation}
When \(A\) is a nonpositive integer, this is a terminating multiple
Kamp\'e de F\'eriet polynomial,
\(F_{0:0;\ldots;0}^{1:1;\ldots;1}\). Thus
\(\mathfrak K_1(A;b;x)={}_2F_0(A,b;-;x)\), while
\(\mathfrak K_2\) is the classical bivariate
\(F_{0:0;0}^{1:1;1}\). We use the notation only in terminating cases.

The confluent Lauricella function needed below is
\begin{equation}
 \Phi_2^{(R)}
 (b_1,\ldots,b_R;c;x_1,\ldots,x_R)
 \coloneq\sum_{r_1,\ldots,r_R\ge0}
 \frac{\prod_{v=1}^R(b_v)_{r_v}}
 {(c)_{r_1+\cdots+r_R}}
 \prod_{v=1}^R\frac{x_v^{r_v}}{r_v!}.
 \label{eq:Phi2-definition}
\end{equation}
The normalization in \eqref{eq:Phi2-definition} is the one used in all
Laguerre-II-like component formulas below.

For later coefficient formulas we set
\begin{equation}
 b_{a,b}(y)\coloneq
 \frac{\Gamma(a+b)}{\Gamma(a)\Gamma(b)}
 y^{a-1}(1-y)^{b-1},\quad 0<y<1,
 \label{eq:beta-kernel}
\end{equation}
for \(a,b>0\), and
\begin{equation}
 g_{\rho,d}(x)\coloneq\frac{\rho^d}{\Gamma(d)}x^{d-1}\e^{-\rho x}
 \mathbf1_{\{x>0\}},\qquad
 p_a(k)\coloneq\e^{-a}\frac{a^k}{k!},\quad k\in\Nzero,
 \label{eq:elementary-positive-kernels}
\end{equation}
\begin{equation}
 r_{d,c}(k)\coloneq(1-c)^d\frac{(d)_k}{k!}c^k,\qquad k\in\Nzero,
 \label{eq:elementary-discrete-weights}
\end{equation}
for \(a,\rho,d>0\) and \(0<c<1\). The beta density
\eqref{eq:beta-kernel} and the gamma density in
\eqref{eq:elementary-positive-kernels} are normalized. The Poisson and
negative-binomial sequences in \eqref{eq:elementary-positive-kernels} and
\eqref{eq:elementary-discrete-weights} have generating functions
\(\e^{a(z-1)}\) and \(((1-c)/(1-cz))^d\), respectively. These four
normalized kernels will be used throughout the paper.
The elementary product linearization
\begin{equation}
 \fall{x}{r}\fall{x}{\ell}
 =\sum_{u=0}^{\min\{r,\ell\}}
 \binom ru\binom\ell u u!\,
 \fall{x}{r+\ell-u}
 \label{eq:falling-linearization}
\end{equation}
will be used repeatedly. We work with multi-indices whose entries differ
by at most one.

\subsection{The normalized Jacobi-like weights}

The finite weights will be built from the following continuous functions.
We normalize them by \(\int_0^1\omega_j(x)\dd x=1\) and record their moments.

Fix parameters
\begin{equation}
 0<A_h<\beta_h,
 \qquad h\in\{1,\ldots,q\}.
 \label{eq:admissible-A-beta}
\end{equation}
We use \eqref{eq:admissible-A-beta} only as a sufficient,
\(N\)-independent construction domain; for \(q\ge2\) it makes every beta
factor in \eqref{eq:beta-product-integral} positive. It is not asserted to
characterize positivity of the finite sequences \(v_{j,N}\): after
continuation, the admissible domain can be strictly larger and may depend
on \(N\).\footnote{For example, let \(q=2\),
\(\A=(3,\frac34)\), and \(\bbeta=(2,\frac52)\), so that
\(A_1>\beta_1\). The two continued rows are the Bernstein transforms of
\(b_{\frac34,\frac74}\) and
\(\frac38b_{\frac74,\frac74}+\frac58b_{\frac34,\frac{11}4}\), respectively,
as follows by comparing their moments with
\eqref{eq:continuous-row-moments}; hence both rows are strictly positive
for every \(N\). Dependence on \(N\) is genuine: for
\(\A=(\frac14,\frac52)\) and \(\bbeta=(\frac34,\frac34)\), direct
evaluation of \eqref{eq:Hahn-like-mass} gives two strictly positive rows
for \(0\le N\le3\), whereas
\(v_{1,4}(3)=v_{2,4}(3)=-96/16093\).}
For \(j\in\{1,\ldots,q\}\), define the normalized positive weight
\(\omega_j\) on \((0,1)\) by
\begin{equation}
 \int_0^1 f(x)\omega_j(x)\dd x
 =\int_{(0,1)^q}f(y_1\cdots y_q)
 \prod_{h=1}^q
 b_{A_h,\,\beta_h+\delta_{h,j}-A_h}(y_h)\dd\boldsymbol y
 \label{eq:beta-product-integral}
\end{equation}
for every continuous \(f\). Taking \(f(x)=x^r\) gives
\begin{equation}
 \int_0^1x^r\omega_j(x)\dd x
 =\frac{(\A)_r}{(\bbeta+\ee_j)_r}
 =\frac{\beta_j}{r+\beta_j}
 \frac{(\A)_r}{(\bbeta)_r},
 \qquad r\in\Nzero.
 \label{eq:continuous-row-moments}
\end{equation}
Thus \(\int_0^1\omega_j(x)\dd x=1\). These are normalized versions of the
Jacobi-like weights of Wolfs
\cite{Wolfs2024}. Indeed, if
\(\mathcal M[w_j](s)=
\Gamma(s\one+\boldsymbol a)/
\Gamma(s\one+\boldsymbol b+\ee_j)\)
and \(A_h=a_h+\alpha+1\), \(\beta_h=b_h+\alpha+1\), then \(\omega_j\) is a
positive constant multiple of \(x^\alpha w_j(x)\).

For every \(j\in\{1,\ldots,q\}\), the type-I calculations use the weight
normalization
\begin{equation}
 \widetilde\omega_j=\frac{1}{\beta_j}\omega_j,
 \qquad
 \int_0^1x^r\widetilde\omega_j(x)\dd x
 =\frac{h(r)}{r+\beta_j},
 \qquad
 h(r)=\frac{(\A)_r}{(\bbeta)_r},\qquad r\in\Nzero.
 \label{eq:canonical-continuous-rows}
\end{equation}
Multiplying a weight by a positive constant changes neither its multiple
orthogonal polynomial nor normality. The type-I polynomials are rescaled
inversely.

For a near-diagonal \(\mm\), put \(d=\abs\mm\). The continuous
Jacobi-like polynomial normalized by its value at zero is
\begin{equation}
A^{\mathrm J}_{\mm}(x)
 =\pFq{q+1}{q}{-d,\bbeta+\mm}{\A}{x}.
 \label{eq:continuous-A}
\end{equation}
The polynomial \eqref{eq:continuous-A} is Wolfs's Jacobi-like type-II polynomial
\cite[Theorem~2.13]{Wolfs2024}, written in the present normalization.
Its orthogonality is used here as limiting data and follows from
\cite[Theorem~2.13]{Wolfs2024}.

\subsection{Normalization at the Jacobi-like limit}

We also fix the scale of the continuous function used in the
Hahn-to-Jacobi limit. Its moments determine the normalization.

Write \(\mathcal M[f](s)=\int_0^1x^{s-1}f(x)\dd x\).
Let \(\mm\) be near the diagonal, and put \(d=\abs\mm\) and \(n=d-1\).
The normalization of the known Jacobi-like type-I linear form is fixed by
\begin{equation}
 \mathcal M[\mathcal H^{\mathrm J}_{\mm}](s)
 =-(1-s)_n\prod_{h=1}^q
 \frac{\Gamma(\beta_h)\Gamma(A_h+s-1)}
 {\Gamma(A_h)\Gamma(\beta_h+m_h+s-1)}.
 \label{eq:Jacobi-B-Mellin}
\end{equation}
and hence by the moment identity
\begin{equation}
 \int_0^1x^r\mathcal H^{\mathrm J}_{\mm}(x)\dd x
 =-\frac{(\A)_r}{(\bbeta)_r}
 \frac{(-r)_n}{\prod_{h=1}^q(r+\beta_h)_{m_h}}.
 \label{eq:Jacobi-B-moments}
\end{equation}
Its moments of every order \(r\in\Nzero\) with \(r<n\) vanish, and its moment of order
\(n\) is nonzero. These formulas identify the limiting normalization; they
do not define a new
construction of the continuous type-I theory; see
\cite[Theorem~2.6 and Corollary~2.8]{Wolfs2024}.

\section{The Hahn-like finite-lattice weights}
\label{sec:Hahn-weights}

We now construct the finite Hahn-like weights. We establish positivity and
factorial moments, identify their reduction to the scalar Hahn weight, and
prove convergence
to the Jacobi-like weights.

\subsection{Definition, positivity, and factorial moments}

The finite weights are obtained by applying the Bernstein kernel to each
continuous weight. The Bernstein representation gives positivity, and the
kernel identity gives the factorial moments.

For fixed \(N\), the Bernstein basis on \([0,1]\) is
\(b_{N,k}(x)=\binom Nkx^k(1-x)^{N-k}\),
\(k\in\{0,\ldots,N\}\).
Its elements are nonnegative and sum to one. If \(\mu\) is a finite measure
on \([0,1]\), its Bernstein transform at level \(N\) is the lattice measure
whose weights are
\begin{equation}
 (\mathscr B_N^*\mu)(k)=\int_0^1b_{N,k}(x)\dd\mu(x).
 \label{eq:Bernstein-transform-definition}
\end{equation}
The ordinary Bernstein operator and its duality with the transformation
\eqref{eq:Bernstein-transform-definition} are
\begin{equation}
 (\mathscr B_NQ)(x)=\sum_{k=0}^NQ(k)b_{N,k}(x),\qquad
 \sum_{k=0}^NQ(k)(\mathscr B_N^*\mu)(k)
 =\int_0^1(\mathscr B_NQ)(x)\dd\mu(x).
 \label{eq:Bernstein-operator}
\end{equation}
Thus the transform of measures and the Bernstein operator are adjoint
\cite{Lorentz1986}.  The two kernel identities needed below are
\begin{equation}
 \sum_{k=0}^Nb_{N,k}(x)=1,
 \qquad
 \sum_{k=0}^N\fall{k}{r}b_{N,k}(x)=\fall{N}{r}x^r,
 \qquad r\in\{0,\ldots,N\}.
 \label{eq:Bernstein-kernel-identities}
\end{equation}
The identities \eqref{eq:Bernstein-kernel-identities} give preservation of
positivity and total mass and the factorial moments below. More explicitly,
\(\sum_{k=0}^N\fall{k}{r}(\mathscr B_N^*\mu)(k)
=\fall{N}{r}\int_0^1x^r\dd\mu(x)\) for
\(r\in\{0,\ldots,N\}\).
For \(Q_N(k)=f(k/N)\), Bernstein's approximation theorem gives
\(\mathscr B_NQ_N\xrightarrow[N\to\infty]{} f\) uniformly for every
\(f\in C[0,1]\).  By the
duality above, this is precisely the weak convergence
\(
 \sum_{k=0}^N(\mathscr B_N^*\mu)(k)\,\delta_{k/N}
 \xrightarrow[N\to\infty]{\mathrm w} \mu
\).

Let \(N\in\Nzero\). We call \((v_{1,N},\ldots,v_{q,N})\) the
\emph{Hahn-like system of weights}, where
\begin{equation}
 v_{j,N}(k)\coloneq\binom Nk\int_0^1
 x^k(1-x)^{N-k}\omega_j(x)\dd x.
 \label{eq:Bernstein-lift}
\end{equation}
Thus \(v_{j,N}=\mathscr B_N^*(\omega_j(x)\dd x)\).  These Bernstein
transforms are positive and have explicit formulas and
moments.

\begin{theorem}[Hahn-like weights]
\label{thm:Hahn-like-rows}
Let \(N\in\Nzero\). Under \eqref{eq:admissible-A-beta}, for every
\(j\in\{1,\ldots,q\}\), the values \(v_{j,N}(k)\) are strictly positive
for all \(k\in\{0,\ldots,N\}\), and
\begin{align}
 \mathcal V_{j,N}(z)
 &=\sum_{k=0}^Nv_{j,N}(k)z^k
 =\pFq{q+1}{q}{-N,\A}{\bbeta+\ee_j}{1-z},
 \label{eq:Hahn-like-pgf}\\
 v_{j,N}(k)
 &=\binom Nk\frac{(\A)_k}{(\bbeta+\ee_j)_k}
 \pFq{q+1}{q}
 {k-N,\A+k}{\bbeta+\ee_j+k}{1},
 \label{eq:Hahn-like-mass}\\
 \sum_{k=0}^N\fall{k}{r}v_{j,N}(k)
 &=\fall{N}{r}\frac{(\A)_r}{(\bbeta+\ee_j)_r},
 \qquad r\in\{0,\ldots,N\}.
 \label{eq:Hahn-like-factorial-moments}
\end{align}
\end{theorem}

All hypergeometric series in the theorem terminate.

\begin{proof}
Summing \eqref{eq:Bernstein-lift} against \(z^k\) and applying the binomial
theorem gives
\(\mathcal V_{j,N}(z)=
\int_0^1(1-(1-z)x)^N\omega_j(x)\dd x\).
Expanding the last power and using \eqref{eq:continuous-row-moments} proves
\eqref{eq:Hahn-like-pgf}. Expanding \((1-x)^{N-k}\) in
\eqref{eq:Bernstein-lift} proves \eqref{eq:Hahn-like-mass}. Finally,
the elementary binomial identity
\(\sum_{k=0}^N\fall{k}{r}\binom Nkx^k(1-x)^{N-k}
=\fall{N}{r}x^r\)
together with \eqref{eq:continuous-row-moments} gives
\eqref{eq:Hahn-like-factorial-moments}. Strict positivity follows directly
from the integral in \eqref{eq:Bernstein-lift}.
\end{proof}

For \(j\in\{1,\ldots,q\}\), define the rescaled weights
\begin{equation}
 \widetilde v_{j,N}=\frac1{\beta_j}v_{j,N},
 \qquad
 h_N(r)=\fall{N}{r}\frac{(\A)_r}{(\bbeta)_r}.
 \label{eq:canonical-finite-row}
\end{equation}
Then, for every \(j\in\{1,\ldots,q\}\) and
\(r\in\{0,\ldots,N\}\),
\begin{equation}
 \sum_{k=0}^N\fall{k}{r}\widetilde v_{j,N}(k)
 =\frac{h_N(r)}{r+\beta_j}.
 \label{eq:resolvent-factorial-moment}
\end{equation}
If \(v_{0,N}\) denotes the weight obtained from
\eqref{eq:beta-product-integral} without the extra unit shift, then, for
every \(j\in\{1,\ldots,q\}\), the generating functions satisfy the exact
differential relation
\begin{equation}
 \bigl((z-1)\partial_z+\beta_j\bigr)\mathcal V_{j,N}(z)
 =\beta_j\mathcal V_{0,N}(z).
 \label{eq:finite-resolvent-relation}
\end{equation}
Equivalently, for \(j\in\{1,\ldots,q\}\) and
\(k\in\{0,\ldots,N\}\),
\((k+\beta_j)v_{j,N}(k)-(k+1)v_{j,N}(k+1)
=\beta_jv_{0,N}(k)\), where \(v_{j,N}(N+1)=0\).
Thus \eqref{eq:finite-resolvent-relation} is the finite counterpart of the simple pole
\((r+\beta_j)^{-1}\) in \eqref{eq:canonical-continuous-rows}.

\subsection{Reduction to the classical Hahn weight}

For one weight, the construction reduces to the classical Hahn weight
throughout its standard connected positivity region.

\begin{corollary}[The case \(q=1\)]
\label{cor:q1-Hahn-weight}
Let \(N\in\Nzero\), \(q=1\), \(A=A_1>0\), \(\beta=\beta_1\), and
\(\gamma=\beta-A>-1\). Then, for every
\(k\in\{0,\ldots,N\}\),
\begin{equation}
 v_{1,N}(k)=
 \frac{N!}{(\beta+1)_N}
 \frac{(A)_k}{k!}
 \frac{(\gamma+1)_{N-k}}{(N-k)!}.
 \label{eq:exact-Hahn-weight}
\end{equation}
\end{corollary}

\begin{proof}
Apply Chu--Vandermonde to the \({}_2F_1(1)\) in
\eqref{eq:Hahn-like-mass}. The normalizing identity is the same summation
at \(k=0\).
\end{proof}

Apart from its displayed normalization, \eqref{eq:exact-Hahn-weight} is
the classical Hahn weight with parameters \(\alpha_{\mathrm H}=A-1\) and
\(\beta_{\mathrm H}=\gamma\).

\subsection{Confluence to the Jacobi-like weights}

The finite measures constructed in Theorem~\ref{thm:Hahn-like-rows}
recover the continuous weights as \(N\to\infty\). We prove weak
convergence, a local limit for the lattice values, and convergence of
every fixed moment.

Assume \eqref{eq:admissible-A-beta}. For every \(N\in\N\) and
\(j\in\{1,\ldots,q\}\), define
\begin{equation}
 \mu_{j,N}\coloneq\sum_{k=0}^Nv_{j,N}(k)\,\delta_{k/N}.
 \label{eq:Hahn-discrete-measure}
\end{equation}

\begin{theorem}[Bernstein confluence]
\label{thm:Hahn-to-Jacobi-rows}
Under \eqref{eq:admissible-A-beta}, for every
\(j\in\{1,\ldots,q\}\), as \(N\to\infty\), the measures
\eqref{eq:Hahn-discrete-measure} converge weakly to
\(\omega_j(x)\dd x\). Locally uniformly for \(x\) in
compact subsets of \((0,1)\),
\begin{equation}
 (N+1)v_{j,N}(\lfloor Nx\rfloor)
 \xrightarrow[N\to\infty]{}\omega_j(x).
 \label{eq:local-Hahn-Jacobi-limit}
\end{equation}
For every fixed \(r\in\Nzero\), every \(N\ge r\), and every
\(j\in\{1,\ldots,q\}\),
\[
 \frac{1}{\fall{N}{r}}
 \sum_{k=0}^N\fall{k}{r}v_{j,N}(k)
 =\int_0^1x^r\omega_j(x)\dd x.
\]
\end{theorem}

\begin{proof}
For \(f\in C[0,1]\), the definition \eqref{eq:Bernstein-lift} gives
\[
 \int f\dd\mu_{j,N}
 =\int_0^1\left[\sum_{k=0}^Nf(k/N)\binom Nk
 x^k(1-x)^{N-k}\right]\omega_j(x)\dd x.
\]
The expression in brackets is the Bernstein polynomial of \(f\), hence it
converges uniformly to \(f\) as \(N\to\infty\). This proves weak convergence. The density
\(\omega_j\) is continuous in the interior: for \(q=1\) this follows from
the beta formula, while for \(q\ge2\) the change \(x=\e^{-t}\) identifies
its density, up to the continuous Jacobian, with a convolution of
integrable densities on \((0,\infty)\). Moreover,
\((N+1)\binom Nkx^k(1-x)^{N-k}\) is the normalized beta kernel with parameters
\(k+1,N-k+1\). Equation \eqref{eq:Bernstein-lift} is therefore an
approximate-identity representation, which proves
\eqref{eq:local-Hahn-Jacobi-limit}. The moment assertion is immediate
from \eqref{eq:Hahn-like-factorial-moments}.
\end{proof}

\section{Explicit Hahn-like type-II and type-I formulas}
\label{sec:Hahn-forms}

This section gives explicit formulas for the near-diagonal
multiple-orthogonality problem. We derive the type-II polynomial \(A\), a
normalized signed sequence satisfying the type-I moment conditions, and,
when the relevant moment matrix is nonsingular, its type-I polynomials.

\subsection{The type-II polynomial and the inverse Bernstein relation}

We now derive the finite polynomial \(A\). The Bernstein operator
\(\mathscr B_N\) defined in \eqref{eq:Bernstein-operator} turns it into the
corresponding Jacobi-like polynomial.  On normalized falling factorials it
satisfies
\begin{equation}
 \mathscr B_N\!\left[\frac{\fall{k}{\ell}}{\fall{N}{\ell}}\right](x)
 =x^\ell.
 \label{eq:Bernstein-factorial-basis}
\end{equation}

The Bernstein identity above gives the Hahn-like polynomial and its
orthogonality.

\begin{theorem}[Hahn-like polynomial]
\label{thm:Hahn-A}
Let \(\mm\) be near the diagonal and let \(d=\abs\mm\le N\). The polynomial
\begin{equation}
 A^{\mathrm H}_{\mm,N}(k)=
 \sum_{\ell=0}^d
 \frac{(-d)_\ell}{\ell!}
 \frac{(\bbeta+\mm)_\ell}{(\A)_\ell}
 \frac{\fall{k}{\ell}}{\fall{N}{\ell}},
 \label{eq:Hahn-A-factorial}
\end{equation}
can equivalently be written as
\(A^{\mathrm H}_{\mm,N}(k)=
\pFq{q+2}{q+1}{-d,-k,\bbeta+\mm}{-N,\A}{1}\).
It has degree \(d\), satisfies \(A^{\mathrm H}_{\mm,N}(0)=1\), and, for
every \(j\in\{1,\ldots,q\}\) with \(m_j>0\),
\begin{equation}
 \sum_{k=0}^N\fall{k}{r}
 A^{\mathrm H}_{\mm,N}(k)v_{j,N}(k)=0,
 \qquad r\in\{0,\ldots,m_j-1\}.
 \label{eq:Hahn-A-orthogonality}
\end{equation}
Moreover,
\begin{equation}
 \mathscr B_NA^{\mathrm H}_{\mm,N}=A^{\mathrm J}_{\mm}.
 \label{eq:inverse-Bernstein-identity}
\end{equation}
Its monic normalization is
\begin{equation}
 \widehat A^{\mathrm H}_{\mm,N}(k)
 =(-1)^d\fall{N}{d}
 \frac{(\A)_d}{(\bbeta+\mm)_d}
 A^{\mathrm H}_{\mm,N}(k).
 \label{eq:Hahn-monic-A}
\end{equation}
\end{theorem}

\begin{proof}
The Bernstein identity, the value at zero, and the leading coefficient
follow immediately from \eqref{eq:Bernstein-factorial-basis}. It remains
to prove \eqref{eq:Hahn-A-orthogonality}. Fix
\(j\in\{1,\ldots,q\}\) with \(m_j>0\) and
\(r\in\Nzero\) with \(r<m_j\), and write
\(b_h=\beta_h+\delta_{h,j}\),
\(M_h=m_h-\delta_{h,j}\), and
\(C_{j,\mm}=\prod_{h=1}^q(b_h)_{M_h}\).
Near-diagonality gives \(M_h\ge r\) for every
\(h\in\{1,\ldots,q\}\). As recorded in
\eqref{eq:near-diagonal-proof-condition}, these inequalities, for all
active \(j\) and all \(0\le r<m_j\), are equivalent to
near-diagonality. Thus this is exactly the combinatorial hypothesis used
in the orthogonality proof. Substitution of
\eqref{eq:Hahn-A-factorial},
\eqref{eq:Hahn-like-factorial-moments}, and
\eqref{eq:falling-linearization} gives the following finite calculation.
We use \eqref{eq:Hahn-like-factorial-moments} also for orders exceeding
\(N\), where both sides are zero by the falling-factorial convention.
Denote the left-hand side of \eqref{eq:Hahn-A-orthogonality} by
\(I_{j,r}\).  After putting \(s=r-u\) and interchanging the two finite
sums, one obtains
\begin{align*}
 I_{j,r}
 &=\sum_{\ell=0}^d\frac{(-d)_\ell}{\ell!}
   \frac{(\bbeta+\mm)_\ell}{(\A)_\ell\fall{N}{\ell}}
   \sum_{u=0}^{\min\{r,\ell\}}\binom ru\binom\ell u u!\,
   \fall{N}{r+\ell-u}
   \frac{(\A)_{r+\ell-u}}{\prod_h(b_h)_{r+\ell-u}}\\
 &=\sum_{s=0}^r\binom rs\sum_{\ell=0}^d(-1)^\ell\binom d\ell
   \fall{\ell}{r-s}\frac{\fall{N}{\ell+s}}{\fall{N}{\ell}}
   \frac{(\A)_{\ell+s}}{(\A)_\ell}
   \prod_h\frac{(b_h+M_h)_\ell}{(b_h)_{\ell+s}}\\
 &=\frac1{C_{j,\mm}}\sum_{\ell=0}^d(-1)^\ell\binom d\ell
   \sum_{s=0}^r(-1)^s\binom rs\fall{\ell}{r-s}(\ell-N)_s
   \prod_h(\ell+A_h)_s(\ell+b_h+s)_{M_h-s}.
\end{align*}
Here the terms with \(r-s>\ell\) vanish, and the last equality uses
\begin{equation}
 \frac{\fall{N}{\ell+s}}{\fall{N}{\ell}}=(-1)^s(\ell-N)_s,
 \qquad
 \frac{(\A)_{\ell+s}}{(\A)_\ell}=(\ell+\A)_s,
 \qquad
 \frac{(b_h+M_h)_\ell}{(b_h)_{\ell+s}}
 =\frac{(\ell+b_h+s)_{M_h-s}}{(b_h)_{M_h}}.
 \label{eq:Hahn-A-cancellations}
\end{equation}
The last quotient is legitimate because \(M_h\ge r\ge s\).  Thus the
inner sum is the polynomial
\begin{equation}
 R_{j,r,N}(z)=\sum_{s=0}^r(-1)^s\binom rs\fall{z}{r-s}(z-N)_s
 \prod_{h=1}^q (z+A_h)_s(z+b_h+s)_{M_h-s}.
 \label{eq:finite-difference-kernel-polynomial}
\end{equation}
Consequently,
\begin{equation}
 I_{j,r}
 =\frac1{C_{j,\mm}}\sum_{\ell=0}^d(-1)^\ell\binom d\ell
 R_{j,r,N}(\ell)
 =\frac{(-1)^d}{C_{j,\mm}}\Delta^dR_{j,r,N}(0),
 \label{eq:exact-Hahn-A-difference}
\end{equation}
where \(\Delta f(z)=f(z+1)-f(z)\).  Factoring
\(\fall{z}{r}\prod_h(z+b_h)_{M_h}\) from
\eqref{eq:finite-difference-kernel-polynomial} gives, away from the
removable exceptional values,
\begin{equation}
 R_{j,r,N}(z)=\fall{z}{r}
 \prod_{h=1}^q(z+b_h)_{M_h}
 \,\pFq{q+2}{q+1}
 {-r,z-N,z+\A}{z-r+1,z+\bbeta+\ee_j}{1}.
 \label{eq:finite-difference-kernel}
\end{equation}
Indeed,
\((-1)^s\binom rs=(-r)_s/s!\),
\(\fall{z}{r-s}/\fall{z}{r}=1/(z-r+1)_s\), and
\((z+b_h+s)_{M_h-s}/(z+b_h)_{M_h}=1/(z+b_h)_s\).
At an exceptional value, \eqref{eq:finite-difference-kernel} is understood
through the cancellation-free polynomial
\eqref{eq:finite-difference-kernel-polynomial}; in particular, this removes
the apparent singularity when an integer \(z<r\) occurs in
\((z-r+1)_s\).

We show that its degree is at most \(d-1\). Factor it as
\[
 R_{j,r,N}(z)=\fall{z}{r}\prod_{h=1}^q(z+b_h)_{M_h}
 \sum_{s=0}^r(-1)^s\binom rsG_s(z),
\]
where \(G_s(z)=(z-N)_s\prod_h(z+A_h)_s/
((z-r+1)_s\prod_h(z+b_h)_s)\).
Apply Lemma~\ref{lem:alternating-Pochhammer-estimate} with \(m=r\), numerator
parameters \((-N,A_1,\ldots,A_q)\), and denominator parameters
\((1-r,b_1,\ldots,b_q)\). It gives
\(\sum_{s=0}^r(-1)^s\binom rsG_s(z)
=\mathrm O(z^{-r})\) as \(|z|\to\infty\).
Since the prefactor has degree
\(r+\sum_hM_h=r+d-1\), we obtain
\(\deg R_{j,r,N}\le d-1\). Equation
\eqref{eq:exact-Hahn-A-difference} therefore vanishes.
\end{proof}

The monic polynomial \eqref{eq:Hahn-monic-A} will be used when a limit
must preserve the leading coefficient.  For \(q=1\),
\eqref{eq:inverse-Bernstein-identity} is the classical Hahn--Jacobi
Bernstein-coefficient relation \cite{Ciesielski1987,Waldron2006}.  The
published multiple-Hahn comparison
\cite[Proposition~3.18]{BranquinhoDiazFoulquieManasWolfs2025} writes
\(Q_{\boldsymbol n}(x)=\sum_{\ell=0}^d
Q_{\boldsymbol n}[\ell](-x)_\ell\) and
\(P_{\boldsymbol n}(x)=\sum_{\ell=0}^d
P_{\boldsymbol n}[\ell]x^\ell\), and obtains
\(Q_{\boldsymbol n}[\ell]=(-1)^\ell
(N-\ell)!P_{\boldsymbol n}[\ell]/(N-d)!\).
That article does not formulate the result in terms of the Bernstein
operator.  Since \((-k)_\ell=(-1)^\ell\fall{k}{\ell}\), its identity and
\eqref{eq:Bernstein-factorial-basis} imply
\(\mathscr B_N[Q_{\boldsymbol n}/\fall{N}{d}]=P_{\boldsymbol n}\), which
gives \eqref{eq:inverse-Bernstein-identity} in its normalization.

\begin{remark}[Sharpness of the near-diagonal hypothesis]
\label{rem:Hahn-near-diagonal-sharpness}
The near-diagonal assumption cannot simply be removed from the closed
formula \eqref{eq:Hahn-A-factorial}. Take
\(\A=(\frac12,\frac34)\), \(\bbeta=(2,\frac52)\), \(N=7\), and
\(\mm=(0,2)\).
These parameters are admissible and \(d=2\le N\), but \(\mm\) is not
near the diagonal. Formula \eqref{eq:Hahn-A-factorial} gives
\(A(k)=1-\frac{48}{7}k+\frac{176}{49}\fall{k}{2}\).
If
\(\mu_r=\sum_{k=0}^7\fall{k}{r}v_{2,7}(k)\), then
\eqref{eq:Hahn-like-factorial-moments} gives \(\mu_1=\frac38\),
\(\mu_2=\frac7{16}\), and \(\mu_3=\frac{175}{256}\).
Using \(\bigl(\fall{k}{1}\bigr)^2=\fall{k}{2}+\fall{k}{1}\) and
\(\fall{k}{1}\fall{k}{2}=\fall{k}{3}+2\fall{k}{2}\), one obtains
\[
 \sum_{k=0}^7 kA(k)v_{2,7}(k)
 =\mu_1-\frac{48}{7}(\mu_2+\mu_1)
  +\frac{176}{49}(\mu_3+2\mu_2)
 =\frac{45}{112}\ne0.
\]
The moment of order zero does vanish, but the displayed nonzero first
moment violates the condition required by \(m_2=2\).  Thus
near-diagonality is genuinely needed for the universal formula in
Theorem~\ref{thm:Hahn-A}; this example does not assert nonexistence of a
different type-II polynomial for the same off-diagonal index.
\end{remark}

Under the same scaling, the polynomial converges to its Jacobi-like
counterpart.

\begin{corollary}[Polynomial confluence]
\label{cor:Hahn-Jacobi-polynomial-confluence}
Let \(\mm\in\Nzero^q\) be a fixed near-diagonal index and put
\(d=\abs\mm\). As \(N\to\infty\), with \(N\ge d\),
\[
 A^{\mathrm H}_{\mm,N}(\lfloor Nx\rfloor)
 \xrightarrow[N\to\infty]{} A^{\mathrm J}_{\mm}(x)
\]
locally uniformly for \(x\) in compact subsets of \(\R\). Moreover, for
every \(\ell\in\{0,\ldots,d\}\), the coefficient of
\(\fall{k}{\ell}/\fall{N}{\ell}\) in
\(A^{\mathrm H}_{\mm,N}(k)\) equals the coefficient of \(x^\ell\) in
\(A^{\mathrm J}_{\mm}(x)\).
\end{corollary}

\begin{proof}
For every fixed \(\ell\),
\(\fall{\lfloor Nx\rfloor}{\ell}/\fall{N}{\ell}
\xrightarrow[N\to\infty]{} x^\ell\) locally uniformly for \(x\in\R\). The sum
\eqref{eq:Hahn-A-factorial} has only the fixed terms
\(0\le\ell\le d\), so termwise passage to the limit gives
\(\sum_{\ell=0}^d(-d)_\ell(\bbeta+\mm)_\ell
x^\ell/[\ell!(\A)_\ell]=A^{\mathrm J}_{\mm}(x)\).
The same formula shows coefficientwise convergence in the scaled
factorial basis.

This limit also gives a self-contained proof of the Jacobi-like
orthogonality quoted after \eqref{eq:continuous-A}. Indeed, for an active
row \(j\) and \(0\le r<m_j\), divide
\eqref{eq:Hahn-A-orthogonality} by \(N^r\). Expressing the polynomial in
the scaled factorial basis, the preceding coefficient identity and the
exact moment relation in Theorem~\ref{thm:Hahn-to-Jacobi-rows} permit
termwise passage to the limit, using
\eqref{eq:falling-linearization} for the products of falling factorials.
The result is
\[
 \int_0^1x^rA_{\mm}^{\mathrm J}(x)\omega_j(x)\dd x=0,
 \qquad 0\le r<m_j.
\]
\end{proof}

\subsection{The normalized signed sequence and its type-I representation}

We first determine a normalized signed sequence satisfying the type-I
moment conditions. When the type-I moment matrix is nonsingular, the
polynomials in its type-I representation can then be reconstructed
explicitly.

Assume \(1\le d\coloneq\abs\mm\le N+1\) and put
\(n=d-1\le N\). A type-I linear form is a signed
sequence \(\mathcal B_{\mm,N}^{\mathrm H}(k)\) with a representation
\begin{equation}
 \mathcal B_{\mm,N}^{\mathrm H}(k)
 =\sum_{\substack{j\in\{1,\ldots,q\}\\m_j>0}}
 B_{j,N}(k)\widetilde v_{j,N}(k).
 \label{eq:B-component-decomposition}
\end{equation}
Here \(\deg B_{j,N}<m_j\) for every
\(j\in\{1,\ldots,q\}\) with \(m_j>0\).

For any finitely supported signed sequence with generating polynomial
\(\mathcal H(z)=\sum_k\mathcal B(k)z^k\), one has
\(\mathcal H^{(r)}(1)=\sum_k\fall{k}{r}\mathcal B(k)\).
Consequently, a factorization
\(\mathcal H(z)=C(1-z)^nF(1-z)\), with \(CF(0)\ne0\), forces the moments
of every order \(r\in\Nzero\) with \(r<n\) to vanish and gives the nonzero moment
\((-1)^n n!CF(0)\) at order \(n\).  Thus the moment conditions below are
already encoded by the factor \((1-z)^n\).  What is not automatic is the
representation \eqref{eq:B-component-decomposition} with
\(\deg B_{j,N}<m_j\): proving membership in this finite-dimensional span
and reconstructing the individual components is the content of
Theorem~\ref{thm:Hahn-B-components} and its corollaries.

Let \(\mm\) be near the diagonal and let \(d=\abs\mm\) satisfy
\(d\in\{1,\ldots,N+1\}\). Put \(n=d-1\), and define
\begin{equation}
 \mathcal H^{\mathrm H}_{\mm,N}(z)
 \coloneq\sum_{k=0}^N\mathcal B_{\mm,N}^{\mathrm H}(k)z^k
 =-\fall{N}{n}\frac{(\A)_n}{(\bbeta)_{\mm+n}}(1-z)^n
 \,\pFq{q+1}{q}{n-N,\A+n}{\bbeta+\mm+n}{1-z}
 \label{eq:Hahn-complete-B-pgf}
\end{equation}

\begin{theorem}[Normalized Hahn-like signed sequence]
\label{thm:Hahn-complete-B}
Let \(\mm\) be near the diagonal with
\(1\le d\coloneq\abs\mm\le N+1\), and set \(n=d-1\). The
factorial moments of the sequence defined by
\eqref{eq:Hahn-complete-B-pgf} are
\begin{equation}
 \sum_{k=0}^N\fall{k}{r}\mathcal B_{\mm,N}^{\mathrm H}(k)
 =-h_N(r)\frac{(-r)_n}
 {\prod_{h=1}^q(r+\beta_h)_{m_h}},
 \qquad r\in\Nzero.
 \label{eq:Hahn-complete-B-moments}
\end{equation}
In particular, its first \(n\) moments vanish and its moment of order \(n\)
is nonzero.
\end{theorem}

\begin{proof}
Put \(u=1-z\) and expand the terminating series in
\eqref{eq:Hahn-complete-B-pgf} as
\(\sum_{s=0}^{N-n}
\frac{(n-N)_s(\A+n)_s}{(\bbeta+\mm+n)_s}\frac{u^s}{s!}\).
The factorial moment of order \(r\) is the \(r\)-th derivative of the
generating function at \(z=1\). Since \(\dd/\dd z=-\dd/\dd u\), it vanishes
for \(r<n\). For \(n\le r\le N\), the coefficient of \(u^r\) is obtained by
setting \(s=r-n\), and hence
\[
 \sum_{k=0}^N\fall{k}{r}\mathcal B_{\mm,N}^{\mathrm H}(k)
 =(-1)^r r!\left[-\fall{N}{n}
 \frac{(\A)_n}{(\bbeta)_{\mm+n}}
 \frac{(n-N)_{r-n}(\A+n)_{r-n}}
 {(\bbeta+\mm+n)_{r-n}(r-n)!}\right].
\]
Use \(\frac{r!}{(r-n)!}=(-1)^n(-r)_n\) and
\((n-N)_{r-n}=(-1)^{r-n}\frac{\fall{N}{r}}{\fall{N}{n}}\), together
with \((\A)_n(\A+n)_{r-n}=(\A)_r\) and
\((\bbeta)_{\mm+n}(\bbeta+\mm+n)_{r-n}
=(\bbeta)_r\prod_{h=1}^q(r+\beta_h)_{m_h}\).
After cancellation this is \eqref{eq:Hahn-complete-B-moments}. At \(r=n\)
all remaining factors are nonzero in the admissible parameter range.
\end{proof}

These moments converge to those of the Jacobi-like type-I linear form.

\begin{corollary}[Hahn-to-Jacobi type-I limit]
\label{cor:Hahn-Jacobi-dual}
Fix a near-diagonal \(\mm\) with \(d\coloneq\abs\mm\ge1\). For every
\(N\in\Nzero\) with \(N\ge d-1\), let
\(\mathcal B_{\mm,N}^{\mathrm H}\) be the sequence in
Theorem~\ref{thm:Hahn-complete-B}. As \(N\to\infty\), the push-forward
signed measures \(\sum_{k=0}^N\mathcal B_{\mm,N}^{\mathrm H}(k)
\,\delta_{k/N}\) converge in every fixed moment to the Jacobi-like type-I
linear form fixed by
\eqref{eq:Jacobi-B-Mellin}.
\end{corollary}

\begin{proof}
For fixed \(r\), divide \eqref{eq:Hahn-complete-B-moments} by \(N^r\).
Since \(\fall{N}{r}/N^r\xrightarrow[N\to\infty]{}1\), the factor
\(h_N(r)/N^r\) tends to
\((\A)_r/(\bbeta)_r\), and the limit is the right-hand side of
\eqref{eq:Jacobi-B-moments}.

For ordinary moments use
\(k^r=\sum_{\ell=0}^r
\genfrac\{\}{0pt}{}{r}{\ell}\fall{k}{\ell}\), where the braces denote
Stirling numbers of the second kind. After division
by \(N^r\), the term \(\ell=r\) has the limit just computed, whereas each
term with \(\ell<r\) is \(\mathrm O(N^{\ell-r})\) as
\(N\to\infty\). Consequently
\(\sum_k(k/N)^r\mathcal B_{\mm,N}^{\mathrm H}(k)\) converges to the \(r\)-th moment of
\(\mathcal H^{\mathrm J}_{\mm}\).
\end{proof}

The next result recovers the type-I polynomials and states when they
are unique.

Let \(\mm\) be near the diagonal, let \(1\le d=\abs\mm\le N+1\), and
put \(n=d-1\). For every \(j\in\{1,\ldots,q\}\) and coefficients
\(b_{j,\ell}\), set
\begin{equation}
 B_{j,N}(k)=\sum_{\ell=0}^{m_j-1}b_{j,\ell}\fall{k}{\ell}.
 \label{eq:B-factorial-components}
\end{equation}
For every \(j\in\{1,\ldots,q\}\) with \(m_j>0\) and every
\(\ell\in\{0,\ldots,m_j-1\}\), define the rational function
\begin{equation}
 \begin{split}
 R_{j,\ell}^{(N)}(r)={}&
 \sum_{u=0}^{\ell}\binom\ell u
 \fall{r}{\ell-u}\fall{N-r}{u}
 \frac{\prod_h(r+A_h)_u}
 {\prod_h(r+\beta_h)_u(r+\beta_j+u)}\\
 ={}&\frac{\fall{r}{\ell}}{r+\beta_j}
 \pFq{q+2}{q+1}
 {-\ell,r-N,r+\A}
 {r-\ell+1,r+\beta_j+1,r+\bbeta^{*j}}{1},
 \end{split}
 \label{eq:R-kernel}
\end{equation}
where \(r+\bbeta^{*j}\) means that \(r+\beta_j\) is omitted, and the
second line is interpreted by continuation when \(\fall{r}{\ell}=0\).
For every \(r\in\{0,\ldots,N\}\),
\begin{equation}
 \frac{1}{h_N(r)}
 \sum_{k=0}^N\fall{k}{r}\fall{k}{\ell}\widetilde v_{j,N}(k)
 =R_{j,\ell}^{(N)}(r).
 \label{eq:R-pairing}
\end{equation}

Suppose that the sets of poles
\(\{-\beta_j-K:K\in\{0,\ldots,m_j-1\}\}\), for
\(j\in\{1,\ldots,q\}\) with \(m_j>0\), are pairwise disjoint and that the
quantities \(d_{J,K}\) defined below do not vanish for every
\(J\in\{1,\ldots,q\}\) with \(m_J>0\) and every
\(K\in\{0,\ldots,m_J-1\}\). For these \(J\) and \(K\), put
\begin{align}
 \pi_{J,K}
 &=\frac{(-1)^{K+1}(\beta_J+K)_n}
 {K!(m_J-1-K)!\prod_{h\ne J}
 (\beta_h-\beta_J-K)_{m_h}},
 \label{eq:target-residue}\\
 d_{J,K}
 &=\frac{(N+\beta_J+1)_K
 \prod_h(A_h-\beta_J-K)_K}
 {(-1)^KK!\prod_{h\ne J}(\beta_h-\beta_J-K)_K}.
 \label{eq:diagonal-residue}
\end{align}
Define \(b_{J,K}\) recursively in descending order of \(K\) by
\begin{equation}
 b_{J,K}=\frac{1}{d_{J,K}}
 \left(\pi_{J,K}-
 \sum_{j=1}^q\sum_{\ell=K+1}^{m_j-1}
 b_{j,\ell}\,
 \Res_{r=-\beta_J-K}R_{j,\ell}^{(N)}(r)\right).
 \label{eq:triangular-B-reconstruction}
\end{equation}

\begin{theorem}[Type-I polynomials from residues]
\label{thm:Hahn-B-components}
Let \(\mm\) be near the diagonal, let \(1\le\abs\mm\le N+1\), and assume
that the pole sets
\(\{-\beta_j-K:K\in\{0,\ldots,m_j-1\}\}\), indexed by
\(j\in\{1,\ldots,q\}\) with \(m_j>0\), are pairwise disjoint and that
\(d_{J,K}\ne0\) for every \(J\in\{1,\ldots,q\}\) with \(m_J>0\) and
every \(K\in\{0,\ldots,m_J-1\}\). Then the coefficients
defined by \eqref{eq:triangular-B-reconstruction} satisfy
\eqref{eq:B-component-decomposition} for the signed sequence of
Theorem~\ref{thm:Hahn-complete-B}, and that sequence is the type-I linear
form for the normalized weights.
\end{theorem}

Every residue in \eqref{eq:triangular-B-reconstruction} is a terminating
\({}_{q+2}F_{q+1}(1)\) with one upper parameter a nonpositive integer.

\begin{proof}
Equation \eqref{eq:R-pairing} follows from
\eqref{eq:falling-linearization} and
\eqref{eq:resolvent-factorial-moment}; rewriting its finite sum gives the
hypergeometric expression.

For distinct indices \(h,J\) with \(m_hm_J>0\), disjointness at \(K=0\) implies
\(\beta_h\ne\beta_J\). More generally, for
\(K\in\{0,\ldots,m_J-1\}\),
\begin{equation}
 (\beta_h-\beta_J-K)_{m_h}\ne0,
 \qquad
 (\beta_h-\beta_J-K)_K\ne0
 \quad(h\ne J).
 \label{eq:pole-separation-nonzero-Pochhammers}
\end{equation}
Indeed, a zero factor would give
\(-\beta_J-K=-\beta_h-t\). For the first Pochhammer,
\(t\in\{0,\ldots,m_h-1\}\), contradicting disjointness of the pole sets.
For the second, near-diagonality gives
\(K\le m_J-1\le m_h\); the case \(K=0\) is the empty product.
Thus all Pochhammer factors in the denominators are nonzero. Since
\(N+\beta_J+1>0\), the assumption \(d_{J,K}\ne0\) excludes the remaining
possible zeros in its numerator.

The required type-I identity is equivalent
to the rational interpolation problem
\begin{equation}
 -\frac{(-r)_n}{\prod_h(r+\beta_h)_{m_h}}
 =\sum_{j=1}^q\sum_{\ell=0}^{m_j-1}
 b_{j,\ell}R_{j,\ell}^{(N)}(r).
 \label{eq:rational-component-identity}
\end{equation}
We record the decay at infinity needed here and in the normality argument.
In the hypergeometric expression in \eqref{eq:R-kernel}, set
\[
 Q_u(r)=
 \frac{(r-N)_u\prod_{h=1}^q(r+A_h)_u}
 {(r-\ell+1)_u(r+\beta_j+1)_u
  \prod_{h\ne j}(r+\beta_h)_u}.
\]
Apply Lemma~\ref{lem:alternating-Pochhammer-estimate} with \(m=\ell\),
numerator parameters \((-N,A_1,\ldots,A_q)\), and denominator parameters
\((1-\ell,\beta_j+1,\bbeta^{*j})\). Since
\((-\ell)_u/u!=(-1)^u\binom\ell u\), it gives
\(\sum_{u=0}^{\ell}(-1)^u\binom\ell u Q_u(r)
=\mathrm O(r^{-\ell})\) as \(|r|\to\infty\).
The prefactor \(\fall{r}{\ell}/(r+\beta_j)\) in \eqref{eq:R-kernel} is
\(\mathrm O(r^{\ell-1})\) as \(|r|\to\infty\). Therefore
\begin{equation}
 R_{j,\ell}^{(N)}(r)=\mathrm O(r^{-1}),
 \qquad |r|\to\infty.
 \label{eq:R-decay}
\end{equation}
The pole-separation hypothesis makes every pole of the finite-sum
representation of \(R_{j,\ell}^{(N)}\) simple. Indeed, the factors with
index \(j\) combine as
\((r+\beta_j)_u(r+\beta_j+u)=(r+\beta_j)_{u+1}\), whose possible pole
indices satisfy \(0\le t\le u\le\ell\le m_j-1\). For \(h\ne j\) and
\(u\ge1\), they satisfy
\(0\le t\le u-1\le\ell-1\le m_j-2\le m_h-1\).
Thus every possible pole lies in one of the prescribed pole sets, and a
coincidence would identify two points in those pairwise disjoint sets.
Thus matching all residues removes all poles.
At the pole \(r=-\beta_J-K\), the rational functions with
\(j=J,\ell<K\), or with
\(j\ne J,\ell<K+1\), have zero residue. The diagonal residue is
\eqref{eq:diagonal-residue}, and the residue required by the left-hand side is
\eqref{eq:target-residue}. Matching residues in decreasing order of \(K\)
gives \eqref{eq:triangular-B-reconstruction}. The difference between the
two sides of \eqref{eq:rational-component-identity} has no poles and tends
to zero as \(\abs r\to\infty\), hence is zero. Multiplying by \(h_N(r)\) and using
\eqref{eq:R-pairing} proves equality of all moments and therefore of the
finite sequences.
\end{proof}

The recursive residue formula can be summed in closed form. In the classical
multiple-Hahn system, every type-I polynomial is a single terminating
generalized hypergeometric polynomial
\cite[Theorem~2.1, equation~(21)]{BranquinhoDiazFoulquieManasWolfs2025}.
The weights in the present paper are different, so that result cannot be
transferred by a substitution of parameters. The corresponding Hahn-like
formula follows instead from the Cauchy transform written in the Newton
basis.

Under the hypotheses of Theorem~\ref{thm:Hahn-B-components}, for every
\(J\in\{1,\ldots,q\}\) with \(m_J>0\) and every
\(K\in\Nzero\) with \(K<m_J\), put
\(\xi_{J,K}=\beta_J+K\). For every
pair \(j,J\in\{1,\ldots,q\}\) with \(m_jm_J>0\), define
\(\mathcal C_{J,J,0}^{(N)}(k)=1\) and
\(\mathcal C_{j,J,0}^{(N)}(k)=0\) for \(j\ne J\).
For every \(j,J\in\{1,\ldots,q\}\) with \(m_jm_J>0\) and every
\(K\in\mathbb N\) with \(K<m_J\), set
\begin{equation}
 \mathcal D_{j,J,K}^{(N)}\coloneq
 \frac{(N+\beta_j)\prod_{h=1}^q(A_h-\beta_j)
       \prod_{h\ne j}(\beta_h-\xi_{J,K})}
 {(N+\xi_{J,K})\prod_{h\ne j}(\beta_h-\beta_j)
       \prod_{h=1}^q(A_h-\xi_{J,K})}
 \label{eq:Hahn-C-block-prefactor}
\end{equation}
and
\begin{equation}
 \mathcal C_{j,J,K}^{(N)}(k)\coloneq
 \mathcal D_{j,J,K}^{(N)}\;
 \pFq{q+2}{q+1}
 {1,1-k-\xi_{J,K},
  \bbeta+(1-\xi_{J,K})\one-\ee_j}
 {1-N-\xi_{J,K},
 \A+(1-\xi_{J,K})\one}{1}.
 \label{eq:Hahn-C-hypergeometric-block}
\end{equation}
Here the hypergeometric expression denotes the finite sum truncated at
\(s=K\) when \(j=J\) and at \(s=K-1\) when \(j\ne J\). This truncation is
imposed before any coincident upper and lower Pochhammer parameters are
simplified, including when
\(A_h+1-\xi_{J,K}\in\mathbb Z_{\le0}\) for some
\(h\in\{1,\ldots,q\}\). Any removable parameter
singularity is interpreted by continuation of the complete right-hand side
of \eqref{eq:Hahn-C-hypergeometric-block}, rather than by cancelling
parameter strings before fixing the finite range. Equivalently, the finite
recurrence formula \eqref{eq:Hahn-C-finite-recurrence} below gives the same
polynomial without hypergeometric ambiguity.

\begin{corollary}[Closed hypergeometric Hahn-like type-I polynomials]
\label{cor:Hahn-B-closed-components}
Under the hypotheses of Theorem~\ref{thm:Hahn-B-components}, the type-I
polynomials for the rescaled weights
\(\widetilde v_{j,N}=v_{j,N}/\beta_j\) are
\begin{equation}
 B_{j,N}(k)=
 \sum_{J=1}^q\sum_{K=0}^{m_J-1}
 \pi_{J,K}\mathcal C_{j,J,K}^{(N)}(k),
 \qquad j\in\{1,\ldots,q\},\quad m_j>0.
 \label{eq:Hahn-B-closed-components}
\end{equation}
For every \(J\in\{1,\ldots,q\}\) with \(m_J>0\) and every
\(K\in\{0,\ldots,m_J-1\}\),
\(\deg \mathcal C_{J,J,K}^{(N)}\le K\). For every such \(J,K\) with
\(K\ge1\), and every \(j\in\{1,\ldots,q\}\) with \(m_j>0\) and
\(j\ne J\), one has \(\deg \mathcal C_{j,J,K}^{(N)}\le K-1\).
Furthermore, \(\mathcal C_{j,J,0}^{(N)}=0\) for every pair of distinct
active indices \(j,J\), and \(\deg B_{j,N}<m_j\) for every
\(j\in\{1,\ldots,q\}\) with \(m_j>0\).
\end{corollary}

\begin{proof}
If some \(m_h=0\), near-diagonality gives
\(m_j\in\{0,1\}\) for every \(j\in\{1,\ldots,q\}\). Thus every active
pole has \(K=0\), the recursive formula
\eqref{eq:triangular-B-reconstruction} reduces to
\(b_{j,0}=\pi_{j,0}\) for every \(j\in\{1,\ldots,q\}\) with \(m_j>0\).
The assertion then follows from
\(\mathcal C_{j,J,0}^{(N)}=\delta_{j,J}\) for every pair of active
indices \(j,J\). We may therefore assume for the rest of the proof that
\(m_h>0\) for every \(h\in\{1,\ldots,q\}\).
The pole-separation hypothesis now implies that
\(\beta_1,\ldots,\beta_q\) are pairwise distinct. For every
\(J\in\{1,\ldots,q\}\), every \(K\in\mathbb N\) with \(K<m_J\), and every
\(t\in\{0,\ldots,K-1\}\), it also gives, with \(c=\beta_J+t\),
\(c+1\ne\beta_j\) for every \(j\in\{1,\ldots,q\}\). Hence all
denominators in the following partial-fraction identities are nonzero.

For a sequence \(u=(u(0),\ldots,u(N))\), define its Cauchy transform by
\(\mathscr C[u](z)=\sum_{k=0}^Nu(k)/(z-k)\).
The finite Newton identity
\(1/(z-k)=\sum_{r=0}^k\fall{k}{r}/\fall{z}{r+1}\)
gives
\begin{equation}
 \mathscr C[u](z)=\sum_{r=0}^N
 \frac{\sum_{k=0}^N\fall{k}{r}u(k)}{\fall{z}{r+1}}.
 \label{eq:Cauchy-Newton-transform}
\end{equation}
For a parameter \(c\), put
\(I_c(z)=\sum_{r=0}^Nh_N(r)/((r+c)\fall{z}{r+1})\).
The factorial moments \eqref{eq:resolvent-factorial-moment} imply
\(I_{\beta_j}=\mathscr C[\widetilde v_{j,N}]\). Write
\(\rho_N(r)=h_N(r+1)/h_N(r)
=(N-r)\prod_h(r+A_h)/\prod_h(r+\beta_h)\).
Partial fractions give
\[
 \frac{\rho_N(r)}{r+c+1}
 =-1+\sum_{j=1}^q\frac{\alpha_j(c)}{r+\beta_j}
 +\frac{\alpha_0(c)}{r+c+1},
\]
where
\[
 \alpha_j(c)=
 \frac{(N+\beta_j)\prod_h(A_h-\beta_j)}
 {(c+1-\beta_j)\prod_{h\ne j}(\beta_h-\beta_j)},
 \qquad
 \alpha_0(c)=
 \frac{(N+c+1)\prod_h(A_h-c-1)}
 {\prod_h(\beta_h-c-1)}.
\]
Put \(F_r(z)=1/\fall{z}{r+1}\). Since
\((z+c)F_r(z)/(r+c)=F_r(z)+1/((r+c)\fall{z}{r})\), the value
\(h_N(0)=1\) and an index shift give
\begin{align*}
 (z+c)I_c(z)
 =\sum_{r=0}^Nh_N(r)F_r(z)+\frac1c
   +\sum_{r=0}^{N-1}
   \frac{h_N(r+1)}{(r+c+1)\fall{z}{r+1}}=\frac1c+\sum_{r=0}^Nh_N(r)
   \left(1+\frac{\rho_N(r)}{r+c+1}\right)F_r(z).
\end{align*}
In the second equality the last sum was extended to \(r=N\), which is
valid because \(\rho_N(N)=0\).
Substitution of the partial-fraction identity yields
\begin{equation}
 (z+c)I_c(z)=\frac1c+
 \sum_{j=1}^q\alpha_j(c)I_{\beta_j}(z)
 +\alpha_0(c)I_{c+1}(z).
 \label{eq:Hahn-Cauchy-contiguous-recurrence}
\end{equation}
Equivalently,
\begin{equation}
 I_{c+1}(z)=\frac{z+c}{\alpha_0(c)}I_c(z)
 -\sum_{j=1}^q\frac{\alpha_j(c)}{\alpha_0(c)}I_{\beta_j}(z)
 -\frac1{c\alpha_0(c)}.
 \label{eq:Hahn-Cauchy-iteration}
\end{equation}
For every \(J\in\{1,\ldots,q\}\) and every
\(K\in\{0,\ldots,m_J-1\}\), direct multiplication of the displayed
formula for \(\alpha_0\) gives
\(\prod_{t=0}^{K-1}\alpha_0(\beta_J+t)=d_{J,K}\). Thus the nonvanishing
hypothesis on \(d_{J,K}\) justifies every division in the iteration.

Starting with \(c=\beta_J\) and iterating
\eqref{eq:Hahn-Cauchy-iteration} gives
\(I_{\xi_{J,K}}(z)=
\sum_{j=1}^q\mathcal C_{j,J,K}^{(N)}(z)I_{\beta_j}(z)
+P_{J,K}(z)\),
where \(P_{J,K}\) is a polynomial and
\begin{equation}
 \mathcal C_{j,J,K}^{(N)}(z)=
 \delta_{j,J}\prod_{s=0}^{K-1}
 \frac{z+\beta_J+s}{\alpha_0(\beta_J+s)}-\sum_{t=0}^{K-1}
 \frac{\alpha_j(\beta_J+t)}{\alpha_0(\beta_J+t)}
 \prod_{s=t+1}^{K-1}
 \frac{z+\beta_J+s}{\alpha_0(\beta_J+s)}.
 \label{eq:Hahn-C-finite-recurrence}
\end{equation}
For \(K\ge1\), substitute the formulas for \(\alpha_j\) and
\(\alpha_0\), change the index to \(t=K-1-\ell\), and denote the
resulting summands by \(T_\ell\), with
\(\ell\in\{0,\ldots,K-1\}\). The term \(T_0\) is
\(\mathcal D_{j,J,K}^{(N)}\). If \(j\ne J\) and \(K\ge2\), consecutive
terms satisfy, for every \(\ell\in\{0,\ldots,K-2\}\),
\begin{equation}
 \frac{T_{\ell+1}}{T_\ell}=
 \frac{(1-z-\xi_{J,K}+\ell)
 \prod_{h=1}^q(\beta_h+1-\xi_{J,K}-\delta_{h,j}+\ell)}
 {(1-N-\xi_{J,K}+\ell)
 \prod_{h=1}^q(A_h+1-\xi_{J,K}+\ell)}.
 \label{eq:Hahn-C-term-ratio}
\end{equation}
If \(j=J\), let \(T_K\) denote the first product in
\eqref{eq:Hahn-C-finite-recurrence}; then
\eqref{eq:Hahn-C-term-ratio} holds for every
\(\ell\in\{0,\ldots,K-1\}\). This is precisely the consecutive-term
ratio in \eqref{eq:Hahn-C-hypergeometric-block}, because
\((1)_\ell/\ell!=1\). The upper parameter
in position \(J\) is \(-K\) when \(j=J\), and \(1-K\) otherwise; this
proves the degree bounds. Near-diagonality gives
\(K-1\le m_J-2\le m_j-1\) in the off-diagonal case.

Finally, the target rational function has the partial-fraction expansion
\[
 -\frac{(-r)_n}{\prod_h(r+\beta_h)_{m_h}}
 =\sum_{J=1}^q\sum_{K=0}^{m_J-1}
 \frac{\pi_{J,K}}{r+\xi_{J,K}}.
\]
Equations \eqref{eq:Hahn-complete-B-moments} and
\eqref{eq:Cauchy-Newton-transform} therefore give
\[
 \mathscr C[\mathcal B_{\mm,N}^{\mathrm H}](z)
 =\sum_{J=1}^q\sum_{K=0}^{m_J-1}
 \pi_{J,K}I_{\xi_{J,K}}(z)
 =\sum_{\substack{j\in\{1,\ldots,q\}\\m_j>0}}
 B_{j,N}(z)\mathscr C[\widetilde v_{j,N}](z)+P(z)
\]
for a polynomial \(P\). For every polynomial \(B\), the difference
\(B(z)\mathscr C[u](z)-\mathscr C[Bu](z)
=\sum_{k=0}^N[B(z)-B(k)]u(k)/(z-k)\) is also a polynomial. At every
\(k\in\{0,\ldots,N\}\), the residues of
\(\mathscr C[\mathcal B_{\mm,N}^{\mathrm H}]\) and
\(B_{j,N}\mathscr C[\widetilde v_{j,N}]\) are, respectively,
\(\mathcal B_{\mm,N}^{\mathrm H}(k)\) and
\(B_{j,N}(k)\widetilde v_{j,N}(k)\), while \(P\) has no residue.
Their equality proves \eqref{eq:B-component-decomposition} with the
polynomials \eqref{eq:Hahn-B-closed-components}.
\end{proof}

For every \(j\in\{1,\ldots,q\}\) with \(m_j>0\), the type-I polynomial
relative to the original weight \(v_{j,N}\), rather than the rescaled
weight, is \(B_{j,N}/\beta_j\). To normalize the factorial moment of order
\(n\) to one, multiply every such \(B_{j,N}\) by
\((-1)^{n+1}\prod_h(n+\beta_h)_{m_h}/(n!h_N(n))\).

Grouping the terms in \eqref{eq:Hahn-B-closed-components} by the index
\(J\) gives a shorter bivariate hypergeometric representation.

Under the hypotheses of Corollary~\ref{cor:Hahn-B-closed-components},
fix \(j\in\{1,\ldots,q\}\) with \(m_j>0\). For every
\(J\in\{1,\ldots,q\}\), put
\(\varepsilon_{j,J}=1-\delta_{j,J}\),
and, whenever \(m_J\ge\varepsilon_{j,J}+1\), abbreviate
\(\varepsilon=\varepsilon_{j,J}\) and define
\begin{align*}
 \boldsymbol a_{j,J}(k)
 &\coloneq{}
 \left(
  \beta_J+\varepsilon+n,\,
  \varepsilon+1-m_J,\,
  k+\beta_J+\varepsilon,\,
  \{\beta_J-\beta_h+\varepsilon+1-m_h\}_{h\ne J}
 \right),\\
 \boldsymbol b_{j,J}^{(N)}
 &\coloneq{}
 \left(
  N+\beta_J+\varepsilon,\,
  \{\beta_J-A_h+\varepsilon\}_{h=1}^q
 \right),\\
 \boldsymbol c_{j,J}^{(N)}
 &\coloneq{}
 \left(
  \beta_J+\varepsilon,\,
  N+\beta_J+\varepsilon+1,\,
  \{\beta_J-A_h+\varepsilon+1\}_{h=1}^q
 \right),\\
 \boldsymbol d_{j,J}(k)
 &\coloneq{}
 \left(
  k+\beta_J+\varepsilon,\,
  \{\beta_J-\beta_h+\varepsilon+\delta_{h,j}\}_{h\ne J}
 \right).
\end{align*}
For \(J=j\), extend \eqref{eq:Hahn-C-block-prefactor} to \(K=0\) by
\(\mathcal D_{j,j,0}^{(N)}=1\). Define
\begin{equation}
 S_{j,J}^{(N)}(k)\coloneq
 \pi_{J,\varepsilon}\mathcal D_{j,J,\varepsilon}^{(N)}
 F_{q+2:q;0}^{q+2:q+1;1}
 \left(
 \begin{matrix}
  \boldsymbol a_{j,J}(k):
  \boldsymbol b_{j,J}^{(N)};1\\
  \boldsymbol c_{j,J}^{(N)}:
  \boldsymbol d_{j,J}(k);\text{--}
 \end{matrix}
 \middle|1,1
 \right).
 \label{eq:Hahn-B-KdF-block}
\end{equation}
\begin{corollary}[Kamp\'e de F\'eriet form of the Hahn-like type-I polynomials]
\label{cor:Hahn-B-KdF}
Under the hypotheses of Corollary~\ref{cor:Hahn-B-closed-components}, for
every \(j\in\{1,\ldots,q\}\) with \(m_j>0\),
\begin{equation}
 B_{j,N}(k)=
 \sum_{\substack{J\in\{1,\ldots,q\}\\
                   m_J\ge\varepsilon_{j,J}+1}}
 S_{j,J}^{(N)}(k).
 \label{eq:Hahn-B-KdF-sum}
\end{equation}
\end{corollary}

\begin{proof}
Fix \(j,J\), write \(\varepsilon=\varepsilon_{j,J}\), and expand the
generalized hypergeometric polynomial in
\eqref{eq:Hahn-C-hypergeometric-block}. Its upper parameter in position
\(J\) is
\(\beta_J+1-(\beta_J+K)-\delta_{J,j}=\varepsilon-K\).
Consequently, the change of variables \(K=\varepsilon+r+t\), with
\(r,t\ge0\) and \(r+t\le m_J-1-\varepsilon\),
parametrizes the full finite summation range. Using
\((a+s)_d/(a)_d=(a+d)_s/(a)_s\),
\((a)_d/(a-s)_d=(1-a-d)_s/(1-a)_s\), and
\((a-r-t)_t=(-1)^t(1-a)_{r+t}/(1-a)_r\), and cancelling common factors
turns its \((r,t)\)-term into
\[
 \pi_{J,\varepsilon}\mathcal D_{j,J,\varepsilon}^{(N)}
 \frac{(\boldsymbol a_{j,J}(k))_{r+t}
       (\boldsymbol b_{j,J}^{(N)})_r(1)_t}
      {(\boldsymbol c_{j,J}^{(N)})_{r+t}
       (\boldsymbol d_{j,J}(k))_r}
 \frac1{r!}\frac1{t!}.
\]
Summing over \(r,t\) and then over the nonempty \(J\)-indexed groups proves
\eqref{eq:Hahn-B-KdF-block}--\eqref{eq:Hahn-B-KdF-sum}. Termination follows
from \(\varepsilon+1-m_J=-(m_J-1-\varepsilon)\). Finally,
\((k+\beta_J+\varepsilon)_{r+t}/
(k+\beta_J+\varepsilon)_r=(k+\beta_J+\varepsilon+r)_t\), which gives the
degree bounds already proved in
Corollary~\ref{cor:Hahn-B-closed-components}.
\end{proof}

Thus each Hahn-like type-I polynomial is a sum of at most \(q\)
terminating Kamp\'e de F\'eriet polynomials. The \(J\)-indexed group is
absent when \(j\ne J\) and \(m_J=1\).

For \(q=2\), the summands in \eqref{eq:Hahn-B-KdF-sum} are
\(F_{4:2;0}^{4:3;1}\); for \(q=3\), they are
\(F_{5:3;0}^{5:4;1}\). When \(q=1\), only the diagonal summand remains and,
by scalar uniqueness, reduces, with the normalization used here, to the usual
terminating \({}_3F_2(1)\) Hahn polynomial.

The extra numerator--denominator pair in these polynomials has the following
Jacobi limit. For \(j,J\in\{1,\ldots,q\}\) with \(m_jm_J>0\) and
\(K\in\Nzero\) with \(K<m_J\), put
\[
 \mathcal D_{j,J,K}^{\mathrm J}\coloneq
 \frac{\prod_h(A_h-\beta_j)\prod_{h\ne j}(\beta_h-\xi_{J,K})}
 {\prod_{h\ne j}(\beta_h-\beta_j)\prod_h(A_h-\xi_{J,K})}.
\]
For these indices and \(0<x<1\),
\(\frac{(1-\lfloor Nx\rfloor-\xi_{J,K})_\ell}
{(1-N-\xi_{J,K})_\ell}\xrightarrow[N\to\infty]{} x^\ell\) and
\(\frac{N+\beta_j}{N+\xi_{J,K}}\xrightarrow[N\to\infty]{}1\).
Thus, if \(K\ge1\),
\[
 \mathcal C_{j,J,K}^{(N)}(\lfloor Nx\rfloor)
 \xrightarrow[N\to\infty]{}
 \mathcal D_{j,J,K}^{\mathrm J}
 \pFq{q+1}{q}
 {1,\bbeta+(1-\xi_{J,K})\one-\ee_j}
 {\A+(1-\xi_{J,K})\one}{x},
\]
with the same diagonal/off-diagonal convention at \(K=0\). This is the
corresponding Jacobi-like type-I term.

When the pole sets overlap, the reconstruction requires a different analysis.

\begin{remark}[Overlapping pole sets]
\label{rem:Hahn-overlapping-poles}
The reconstruction in Theorem~\ref{thm:Hahn-B-components} uses the simple,
pairwise distinct poles
\(\{-\beta_j-\ell:\ell\in\Nzero,\ \ell<m_j\}\), with
\(j\in\{1,\ldots,q\}\).
If \(\beta_j+\ell=\beta_h+t\) for two different pairs
\((j,\ell)\ne(h,t)\), the simple-pole reconstruction does not apply and no
normality conclusion is drawn here. This separation condition is the finite
counterpart of the nonintegrality condition on the differences of the
\(b_j\)'s used in the Jacobi-like type-I normality argument
\cite[Proposition~2.7]{Wolfs2024}.

The equality \(\beta_j=\beta_h\), with \(j\ne h\), is genuinely
degenerate when \(m_j,m_h>0\). Indeed,
\eqref{eq:Hahn-like-factorial-moments} then gives identical factorial
moments for \(v_{j,N}\) and \(v_{h,N}\) through order \(N\). Finite support
implies \(v_{j,N}=v_{h,N}\), so the type-I moment matrix has identical
columns and is singular. More general overlapping-pole cases are outside
the scope of the present reconstruction.
\end{remark}

The reconstruction gives normality under its explicit hypotheses and the classical
formulas when \(q=1\).

\begin{corollary}[Normality and the classical Hahn reduction]
\label{cor:Hahn-normality}
Let \(N\in\N\).
\begin{enumerate}[label=\textup{(\roman*)}]
\item Let \(q>1\) and let \(\mm\in\Nzero^q\) be near the diagonal, with
\(1\le\abs\mm\le N+1\). If the pole sets
\(\{-\beta_j-K:K\in\{0,\ldots,m_j-1\}\}\), indexed by
\(j\in\{1,\ldots,q\}\) with \(m_j>0\), are pairwise disjoint and all the
coefficients \(d_{J,K}\) in \eqref{eq:diagonal-residue} are nonzero, then
\(\mm\) is normal.
\item Let \(q=1\), \(m\in\{1,\ldots,N+1\}\), and \(n=m-1\). Throughout
the admissible scalar positivity region, \(m\) is normal and
\[
 B_{1,N}(k)\propto\pFq{3}{2}{-n,n+\beta,-k}{A,-N}{1}.
\]
If \(m\le N\), the value-normalized type-II polynomial is
  \[
  A^{\mathrm H}_{m,N}(k)=
  \pFq{3}{2}{-m,-k,\beta+m}{-N,A}{1}.
  \]
If \(m=N+1\), its monic replacement is the nodal polynomial
\(\widehat A^{\mathrm H}_{N+1,N}=\Pi_{N+1}\) from
\eqref{eq:finite-lattice-nodal-polynomial}; no normalization with
\(A(0)=1\) is possible.
\end{enumerate}
\end{corollary}

\begin{proof}
Consider a homogeneous null vector for the type-I moment matrix. By
\eqref{eq:R-pairing}, it gives a linear combination of the rational functions
\(R_{j,\ell}^{(N)}(r)\) that vanishes for every
\(r\in\{0,\ldots,d-1\}\). After multiplication by
\(\prod_h(r+\beta_h)_{m_h}\), its numerator has degree at most \(d-1\)
by \eqref{eq:R-decay}; hence those \(d\) zeros make it identically zero.
This includes
\(d=N+1\), because the required values \(r=0,\ldots,d-1\) are then
exactly \(0,\ldots,N\), the full finite-lattice range of
\eqref{eq:R-pairing}. Descending through the
sets of poles as in \eqref{eq:triangular-B-reconstruction}, the nonzero
diagonal residues \(d_{J,K}\) force every coefficient of the null vector
to vanish. Thus the moment matrix, and therefore its transpose, is
nonsingular. For \(q=1\), the displayed formulas through degree \(N\)
are the defining \({}_3F_2\) Hahn expressions, while the endpoint
type-II statement is \eqref{eq:finite-lattice-nodal-polynomial}. At the
type-I endpoint \(m=N+1\), so that \(n=N\), the finite-sum convention and
Chu--Vandermonde give, for \(0\le k\le N\),
\[
 \pFq{3}{2}{-N,N+\beta,-k}{A,-N}{1}
 =
 \pFq{2}{1}{-k,N+\beta}{A}{1}
 =\frac{(A-N-\beta)_k}{(A)_k}.
\]
Since \(A-N-\beta=-N-\gamma\), \eqref{eq:exact-Hahn-weight} then yields
\(B_{1,N}(k)v_{1,N}(k)\propto(-1)^k\binom Nk\). Its generating
polynomial is \((1-z)^N\), which directly verifies the endpoint
type-I moment conditions. In the scalar Hahn positivity region,
normality follows
independently of the residue argument: for every nonzero polynomial
\(p\) of degree at most \(m-1\),
\(\sum_{k=0}^Np(k)^2v_{1,N}(k)>0\),
because the Hahn weight is strictly positive at all \(N+1\) lattice
points. Hence the scalar moment matrix is positive definite.
\end{proof}

Corollaries~\ref{cor:q1-Hahn-weight} and \ref{cor:Hahn-normality} give the
standard scalar Hahn weight and, through degree \(N\), the standard Hahn
polynomial formulas; the degree-\(N+1\) type-II endpoint is instead the
nodal representative \eqref{eq:finite-lattice-nodal-polynomial}. No
additional pole-separation condition is needed in the admissible scalar
parameter range.

For every \(j\in\{1,\ldots,q\}\) with \(m_j>0\), if the original weight
\(v_{j,N}\), rather than \(\widetilde v_{j,N}\), is used in
\eqref{eq:B-component-decomposition}, replace \(B_{j,N}\) by
\(B_{j,N}/\beta_j\). The type-I linear form itself is unchanged.

For one weight, positivity holds on an additional parameter range.

\begin{remark}[Scalar Hahn positivity regions]
\label{rem:q1-larger-region}
For \(q\ge2\), the conditions
\(0<A_h<\beta_h\) make all kernels in
\eqref{eq:beta-product-integral} positive. For \(q=1\), the connected
Hahn positivity region is larger: \(A>0\) and
\(\gamma=\beta-A>-1\).
Equivalently, the Hahn parameters satisfy
\(\alpha_{\mathrm H}=A-1>-1\) and \(\gamma>-1\). This is the standard
positivity region uniform in \(N\). For a fixed lattice size \(N\), there
is also the disconnected Hahn positivity region
\(\alpha_{\mathrm H}<-N\), \(\gamma<-N\),
equivalently \(A<1-N\) and \(\beta-A<-N\). The results below use the
connected region uniform in \(N\). The normalization
\(\widetilde v=v/\beta\) requires \(\beta\ne0\); at \(\beta=0\), the
identities for \(v\) and the \(A\)-polynomial remain valid, but the signed
sequence expressed with the rescaled weight \(\widetilde v=v/\beta\) is
singular. A finite normalization for the original weight \(v\) is
  \[
   \mathcal H^{\mathrm H,0}_{m,N}(z)
   \coloneq\lim_{\beta\to0}\beta\,
   \mathcal H^{\mathrm H}_{m,N}(z)
   =-\fall{N}{n}\frac{(A)_n}{(m+n-1)!}(1-z)^n
   \pFq{2}{1}{n-N,A+n}{m+n}{1-z},
   \qquad n=m-1.
  \]
Equivalently, one multiplies the signed sequence of
Theorem~\ref{thm:Hahn-complete-B} by \(\beta\) before returning to the
original weight \(v\).
\end{remark}

\subsection{Componentwise confluence to the Jacobi-like system}
\label{subsec:type-I-special-function-confluence}

The Hahn-to-Jacobi limit follows term by term from the Kamp\'e de F\'eriet formula.
The other limits require one qualification: whenever a parameter pair is
sent to infinity, the terms associated with its poles must be combined
before taking the limit.

For \(j\in\{1,\ldots,q\}\) with \(m_j>0\) and
\(J\in\{1,\ldots,q\}\), put
\(\varepsilon_{j,J}=1-\delta_{j,J}\),
\(\eta_{j,J}=\beta_J+\varepsilon_{j,J}\), and
\(M_{j,J}=m_J-1-\varepsilon_{j,J}\).
For \(M_{j,J}\ge0\), define the parameter lists
\begin{align*}
 \boldsymbol a_{j,J}^{\mathrm{Jac}}
 &\coloneq{}
 \left(
  \eta_{j,J}+n,\,-M_{j,J},\,
  (\eta_{j,J}+1-\beta_h-m_h)_{h\ne J}
 \right),\\
 \boldsymbol b_{j,J}^{\mathrm{Jac}}
 &\coloneq{}\bigl(\eta_{j,J}-A_h\bigr)_{h=1}^q,\\
 \boldsymbol c_{j,J}^{\mathrm{Jac}}
 &\coloneq{}
 \left(
  \eta_{j,J},\,
  (\eta_{j,J}+1-A_h)_{h=1}^q
 \right),\\
 \boldsymbol d_{j,J}^{\mathrm{Jac}}
 &\coloneq{}\bigl(\eta_{j,J}-\beta_h+\delta_{h,j}\bigr)_{h\ne J}.
\end{align*}
Set \(\mathcal D_{J,J}^{\mathrm{Jac}}=1\); for \(j\ne J\), set
\begin{equation}
 \mathcal D_{j,J}^{\mathrm{Jac}}
 =
 \frac{
  \prod_h(A_h-\beta_j)
  \prod_{h\ne j}(\beta_h-\beta_J-1)
 }{
  \prod_{h\ne j}(\beta_h-\beta_j)
  \prod_h(A_h-\beta_J-1)
 }.
 \label{eq:Jacobi-KdF-sector-prefactor}
\end{equation}
For \(M_{j,J}\ge0\), define
\begin{equation}
 S_{j,J}^{\mathrm{Jac}}(x)\coloneq
 \pi_{J,\varepsilon_{j,J}}\mathcal D_{j,J}^{\mathrm{Jac}}
 F_{q+1:q-1;0}^{q+1:q;1}
 \left(
 \begin{matrix}
  \boldsymbol a_{j,J}^{\mathrm{Jac}}:
  \boldsymbol b_{j,J}^{\mathrm{Jac}};1\\
  \boldsymbol c_{j,J}^{\mathrm{Jac}}:
  \boldsymbol d_{j,J}^{\mathrm{Jac}};\text{--}
 \end{matrix}
 \,;\,1,x
 \right).
 \label{eq:Jacobi-KdF-sector}
\end{equation}
When \(M_{j,J}<0\), set both
\(S_{j,J}^{(N)}=0\) and \(S_{j,J}^{\mathrm{Jac}}=0\). Finally, set
\begin{equation}
 B_j^{\mathrm{Jac}}(x)\coloneq\sum_{J=1}^qS_{j,J}^{\mathrm{Jac}}(x).
 \label{eq:Jacobi-KdF-components}
\end{equation}
For every \(j\in\{1,\ldots,q\}\) with \(m_j=0\), set
\(B_j^{\mathrm{Jac}}=0\).

Corollary~\ref{cor:Hahn-Jacobi-polynomial-confluence} treats the type-II
polynomial. We now prove the corresponding sectorwise and componentwise
type-I confluence.

\begin{proposition}[Locally uniform Hahn-to-Jacobi component convergence]
\label{prop:Hahn-Jacobi-KdF-confluence}
Assume the pole-separation conditions in the hypotheses of
Corollary~\ref{cor:Hahn-B-KdF}, and assume that all its remaining
nonvanishing conditions hold for every sufficiently large \(N\). Then, as
\(N\to\infty\), the following assertions hold.
\begin{enumerate}[label=\textnormal{(\roman*)}]
\item For every \(j\in\{1,\ldots,q\}\) with \(m_j>0\), every
\(J\in\{1,\ldots,q\}\), and every compact set
\(\mathcal K\subset(0,1)\),
\begin{equation}
 \sup_{x\in\mathcal K}\left|
 S_{j,J}^{(N)}(\lfloor Nx\rfloor)-S_{j,J}^{\mathrm{Jac}}(x)
 \right|\xrightarrow[N\to\infty]{}0.
 \label{eq:Hahn-Jacobi-KdF-sector-limit}
\end{equation}
The limiting block is terminating and satisfies
\[
 \deg S_{j,J}^{\mathrm{Jac}}
 \le M_{j,J}=m_J-1-\varepsilon_{j,J}
\]
whenever \(M_{j,J}\ge0\); for \(M_{j,J}<0\), both blocks are zero by
definition.
\item For every \(j\in\{1,\ldots,q\}\) with \(m_j>0\) and every
compact set \(\mathcal K\subset(0,1)\),
\begin{equation}
 \sup_{x\in\mathcal K}\left|
 B_{j,N}(\lfloor Nx\rfloor)-B_j^{\mathrm{Jac}}(x)
 \right|\xrightarrow[N\to\infty]{}0.
 \label{eq:Hahn-Jacobi-KdF-component-limit}
\end{equation}
and \(\deg B_j^{\mathrm{Jac}}<m_j\).
\item If the limiting type-I moment matrix is nonsingular, then
\[
 \mathcal H_{\boldsymbol m}^{\mathrm J}(x)
 =\sum_{j=1}^qB_j^{\mathrm{Jac}}(x)\widetilde\omega_j(x),
 \qquad \widetilde\omega_j=\omega_j/\beta_j.
\]
This linear form has the moments in \eqref{eq:Jacobi-B-moments},
\(\boldsymbol m\) is normal, and
\((B_1^{\mathrm{Jac}},\ldots,B_q^{\mathrm{Jac}})\) is the unique tuple of
type-I polynomials with the stated degree bounds.
Moreover,
\begin{equation}
 (N+1)\mathcal B_{\mm,N}^{\mathrm H}(\lfloor Nx\rfloor)
 \xrightarrow[N\to\infty]{}\mathcal H_{\mm}^{\mathrm J}(x)
 \label{eq:Hahn-Jacobi-form-local-limit}
\end{equation}
locally uniformly for \(x\) in compact subsets of \((0,1)\).
\end{enumerate}
\end{proposition}

For every \(j\in\{1,\ldots,q\}\) with \(m_j>0\), the component relative
to the original weight \(\omega_j\) is \(B_j^{\mathrm{Jac}}/\beta_j\).

\begin{proof}
Put \(k_N=\lfloor Nx\rfloor\), and let \(r,t\) be the two summation indices in
\eqref{eq:Hahn-B-KdF-block}. The factor
\((-M_{j,J})_{r+t}\) makes the sum finite. After cancellation, the only
quotient involving both \(N\) and the lattice variable is
\[
 \frac{(k_N+\eta_{j,J})_{r+t}}
      {(k_N+\eta_{j,J})_r}
 \frac{(N+\eta_{j,J})_r}
      {(N+\eta_{j,J}+1)_{r+t}}
 \xrightarrow[N\to\infty]{}x^t.
\]
The convergence is uniform for \(x\) in each compact subset of \((0,1)\).
The prefactor tends to \eqref{eq:Jacobi-KdF-sector-prefactor}, and all
other parameters are fixed. Since the sum is finite, this proves
\eqref{eq:Hahn-Jacobi-KdF-sector-limit}. Summing over \(J\) in
\eqref{eq:Hahn-B-KdF-sum} proves
\eqref{eq:Hahn-Jacobi-KdF-component-limit}. The upper parameter
\(-M_{j,J}\) gives \(\deg S_{j,J}^{\mathrm{Jac}}\le M_{j,J}\).
If \(j=J\), then \(M_{j,J}=m_j-1\); if \(j\ne J\), near-diagonality gives
\(M_{j,J}=m_J-2\le m_j-1\). This proves the degree bound.
To make the coefficient passage explicit, put
\(\widehat B_{j,N}(x)=B_{j,N}(Nx)\). Choose \(m_j\) distinct points
\(x_1,\ldots,x_{m_j}\) in \((0,1)\). For all sufficiently large \(N\), the
grid points \(\lfloor Nx_a\rfloor/N\), \(1\le a\le m_j\), are distinct.
Equation~\eqref{eq:Hahn-Jacobi-KdF-component-limit}, evaluated at these
grid points, gives convergence of the corresponding \(m_j\) polynomial
values. The associated Vandermonde matrices converge to an invertible
matrix, so \(\widehat B_{j,N}\to B_j^{\mathrm{Jac}}\)
coefficientwise.
Together with the exact row moments in
Theorem~\ref{thm:Hahn-to-Jacobi-rows}, this permits termwise passage in
every fixed moment of the component decomposition.
Corollary~\ref{cor:Hahn-Jacobi-dual} then gives precisely
\eqref{eq:Jacobi-B-moments}. Under nonsingularity, these moments identify
\eqref{eq:Jacobi-KdF-components} as the unique tuple of type-I
polynomials. Finally,
\eqref{eq:Hahn-Jacobi-KdF-component-limit} and the local row limit
\eqref{eq:local-Hahn-Jacobi-limit}, applied to the finite component sum,
give \eqref{eq:Hahn-Jacobi-form-local-limit}.
\end{proof}

\paragraph{Classical Jacobi reduction.}
For \(q=1\), under the principal admissibility condition \(0<A<\beta\),
write \(m=m_1\ge1\) and \(n=m-1\). Then
\begin{align}
 A_m^{\mathrm J}(x)
 &=\pFq{2}{1}{-m,m+\beta}{A}{x}
 =\frac{m!}{(A)_m}
 P_m^{(A-1,\,\beta-A)}(1-2x),
 \label{eq:q1-Jacobi-A}\\
 B_1^{\mathrm{Jac}}(x)
 &=-\frac{(A)_n}{(\beta-A+1)_n}
 \pFq{2}{1}{-n,n+\beta}{A}{x}
 =-\frac{n!}{(\beta-A+1)_n}
 P_n^{(A-1,\,\beta-A)}(1-2x).
 \label{eq:q1-Jacobi-B}
\end{align}
Here \(B_1^{\mathrm{Jac}}\) is the component relative to
\(\widetilde\omega_1=\omega_1/\beta\); relative to the original
probability weight it is \(B_1^{\mathrm{Jac}}/\beta\). Thus the two
polynomials reduce to the ordinary Jacobi polynomials of degrees \(m\) and
\(m-1\), respectively. Positivity of the scalar Jacobi weight also proves
normality throughout the admissible scalar parameter range.

\section{Finite reconstruction of the unreflected type-I components}
\label{sec:explicit-unreflected-B}

\subsection{Rational moment reconstruction}
\label{subsec:unreflected-rational-reconstruction}

We establish first the finite algebraic reconstruction shared by the
Kravchuk-like, Meixner-II-like, and Charlier-II-like families.  Separating
the common argument from the three applications avoids repetition.  The formulas
contain only terminating generalized hypergeometric polynomials and finite
Horn sums.

The families use different weight normalizations. Fix a near-diagonal
multi-index \(\mm\) with \(\abs\mm\ge1\); in the Kravchuk case assume
also \(\abs\mm\le N+1\). Throughout this subsection,
\(X\in\{\mathrm K,\mathrm{MII},\mathrm{CII}\}\),
\(j\in\{1,\ldots,q\}\), and \(k,r,\ell,u\in\Nzero\), unless a smaller
range is displayed. Here
\(D_j\) denotes a component relative to a rescaled weight in the
Kravchuk-, Meixner-II-, Charlier-II-, and Laguerre-I-like formulas; the
conversion to the original weight is stated in each case. The symbols
\(B_j\) in the Hahn-, Meixner-I-, Charlier-I-, and Jacobi-like formulas
retain the normalizations fixed in the preceding sections. Put
\(n=\sum_{j=1}^q m_j-1\),
\(D_-(r)=\prod_{h=1}^{q-1}(r+\beta_h)_{m_h}\), and
\(\mathcal R_{\mm}(r)=-(-r)_n/D_-(r)\).
Here \(v_j^X\) and \(\mathcal B_{\mm}^X\) denote the weights and signed
sequences specified in Theorems~\ref{thm:K-rows}--\ref{thm:K-forms},
\ref{thm:MII-rows}--\ref{thm:MII-forms}, and
\ref{thm:CII-system}, respectively. For the
Kravchuk-like family, suppress the fixed lattice size by writing
\(v_j^{\mathrm K}\coloneq v_{j,N}^{\mathrm K}\) and
\(\mathcal B_{\mm}^{\mathrm K}\coloneq\mathcal B_{\mm,N}^{\mathrm K}\), and
extend both sequences by zero to \(k>N\).
For each \(X\in\{\mathrm K,\mathrm{MII},\mathrm{CII}\}\), define the
rescaled weights
\begin{equation}
 w_j^X=\frac{v_j^X}{\beta_j}\quad
 (j\in\{1,\ldots,q-1\}),
 \qquad w_q^X=v_q^X.
 \label{eq:unreflected-canonical-rows}
\end{equation}
Their factorial moments are
\begin{equation}
 \sum_{k\ge0}\fall{k}{r}w_j^X(k)=
 \begin{cases}
 h_X(r)/(r+\beta_j),&j\in\{1,\ldots,q-1\},\\
 h_X(r),&j=q,
 \end{cases}
 \label{eq:unreflected-canonical-moments}
\end{equation}
where
\(h_{\mathrm K}(r)=\fall{N}{r}c^r(\A_{<q})_r/(\bbeta_{<q})_r\),
\(h_{\mathrm{MII}}(r)=\tau^r(\A)_r/(\bbeta_{<q})_r\), and
\(h_{\mathrm{CII}}(r)=a^r(\A_{<q})_r/(\bbeta_{<q})_r\).
Whenever a normalized identity below contains \(h_{\mathrm K}(r)^{-1}\),
it is asserted for \(r\in\{0,\ldots,N\}\). This is the entire range needed
later, since \(n\le N\); the corresponding unnormalized moment identities
remain valid for every \(r\in\Nzero\).

The following rational functions encode the factorial moments of polynomial
multiples of these weights:
\begin{align}
 R_{j,\ell}^X(r)
 &=\sum_{u=0}^{\ell}\binom\ell u\fall{r}{\ell-u}
   \frac{\Theta_u^X(r)}{r+\beta_j+u},&&
   j\in\{1,\ldots,q-1\},
 \label{eq:unreflected-R-finite-row}\\
 R_{q,\ell}^X(r)
 &=\sum_{u=0}^{\ell}\binom\ell u\fall{r}{\ell-u}
   \Theta_u^X(r),
 \label{eq:unreflected-R-base-row}
\end{align}
with
\(\Theta_u^{\mathrm K}(r)=c^u\fall{N-r}{u}
(r+\A_{<q})_u/(r+\bbeta_{<q})_u\),
\(\Theta_u^{\mathrm{MII}}(r)=\tau^u(r+\A)_u/(r+\bbeta_{<q})_u\), and
\(\Theta_u^{\mathrm{CII}}(r)=a^u(r+\A_{<q})_u/(r+\bbeta_{<q})_u\).
Thus
\begin{equation}
 \frac1{h_X(r)}\sum_{k\ge0}\fall{k}{r}\fall{k}{\ell}w_j^X(k)
 =R_{j,\ell}^X(r).
 \label{eq:unreflected-R-pairing}
\end{equation}
For example, in the Meixner-II-like case these functions are
\begin{align*}
 R_{j,\ell}^{\mathrm{MII}}(r)
 &=\frac{\fall{r}{\ell}}{r+\beta_j}
 \pFq{q+1}{q}
 {-\ell,r+\A}
 {r-\ell+1,r+\beta_j+1,r+\bbeta_{<q}^{*j}}{-\tau},
 &&j\in\{1,\ldots,q-1\},\\
 R_{q,\ell}^{\mathrm{MII}}(r)
 &=\fall{r}{\ell}
 \pFq{q+1}{q}
 {-\ell,r+\A}
 {r-\ell+1,r+\bbeta_{<q}}{-\tau}.
\end{align*}
The corresponding Kravchuk formulas are obtained by adjoining the upper
parameter \(r-N\) and replacing \(-\tau\) by \(c\); the
Charlier-II-like formulas are obtained by deleting \(A_q\) and
replacing \(-\tau\) by \(-a\).  The finite sums
\eqref{eq:unreflected-R-finite-row}--\eqref{eq:unreflected-R-base-row}
define them at removable exceptional values.

\subsubsection{The partial-fraction contributions}

Let \(J\in\{1,\ldots,q-1\}\), \(K\in\Nzero\) with \(K<m_J\), and
\(j\in\{1,\ldots,q\}\), and put
\(\xi_{J,K}=\beta_J+K\) and
\(\varepsilon_{j,J}=1-\delta_{j,J}\).
Write \(\ee_j^-\) for the \(j\)-th coordinate vector in
\(\mathbb R^{q-1}\) when \(j\in\{1,\ldots,q-1\}\), and put
\(\ee_q^-=\boldsymbol0\).
Set \(\mathcal C_{j,J,K}^X=0\) when
\(K<\varepsilon_{j,J}\).  Otherwise define
\begin{align}
 \mathcal C_{j,J,K}^{\mathrm K}(k)
 &=\Delta_{j,J,K}^{\mathrm K}
 \pFq{q+1}{q}
 {1,1-k-\xi_{J,K},
  \bbeta_{<q}+(1-\xi_{J,K})\one_{q-1}-\ee_j^-}
 {1-N-\xi_{J,K},
  \A_{<q}+(1-\xi_{J,K})\one_{q-1}}{c^{-1}},
 \label{eq:K-finite-pole-block}\\
 \mathcal C_{j,J,K}^{\mathrm{MII}}(k)
 &=\Delta_{j,J,K}^{\mathrm{MII}}
 \pFq{q+1}{q}
 {1,1-k-\xi_{J,K},
  \bbeta_{<q}+(1-\xi_{J,K})\one_{q-1}-\ee_j^-}
 {\A+(1-\xi_{J,K})\one_q}{-\tau^{-1}},
 \label{eq:MII-finite-pole-block}\\
 \mathcal C_{j,J,K}^{\mathrm{CII}}(k)
 &=\Delta_{j,J,K}^{\mathrm{CII}}
 \pFq{q+1}{q-1}
 {1,1-k-\xi_{J,K},
  \bbeta_{<q}+(1-\xi_{J,K})\one_{q-1}-\ee_j^-}
 {\A_{<q}+(1-\xi_{J,K})\one_{q-1}}{-a^{-1}}.
 \label{eq:CII-finite-pole-block}
\end{align}
When the blocks are nonzero, all three hypergeometric series terminate. As
polynomials in \(k\), their degrees are at most \(K\) for \(j=J\) and at
most \(K-1\) for \(j\ne J\). Their prefactors are, for
\(j\in\{1,\ldots,q-1\}\),
\begin{align}
 \Delta_{j,J,K}^{\mathrm K}
 &=\frac{(N+\beta_j)
   \prod_{h=1}^{q-1}(A_h-\beta_j)
   \prod_{\substack{1\le h\le q-1\\h\ne j}}(\beta_h-\xi_{J,K})}
  {(N+\xi_{J,K})
   \prod_{\substack{1\le h\le q-1\\h\ne j}}(\beta_h-\beta_j)
   \prod_{h=1}^{q-1}(A_h-\xi_{J,K})},
 \label{eq:K-finite-pole-prefactor-row}\\
 \Delta_{j,J,K}^{\mathrm{MII}}
 &=\frac{\prod_{h=1}^{q}(A_h-\beta_j)
   \prod_{\substack{1\le h\le q-1\\h\ne j}}(\beta_h-\xi_{J,K})}
  {\prod_{\substack{1\le h\le q-1\\h\ne j}}(\beta_h-\beta_j)
   \prod_{h=1}^{q}(A_h-\xi_{J,K})},
 \label{eq:MII-finite-pole-prefactor-row}\\
 \Delta_{j,J,K}^{\mathrm{CII}}
 &=\frac{\prod_{h=1}^{q-1}(A_h-\beta_j)
   \prod_{\substack{1\le h\le q-1\\h\ne j}}(\beta_h-\xi_{J,K})}
  {\prod_{\substack{1\le h\le q-1\\h\ne j}}(\beta_h-\beta_j)
   \prod_{h=1}^{q-1}(A_h-\xi_{J,K})}.
 \label{eq:CII-finite-pole-prefactor-row}
\end{align}
For the \(q\)-th weight they are
\begin{align}
 \Delta_{q,J,K}^{\mathrm K}
 &=(1-c^{-1})
   \frac{\prod_{h=1}^{q-1}(\beta_h-\xi_{J,K})}
   {(N+\xi_{J,K})\prod_{h=1}^{q-1}(A_h-\xi_{J,K})},
 \label{eq:K-finite-pole-prefactor-base}\\
 \Delta_{q,J,K}^{\mathrm{MII}}
 &=-\frac{1+\tau}{\tau}
   \frac{\prod_{h=1}^{q-1}(\beta_h-\xi_{J,K})}
   {\prod_{h=1}^{q}(A_h-\xi_{J,K})},
 \label{eq:MII-finite-pole-prefactor-base}\\
 \Delta_{q,J,K}^{\mathrm{CII}}
 &=-\frac1a
   \frac{\prod_{h=1}^{q-1}(\beta_h-\xi_{J,K})}
   {\prod_{h=1}^{q-1}(A_h-\xi_{J,K})}.
 \label{eq:CII-finite-pole-prefactor-base}
\end{align}
In the diagonal case \(j=J\) and \(K=0\), the prefactor is interpreted by
continuation as one, while the hypergeometric expression is its
degree-zero finite sum and hence also equals one. For every
\(J\in\{1,\ldots,q-1\}\) and \(K\in\Nzero\) with \(K<m_J\), these
polynomials satisfy
\begin{equation}
 \frac1{r+\xi_{J,K}}
 =\sum_{j=1}^q\frac1{h_X(r)}
   \sum_{k\ge0}\fall{k}{r}\mathcal C_{j,J,K}^X(k)w_j^X(k).
 \label{eq:finite-pole-block-pairing}
\end{equation}

For every \(J\in\{1,\ldots,q-1\}\) and \(K\in\Nzero\) with
\(K<m_J\), the coefficients of the simple-pole part of
\(\mathcal R_{\mm}\) are
\begin{equation}
 \pi_{J,K}^{-}=
 \frac{(-1)^{K+1}(\beta_J+K)_n}
 {K!(m_J-1-K)!
  \prod_{\substack{1\le h\le q-1\\h\ne J}}
  (\beta_h-\beta_J-K)_{m_h}}.
 \label{eq:unreflected-finite-pole-coefficient}
\end{equation}

\subsubsection{The polynomial contribution at infinity}

For a scalar \(x\) and \(u\in\Nzero\), let
\([x]_u=(x,x+1,\ldots,x+u-1)\), with \([x]_0\) empty, and put
\(\mathsf f_0\) equal to the empty list and
\(\mathsf f_s=(0,-1,\ldots,1-s)\) for \(s\ge1\). Repeated entries are
retained. For finite lists
\(\boldsymbol\xi=(\xi_1,\ldots,\xi_p)\) and
\(\boldsymbol\eta=(\eta_1,\ldots,\eta_t)\), and for
\(\nu\in\mathbb Z\), define
\begin{equation}
 \mathfrak h_\nu(\boldsymbol\xi;\boldsymbol\eta)
 =[z^\nu]\frac{\prod_{\xi\in\boldsymbol\xi}(1+\xi z)}
 {\prod_{\eta\in\boldsymbol\eta}(1+\eta z)},
 \qquad \mathfrak h_\nu=0\quad(\nu<0).
 \label{eq:unreflected-Horn-coefficient}
\end{equation}
Equivalently,
\begin{equation}
 \mathfrak h_\nu(\boldsymbol\xi;\boldsymbol\eta)
 =\sum_{\substack{\boldsymbol\epsilon\in\{0,1\}^{p},
                   \,\boldsymbol\kappa\in\Nzero^{t}\\
                   \sum_{i=1}^p\epsilon_i+\sum_{i=1}^t\kappa_i=\nu}}
 (-1)^{\sum_{i=1}^t\kappa_i}
 \prod_{i=1}^{p}\xi_i^{\epsilon_i}
 \prod_{i=1}^{t}\eta_i^{\kappa_i}.
 \label{eq:unreflected-Horn-finite-sum}
\end{equation}
Thus \(\mathfrak h_\nu\) is a terminating multivariate Horn sum.

For \(\ell\in\Nzero\) and \(u\in\{0,\ldots,\ell\}\), introduce
\begin{align*}
 \boldsymbol\Xi_{\ell,u}^{\mathrm{MII}}
 &=(\mathsf f_{\ell-u},[A_1]_u,\ldots,[A_q]_u),
 &\sigma_u^{\mathrm{MII}}&=\tau^u,\\
 \boldsymbol\Xi_{\ell,u}^{\mathrm{CII}}
 &=(\mathsf f_{\ell-u},[A_1]_u,\ldots,[A_{q-1}]_u),
 &\sigma_u^{\mathrm{CII}}&=a^u,\\
 \boldsymbol\Xi_{\ell,u}^{\mathrm K}
 &=(\mathsf f_{\ell-u},[-N]_u,[A_1]_u,\ldots,[A_{q-1}]_u),
 &\sigma_u^{\mathrm K}&=(-c)^u,\\
 \boldsymbol\eta_u&=([\beta_1]_u,\ldots,[\beta_{q-1}]_u).
\end{align*}
Here parentheses denote concatenation of the displayed finite lists. Set
\(d_{\ell,u}^{\mathrm K}=d_{\ell,u}^{\mathrm{MII}}=\ell\) and
\(d_{\ell,u}^{\mathrm{CII}}=\ell-u\).
For \(K,\ell\in\Nzero\), the coefficient of \(r^K\) in the polynomial part at infinity of
\(R_{q,\ell}^X(r)\) is the explicit Horn sum
\begin{equation}
 p_{K,\ell}^X=
 \sum_{u=0}^{\ell}\binom\ell u\sigma_u^X
 \mathfrak h_{d_{\ell,u}^X-K}
 (\boldsymbol\Xi_{\ell,u}^X;\boldsymbol\eta_u).
 \label{eq:unreflected-infinity-transition}
\end{equation}
It is upper triangular in the sense that \(p_{K,\ell}^X=0\) for
\(K>\ell\), and
\begin{equation}
 p_{\ell,\ell}^{\mathrm K}=(1-c)^\ell,
 \qquad
 p_{\ell,\ell}^{\mathrm{MII}}=(1+\tau)^\ell,
 \qquad
 p_{\ell,\ell}^{\mathrm{CII}}=1.
 \label{eq:unreflected-infinity-pivots}
\end{equation}

For the residues of \(R_{q,\ell}^X\) at the finite poles, let
\(J\in\{1,\ldots,q-1\}\), \(K\in\Nzero\) with \(K<m_J\), and
\(u,\ell\in\Nzero\), and put
\begin{align*}
 W_{J,K,u}^{\mathrm K}
 &=c^u\fall{N+\beta_J+K}{u}
   \prod_{g=1}^{q-1}(A_g-\beta_J-K)_u,\\
 W_{J,K,u}^{\mathrm{MII}}
 &=\tau^u\prod_{g=1}^{q}(A_g-\beta_J-K)_u,\\
 W_{J,K,u}^{\mathrm{CII}}
 &=a^u\prod_{g=1}^{q-1}(A_g-\beta_J-K)_u,\\
 Q_{J,K}(u)
 &=\prod_{\substack{1\le h\le q-1\\h\ne J}}
   (\beta_h-\beta_J-K)_u.
\end{align*}
Then
\begin{equation}
 z_{J,K,\ell}^X=
 \sum_{u=K+1}^{\ell}\binom\ell u
 \fall{-\beta_J-K}{\ell-u}
 \frac{W_{J,K,u}^X}
 {(-1)^K K!(u-K-1)!Q_{J,K}(u)}
 \label{eq:unreflected-base-residue}
\end{equation}
is \(\Res_{r=-\beta_J-K}R_{q,\ell}^X(r)\); an empty sum is zero.
Consequently, the component vector
\begin{equation}
 \mathcal E_{j,\ell}^X(k)
 =\delta_{j,q}\fall{k}{\ell}
 -\sum_{J=1}^{q-1}\sum_{K=0}^{\min(m_J-1,\ell-1)}
 z_{J,K,\ell}^X\mathcal C_{j,J,K}^X(k)
 \label{eq:unreflected-corrected-infinity-vector}
\end{equation}
has the normalized factorial-moment identity
\begin{equation}
 \sum_{j=1}^q\frac1{h_X(r)}
 \sum_{k\ge0}\fall{k}{r}\mathcal E_{j,\ell}^X(k)w_j^X(k)
 =\sum_{K=0}^{\ell}p_{K,\ell}^Xr^K.
 \label{eq:unreflected-corrected-infinity-pairing}
\end{equation}

The polynomial part of \(\mathcal R_{\mm}\) is
\(\sum_{K=0}^{m_q-1}\pi_{\infty,K}r^K\), where, for every
\(K\in\Nzero\) with \(K<m_q\),
\begin{equation}
 \pi_{\infty,K}=(-1)^{n+1}
 \mathfrak h_{m_q-1-K}
 \left(\mathsf f_n;
 ([\beta_1]_{m_1},\ldots,[\beta_{q-1}]_{m_{q-1}})\right).
 \label{eq:unreflected-target-infinity-coefficients}
\end{equation}
For \(\ell\in\Nzero\) with \(\ell<m_q\), the triangular system
\eqref{eq:unreflected-corrected-infinity-pairing} has the explicit inverse
\begin{equation}
 \gamma_\ell^X=
 \sum_{s=0}^{m_q-1-\ell}(-1)^s
 \sum_{\ell=\ell_0<\ell_1<\cdots<\ell_s\le m_q-1}
 \frac{\pi_{\infty,\ell_s}}{p_{\ell_s,\ell_s}^X}
 \prod_{t=0}^{s-1}
 \frac{p_{\ell_t,\ell_{t+1}}^X}
      {p_{\ell_t,\ell_t}^X}.
 \label{eq:unreflected-infinity-chain}
\end{equation}
For \(s=0\), the inner sum contains only \(\ell_0=\ell\) and the product is
one.  Formula \eqref{eq:unreflected-infinity-chain} is therefore a finite
sum of products of the terminating Horn polynomials
\eqref{eq:unreflected-infinity-transition}.

For every \(X\in\{\mathrm K,\mathrm{MII},\mathrm{CII}\}\) and
\(j\in\{1,\ldots,q\}\), set
\begin{equation}
 D_j^X(k)=
 \sum_{J=1}^{q-1}\sum_{K=0}^{m_J-1}
 \pi_{J,K}^{-}\mathcal C_{j,J,K}^X(k)
 +\sum_{\ell=0}^{m_q-1}
 \gamma_\ell^X\mathcal E_{j,\ell}^X(k).
 \label{eq:unreflected-explicit-components}
\end{equation}
For the original weights, set
\begin{equation}
 B_j^X=D_j^X/\beta_j,\quad j\in\{1,\ldots,q-1\},
 \qquad B_q^X=D_q^X.
 \label{eq:unreflected-probability-components}
\end{equation}

\begin{theorem}[Kravchuk-, Meixner-II-, and Charlier-II-like type-I components]
\label{thm:explicit-unreflected-B}
Let \(\mm\) be near the diagonal with \(\abs\mm\ge1\), and set
\(n=\abs\mm-1\). In the Kravchuk case assume
\(\sum_{j=1}^q m_j\le N+1\). Assume that the sets
\(\{-\beta_J-K:K\in\Nzero,\ K<m_J\}\), indexed by
\(J\in\{1,\ldots,q-1\}\), are pairwise disjoint and that
the denominators in
\eqref{eq:K-finite-pole-block}--\eqref{eq:CII-finite-pole-prefactor-base},
\eqref{eq:unreflected-finite-pole-coefficient}, and
\eqref{eq:unreflected-base-residue} do not vanish at any of their indicated
indices, apart from the removable diagonal cases with \(K=0\). Then, for
every \(X\in\{\mathrm K,\mathrm{MII},\mathrm{CII}\}\),
\(\deg D_j^X<m_j\) for every
\(j\in\{1,\ldots,q\}\) with \(m_j>0\), and
\begin{equation}
 \mathcal B_{\mm}^X(k)=\sum_{j=1}^qD_j^X(k)w_j^X(k),
 \qquad k\in\Nzero.
 \label{eq:unreflected-component-decomposition}
\end{equation}
For every \(r\in\Nzero\), with the additional restriction \(r\le N\)
when \(X=\mathrm K\),
\begin{equation}
 \sum_{k\ge0}r!\binom{k}{r}\mathcal B_{\mm}^X(k)
 =-h_X(r)\frac{(-r)_n}{D_-(r)}.
 \label{eq:unreflected-complete-moments}
\end{equation}
The moments below order \(n\) vanish, whereas the right-hand side of
\eqref{eq:unreflected-complete-moments} at \(r=n\) is nonzero. Moreover,
the type-I moment matrix is nonsingular. Consequently, \(\mm\) is normal
in the sense of Definition~\ref{def:normal-multi-index}, and
\((D_1^X,\ldots,D_q^X)\) is the unique tuple of type-I polynomials with
the stated degree bounds relative to the rescaled weights
\eqref{eq:unreflected-canonical-rows}. The components relative to the
original weights are given by \eqref{eq:unreflected-probability-components}.
\end{theorem}

\begin{proof}
We first derive the three normalized factorial-moment functions. The falling
factorial product identity
\(\fall{k}{r}\fall{k}{\ell}
=\sum_{u=0}^{\ell}\binom\ell u
\fall{r}{\ell-u}\fall{k}{r+u}\)
contains only \(\ell+1\) terms. The required moment quotients are
\(h_X(r+u)/h_X(r)=\Theta_u^X(r)\).
Substitution in \eqref{eq:unreflected-canonical-moments} gives
\eqref{eq:unreflected-R-finite-row}--
\eqref{eq:unreflected-R-base-row}, and hence
\eqref{eq:unreflected-R-pairing}, for all three families.

We next establish the simple-pole moment identity. For every
\(J\in\{1,\ldots,q-1\}\) and \(K\in\Nzero\) with \(K<m_J\), the
Cauchy-transform recurrence for arbitrary Hahn parameters in the proof of
Corollary~\ref{cor:Hahn-B-closed-components} gives
\begin{equation}
 \sum_{j=1}^q\frac1{h_N(r)}
 \sum_{k=0}^N\fall{k}{r}
 \mathcal C_{j,J,K}^{(N)}(k)
 \widetilde v_{j,N}(k)
 =\frac1{r+\xi_{J,K}}.
 \label{eq:unreflected-Hahn-pairing}
\end{equation}
Indeed, \(\mathcal C_{\cdot,J,K}^{(N)}\) is the coefficient vector of
\(I_{\xi_{J,K}}\) in the iterated relation
\eqref{eq:Hahn-Cauchy-contiguous-recurrence}; the polynomial remainder
has no residue at a lattice point, and the factorial moments of
\(I_{\xi_{J,K}}\) are
\(h_N(r)/(r+\xi_{J,K})\).

For the partial Hahn-to-Kravchuk limit
\eqref{eq:K-scaling}, with \(J\in\{1,\ldots,q-1\}\) and
\(K\in\Nzero\) satisfying \(K<m_J\), every summand of the terminating
hypergeometric polynomial has the factor
\(\frac{(\beta_q^{(T)}+1-\xi_{J,K}-\delta_{j,q})_s}
{(A_q^{(T)}+1-\xi_{J,K})_s}
\xrightarrow[T\to\infty]{}c^{-s}\).
The prefactor has the finite limit
\[
 \left.
 (\beta_q^{(T)})^{-\delta_{j,q}}
 \mathcal D_{j,J,K}^{(N)}
 \right|_{A_q=A_q^{(T)},\,\beta_q=\beta_q^{(T)}}
 \xrightarrow[T\to\infty]{}\Delta_{j,J,K}^{\mathrm K};
\]
for \(j\in\{1,\ldots,q-1\}\) all factors containing \(A_q^{(T)}\) or
\(\beta_q^{(T)}\) cancel in ratios, whereas for \(j=q\) the remaining
factor tends to
\(\lim_{T\to\infty}(N+\beta_q^{(T)})(A_q^{(T)}-\beta_q^{(T)})/
[\beta_q^{(T)}A_q^{(T)}]=1-c^{-1}\).
The factors with indices below \(q\) are precisely those in
\eqref{eq:K-finite-pole-prefactor-row} and
\eqref{eq:K-finite-pole-prefactor-base}.  Hence
\[
 \left.
 (\beta_q^{(T)})^{-\delta_{j,q}}
 \mathcal C_{j,J,K}^{(N)}(k)
 \right|_{A_q=A_q^{(T)},\,\beta_q=\beta_q^{(T)}}
 \xrightarrow[T\to\infty]{}\mathcal C_{j,J,K}^{\mathrm K}(k).
\]

For the Hahn-like-to-Meixner-II-like limit \eqref{eq:MII-scaling}, the factor that
contains the two diverging parameters is instead
\(\frac{(\beta_q^{(N)}+1-\xi_{J,K}-\delta_{j,q})_s}
{(1-N-\xi_{J,K})_s}
\xrightarrow[N\to\infty]{}(-\tau^{-1})^s\).
Moreover,
\[
 \left.
 (\beta_q^{(N)})^{-\delta_{j,q}}
 \mathcal D_{j,J,K}^{(N)}
 \right|_{\beta_q=\beta_q^{(N)}}
 \xrightarrow[N\to\infty]{}\Delta_{j,J,K}^{\mathrm{MII}}.
\]
For \(j=q\), the nonconstant part of this limit is
\[
 \lim_{N\to\infty}
 \frac1{\beta_q^{(N)}}
 \frac{(N+\beta_q^{(N)})
       \prod_{g=1}^{q}(A_g-\beta_q^{(N)})}
      {N\prod_{h=1}^{q-1}(\beta_h-\beta_q^{(N)})}
 =-\frac{1+\tau}{\tau};
\]
the factors independent of \(N\) yield the product quotient in
\eqref{eq:MII-finite-pole-prefactor-base}.  Consequently,
\[
 \left.
 (\beta_q^{(N)})^{-\delta_{j,q}}
 \mathcal C_{j,J,K}^{(N)}(k)
 \right|_{\beta_q=\beta_q^{(N)}}
 \xrightarrow[N\to\infty]{}\mathcal C_{j,J,K}^{\mathrm{MII}}(k).
\]

Finally, in the Meixner-II-like-to-Charlier-II-like scaling
\(A_q=R\), \(\tau=a/R\), the term of order \(s\) contains
\(\frac{(-\tau^{-1})^s}{(A_q+1-\xi_{J,K})_s}
=\frac{(-R/a)^s}{(R+1-\xi_{J,K})_s}
\xrightarrow[R\to\infty]{}(-a^{-1})^s\).
The prefactors indexed by \(j\in\{1,\ldots,q-1\}\) lose the quotient
\((R-\beta_j)/(R-\xi_{J,K})\), whose limit is one. For \(j=q\),
\(-\frac{1+\tau}{\tau(A_q-\xi_{J,K})}
=-\frac{R+a}{a(R-\xi_{J,K})}
\xrightarrow[R\to\infty]{}-\frac1a\).
These are \eqref{eq:CII-finite-pole-prefactor-row} and
\eqref{eq:CII-finite-pole-prefactor-base}, so
\(\mathcal C_{j,J,K}^{\mathrm{MII}}\) converges coefficientwise to
\(\mathcal C_{j,J,K}^{\mathrm{CII}}\).

All the series just considered terminate at \(s=K\) when \(j=J\) and
at \(s=K-1\) when \(j\ne J\).  The limits can therefore be taken in
each finite sum.  For each fixed \(r\), only finitely many factorial
moments enter \eqref{eq:unreflected-Hahn-pairing}. The first two limits
in \eqref{eq:unreflected-Hahn-pairing} and then the
Meixner-II-like-to-Charlier-II-like limit prove
\eqref{eq:finite-pole-block-pairing} for \(X=\mathrm K\),
\(X=\mathrm{MII}\), and \(X=\mathrm{CII}\), respectively. In the
removable case \(K=0\), the polynomial equals one in component \(J\) and
zero in every other component.

We now compute the coefficients required by the signed sequence. Along the
Kravchuk limit,
\[
 \frac{(\beta_q^{(T)})_{m_q}}{(r+\beta_q^{(T)})_{m_q}}
 \left(-\frac{(-r)_n}{D_-(r)}\right)
 \xrightarrow[T\to\infty]{}\mathcal R_{\mm}(r).
\]
Along the Hahn-like-to-Meixner-II-like limit,
\[
 \frac{(\beta_q^{(N)})_{m_q}}{(r+\beta_q^{(N)})_{m_q}}
 \left(-\frac{(-r)_n}{D_-(r)}\right)
 \xrightarrow[N\to\infty]{}\mathcal R_{\mm}(r).
\]
In the subsequent Charlier-II-like limit the
rational factor \(\mathcal R_{\mm}\) contains neither \(A_q\) nor \(\tau\), so
it is unchanged. At \(r=-\beta_J-K\),
\(\prod_{\substack{s\in\{0,\ldots,m_J-1\}\\s\ne K}}(s-K)
=(-1)^K K!(m_J-1-K)!\).
Evaluation of the remaining factors gives
\[
 \Res_{r=-\beta_J-K}\mathcal R_{\mm}(r)
 =\frac{(-1)^{K+1}(\beta_J+K)_n}
 {K!(m_J-1-K)!
  \prod_{\substack{1\le h\le q-1\\h\ne J}}
  (\beta_h-\beta_J-K)_{m_h}}
 =\pi_{J,K}^{-}.
\]
After subtracting these simple fractions, the remainder is a polynomial
of degree at most \(m_q-1\). To determine its coefficients, write
\(z=r^{-1}\) and use
\(-(-r)_n=(-1)^{n+1}r^n\prod_{s=0}^{n-1}(1-sz)\) and
\(D_-(r)=r^{\sum_{h=1}^{q-1}m_h}
\prod_{h=1}^{q-1}\prod_{s=0}^{m_h-1}(1+(\beta_h+s)z)\).
the coefficient of \(r^K\) in the polynomial part is
\[
 (-1)^{n+1}
 \mathfrak h_{m_q-1-K}
 \left(\mathsf f_n;
 ([\beta_1]_{m_1},\ldots,[\beta_{q-1}]_{m_{q-1}})\right)
 =\pi_{\infty,K}.
\]
We have therefore obtained the labeled identity
\begin{equation}
 \mathcal R_{\mm}(r)=
 \sum_{J=1}^{q-1}\sum_{K=0}^{m_J-1}
 \frac{\pi_{J,K}^{-}}{r+\beta_J+K}
 +\sum_{K=0}^{m_q-1}\pi_{\infty,K}r^K.
 \label{eq:unreflected-target-partial-fractions}
\end{equation}
Here every \(\mathfrak h_\nu\) is given by the finite formula
\eqref{eq:unreflected-Horn-finite-sum}.

It remains to construct component vectors for the polynomial part.  For
\(u>K\), the only singular factor in the \(u\)-th summand of
\(R_{q,\ell}^X\) at \(r=-\beta_J-K\) is \((r+\beta_J)_u^{-1}\), whose
residue is \(1/[(-1)^K K!(u-K-1)!]\).
Thus the residues for the three families are
\begin{align*}
 z_{J,K,\ell}^{\mathrm K}
 &=\sum_{u=K+1}^{\ell}\binom\ell u
 \fall{-\beta_J-K}{\ell-u}
 \frac{c^u\fall{N+\beta_J+K}{u}
       \prod_{g=1}^{q-1}(A_g-\beta_J-K)_u}
 {(-1)^K K!(u-K-1)!
  \prod_{\substack{1\le h\le q-1\\h\ne J}}
  (\beta_h-\beta_J-K)_u},\\
 z_{J,K,\ell}^{\mathrm{MII}}
 &=\sum_{u=K+1}^{\ell}\binom\ell u
 \fall{-\beta_J-K}{\ell-u}
 \frac{\tau^u\prod_{g=1}^{q}(A_g-\beta_J-K)_u}
 {(-1)^K K!(u-K-1)!
  \prod_{\substack{1\le h\le q-1\\h\ne J}}
  (\beta_h-\beta_J-K)_u},\\
 z_{J,K,\ell}^{\mathrm{CII}}
 &=\sum_{u=K+1}^{\ell}\binom\ell u
 \fall{-\beta_J-K}{\ell-u}
 \frac{a^u\prod_{g=1}^{q-1}(A_g-\beta_J-K)_u}
 {(-1)^K K!(u-K-1)!
  \prod_{\substack{1\le h\le q-1\\h\ne J}}
  (\beta_h-\beta_J-K)_u}.
\end{align*}
They are the three specializations of
\eqref{eq:unreflected-base-residue}.

For the expansion at infinity, the \(u\)-th summand of
\(R_{q,\ell}^X\) can be written as
\[
 \binom\ell u\sigma_u^X
 r^{d_{\ell,u}^X}
 \frac{\prod_{\xi\in\boldsymbol\Xi_{\ell,u}^X}(1+\xi/r)}
      {\prod_{\eta\in\boldsymbol\eta_u}(1+\eta/r)},
\]
where \(d_{\ell,u}^X\) was defined above. Its coefficient of \(r^K\) is
\(\binom\ell u\sigma_u^X
\mathfrak h_{d_{\ell,u}^X-K}
(\boldsymbol\Xi_{\ell,u}^X;\boldsymbol\eta_u)\).
Summing over \(u\) yields
\eqref{eq:unreflected-infinity-transition}.  At degree \(\ell\), all
values of \(u\) contribute in the Kravchuk-like and Meixner-II-like cases, whereas
only \(u=0\) contributes in the Charlier-II-like case. Hence
\(\sum_{u=0}^{\ell}\binom\ell u(-c)^u=(1-c)^\ell\),
\(\sum_{u=0}^{\ell}\binom\ell u\tau^u=(1+\tau)^\ell\), and
\(p_{\ell,\ell}^{\mathrm{CII}}=1\),
which proves all three assertions in
\eqref{eq:unreflected-infinity-pivots}.

Near-diagonality gives \(\ell\le m_q-1\le m_J\).  Therefore the poles
of \(R_{q,\ell}^X\) are exactly among
\(-\beta_J-K\), with
\(K\in\Nzero\) and \(K\le\min(m_J-1,\ell-1)\). Subtracting
\(z_{J,K,\ell}^X/(r+\xi_{J,K})\) for every such pole leaves the
polynomial part just computed.  Combining this observation with
\eqref{eq:finite-pole-block-pairing} proves
\eqref{eq:unreflected-corrected-infinity-pairing}.

For completeness, set
\(P_\ell^X(r)=\sum_{K=0}^{\ell}p_{K,\ell}^Xr^K\).  Since the diagonal
coefficient in each family is nonzero, the coefficients of an expansion
\(\sum_\ell\gamma_\ell^XP_\ell^X\) must obey
\begin{equation}
 \gamma_\ell^X
 =\frac{\pi_{\infty,\ell}}{p_{\ell,\ell}^X}
 -\frac1{p_{\ell,\ell}^X}
  \sum_{s=\ell+1}^{m_q-1}p_{\ell,s}^X\gamma_s^X.
 \label{eq:unreflected-infinity-recurrence}
\end{equation}
Every increasing index sequence of positive length in
\eqref{eq:unreflected-infinity-chain} has a unique second index \(s\).
Its remaining indices form a sequence occurring in \(\gamma_s^X\).
Grouping the sum by \(s\) turns
\eqref{eq:unreflected-infinity-chain} into
\eqref{eq:unreflected-infinity-recurrence}. Descending from
\(\ell=m_q-1\) proves
\(\sum_{s=\ell}^{m_q-1}p_{\ell,s}^X\gamma_s^X
=\pi_{\infty,\ell}\) for \(0\le\ell<m_q\).
Thus the second term in
\eqref{eq:unreflected-explicit-components} contributes
\(\sum_K\pi_{\infty,K}r^K\) to the normalized moment function, while the
first term contributes the simple fractions in
\eqref{eq:unreflected-target-partial-fractions}. Their sum gives
\(\mathcal R_{\mm}\), proving
\eqref{eq:unreflected-complete-moments}.

This moment identity also proves equality with the signed sequences
specified in the corresponding family sections.  In the Kravchuk case both sides are supported on
\(\{0,\ldots,N\}\), so their factorial moments of orders
\(0,\ldots,N\) determine every coefficient.  In the Meixner-II-like and
Charlier-II-like cases, multiplication of a weight by a polynomial
corresponds in its generating function to applying a polynomial in
\(z\partial_z\). The weight generating functions and the signed generating
functions are analytic in a neighborhood of \(z=1\).  Equation
\eqref{eq:unreflected-complete-moments} says that their Taylor
coefficients at \(z=1\) agree; the identity theorem gives equality of the
generating functions and hence of their coefficients.

The degree restrictions follow directly from the summands. A term
\(\mathcal C_{j,J,K}^X\) has degree at most \(K\) when \(j=J\) and degree
at most \(K-1\) otherwise. In \(\mathcal E_{j,\ell}^X\), the \(q\)-th component
contains the monomial of degree \(\ell<m_q\), and every correction has
\(K<\ell\). If a correction belongs to a component other than \(J\),
near-diagonality gives
\(K-1\le m_J-2\le m_j-1\).  Hence
\(\deg D_j^X<m_j\) for every \(j\in\{1,\ldots,q\}\). Since
\((-r)_n=0\) for every \(r\in\Nzero\) with \(r<n\), the first \(n\)
factorial moments
vanish. At \(r=n\),
\(-h_X(n)\frac{(-n)_n}{D_-(n)}
=(-1)^{n+1}\frac{n!h_X(n)}{D_-(n)}\ne0\), which is the asserted
normalization.

We finish with uniqueness. The space of tuples
\((P_1,\ldots,P_q)\) with \(\deg P_j<m_j\) has dimension
\(\sum_{j=1}^q m_j=n+1\). The component vectors
\[
 \{\mathcal C_{\cdot,J,K}^X:J\in\{1,\ldots,q-1\},\
 K\in\Nzero,\ K<m_J\}
 \ \cup\
 \{\mathcal E_{\cdot,\ell}^X:\ell\in\Nzero,\ \ell<m_q\}
\]
also number \(n+1\). Their normalized moment functions are the distinct simple
fractions \((r+\xi_{J,K})^{-1}\) and the polynomials
\(P_\ell^X\), whose leading coefficients are nonzero.  These rational
functions are linearly independent, so the displayed component vectors
form a basis of the component space.

Suppose a component vector in that space has factorial moments zero for
every \(r\in\{0,\ldots,n\}\), and let \(G(r)\) be its factorial-moment
function divided by \(h_X(r)\).
The formulas for \(R_{j,\ell}^X\) show that \(D_-(r)G(r)\) is a
polynomial. Its degree is at most \(n\): for
\(j\in\{1,\ldots,q-1\}\),
\(R_{j,\ell}^X(r)=\mathrm O(r^{\ell-1})\), whereas
\(R_{q,\ell}^X(r)=\mathrm O(r^\ell)\) as \(r\to\infty\), and
near-diagonality gives the required bounds on \(\ell\).  Since
\(D_-(r)h_X(r)\ne0\) for every \(r\in\{0,\ldots,n\}\), this polynomial has \(n+1\)
distinct zeros and is identically zero. Expansion in the preceding basis
forces every component coefficient to vanish. Thus the kernel of the
square type-I moment map is trivial, so its moment matrix is nonsingular.
By Definition~\ref{def:normal-multi-index}, \(\mm\) is normal and both the
normalized type-I and monic type-II problems are unique.
\end{proof}

\begin{corollary}[Component limits for the three families]
\label{cor:unreflected-component-limits}
Assume the hypotheses of Theorem~\ref{thm:explicit-unreflected-B}, and let
\(j\in\{1,\ldots,q\}\) with \(m_j>0\).
Under the Kravchuk scaling \eqref{eq:K-scaling}, with \(N\) fixed,
\begin{equation}
 \frac{(\beta_q^{(T)})_{m_q}}
      {(\beta_q^{(T)})^{\delta_{j,q}}}
 \left.B_{j,N}(k)\right|_{
 A_q=A_q^{(T)},\,\beta_q=\beta_q^{(T)}}
 \xrightarrow[T\to\infty]{}D_j^{\mathrm K}(k).
 \label{eq:Hahn-K-component-limit}
\end{equation}
Under the Meixner-II-like scaling \eqref{eq:MII-scaling},
\begin{equation}
 \frac{(\beta_q^{(N)})_{m_q}}
      {(\beta_q^{(N)})^{\delta_{j,q}}}
 \left.B_{j,N}(k)\right|_{\beta_q=\beta_q^{(N)}}
 \xrightarrow[N\to\infty]{}D_j^{\mathrm{MII}}(k).
 \label{eq:Hahn-MII-component-limit}
\end{equation}
Finally, under the Charlier-II-like scaling \eqref{eq:CII-scaling},
\begin{equation}
 D_j^{\mathrm{MII}}(k)\xrightarrow[R\to\infty]{}D_j^{\mathrm{CII}}(k).
 \label{eq:MII-CII-component-limit}
\end{equation}
All three limits are coefficientwise in the polynomial variable and,
equivalently, locally uniform on compact subsets of \(\mathbb C\).
\end{corollary}

\begin{proof}
We finally identify each component limit from Hahn. Along the
Kravchuk path, \(A_q^{(T)}=Ta\), \(\beta_q^{(T)}=Tb\), and \(N\) is fixed.
For \(J\in\{1,\ldots,q-1\}\) and \(K\in\Nzero\) with \(K<m_J\),
equation \eqref{eq:target-residue} gives
\[
 (\beta_q^{(T)})_{m_q}
 \left.\pi_{J,K}\right|_{\beta_q=\beta_q^{(T)}}
 =\pi_{J,K}^{-}
 \frac{(\beta_q^{(T)})_{m_q}}
      {(\beta_q^{(T)}-\xi_{J,K})_{m_q}}
 \xrightarrow[T\to\infty]{}\pi_{J,K}^{-}.
\]
Together with the Kravchuk fixed-pole limit, this proves convergence of
every fixed-pole term after multiplication of the \(j\)-th Hahn component
by \((\beta_q^{(T)})_{m_q}/
(\beta_q^{(T)})^{\delta_{j,q}}\).
The combined \(J=q\) contribution has no finite pole after passage to the
limit. Its rational coordinates are the family
\(\bigl(\pi_{\infty,\ell}\bigr)_{\ell\in\Nzero,\,\ell<m_q}\), calculated
above from the expansion at infinity. Recursion
\eqref{eq:unreflected-infinity-recurrence} gives its unique coordinates
\(\bigl(\gamma_\ell^{\mathrm K}\bigr)_{\ell\in\Nzero,\,\ell<m_q}\) in the
polynomial basis \(\mathcal E_{\cdot,\ell}^{\mathrm K}\). This proves
\eqref{eq:Hahn-K-component-limit}.

Along the Meixner-II-like path,
\(\beta_q^{(N)}=A_q+N/\tau\) and \(N\to\infty\).
For \(J\in\{1,\ldots,q-1\}\) and \(K\in\Nzero\) with \(K<m_J\),
\[
 (\beta_q^{(N)})_{m_q}
 \left.\pi_{J,K}\right|_{\beta_q=\beta_q^{(N)}}
 =\pi_{J,K}^{-}
 \frac{(\beta_q^{(N)})_{m_q}}
      {(\beta_q^{(N)}-\xi_{J,K})_{m_q}}
 \xrightarrow[N\to\infty]{}\pi_{J,K}^{-}.
\]
The Meixner-II-like fixed-pole limit and the triangular inversion for the
polynomial part prove \eqref{eq:Hahn-MII-component-limit}.

Under \(A_q=R\), \(\tau=a/R\), each residue
\(z_{J,K,\ell}^{\mathrm{MII}}\) tends to
\(z_{J,K,\ell}^{\mathrm{CII}}\), because
\(\tau^u(A_q-\xi_{J,K})_u
=\left(\frac aR\right)^u(R-\xi_{J,K})_u
 \xrightarrow[R\to\infty]{}a^u\).
The corresponding identity for every coefficient
\(p_{K,\ell}^{\mathrm{MII}}\) follows term by term from
\eqref{eq:unreflected-infinity-transition}, and its diagonal coefficient tends
from \((1+\tau)^\ell\) to one. Descending recurrence
\eqref{eq:unreflected-infinity-recurrence} therefore
gives \(\gamma_\ell^{\mathrm{MII}}\xrightarrow[R\to\infty]{}
\gamma_\ell^{\mathrm{CII}}\), which proves
\eqref{eq:MII-CII-component-limit}.
\end{proof}

In \eqref{eq:unreflected-explicit-components}, the first sum contains the
terms associated with the finite poles. The second represents the
polynomial part at infinity and is a finite triangular combination of
terminating Horn polynomials.

\subsection{Compact Kamp\'e de F\'eriet blocks for the finite poles}

The double sum over the finite-pole labels can be regrouped before any
limit is taken.  This retains the compact structure of the Hahn-like
components.  Fix $J\in\{1,\ldots,q-1\}$, put
\(\varepsilon=1-\delta_{j,J}\), and assume
\(m_J\ge\varepsilon+1\).  Define
\begin{align*}
 \boldsymbol a_{j,J}^{-}(k)
 &\coloneq\left(
 \beta_J+\varepsilon+n,\ \varepsilon+1-m_J,\
 k+\beta_J+\varepsilon,
 \{\beta_J-\beta_h+\varepsilon+1-m_h\}_{
                 1\le h<q,\ h\ne J}\right),\\
 \boldsymbol d_{j,J}^{-}(k)
 &\coloneq\left(k+\beta_J+\varepsilon,
 \{\beta_J-\beta_h+\varepsilon+\delta_{h,j}\}_{
                 1\le h<q,\ h\ne J}\right),\\
 \boldsymbol b_{j,J}^{\mathrm K}
 &\coloneq\left(N+\beta_J+\varepsilon,
 \{\beta_J-A_h+\varepsilon\}_{1\le h<q}\right),\\
 \boldsymbol c_{j,J}^{\mathrm K}
 &\coloneq\left(\beta_J+\varepsilon,
 N+\beta_J+\varepsilon+1,
 \{\beta_J-A_h+\varepsilon+1\}_{1\le h<q}\right),\\
 \boldsymbol b_{j,J}^{\mathrm{MII}}
 &\coloneq\left(\{\beta_J-A_h+\varepsilon\}_{1\le h\le q}\right),\\
 \boldsymbol c_{j,J}^{\mathrm{MII}}
 &\coloneq\left(\beta_J+\varepsilon,
 \{\beta_J-A_h+\varepsilon+1\}_{1\le h\le q}\right),\\
 \boldsymbol b_{j,J}^{\mathrm{CII}}
 &\coloneq\left(\{\beta_J-A_h+\varepsilon\}_{1\le h<q}\right),\\
 \boldsymbol c_{j,J}^{\mathrm{CII}}
 &\coloneq\left(\beta_J+\varepsilon,
 \{\beta_J-A_h+\varepsilon+1\}_{1\le h<q}\right).
\end{align*}
Set
\begin{align}
 \mathscr K_{j,J}(k)
 &\coloneq \pi_{J,\varepsilon}^{-}
 \Delta_{j,J,\varepsilon}^{\mathrm K}
 F_{q+1:q-1;0}^{q+1:q;1}
 \left(
 \begin{matrix}
  \boldsymbol a_{j,J}^{-}(k):
  \boldsymbol b_{j,J}^{\mathrm K};1\\
  \boldsymbol c_{j,J}^{\mathrm K}:
  \boldsymbol d_{j,J}^{-}(k);\text{--}
 \end{matrix}
 \middle|1,c^{-1}\right),
 \label{eq:K-compact-finite-pole-block}\\
 \mathscr M_{j,J}(k)
 &\coloneq \pi_{J,\varepsilon}^{-}
 \Delta_{j,J,\varepsilon}^{\mathrm{MII}}
 F_{q+1:q-1;0}^{q+1:q;1}
 \left(
 \begin{matrix}
  \boldsymbol a_{j,J}^{-}(k):
  \boldsymbol b_{j,J}^{\mathrm{MII}};1\\
  \boldsymbol c_{j,J}^{\mathrm{MII}}:
  \boldsymbol d_{j,J}^{-}(k);\text{--}
 \end{matrix}
 \middle|1,-\tau^{-1}\right).
 \label{eq:MII-compact-finite-pole-block}
\end{align}
\begin{equation}
 \mathscr C_{j,J}(k)
 \coloneq \pi_{J,\varepsilon}^{-}
 \Delta_{j,J,\varepsilon}^{\mathrm{CII}}
 F_{q:q-1;0}^{q+1:q-1;1}
 \left(
 \begin{matrix}
  \boldsymbol a_{j,J}^{-}(k):
  \boldsymbol b_{j,J}^{\mathrm{CII}};1\\
  \boldsymbol c_{j,J}^{\mathrm{CII}}:
  \boldsymbol d_{j,J}^{-}(k);\text{--}
 \end{matrix}
 \middle|1,a^{-1}\right).
 \label{eq:CII-compact-finite-pole-block}
\end{equation}
If $m_J<\varepsilon+1$, set all three blocks equal to zero.

\begin{proposition}[Regrouping of the finite-pole contributions]
\label{prop:compact-unreflected-finite-poles}
Under the hypotheses of
Theorem~\ref{thm:explicit-unreflected-B}, for every
\(J\in\{1,\ldots,q-1\}\) and every active component $j$,
\begin{align}
 \sum_{K=0}^{m_J-1}\pi_{J,K}^{-}
 \mathcal C_{j,J,K}^{\mathrm K}(k)
 &=\mathscr K_{j,J}(k),
 \label{eq:K-compact-finite-pole-identity}\\
 \sum_{K=0}^{m_J-1}\pi_{J,K}^{-}
 \mathcal C_{j,J,K}^{\mathrm{MII}}(k)
 &=\mathscr M_{j,J}(k).
 \label{eq:MII-compact-finite-pole-identity}\\
 \sum_{K=0}^{m_J-1}\pi_{J,K}^{-}
 \mathcal C_{j,J,K}^{\mathrm{CII}}(k)
 &=\mathscr C_{j,J}(k).
 \label{eq:CII-compact-finite-pole-identity}
\end{align}
Consequently, the explicit components may be written with one outer sum,
\begin{align}
 D_j^{\mathrm K}(k)
 &=\sum_{J=1}^{q-1}\mathscr K_{j,J}(k)
 +\sum_{\ell=0}^{m_q-1}\gamma_\ell^{\mathrm K}
   \mathcal E_{j,\ell}^{\mathrm K}(k),\\
 D_j^{\mathrm{MII}}(k)
 &=\sum_{J=1}^{q-1}\mathscr M_{j,J}(k)
 +\sum_{\ell=0}^{m_q-1}\gamma_\ell^{\mathrm{MII}}
   \mathcal E_{j,\ell}^{\mathrm{MII}}(k),\\
 D_j^{\mathrm{CII}}(k)
 &=\sum_{J=1}^{q-1}\mathscr C_{j,J}(k)
 +\sum_{\ell=0}^{m_q-1}\gamma_\ell^{\mathrm{CII}}
   \mathcal E_{j,\ell}^{\mathrm{CII}}(k).
\end{align}
\end{proposition}

\begin{proof}
In the Hahn block \eqref{eq:Hahn-B-KdF-block}, take first $J<q$.
Under the Kravchuk scaling, the two parameter pairs containing
\(A_q\) and \(\beta_q\) contribute $c^{-t}$ to the term indexed by
\((r,t)\); all remaining parameters give precisely the strings in
\eqref{eq:K-compact-finite-pole-block}.  Under the Meixner-II scaling,
the pairs containing $N$ and \(\beta_q\) contribute
\((-\tau^{-1})^t\), giving
\eqref{eq:MII-compact-finite-pole-block}.  The sums terminate at
\(r+t=m_J-1-\varepsilon\), so both limits are termwise.  The normalized
prefactors tend respectively to
\(\pi_{J,\varepsilon}^{-}\Delta_{j,J,\varepsilon}^{\mathrm K}\)
and
\(\pi_{J,\varepsilon}^{-}\Delta_{j,J,\varepsilon}^{\mathrm{MII}}\).
Equations \eqref{eq:K-compact-finite-pole-identity} and
\eqref{eq:MII-compact-finite-pole-identity} now follow from the change of
indices $K=\varepsilon+r+t$ used in the proof of
Corollary~\ref{cor:Hahn-B-KdF}. Under
\(A_q=R\), \(\tau=a/R\), the parameter pair containing \(A_q\)
contributes
\[
 \frac{(\beta_J-R+\varepsilon)_r}
 {(\beta_J-R+\varepsilon+1)_{r+t}}
 \left(-\frac Ra\right)^t\longrightarrow a^{-t}.
\]
Deleting this pair gives the strings in
\eqref{eq:CII-compact-finite-pole-block}, and the prefactor has the
limit \(\Delta_{j,J,\varepsilon}^{\mathrm{CII}}\). This proves
\eqref{eq:CII-compact-finite-pole-identity}. Substitution in
\eqref{eq:unreflected-explicit-components} proves the three component
formulas.
\end{proof}

\paragraph{Classical one-weight reductions.}
For \(q=1\), let \(n=m_1-1\), with every empty list or product equal
to one. The Kravchuk-, Meixner-II-, and Charlier-II-like formulas reduce to
\begin{align}
 D_1^{\mathrm K}(k)
 &=-n!\binom Nn\left(\frac{c}{1-c}\right)^n
 \pFq{2}{1}{-n,-k}{-N}{c^{-1}},
 \label{eq:unreflected-q1-K-B}\\
 D_1^{\mathrm{MII}}(k)
 &=-(A_1)_n\left(\frac{\tau}{1+\tau}\right)^n
 \pFq{2}{1}{-n,-k}{A_1}{-\tau^{-1}},
 \label{eq:unreflected-q1-MII-B}\\
 D_1^{\mathrm{CII}}(k)
 &=-a^n\pFq{2}{0}{-n,-k}{-}{-a^{-1}}.
 \label{eq:unreflected-q1-CII-B}
\end{align}

\paragraph{Verification.}
When \(q=1\), there are no finite poles, so
\(\mathcal R_{\mm}(r)=-(-r)_n\) and only the polynomial part remains.
Substitution of the three scalar factorial moments in
\eqref{eq:unreflected-R-base-row}, followed by the terminating
Chu--Vandermonde identity, gives
\eqref{eq:unreflected-q1-K-B}--\eqref{eq:unreflected-q1-CII-B}.
Equivalently, multiplying each polynomial by its scalar weight and summing
its generating series gives exactly the corresponding signed generating
function.

The three displayed formulas are the
classical Kravchuk, Meixner, and Charlier polynomials of degree \(n\), in
the normalization inherited from the signed sequence.

\section{The Kravchuk-like finite-lattice family}
\label{sec:Kravchuk}

A second finite-lattice family is obtained by taking one beta-parameter pair
to infinity
while the lattice size remains fixed. This limit preserves \(q\)
positive weights: the first \(q-1\) retain one shifted beta parameter, whereas
the last becomes the unshifted reference weight. We derive the weights, the
type-II polynomial, the signed sequence satisfying the type-I moments, and
the type-I polynomials under explicit
pole-separation conditions.

\subsection{A partial Hahn-to-Kravchuk limit}

We scale only the last Hahn parameter pair. Scaling all pairs
simultaneously would make every weight converge to the same binomial weight
and would destroy multiple orthogonality with \(q\) distinct weights.

Keep \(N\), \(\A_{<q}\), and \(\bbeta_{<q}\) fixed, and let
\begin{equation}
 A_q^{(T)}=Ta,\qquad
 \beta_q^{(T)}=Tb,\qquad
 0<a<b,\qquad
 c\coloneq\frac ab\in(0,1),\qquad T\to\infty.
 \label{eq:K-scaling}
\end{equation}
For \(h\in\{1,\ldots,q-1\}\) and \(j\in\{1,\ldots,q\}\), put
\begin{equation}
 C_{h,j}\coloneq\beta_h+\delta_{h,j},\qquad
 C_{h,q}\coloneq\beta_h,
 \qquad
 \boldsymbol C_j=(C_{1,j},\ldots,C_{q-1,j}).
 \label{eq:K-Cj}
\end{equation}

The limiting weights remain supported on \(\{0,\ldots,N\}\).

\begin{theorem}[Kravchuk-like weights]
\label{thm:K-rows}
Let \(N\in\Nzero\), use the scaling \eqref{eq:K-scaling}, and assume
\(0<A_h<\beta_h\) for every \(h\in\{1,\ldots,q-1\}\). As
\(T\to\infty\), the Hahn-like weights converge
coefficientwise to normalized strictly positive weights
\(v_{j,N}^{\mathrm K}\), for every \(j\in\{1,\ldots,q\}\). The following
formulas hold.
\begin{enumerate}[label=\textup{(\roman*)}]
\item For every \(j\in\{1,\ldots,q\}\), the generating-function identity
below holds for all \(z\in\mathbb C\), the mass formula holds for every
\(k\in\{0,\ldots,N\}\), and the moment formula holds for every
\(r\in\{0,\ldots,N\}\):
\begin{align}
 G_{j,N}^{\mathrm K}(z)
 &\coloneq\sum_{k=0}^Nv_{j,N}^{\mathrm K}(k)z^k
 =\pFq{q}{q-1}
 {-N,\A_{<q}}{\boldsymbol C_j}{c(1-z)},
 \label{eq:K-pgf}\\
 v_{j,N}^{\mathrm K}(k)
 &=\binom Nk c^k
 \frac{(\A_{<q})_k}{(\boldsymbol C_j)_k}
 \pFq{q}{q-1}
 {k-N,\A_{<q}+k}{\boldsymbol C_j+k}{c},
 \label{eq:K-mass}\\
 \sum_{k=0}^N\fall{k}{r}v_{j,N}^{\mathrm K}(k)
 &=\fall{N}{r}c^r
 \frac{(\A_{<q})_r}{(\boldsymbol C_j)_r},
 \qquad r\in\{0,\ldots,N\}.
 \label{eq:K-factorial-moments}
\end{align}
\item With \(Y=y_1\cdots y_{q-1}\), for every
\(j\in\{1,\ldots,q\}\) and \(z\in\mathbb C\), the generating functions
have the representations below; the mass formula holds for every
\(k\in\{0,\ldots,N\}\):
\begin{align}
 G_{j,N}^{\mathrm K}(z)
 &=
 \int_{(0,1)^{q-1}}
 [1-c(1-z)Y]^N
 \prod_{h=1}^{q-1}b_{A_h,C_{h,j}-A_h}(y_h)
 \dd\boldsymbol y,
 \label{eq:K-positive-pgf}\\
 v_{j,N}^{\mathrm K}(k)
 &=
 \binom Nk
 \int_{(0,1)^{q-1}}
 (cY)^k(1-cY)^{N-k}
 \prod_{h=1}^{q-1}b_{A_h,C_{h,j}-A_h}(y_h)
 \dd\boldsymbol y.
 \label{eq:K-positive-mass}
\end{align}
\end{enumerate}
\end{theorem}

\begin{proof}
For each fixed \(s\),
\((A_q^{(T)})_s/(\beta_q^{(T)}+\delta_{q,j})_s
\xrightarrow[T\to\infty]{} c^s\).
Since the Hahn generating functions terminate at \(s=N\), termwise
passage in \eqref{eq:Hahn-like-pgf} gives \eqref{eq:K-pgf}; the same
argument in \eqref{eq:Hahn-like-factorial-moments} gives
\eqref{eq:K-factorial-moments}. Coefficient extraction gives
\eqref{eq:K-mass}.

Alternatively, as \(T\to\infty\), the last beta factor in
\eqref{eq:beta-product-integral} has moments converging to \(c^s\).
Applying this to the polynomial
\([1-(1-z)y_1\cdots y_q]^N\) gives
\eqref{eq:K-positive-pgf}. Extracting the coefficient of \(z^k\) gives
\eqref{eq:K-positive-mass}. Since \(0<cY<1\) throughout the integration
domain, all lattice weights are strictly positive.
\end{proof}

\subsection{Explicit type-II polynomial and normalized signed sequence}

The terminating Hahn expressions pass directly to the partial parameter
limit. The type-II polynomial requires no multiplicative rescaling,
whereas the signed sequence must be multiplied by
\((\beta_q^{(T)})_{m_q}\).

Let \(\mm\) be near the diagonal, put
\(1\le d=\abs\mm\le N\) and \(n=d-1\), and define
\begin{equation}
 A_{\mm,N}^{\mathrm K}(k)
 \coloneq
 \pFq{q+1}{q}
 {-d,-k,\bbeta_{<q}+\mm_{<q}}
 {-N,\A_{<q}}{c^{-1}}.
 \label{eq:K-A}
\end{equation}
Its monic normalization is
\begin{equation}
 \widehat A_{\mm,N}^{\mathrm K}(k)
 \coloneq
 (-1)^dc^d\fall{N}{d}
 \frac{(\A_{<q})_d}
 {(\bbeta_{<q}+\mm_{<q})_d}
 A_{\mm,N}^{\mathrm K}(k).
 \label{eq:K-monic-A}
\end{equation}
Define the signed sequence \(\mathcal B_{\mm,N}^{\mathrm K}\) by
\begin{equation}
 \mathcal H_{\mm,N}^{\mathrm K}(z)
 \coloneq\sum_{k=0}^N\mathcal B_{\mm,N}^{\mathrm K}(k)z^k
 =-\fall{N}{n}c^n
 \frac{(\A_{<q})_n}{(\bbeta_{<q})_{\mm_{<q}+n}}
 (1-z)^n
 \pFq{q}{q-1}{n-N,\A_{<q}+n}
 {\bbeta_{<q}+\mm_{<q}+n}{c(1-z)}.
 \label{eq:K-B-pgf}
\end{equation}
Put
\begin{equation}
 h_N^{\mathrm K}(r)
 \coloneq
 \fall{N}{r}c^r
 \frac{(\A_{<q})_r}{(\bbeta_{<q})_r}.
 \label{eq:K-h}
\end{equation}

\begin{theorem}[Kravchuk-like type-II polynomial and signed sequence]
\label{thm:K-forms}
As \(T\to\infty\) under the scaling \eqref{eq:K-scaling}:
\begin{enumerate}[label=\textup{(\roman*)}]
\item The Hahn-like type-II polynomials converge coefficientwise to
\(A_{\mm,N}^{\mathrm K}\). This polynomial has degree \(d\), satisfies
\(A_{\mm,N}^{\mathrm K}(0)=1\), and \(\widehat
A_{\mm,N}^{\mathrm K}\) is monic. For every
\(j\in\{1,\ldots,q\}\) with \(m_j>0\),
\begin{equation}
 \sum_{k=0}^N\fall{k}{r}
 A_{\mm,N}^{\mathrm K}(k)v_{j,N}^{\mathrm K}(k)=0,
 \qquad r\in\{0,\ldots,m_j-1\}.
 \label{eq:K-A-orthogonality}
\end{equation}
\item The normalized Hahn-like generating functions satisfy, coefficientwise
in \(z\),
\begin{equation}
 \lim_{T\to\infty}
 (\beta_q^{(T)})_{m_q}
 \mathcal H_{\mm,N}^{\mathrm H}(z)
 =\mathcal H_{\mm,N}^{\mathrm K}(z).
 \label{eq:K-B-limit}
\end{equation}
Its factorial moments are
\begin{equation}
 \sum_{k=0}^N\fall{k}{r}
 \mathcal B_{\mm,N}^{\mathrm K}(k)
 =
 -h_N^{\mathrm K}(r)
 \frac{(-r)_n}
 {\prod_{h=1}^{q-1}(r+\beta_h)_{m_h}},
 \qquad r\in\{0,\ldots,N\}.
 \label{eq:K-B-moments}
\end{equation}
The moments of every order \(r\in\Nzero\) with \(r<n\) vanish, and
\begin{equation}
 \sum_{k=0}^N\fall{k}{n}
 \mathcal B_{\mm,N}^{\mathrm K}(k)
 =
 (-1)^{n+1}n!\,h_N^{\mathrm K}(n)
 \frac{1}{\prod_{h=1}^{q-1}(n+\beta_h)_{m_h}}
 \ne0.
 \label{eq:K-B-leading-moment}
\end{equation}
\end{enumerate}
\end{theorem}

\begin{proof}
For every fixed \(\ell\),
\(\frac{(\beta_q^{(T)}+m_q)_\ell}{(A_q^{(T)})_\ell}
\xrightarrow[T\to\infty]{} c^{-\ell}\).
Termwise passage in \eqref{eq:Hahn-A-factorial} gives \eqref{eq:K-A}.
Since \(N\) is fixed, the Hahn orthogonality relations pass through the
finite sums and prove \eqref{eq:K-A-orthogonality}. The term \(\ell=d\)
gives the leading coefficient and \eqref{eq:K-monic-A}.

For the signed sequence,
\[
 \frac{(\beta_q^{(T)})_{m_q}(A_q^{(T)})_n}
 {(\beta_q^{(T)})_{m_q+n}}
 \xrightarrow[T\to\infty]{} c^n,
 \qquad
 \frac{(A_q^{(T)}+n)_s}
 {(\beta_q^{(T)}+m_q+n)_s}
 \xrightarrow[T\to\infty]{} c^s.
\]
Termwise passage in \eqref{eq:Hahn-complete-B-pgf}, with the normalization
in \eqref{eq:K-B-limit}, gives \eqref{eq:K-B-pgf}. Multiplying
\eqref{eq:Hahn-complete-B-moments} by
\((\beta_q^{(T)})_{m_q}\) and using
\((\beta_q^{(T)})_{m_q}/(r+\beta_q^{(T)})_{m_q}
\xrightarrow[T\to\infty]{}1\)
gives \eqref{eq:K-B-moments}.
\end{proof}

\begin{remark}[The maximal type-I index on the finite lattice]
\label{rem:K-maximal-type-I-index}
The value-normalized type-II assertion in Theorem~\ref{thm:K-forms}
requires \(d\le N\). At \(d=N+1\), the unique monic type-II solution
under normality is the nodal polynomial \(\Pi_{N+1}\) in
\eqref{eq:finite-lattice-nodal-polynomial}. Its restriction to the
lattice is identically zero, so it cannot satisfy \(A(0)=1\), and
\eqref{eq:K-A} reaches the zero denominator \((-N)_{N+1}\). In contrast,
the signed sequence \eqref{eq:K-B-pgf}, its moment formulas
\eqref{eq:K-B-moments}--\eqref{eq:K-B-leading-moment}, and the type-I
reconstruction below remain valid for \(d=N+1\).  Indeed, then \(n=N\),
the hypergeometric factor in \eqref{eq:K-B-pgf} is one because its first
upper parameter is zero, and \(h_N^{\mathrm K}(N)\ne0\).  Thus the
maximal type-I problem uses all \(N+1\) moment conditions. The associated
monic type-II representative exists, but it is precisely the nodal
polynomial and hence carries no nonzero lattice values.
\end{remark}

The coefficients of the signed sequence are also explicit. Put
\(C_{\mm,N}^{\mathrm K}
\coloneq-\fall{N}{n}c^n(\A_{<q})_n/
(\bbeta_{<q})_{\mm_{<q}+n}\). Then
\begin{equation}
 \mathcal B_{\mm,N}^{\mathrm K}(k)
 =
 C_{\mm,N}^{\mathrm K}(-1)^k
 \sum_{s=0}^{N-n}
 \frac{(n-N)_s(\A_{<q}+n)_s}
 {(\bbeta_{<q}+\mm_{<q}+n)_s}
 \frac{c^s}{s!}\binom{n+s}{k},
 \label{eq:K-B-coefficients}
\end{equation}
where \(\binom{n+s}{k}=0\) for \(k>n+s\). Unlike the generating
function, \eqref{eq:K-B-coefficients} displays each lattice coefficient
directly.

\subsection{Type-I polynomials and near-diagonal normality}

The common reconstruction in Section~\ref{sec:explicit-unreflected-B}
applies directly to the Kravchuk-like family. We retain the lattice size
in the notation and use the canonical rows
\begin{equation}
 w_{j,N}^{\mathrm K}\coloneq
 \frac{1}{\beta_j}v_{j,N}^{\mathrm K},
 \quad j\in\{1,\ldots,q-1\},
 \qquad
 w_{q,N}^{\mathrm K}\coloneq v_{q,N}^{\mathrm K}.
 \label{eq:K-canonical-rows}
\end{equation}
Let \(R_{j,\ell}^{\mathrm K}\) denote the \(X=\mathrm K\) specializations
of \eqref{eq:unreflected-R-finite-row}--
\eqref{eq:unreflected-R-base-row}, with their finite-sum values at removable
exceptional parameters. For every active \(J<q\) and \(0\le K<m_J\), put
\begin{equation}
 d_{J,K}^{\mathrm K}\coloneq
 \Res_{r=-\beta_J-K}R_{J,K}^{\mathrm K}(r)
 =\frac{c^K(N+\beta_J+1)_K
 \prod_{h=1}^{q-1}(A_h-\beta_J-K)_K}
 {(-1)^K K!\prod_{\substack{1\le h\le q-1\\h\ne J}}
 (\beta_h-\beta_J-K)_K}.
 \label{eq:K-diagonal-residue}
\end{equation}

On the generic parameter set where the denominators in
\eqref{eq:K-finite-pole-block},
\eqref{eq:K-finite-pole-prefactor-row},
\eqref{eq:K-finite-pole-prefactor-base},
\eqref{eq:unreflected-finite-pole-coefficient}, and
\eqref{eq:unreflected-base-residue} do not vanish at their indicated
indices, apart from the removable diagonal cases with \(K=0\), the
components have the following closed form:
\begin{equation}
 D_{j,N}^{\mathrm K}(k)=
 \sum_{J=1}^{q-1}\sum_{K=0}^{m_J-1}
 \pi_{J,K}^{-}\mathcal C_{j,J,K}^{\mathrm K}(k)
 +\sum_{\ell=0}^{m_q-1}
 \gamma_\ell^{\mathrm K}\mathcal E_{j,\ell}^{\mathrm K}(k).
 \label{eq:K-specialized-components}
\end{equation}
At nongeneric removable values satisfying the hypotheses of
Theorem~\ref{thm:K-components}, \(D_{j,N}^{\mathrm K}\) denotes the unique
polynomial furnished there.
The components relative to the original probability weights are
\begin{equation}
 B_{j,N}^{\mathrm K}=D_{j,N}^{\mathrm K}/\beta_j
 \quad(j\in\{1,\ldots,q-1\}),\qquad
 B_{q,N}^{\mathrm K}=D_{q,N}^{\mathrm K}.
 \label{eq:K-probability-components}
\end{equation}

\begin{theorem}[Kravchuk-like type-I polynomials and normality]
\label{thm:K-components}
Let \(\mm\) be near the diagonal, let
\(1\le d=\abs\mm\le N+1\), and put \(n=d-1\). Assume that the pole strings
\[
 \{-\beta_J-K:0\le K<m_J\},
 \qquad J\in\{1,\ldots,q-1\},\quad m_J>0,
\]
are pairwise disjoint and that \(d_{J,K}^{\mathrm K}\ne0\) for every
active \(J<q\) and \(0\le K<m_J\).
Then
\begin{equation}
 \mathcal B_{\mm,N}^{\mathrm K}(k)
 =\sum_{j=1}^qD_{j,N}^{\mathrm K}(k)w_{j,N}^{\mathrm K}(k)
 =\sum_{j=1}^qB_{j,N}^{\mathrm K}(k)v_{j,N}^{\mathrm K}(k).
 \label{eq:K-component-decomposition}
\end{equation}
For every active \(j\),
\(\deg D_{j,N}^{\mathrm K}<m_j\) and
\(\deg B_{j,N}^{\mathrm K}<m_j\); both component tuples are unique, and
\(\mm\) is normal.

If, in addition, the generic nonvanishing conditions preceding
\eqref{eq:K-specialized-components} hold, the unique components are given
by \eqref{eq:K-specialized-components}. Under these additional hypotheses
and the Hahn-to-Kravchuk scaling \eqref{eq:K-scaling}, the Hahn
components satisfy, for every active \(j\),
\[
 \frac{(\beta_q^{(T)})_{m_q}}
      {(\beta_q^{(T)})^{\delta_{j,q}}}
 \left.B_{j,N}(k)\right|_{
 A_q=A_q^{(T)},\,\beta_q=\beta_q^{(T)}}
 \xrightarrow[T\to\infty]{}D_{j,N}^{\mathrm K}(k)
\]
coefficientwise in \(k\).
\end{theorem}

\begin{proof}
Put \(D_-(r)=\prod_{h<q}(r+\beta_h)_{m_h}\). By
\eqref{eq:unreflected-R-pairing}, coefficients
\(D_{j,N}^{\mathrm K}(k)=\sum_{\ell<m_j}b_{j,\ell}^{\mathrm K}
\fall{k}{\ell}\) have the required normalized moments precisely when
\[
 -\frac{(-r)_n}{D_-(r)}
 =\sum_{j=1}^q\sum_{\ell=0}^{m_j-1}
 b_{j,\ell}^{\mathrm K}R_{j,\ell}^{\mathrm K}(r).
\]
The finite sums defining the \(R_{j,\ell}^{\mathrm K}\), together with
near-diagonality, show that multiplication by \(D_-\) clears every
denominator and gives degree at most \(n\); moreover
\(R_{j,\ell}^{\mathrm K}=O(r^{\ell-1})\) for \(j<q\), while
\(R_{q,\ell}^{\mathrm K}=(1-c)^\ell r^\ell+O(r^{\ell-1})\).
At \(r=-\beta_J-K\), the diagonal residue is
\eqref{eq:K-diagonal-residue}; terms with the same source and lower degree,
or another source and degree at most \(K\), have zero residue there.
Separation and \(d_{J,K}^{\mathrm K}\ne0\) therefore determine, in
descending \(K\), all finite-pole coefficients and leave a polynomial of
degree at most \(m_q-1\), interpreted as zero if \(m_q=0\). Cancelling the poles of each
\(R_{q,\ell}^{\mathrm K}\) by the same triangular elimination produces a
polynomial with leading term \((1-c)^\ell r^\ell\). These polynomials form
a triangular basis, so comparison at infinity determines the remaining
coefficients uniquely.

Multiplication by \(h_{\mathrm K}(r)\), followed by
\eqref{eq:unreflected-R-pairing}, gives the moments
\eqref{eq:K-B-moments}; since both sequences are supported on
\(\{0,\ldots,N\}\), these moments identify
\(\mathcal B_{\mm,N}^{\mathrm K}\). The same rational argument applied to
a homogeneous component vector shows that \(D_-(r)G(r)\), of degree at
most \(n\), vanishes at \(0,\ldots,n\); it is zero, and the preceding
triangular elimination makes every coefficient zero. Hence the moment map
is nonsingular, proving uniqueness and normality, also when \(n=N\).

Under the additional generic hypotheses, the \(X=\mathrm K\) argument in
the proof of Theorem~\ref{thm:explicit-unreflected-B} identifies this unique tuple with
\eqref{eq:K-specialized-components};
\eqref{eq:unreflected-probability-components} gives the stated conversion.
Finally, the \(X=\mathrm K\) argument in the proof of
Corollary~\ref{cor:unreflected-component-limits} gives
\eqref{eq:Hahn-K-component-limit}.
\end{proof}

\paragraph{The moving \(J=q\) Kamp\'e de F\'eriet block.}
Under the additional generic nonvanishing conditions preceding
\eqref{eq:K-specialized-components}, the Kravchuk-like components can be
indexed by the same \(q\) blocks as the Hahn formula
\eqref{eq:Hahn-B-KdF-sum}.  For every active component \(j\), define the
block at infinity by
\begin{equation}
 \mathscr K_{j,q}^{\infty}(k)
 \coloneq
 \sum_{\ell=0}^{m_q-1}
 \gamma_\ell^{\mathrm K}\mathcal E_{j,\ell}^{\mathrm K}(k),
 \label{eq:K-moving-block-definition}
\end{equation}
with the empty sum interpreted as zero.  Thus the compact component
formula becomes
\begin{equation}
 D_{j,N}^{\mathrm K}(k)
 =\sum_{J=1}^{q-1}\mathscr K_{j,J}(k)
  +\mathscr K_{j,q}^{\infty}(k).
 \label{eq:K-all-sector-components}
\end{equation}

\begin{corollary}[Confluence of the moving Hahn Kamp\'e de F\'eriet block]
\label{cor:K-moving-KdF-block}
Assume the hypotheses of Theorem~\ref{thm:K-components} and the
additional generic nonvanishing conditions preceding
\eqref{eq:K-specialized-components}.  Let \(j\) be active and put
\(\varepsilon=1-\delta_{j,q}\).  If
\(m_q<\varepsilon+1\), interpret the absent Hahn block
\(S_{j,q}^{(N)}\) as zero; then both sides below are zero.  Otherwise,
under \eqref{eq:K-scaling},
\begin{equation}
 \frac{(\beta_q^{(T)})_{m_q}}
      {(\beta_q^{(T)})^{\delta_{j,q}}}
 \left.S_{j,q}^{(N)}(k)\right|_{
 A_q=A_q^{(T)},\,\beta_q=\beta_q^{(T)}}
 \xrightarrow[T\to\infty]{}
 \mathscr K_{j,q}^{\infty}(k)
 \label{eq:Hahn-K-moving-block-limit}
\end{equation}
coefficientwise in \(k\).

This limit is a single coefficient extraction from the terminating Hahn
Kamp\'e de F\'eriet block.  More precisely, set
\(M=m_q-1-\varepsilon\), \(s=\sum_{h=1}^{q-1}m_h\), and let
\(\Phi_{j,q}^{(N)}(u;k)\) denote the Kamp\'e de F\'eriet factor in
\eqref{eq:Hahn-B-KdF-block}, without its prefactor
\(\pi_{q,\varepsilon}\mathcal D_{j,q,\varepsilon}^{(N)}\), after
\(\beta_q=u^{-1}\) and \(A_q=cu^{-1}\).  Rewrite every moving
Pochhammer factor by
\[
 (\lambda u^{-1}+a)_r
 =(\lambda u^{-1})^r
  \prod_{v=0}^{r-1}\left(1+\frac{a+v}{\lambda}u\right),
 \qquad \lambda\ne0,
\]
and cancel the resulting powers of \(u\) in each finite term.  This
defines \(\Phi_{j,q}^{(N)}\) as a finite rational expression analytic at
\(u=0\).  Then, coefficientwise in \(k\),
\begin{equation}
 \Phi_{j,q}^{(N)}(u;k)=\mathrm O(u^{2M}),
 \qquad
 \mathscr K_{j,q}^{\infty}(k)
 =\Lambda_j^{\mathrm K}[u^{2M}]\Phi_{j,q}^{(N)}(u;k)
 =\frac{\Lambda_j^{\mathrm K}}{(2M)!}
   \left.\frac{\partial^{2M}}{\partial u^{2M}}
   \Phi_{j,q}^{(N)}(u;k)\right|_{u=0},
 \label{eq:K-moving-block-extraction}
\end{equation}
where
\begin{equation}
 \Lambda_j^{\mathrm K}=
 \begin{cases}
 \displaystyle \frac{(-1)^{s+1}}{M!},&j=q,\\[6pt]
 \displaystyle \frac{(-1)^s}{M!}\frac{c}{c-1}
 \frac{(N+\beta_j)\prod_{h=1}^{q-1}(A_h-\beta_j)}
 {\prod_{\substack{1\le h\le q-1\\h\ne j}}
  (\beta_h-\beta_j)},&j<q.
 \end{cases}
 \label{eq:K-moving-block-prefactor}
\end{equation}
Moreover,
\(\deg\mathscr K_{j,q}^{\infty}\le m_q-1\) for \(j=q\), whereas
\(\deg\mathscr K_{j,q}^{\infty}\le m_q-2\) for \(j<q\) whenever that
block is present.
Consequently, every summand in \eqref{eq:K-all-sector-components} is the
ordinary coefficientwise confluence of one terminating Hahn Kamp\'e de
F\'eriet block.  For \(J<q\) the limit remains a Kamp\'e de F\'eriet
polynomial; the moving block \(J=q\) is the coefficient extraction
\eqref{eq:K-moving-block-extraction} and equals the finite combination
of products of terminating Horn polynomials in
\eqref{eq:K-moving-block-definition}.
\end{corollary}

\begin{proof}
For every \(J<q\), the \(X=\mathrm K\) argument in the proof of
Proposition~\ref{prop:compact-unreflected-finite-poles} gives the
coefficientwise limit of the normalized Hahn block
\(S_{j,J}^{(N)}\) as \(\mathscr K_{j,J}\).  Subtracting these finitely
many limits from the component limit in
Theorem~\ref{thm:K-components}, and then using
\eqref{eq:K-all-sector-components}, proves
\eqref{eq:Hahn-K-moving-block-limit}.

Put \(B=\beta_q=u^{-1}\), so that \(A_q=cB\).  Directly from
\eqref{eq:target-residue} and \eqref{eq:Hahn-C-block-prefactor},
\[
 \frac{(B)_{m_q}}{B^{\delta_{j,q}}}
 \pi_{q,\varepsilon}\mathcal D_{j,q,\varepsilon}^{(N)}
 =\Lambda_j^{\mathrm K}u^{-2M}
  +\mathrm O(u^{-2M+1}).
\]
The remaining Kamp\'e de F\'eriet factor is a finite sum of rational
functions analytic at \(u=0\).  On the open parameter subset where
\(\Lambda_j^{\mathrm K}\ne0\), the finite limit
\eqref{eq:Hahn-K-moving-block-limit} forces its coefficients of orders
below \(2M\) to vanish. Those coefficients are rational functions of the
fixed parameters, so the same vanishing holds by rational continuation at
every parameter point at which the displayed expressions are regular.
Taking the constant term of the product proves
\eqref{eq:K-moving-block-extraction} at every such regular point, including
those at which \(\Lambda_j^{\mathrm K}=0\); the displayed asymptotic
evaluation also gives \eqref{eq:K-moving-block-prefactor}. The degree
bounds follow from
the terminating parameter \(-M\) in the source block.
\end{proof}

\subsection{Classical reduction and the two Charlier-like limits}

For one weight the partial parameter limit is the ordinary Hahn-to-Kravchuk
limit. For \(q>1\), direct scaling of the same finite-lattice family leads
to the Charlier-II-like system, whereas reflection followed by scaling leads
to the Charlier-I-like system.

\begin{corollary}[Reduction to the scalar Kravchuk case]
\label{cor:q1-Kravchuk}
Let \(q=1\), \(m\in\{1,\ldots,N\}\), and \(n=m-1\). Then
the identities involving \(k\) hold for every
\(k\in\{0,\ldots,N\}\), and the generating-function identity holds for
every \(z\in\mathbb C\):
\begin{align}
 v_N^{\mathrm K}(k)
 &=\binom Nk c^k(1-c)^{N-k},
 \label{eq:q1-K-weight}\\
 A_{m,N}^{\mathrm K}(k)
 &=\pFq{2}{1}{-m,-k}{-N}{c^{-1}},
 \label{eq:q1-K-A}\\
 \mathcal H_{m,N}^{\mathrm K}(z)
 &=-\fall{N}{n}c^n
 (1-z)^n(1-c+cz)^{N-n},
 \label{eq:q1-K-B-pgf}\\
 B_{m,N}^{\mathrm K}(k)
 &=
 -\fall{N}{n}
 \left(\frac{c}{1-c}\right)^n
 \pFq{2}{1}{-n,-k}{-N}{c^{-1}}.
 \label{eq:q1-K-B}
\end{align}
\end{corollary}

\begin{proof}
When \(q=1\), \eqref{eq:K-pgf} becomes
\(\pFq{1}{0}{-N}{}{c(1-z)}=(1-c+cz)^N\), which gives
\eqref{eq:q1-K-weight}. The formulas for \(A\) and the signed
sequence are the corresponding empty-product cases of
Theorem~\ref{thm:K-forms}; in particular, \eqref{eq:q1-K-A} is the
ordinary scalar Kravchuk polynomial. Finally,
\[
 \sum_{k=0}^Nv_N^{\mathrm K}(k)
 \pFq{2}{1}{-n,-k}{-N}{c^{-1}}z^k
 =
 (1-c)^n(1-z)^n(1-c+cz)^{N-n}.
\]
Comparison with \eqref{eq:q1-K-B-pgf} gives \eqref{eq:q1-K-B}.
\end{proof}

The polynomials in Corollary~\ref{cor:q1-Kravchuk},
\(A_{m,N}^{\mathrm K}\) and \(B_{m,N}^{\mathrm K}\), are the ordinary
Kravchuk polynomials of degrees \(m\) and \(m-1\), respectively.

For the two limits it is clearest to normalize the signed sequence so that
its factorial moment of order \(n\) equals one:
\begin{equation}
 \mathcal H_{\mm,N}^{\mathrm K,\mathrm{unit}}(z)
 \coloneq
 \frac{(z-1)^n}{n!}
 \pFq{q}{q-1}
 {n-N,\A_{<q}+n}
 {\bbeta_{<q}+\mm_{<q}+n}
 {c(1-z)}.
 \label{eq:K-unit-B}
\end{equation}
Its derivatives at \(z=1\) of every order \(r\in\Nzero\) with \(r<n\) vanish, and its
\(n\)-th derivative equals one.
The exact conversion from the normalization in \eqref{eq:K-B-pgf} is
\begin{equation}
 \mathcal H_{\mm,N}^{\mathrm K,\mathrm{unit}}
 =\mathfrak u_{\mm,N}^{\mathrm K}\mathcal H_{\mm,N}^{\mathrm K},
 \qquad
 \mathfrak u_{\mm,N}^{\mathrm K}
 \coloneq
 \frac{(-1)^{n+1}(\bbeta_{<q})_{\mm_{<q}+n}}
 {n!\fall{N}{n}c^n(\A_{<q})_n}.
 \label{eq:K-unit-normalizing-factor}
\end{equation}

Let \(\mm\) be fixed and near the diagonal, put
\(d=\abs\mm\ge1\) and \(n=d-1\), fix \(a>0\), and set
\begin{equation}
 c_N=\frac aN.
 \label{eq:K-CII-scaling}
\end{equation}
Define
\[
 \widehat A_{\mm}^{\mathrm{CII}}(k)
 \coloneq(-1)^da^d
 \frac{(\A_{<q})_d}{(\bbeta_{<q}+\mm_{<q})_d}
 A_{\mm}^{\mathrm{CII}}(k).
\]

\begin{corollary}[Kravchuk-to-Charlier-II-like confluence]
\label{cor:K-to-CII}
Under the scaling \eqref{eq:K-CII-scaling}, as \(N\to\infty\), with
\(\A_{<q}\) and \(\bbeta_{<q}\) fixed:
\begin{enumerate}[label=\textup{(\roman*)}]
\item For every \(j\in\{1,\ldots,q\}\),
\(G_{j,N}^{\mathrm K}\) converges locally uniformly to
\(G_j^{\mathrm{CII}}\) in \eqref{eq:CII-pgf}, coefficientwise and in every
fixed factorial moment.
\item The monic type-II polynomials satisfy
\begin{equation}
 \widehat A_{\mm,N}^{\mathrm K}(k)
 \xrightarrow[N\to\infty]{}
 \widehat A_{\mm}^{\mathrm{CII}}(k)
 \label{eq:K-CII-A-limit}
\end{equation}
coefficientwise in \(k\).
\item The unit-normalized signed sequences satisfy
\begin{equation}
 \mathcal H_{\mm,N}^{\mathrm K,\mathrm{unit}}(z)
 \xrightarrow[N\to\infty]{}
 \frac{(z-1)^n}{n!}
 \pFq{q-1}{q-1}
 {\A_{<q}+n}
 {\bbeta_{<q}+\mm_{<q}+n}
 {a(z-1)}.
 \label{eq:K-CII-B-limit}
\end{equation}
\item Suppose in addition that the pole-separation and nonvanishing
hypotheses of Theorem~\ref{thm:explicit-unreflected-B} hold for the
Charlier-II-like blocks. Then the corresponding hypotheses hold for the Kravchuk-like
specialization for every sufficiently large \(N\). Restoring the lattice
size in the notation of \eqref{eq:unreflected-explicit-components} and
\eqref{eq:unreflected-probability-components}, write
\(\mathscr K_{j,J;N}\) and
\(\mathscr K_{j,q;N}^{\infty}\) for the blocks
\eqref{eq:K-compact-finite-pole-block} and
\eqref{eq:K-moving-block-definition}, respectively, with their
\(N\)-dependence displayed; the second is the moving Hahn-block limit of
Corollary~\ref{cor:K-moving-KdF-block}. For every active \(j\), each
sector has a separate coefficientwise limit:
\begin{align}
 \left.\mathscr K_{j,J;N}(k)\right|_{c=a/N}
 &\xrightarrow[N\to\infty]{}\mathscr C_{j,J}(k),
 &&J\in\{1,\ldots,q-1\},
 \label{eq:K-CII-finite-sector-limit}\\
 \left.\mathscr K_{j,q;N}^{\infty}(k)\right|_{c=a/N}
 &\xrightarrow[N\to\infty]{}\mathscr C_{j,q}^{\infty}(k).
 \label{eq:K-CII-infinity-sector-limit}
\end{align}
Here \(\mathscr C_{j,q}^{\infty}\) denotes the Charlier-II-like block at
infinity defined explicitly in
\eqref{eq:CII-infinity-block-definition}.
Consequently,
\begin{align}
 D_{j,N}^{\mathrm K}(k)\big|_{c=a/N}
 &\xrightarrow[N\to\infty]{}D_j^{\mathrm{CII}}(k),
 &
 B_{j,N}^{\mathrm K}(k)\big|_{c=a/N}
 &\xrightarrow[N\to\infty]{}B_j^{\mathrm{CII}}(k)
 \label{eq:K-CII-component-limit}
\end{align}
coefficientwise in \(k\).  These are the components in the
normalizations \eqref{eq:K-B-pgf} and \eqref{eq:CII-B}, respectively.
Their unit-normalized versions satisfy the same componentwise limit after
multiplication by
\begin{equation}
 \mathfrak u_{\mm,N}^{\mathrm K}\big|_{c=a/N}
 \longrightarrow
 \mathfrak u_{\mm}^{\mathrm{CII}}
 \coloneq
 \frac{(-1)^{n+1}(\bbeta_{<q})_{\mm_{<q}+n}}
 {n!a^n(\A_{<q})_n}.
 \label{eq:K-CII-unit-factor-limit}
\end{equation}
Explicitly, for \(X_{j,N}^{\mathrm K}=D_{j,N}^{\mathrm K}\) or
\(B_{j,N}^{\mathrm K}\), and the corresponding
\(X_j^{\mathrm{CII}}\),
\begin{equation}
 \left.\mathfrak u_{\mm,N}^{\mathrm K}
 X_{j,N}^{\mathrm K}(k)\right|_{c=a/N}
 \longrightarrow
 \mathfrak u_{\mm}^{\mathrm{CII}}X_j^{\mathrm{CII}}(k)
 \quad\hbox{coefficientwise}.
 \label{eq:K-CII-unit-component-limit}
\end{equation}
\end{enumerate}
\end{corollary}

The limit in \eqref{eq:K-CII-B-limit} is \eqref{eq:CII-B} normalized so
that its factorial moment of order \(n\) equals one.

\begin{proof}
Under the scaling \eqref{eq:K-CII-scaling}, the positive representation gives
\[
 G_{j,N}^{\mathrm K}(z)
 =
 \int_{(0,1)^{q-1}}
 \left(1+\frac aN Y(z-1)\right)^N
 \prod_{h=1}^{q-1}b_{A_h,C_{h,j}-A_h}(y_h)\dd\boldsymbol y.
\]
As \(N\to\infty\), the integrand converges uniformly on compact \(z\)-sets to
\(\exp(aY(z-1))\), and it has a uniform compact majorant. This proves the
weight limit. The hypergeometric factor in \eqref{eq:K-unit-B} has the same
integral with lattice size \(N-n\) and beta parameters
\(A_h+n\) and \(\beta_h-A_h+m_h\); it gives
\eqref{eq:K-CII-B-limit}. Finally, the polynomial series terminates at
\(d\), while
\(\frac{(N/a)^s}{(-N)_s}\xrightarrow[N\to\infty]{}(-a^{-1})^s\) and
\(\fall{N}{d}\left(\frac aN\right)^d\xrightarrow[N\to\infty]{}a^d\).
This proves \eqref{eq:K-CII-A-limit}.

For completeness, we verify the individual type-I components before
summing their sectors.  For every fixed \(u\),
\[
 \Theta_u^{\mathrm K}(r)\big|_{c=a/N}
 =\left(\frac aN\right)^u\fall{N-r}{u}
   \frac{(r+\A_{<q})_u}{(r+\bbeta_{<q})_u}
 \longrightarrow
 a^u\frac{(r+\A_{<q})_u}{(r+\bbeta_{<q})_u}
 =\Theta_u^{\mathrm{CII}}(r).
\]
Thus every finite sum
\(R_{j,\ell}^{\mathrm K,N}\) in
\eqref{eq:unreflected-R-finite-row}--
\eqref{eq:unreflected-R-base-row} converges coefficientwise, after a
common denominator is cleared, to \(R_{j,\ell}^{\mathrm{CII}}\).
At a finite pole, the terminating series in
\eqref{eq:K-finite-pole-block} has, in its term of order \(s\), the
factor
\[
 \frac{(N/a)^s}{(1-N-\xi_{J,K})_s}\longrightarrow(-a^{-1})^s.
\]
The remaining parameters are fixed, while
\[
 \frac{N+\beta_j}{N+\xi_{J,K}}\longrightarrow1,
 \qquad
 \frac{1-N/a}{N+\xi_{J,K}}\longrightarrow-\frac1a.
\]
Consequently
\(\mathcal C_{j,J,K}^{\mathrm K}\to
\mathcal C_{j,J,K}^{\mathrm{CII}}\) coefficientwise for every indicated
\((j,J,K)\).

The residues \(z_{J,K,\ell}^{\mathrm K}\) and the polynomial-part
coefficients \(p_{K,\ell}^{\mathrm K}\) are, respectively, residues and
coefficients of these same finite rational functions. Hence they converge
to their Charlier-II-like counterparts. Equivalently, this follows
directly from
\[
 \left(\frac aN\right)^u
 \fall{N+\beta_J+K}{u}\longrightarrow a^u
\]
in \eqref{eq:unreflected-base-residue}.  The triangular pivots satisfy
\((1-a/N)^\ell\to1\), so the finite inverse
\eqref{eq:unreflected-infinity-chain} gives
\(\gamma_\ell^{\mathrm K}\to\gamma_\ell^{\mathrm{CII}}\).
All sector sums are finite. The finite-pole convergence above therefore
gives \eqref{eq:K-CII-finite-sector-limit}; the convergence of the
residues, corrected vectors, and triangular coefficients gives
\eqref{eq:K-CII-infinity-sector-limit}. Summing these limits in
\eqref{eq:K-all-sector-components} and
\eqref{eq:CII-all-sector-components} proves the first limit in
\eqref{eq:K-CII-component-limit}. Since the parameters
\(\beta_j\) are fixed, \eqref{eq:unreflected-probability-components}
gives the second.  The nonzero limiting denominators in the stated
hypotheses also show that the Kravchuk reconstruction is well defined for
all sufficiently large \(N\).  Finally,
\(\fall{N}{n}(a/N)^n\to a^n\), and
\eqref{eq:K-unit-normalizing-factor} gives
\eqref{eq:K-CII-unit-factor-limit} and the asserted unit-normalized
component limit.
\end{proof}

The second limit first reflects the lattice.

Let \(\mm\) be fixed and near the diagonal, put
\(d=\abs\mm\ge1\) and \(n=d-1\), and fix \(a>0\),
\(\delta_h>0\), and \(0<c_h<1\), for every
\(h\in\{1,\ldots,q-1\}\). Put
\begin{equation}
 \lambda_h=\frac{c_h}{1-c_h},\qquad
 c_N=1-\frac aN,\qquad
 A_{h,N}=\frac N{\lambda_h},\qquad
 \beta_{h,N}=A_{h,N}+\delta_h.
 \label{eq:K-CI-scaling}
\end{equation}
Write
\(\A_{<q,N}=(A_{1,N},\ldots,A_{q-1,N})\) and
\(\bbeta_{<q,N}=(\beta_{1,N},\ldots,\beta_{q-1,N})\).
For every \(j\in\{1,\ldots,q\}\), define
\(v_{j,N}^{\mathrm K,\mathrm{ref}}(\ell)
=v_{j,N}^{\mathrm K}(N-\ell)\) and
\(G_{j,N}^{\mathrm K,\mathrm{ref}}(z)=z^NG_{j,N}^{\mathrm K}(z^{-1})\).
For \(h\in\{1,\ldots,q-1\}\), put
\(R_h(z)=(1-c_h)/(1-c_hz)\), and define
\(P_{\mm,N}^{\mathrm K,\mathrm{ref}}(\ell)
\coloneq(-1)^d\widehat A_{\mm,N}^{\mathrm K}(N-\ell)\).
Define the reflected unit-normalized generating function by
\(\mathcal H_{\mm,N}^{\mathrm K,\mathrm{ref},\mathrm{unit}}(z)
\coloneq(-1)^nz^N
\mathcal H_{\mm,N}^{\mathrm K,\mathrm{unit}}(z^{-1})\). Its explicit form is
\begin{equation}
 \mathcal H_{\mm,N}^{\mathrm K,\mathrm{ref},\mathrm{unit}}(z)
 =
 \frac{(z-1)^n}{n!}z^{N-n}
 \pFq{q}{q-1}
 {n-N,\A_{<q,N}+n}
 {\bbeta_{<q,N}+\mm_{<q}+n}
 {c_N(1-z^{-1})}.
 \label{eq:K-reflected-unit-B}
\end{equation}
Write
\begin{equation}
 \mathcal H_{\mm,N}^{\mathrm K,\mathrm{ref},\mathrm{unit}}(z)
 =\sum_{\ell=0}^N
 \mathcal B_{\mm,N}^{\mathrm K,\mathrm{ref},\mathrm{unit}}(\ell)z^\ell.
 \label{eq:K-reflected-unit-coefficients}
\end{equation}
Thus, with \(\mathfrak u_{\mm,N}^{\mathrm K}\) in
\eqref{eq:K-unit-normalizing-factor} evaluated at the parameters
\eqref{eq:K-CI-scaling},
\begin{equation}
 \mathcal B_{\mm,N}^{\mathrm K,\mathrm{ref},\mathrm{unit}}(\ell)
 =(-1)^n\mathfrak u_{\mm,N}^{\mathrm K}
 \mathcal B_{\mm,N}^{\mathrm K}(N-\ell).
 \label{eq:K-reflected-unit-coefficient-identity}
\end{equation}
Put
\begin{equation}
 \mathcal H_{\mm}^{\mathrm{CI}}(z)
 \coloneq
 \frac{(z-1)^n}{n!}\e^{a(z-1)}
 \prod_{h=1}^{q-1}R_h(z)^{\delta_h+m_h}.
\end{equation}

\begin{theorem}[Reflected Kravchuk-to-Charlier-I-like confluence]
\label{thm:K-to-CI}
Under the preceding assumptions and scaling, as \(N\to\infty\):
\begin{enumerate}[label=\textup{(\roman*)}]
\item For every \(j\in\{1,\ldots,q\}\),
\begin{equation}
 G_{j,N}^{\mathrm K,\mathrm{ref}}(z)
 \xrightarrow[N\to\infty]{}
 \e^{a(z-1)}
 \prod_{h=1}^{q-1}R_h(z)^{\delta_h+\delta_{h,j}}
 =G_j^{\mathrm{CI}}(z)
 \label{eq:K-CI-row-limit}
\end{equation}
uniformly on compact subsets of the common disk of analyticity.
\item If the parameters \(c_h\) for which
\(h\in\{1,\ldots,q-1\}\) and \(m_h>0\) are pairwise distinct, then
\begin{equation}
 P_{\mm,N}^{\mathrm K,\mathrm{ref}}(\ell)
 \xrightarrow[N\to\infty]{} P_{\mm}^{\mathrm{CI}}(\ell)
 \label{eq:K-CI-A-limit}
\end{equation}
coefficientwise.
\item The coefficients of
\(\mathcal H_{\mm,N}^{\mathrm K,\mathrm{ref},\mathrm{unit}}\) have factorial
moments zero for every order \(r\in\Nzero\) with \(r<n\) and moment one at order \(n\). Moreover,
\begin{equation}
 \mathcal H_{\mm,N}^{\mathrm K,\mathrm{ref},\mathrm{unit}}(z)
 \xrightarrow[N\to\infty]{} \mathcal H_{\mm}^{\mathrm{CI}}(z)
 \label{eq:K-CI-B-limit}
\end{equation}
locally uniformly.
\item Under the pairwise-distinctness hypothesis in \textup{(ii)}, for every
sufficiently large \(N\), the reflected Kravchuk-like
type-I moment matrix is nonsingular and there is a unique tuple of
polynomials \(B_{j,N}^{\mathrm K,\mathrm{ref},\mathrm{unit}}\), indexed
by the active \(j\)'s, such that
\begin{equation}
 \deg B_{j,N}^{\mathrm K,\mathrm{ref},\mathrm{unit}}<m_j,
 \qquad
 \mathcal B_{\mm,N}^{\mathrm K,\mathrm{ref},\mathrm{unit}}(\ell)
 =\sum_{\substack{1\le j\le q\\m_j>0}}
 B_{j,N}^{\mathrm K,\mathrm{ref},\mathrm{unit}}(\ell)
 v_{j,N}^{\mathrm K,\mathrm{ref}}(\ell).
 \label{eq:K-CI-component-decomposition}
\end{equation}
If \(B_{j,N}^{\mathrm K,\mathrm{unit}}\) denotes the corresponding
unit-normalized component before reflection, then the exact finite-\(N\)
relation is
\begin{equation}
 B_{j,N}^{\mathrm K,\mathrm{ref},\mathrm{unit}}(\ell)
 =(-1)^nB_{j,N}^{\mathrm K,\mathrm{unit}}(N-\ell).
 \label{eq:K-CI-component-reflection}
\end{equation}
For every active \(j\),
\begin{equation}
 B_{j,N}^{\mathrm K,\mathrm{ref},\mathrm{unit}}(\ell)
 \xrightarrow[N\to\infty]{}B_j^{\mathrm{CI}}(\ell)
 \label{eq:K-CI-component-limit}
\end{equation}
coefficientwise in \(\ell\), where the limiting components are the
unit-normalized Charlier-I-like polynomials of
Theorem~\ref{thm:explicit-CI-B}.
\end{enumerate}
\end{theorem}

\begin{proof}
In the reflected version of \eqref{eq:K-positive-pgf}, the polynomial
kernel can be written as \(\left[1+(z-1)(1-c_NY)\right]^N\).
Set \(y_h=1-u_h/N\) in the beta integrals. Uniform gamma-ratio
asymptotics on compact \(u_h\)-sets give
\[
 \frac1N b_{A_{h,N},\,\delta_h+\delta_{h,j}}
       (1-u_h/N)
 \xrightarrow[N\to\infty]{}
 \frac{\lambda_h^{-\delta_h-\delta_{h,j}}}
 {\Gamma(\delta_h+\delta_{h,j})}
 u_h^{\delta_h+\delta_{h,j}-1}
 \e^{-u_h/\lambda_h}.
\]
The transformed beta kernels have a common integrable exponential
majorant. Moreover,
\(N(1-c_N\prod_{h=1}^{q-1}y_h)
\xrightarrow[N\to\infty]{}a+\sum_{h=1}^{q-1}u_h\).
Dominated convergence therefore gives \eqref{eq:K-CI-row-limit}, because
\[
 \frac{1}{\Gamma(\delta)\lambda_h^\delta}
 \int_0^\infty u^{\delta-1}\e^{-u/\lambda_h}\e^{u(z-1)}\dd u
 =R_h(z)^\delta.
\]
For the signed generating function, replace the lattice size by \(N-n\)
and the second beta parameters by \(\delta_h+m_h\).  The transformed
factors retain the same common exponential majorant, while
\((N-n)(1-c_Ny_1\cdots y_{q-1})
\xrightarrow[N\to\infty]{}a+u_1+\cdots+u_{q-1}\).
Dominated convergence in its generating function therefore proves
\eqref{eq:K-CI-B-limit}. The factor
\((-1)^n\) in \eqref{eq:K-reflected-unit-B} preserves the unit
\(n\)-th factorial moment under reflection.

Moreover, \((-1)^d\widehat A_{\mm,N}^{\mathrm K}(N-\ell)\) is the
monic type-II polynomial for the reflected weights: for every
\(j\in\{1,\ldots,q\}\) with \(m_j>0\), the
families
\(\{(N-\ell)^{\underline r}:r\in\Nzero,\ r<m_j\}\) and
\(\{\ell^{\underline r}:r\in\Nzero,\ r<m_j\}\) span the same
polynomial space.
As \(N\to\infty\), the generating functions of the weights and all fixed
derivatives required by the type-II moment matrices converge at \(z=1\). Hence those finite moment
matrices converge to the Charlier-I-like matrix. Its determinant
is nonzero under this pairwise-distinctness hypothesis by
Proposition~\ref{prop:CI-normality}. Continuity of matrix inversion proves
\eqref{eq:K-CI-A-limit}.

It remains to treat the individual type-I components.  Index the columns
by \((j,r)\), where \(m_j>0\) and \(0\le r<m_j\), and the rows by
\(s\in\{0,\ldots,n\}\).  Set
\begin{equation}
 \bigl(M_N^{\mathrm K,\mathrm{ref}}\bigr)_{s,(j,r)}
 \coloneq
 \sum_{\ell=0}^N\fall{\ell}{s}\fall{\ell}{r}
 v_{j,N}^{\mathrm K,\mathrm{ref}}(\ell).
 \label{eq:K-CI-typeI-moment-matrix}
\end{equation}
By \eqref{eq:falling-linearization}, every entry is a fixed linear
combination of derivatives of \(G_{j,N}^{\mathrm K,\mathrm{ref}}\) at
\(z=1\).  The local uniform convergence
\eqref{eq:K-CI-row-limit} and Cauchy's formula therefore give
\[
 M_N^{\mathrm K,\mathrm{ref}}\longrightarrow M^{\mathrm{CI}},
\]
where \(M^{\mathrm{CI}}\) is the Charlier-I-like type-I moment matrix.
It is nonsingular by Proposition~\ref{prop:CI-normality}.  Hence
\(M_N^{\mathrm K,\mathrm{ref}}\) is nonsingular for all sufficiently
large \(N\), and the coefficient vector of the unique unit-normalized
tuple is
\[
 \bigl(M_N^{\mathrm K,\mathrm{ref}}\bigr)^{-1}
 (0,\ldots,0,1)^{\mathsf T}.
\]
Continuity of inversion proves \eqref{eq:K-CI-component-limit}.

We finally verify that this moment solution decomposes the explicit signed
sequence, rather than merely sharing its first \(n+1\) moments. Reflect
the resulting linear form back to the original lattice and multiply it by
\((-1)^n\).  Its component polynomials still have the required degree
bounds, and its factorial moments of orders \(0,\ldots,n\) agree with
those of \(\mathcal B_{\mm,N}^{\mathrm K,\mathrm{unit}}\): reflection
preserves the zero moments below \(n\), while the leading coefficient of
\(\fall{N-k}{n}\) is \((-1)^n\).  Convert the first \(q-1\) component
rows to the canonical weights \eqref{eq:K-canonical-rows}. By
\eqref{eq:unreflected-R-pairing}, after division by
\(h_N^{\mathrm K}(s)\) and multiplication by
\(D_-(s)=\prod_{h=1}^{q-1}(s+\beta_{h,N})_{m_h}\), the difference
between this linear form and
\(\mathfrak u_{\mm,N}^{\mathrm K}\mathcal B_{\mm,N}^{\mathrm K}\)
is a polynomial in \(s\) of degree at most \(n\).  This is the
near-diagonal degree estimate used in the proof of
Theorem~\ref{thm:K-components}; it does not use pole separation.  The
polynomial vanishes at \(s=0,\ldots,n\), hence vanishes identically.
The two signed sequences therefore have all factorial moments of orders
\(0,\ldots,N\) equal.  Both are supported on \(\{0,\ldots,N\}\), so
they coincide. Reflecting once more proves
\eqref{eq:K-CI-component-decomposition} and the exact relation
\eqref{eq:K-CI-component-reflection}.
\end{proof}

\begin{remark}[Relation with the classical multiple Kravchuk system]
The partial limit constructed in Section~\ref{sec:Kravchuk} is not the
classical multiple Kravchuk system. In the latter, the weights are binomial
sequences with distinct parameters \(p_j\)
\cite{ArvesuCoussementVanAssche2003,
BranquinhoDiazFoulquieManasWolfs2024Discrete}. Here every weight still
depends on the \(q-1\) parameter pairs that were kept fixed.

This distinction cannot be removed by sending all parameter pairs to
infinity.
Indeed, let \(A_h=A_h(t)\) and \(\beta_h=\beta_h(t)\) satisfy
\(A_h(t),\beta_h(t)\xrightarrow[t\to\infty]{}\infty\) and
\(A_h(t)/\beta_h(t)\xrightarrow[t\to\infty]{}c_h\). Then, for every
fixed \(r\),
\(\fall{N}{r}\frac{(\A)_r}
{(\bbeta+\ee_j)_r}
\xrightarrow[t\to\infty]{}
\fall{N}{r}\left(\prod_hc_h\right)^r\),
independently of \(j\). Thus all weights converge to the same binomial weight
and the multiple system loses rank. The words \enquote{partial} and
\enquote{-like} are therefore essential.  A simultaneous multi-parameter
limit that retains several independent weights requires a separate
analysis.
\end{remark}

\section{The Meixner-II-like family}
\label{sec:MII}

We first send the right endpoint of the Hahn-like lattice to infinity while
keeping the left endpoint fixed. The limit is a Meixner-II-like family with
an explicit type-II polynomial, an explicit signed sequence satisfying the
type-I moment conditions, and a direct Laguerre-I-like limit.

\subsection{The left-endpoint scaling limit}

The scaling below keeps the left endpoint fixed and sends the opposite
endpoint to infinity. We first identify the limiting weights and their
factorial moments.

Keep \(\A\) and \(\beta_1,\ldots,\beta_{q-1}\) fixed, let
\begin{equation}
 \beta_q^{(N)}=A_q+\frac{N}{\tau},
 \qquad
 \tau=\frac{c}{1-c}>0,
 \label{eq:MII-scaling}
\end{equation}
and send \(N\to\infty\) with the lattice coordinate \(k\) fixed. For
\(h\in\{1,\ldots,q-1\}\) and \(j\in\{1,\ldots,q\}\), put
\(C_{h,j}=\beta_h+\delta_{h,j}\),
where \(C_{h,q}=\beta_h\), and write
\(\boldsymbol C_j=(C_{1,j},\ldots,C_{q-1,j})\) and
\(Y=y_1\cdots y_{q-1}\).

This scaling produces the following Meixner-II-like weights.

\begin{theorem}[Meixner-II-like weights]
\label{thm:MII-rows}
As \(N\to\infty\) under \eqref{eq:MII-scaling}, for every
\(j\in\{1,\ldots,q\}\), the Hahn-like weights converge coefficientwise
and in every fixed moment to a normalized positive weight
\(v_j^{\mathrm{MII}}\) on \(\Nzero\) with generating function
\begin{equation}
 G_j^{\mathrm{MII}}(z)
 =\pFq{q}{q-1}{A_1,\ldots,A_q}
 {C_{1,j},\ldots,C_{q-1,j}}{\tau(z-1)}.
 \label{eq:MII-pgf}
\end{equation}
Its factorial moments are, for every \(r\in\Nzero\),
\begin{equation}
 \sum_{k\ge0}\fall{k}{r}v_j^{\mathrm{MII}}(k)
 =\tau^r\frac{\prod_{h=1}^q(A_h)_r}
 {\prod_{h=1}^{q-1}(C_{h,j})_r}.
 \label{eq:MII-factorial-moments}
\end{equation}
The same generating function has the positive integral representation
\begin{equation}
 G_j^{\mathrm{MII}}(z)
 =\int_{(0,1)^{q-1}}
 \frac{\prod_{h=1}^{q-1}b_{A_h,C_{h,j}-A_h}(y_h)
       \dd\boldsymbol y}
 {(1+\tau(1-z)Y)^{A_q}}.
 \label{eq:MII-positive-integral}
\end{equation}
\end{theorem}

\begin{proof}
For fixed \(r\),
\(\fall{N}{r}/(\beta_q^{(N)}+\delta_{q,j})_r
\xrightarrow[N\to\infty]{}\tau^r\).
The factorial moments therefore converge by
\eqref{eq:Hahn-like-factorial-moments}; ordinary moments follow by the
finite Stirling transform. To justify convergence of the coefficients,
use \eqref{eq:beta-product-integral} in the finite generating-function
integral and set \(y_q=\tau x/N\). As \(N\to\infty\), the transformed beta
kernel tends to \(g_{1,A_q}(x)\), while
\(\bigl[1-(1-z)y_1\cdots y_{q-1}\tau x/N\bigr]^N
\xrightarrow[N\to\infty]{}
\exp\{-\tau(1-z)y_1\cdots y_{q-1}x\}\).
Dominated convergence is locally uniform for \(\abs z<1\), and Cauchy's
coefficient formula proves coefficientwise convergence. Termwise passage
in \eqref{eq:Hahn-like-pgf} also gives \eqref{eq:MII-pgf} first when
\(\abs{\tau(z-1)}<1\). Expanding the integrand in
\eqref{eq:MII-positive-integral} and using the beta integrals gives the
same hypergeometric series. The integral is positive for \(0\le z\le1\),
equals one at \(z=1\), and supplies the analytic continuation throughout
the unit disk. Differentiating at \(z=1\) gives
\eqref{eq:MII-factorial-moments}.
\end{proof}

For every \(j\in\{1,\ldots,q\}\) and \(k\in\Nzero\), coefficient
extraction in \eqref{eq:MII-positive-integral} gives the explicit weights
\begin{equation}
 v_j^{\mathrm{MII}}(k)
 =\frac{\tau^k}{k!}
 \frac{(\A)_k}{(\boldsymbol C_j)_k}
 \pFq{q}{q-1}{\A+k}{\boldsymbol C_j+k}{-\tau}.
 \label{eq:MII-mass}
\end{equation}
The positive integral in \eqref{eq:MII-positive-integral} supplies the
continuation of \eqref{eq:MII-mass} when its displayed series is outside
the disk of convergence.

\subsection{Explicit Meixner-II-like type-II and type-I formulas}

We now give closed formulas for the Meixner-II-like type-II polynomial and
for a signed sequence satisfying the type-I moment conditions.

\begin{theorem}[Meixner-II-like type-II polynomial and signed sequence]
\label{thm:MII-forms}
Let \(\mm\) be fixed and near the diagonal, and put
\(d=\abs\mm\ge1\) and \(n=d-1\). Under \eqref{eq:MII-scaling}, the
following assertions hold as \(N\to\infty\).
\begin{enumerate}[label=\textnormal{(\roman*)}]
\item The Hahn polynomial converges termwise to
\begin{equation}
 A_{\mm}^{\mathrm{MII}}(k)=
 \pFq{q+1}{q}
 {-d,-k,\beta_1+m_1,\ldots,\beta_{q-1}+m_{q-1}}
 {A_1,\ldots,A_q}{-\tau^{-1}}.
 \label{eq:MII-A}
\end{equation}
This polynomial has exact degree \(d\) and satisfies
\(A_{\mm}^{\mathrm{MII}}(0)=1\).
For every \(j\in\{1,\ldots,q\}\) with \(m_j>0\), it satisfies
\begin{equation}
 \sum_{k\ge0}\fall{k}{r}A_{\mm}^{\mathrm{MII}}(k)
 v_j^{\mathrm{MII}}(k)=0,
 \qquad r\in\{0,\ldots,m_j-1\}.
 \label{eq:MII-A-orthogonality}
\end{equation}

\item The generating function of the scaled signed sequence
\((\beta_q^{(N)})_{m_q}\mathcal B_{\mm,N}^{\mathrm H}\) converges to
\begin{equation}
 \mathcal H_{\mm}^{\mathrm{MII}}(z)
 =-\tau^n\frac{(\A)_n}{(\bbeta_{<q})_{\mm_{<q}+n}}(1-z)^n
 \pFq{q}{q-1}
 {\A+n}{\bbeta_{<q}+\mm_{<q}+n}{\tau(z-1)}.
 \label{eq:MII-B-pgf}
\end{equation}
Its factorial moments are, for every \(r\in\Nzero\),
\begin{equation}
 \sum_{k\ge0}\fall{k}{r}\mathcal B_{\mm}^{\mathrm{MII}}(k)
 =-\tau^r(-r)_n
 \frac{\prod_{h=1}^{q}(A_h)_r}
 {\prod_{h=1}^{q-1}(\beta_h)_{m_h+r}}.
 \label{eq:MII-B-moments}
\end{equation}
\end{enumerate}
\end{theorem}

\begin{proof}
In \eqref{eq:Hahn-A-factorial},
\((\beta_q^{(N)}+m_q)_\ell/(-N)_\ell
\xrightarrow[N\to\infty]{}(-1/\tau)^\ell\),
and the sum terminates at \(d\), proving \eqref{eq:MII-A}. Its leading
coefficient is
\[
 [k^d]A_{\mm}^{\mathrm{MII}}(k)
 =(-1)^d\tau^{-d}
  \frac{\prod_{h=1}^{q-1}(\beta_h+m_h)_d}
       {\prod_{h=1}^{q}(A_h)_d}\ne0,
\]
and setting \(k=0\) leaves only the zeroth term. Orthogonality
follows either by passage in fixed moments or by direct substitution of
\eqref{eq:MII-factorial-moments}. The same termwise argument in
\eqref{eq:Hahn-complete-B-pgf}, together with
\(\fall{N}{n}/(\beta_q^{(N)})_n
\xrightarrow[N\to\infty]{}\tau^n\), gives
\eqref{eq:MII-B-pgf} initially for \(\abs{\tau(z-1)}<1\). The
shifted Euler integral, with parameters \(A_h+n\) and
\(\beta_h+m_h+n\), supplies the continuation throughout the unit disk; the
identity theorem gives the displayed equality there. Differentiation at
\(z=1\) gives
\eqref{eq:MII-B-moments}.
\end{proof}

\subsection{Type-I components and near-diagonal normality}

The terms involving the last Hahn parameter must be summed before that
parameter tends to infinity.  The specialization \(X=\mathrm{MII}\) of
\eqref{eq:unreflected-explicit-components} performs this summation and
gives every Meixner-II-like type-I polynomial as a finite combination of
terminating \({}_{q+1}F_q\) and Horn polynomials.  Under the pole-separation
and nonvanishing assumptions in Theorem~\ref{thm:explicit-unreflected-B},
that theorem also proves near-diagonal normality.  The polynomial \(A\) and
the signed sequence \eqref{eq:MII-B-pgf} require no such assumptions.

\begin{corollary}[Meixner-II-like components and their Hahn limit]
Under the hypotheses of Theorem~\ref{thm:explicit-unreflected-B}, the
Meixner-II-like components are the specialization $X=\mathrm{MII}$ of
\eqref{eq:unreflected-explicit-components}. They have degrees $<m_j$,
are unique, and their coefficientwise confluence from the Hahn-like
components is \eqref{eq:Hahn-MII-component-limit}.
\end{corollary}

\begin{proof}
Apply Theorem~\ref{thm:explicit-unreflected-B} with
$X=\mathrm{MII}$, followed by the second part of
Corollary~\ref{cor:unreflected-component-limits}.
\end{proof}

\paragraph{The moving \(J=q\) Kamp\'e de F\'eriet block.}
Under the generic hypotheses of
Theorem~\ref{thm:explicit-unreflected-B}, define, for every active
component \(j\),
\begin{equation}
 \mathscr M_{j,q}^{\infty}(k)
 \coloneq
 \sum_{\ell=0}^{m_q-1}
 \gamma_\ell^{\mathrm{MII}}
 \mathcal E_{j,\ell}^{\mathrm{MII}}(k),
 \label{eq:MII-moving-block-definition}
\end{equation}
with the empty sum interpreted as zero.  The compact component formula is
therefore
\begin{equation}
 D_j^{\mathrm{MII}}(k)
 =\sum_{J=1}^{q-1}\mathscr M_{j,J}(k)
  +\mathscr M_{j,q}^{\infty}(k).
 \label{eq:MII-all-sector-components}
\end{equation}

\begin{corollary}[Confluence of the moving Hahn Kamp\'e de F\'eriet block]
\label{cor:MII-moving-KdF-block}
Assume the hypotheses of
Theorem~\ref{thm:explicit-unreflected-B} for the Meixner-II-like
specialization, let \(j\) be active, and put
\(\varepsilon=1-\delta_{j,q}\) and
\(B_N=\beta_q^{(N)}=A_q+N/\tau\).  If
\(m_q<\varepsilon+1\), interpret the absent Hahn block
\(S_{j,q}^{(N)}\) as zero; then both sides below are zero.  Otherwise,
under \eqref{eq:MII-scaling},
\begin{equation}
 \frac{(B_N)_{m_q}}{B_N^{\delta_{j,q}}}
 S_{j,q}^{(N)}(k)\big|_{\beta_q=B_N}
 \xrightarrow[N\to\infty]{}
 \mathscr M_{j,q}^{\infty}(k)
 \label{eq:Hahn-MII-moving-block-limit}
\end{equation}
coefficientwise in \(k\).

The limit is a single coefficient extraction from the terminating Hahn
Kamp\'e de F\'eriet block.  More precisely, set
\(M=m_q-1-\varepsilon\), \(s=\sum_{h=1}^{q-1}m_h\), and let
\(\Phi_{j,q}^{\mathrm{MII}}(u;k)\) denote the Kamp\'e de F\'eriet
factor in \eqref{eq:Hahn-B-KdF-block}, without its prefactor
\(\pi_{q,\varepsilon}\mathcal D_{j,q,\varepsilon}^{(N)}\), after
\begin{equation*}
 \beta_q=u^{-1},\qquad N=\tau(u^{-1}-A_q).
\end{equation*}
Rewrite every moving Pochhammer factor by
\[
 (\lambda u^{-1}+a)_r
 =(\lambda u^{-1})^r
  \prod_{v=0}^{r-1}\left(1+\frac{a+v}{\lambda}u\right),
 \qquad \lambda\ne0,
\]
and cancel the resulting powers of \(u\) in each finite term.  This
defines \(\Phi_{j,q}^{\mathrm{MII}}\) as a finite rational expression
analytic at \(u=0\).  Then, coefficientwise in \(k\),
\begin{equation}
 \Phi_{j,q}^{\mathrm{MII}}(u;k)=\mathrm O(u^{2M}),
 \qquad
 \mathscr M_{j,q}^{\infty}(k)
 =\Lambda_j^{\mathrm{MII}}[u^{2M}]
  \Phi_{j,q}^{\mathrm{MII}}(u;k)
 =\frac{\Lambda_j^{\mathrm{MII}}}{(2M)!}
  \left.\frac{\partial^{2M}}{\partial u^{2M}}
  \Phi_{j,q}^{\mathrm{MII}}(u;k)\right|_{u=0},
 \label{eq:MII-moving-block-extraction}
\end{equation}
where
\begin{equation}
 \Lambda_j^{\mathrm{MII}}=
 \begin{cases}
 \displaystyle \frac{(-1)^{s+1}}{M!},&j=q,\\[6pt]
 \displaystyle \frac{(-1)^{s+1}}{M!}\frac{\tau}{1+\tau}
 \frac{\prod_{h=1}^{q}(A_h-\beta_j)}
 {\prod_{\substack{1\le h\le q-1\\h\ne j}}
  (\beta_h-\beta_j)},&j<q.
 \end{cases}
 \label{eq:MII-moving-block-prefactor}
\end{equation}
Moreover,
\(\deg\mathscr M_{j,q}^{\infty}\le m_q-1\) for \(j=q\), whereas
\(\deg\mathscr M_{j,q}^{\infty}\le m_q-2\) for \(j<q\) whenever that
block is present.  Thus every term in
\eqref{eq:MII-all-sector-components} is the ordinary coefficientwise
confluence of one terminating Hahn Kamp\'e de F\'eriet block.  The
blocks with \(J<q\) remain Kamp\'e de F\'eriet polynomials; the moving
block is the coefficient extraction
\eqref{eq:MII-moving-block-extraction} and equals the finite combination
of terminating hypergeometric and Horn polynomials in
\eqref{eq:MII-moving-block-definition}.  For \(q=1\), it is
the entire component and reduces to a scalar Meixner polynomial of
degree \(m_1-1\), with the present type-I normalization.
\end{corollary}

\begin{proof}
For every \(J<q\), the Meixner-II-like part of the proof of
Proposition~\ref{prop:compact-unreflected-finite-poles} gives the
coefficientwise limit
\[
 \frac{(B_N)_{m_q}}{B_N^{\delta_{j,q}}}
 S_{j,J}^{(N)}(k)\big|_{\beta_q=B_N}
 \xrightarrow[N\to\infty]{}\mathscr M_{j,J}(k),
\]
with absent blocks interpreted as zero.  Subtracting these finitely many
limits from \eqref{eq:Hahn-MII-component-limit}, and using
\eqref{eq:MII-all-sector-components}, proves
\eqref{eq:Hahn-MII-moving-block-limit}.

Put \(B=\beta_q=u^{-1}\), so that
\(N=\tau(B-A_q)\).  Directly from
\eqref{eq:target-residue} and \eqref{eq:Hahn-C-block-prefactor},
\[
 \frac{(B)_{m_q}}{B^{\delta_{j,q}}}
 \pi_{q,\varepsilon}\mathcal D_{j,q,\varepsilon}^{(N)}
 =\Lambda_j^{\mathrm{MII}}u^{-2M}
  +\mathrm O(u^{-2M+1}).
\]
The remaining Kamp\'e de F\'eriet factor is a finite sum of rational
functions analytic at \(u=0\).  On the open parameter subset where
\(\Lambda_j^{\mathrm{MII}}\ne0\), the finite limit
\eqref{eq:Hahn-MII-moving-block-limit} forces its coefficients of orders
below \(2M\) to vanish. Those coefficients are rational functions of the
fixed parameters, so the same vanishing holds by rational continuation at
every parameter point at which the displayed expressions are regular.
Taking the constant term of the product proves
\eqref{eq:MII-moving-block-extraction} at every such regular point,
including those at which \(\Lambda_j^{\mathrm{MII}}=0\); the displayed
asymptotic evaluation also gives \eqref{eq:MII-moving-block-prefactor}.
The degree bounds follow
from the terminating parameter \(-M\) in the source block.
\end{proof}

\subsection{Classical reduction and the direct Laguerre-I limit}

We first check the one-weight case and then take the direct continuous
limit. The scalar case recovers classical Meixner, while the continuous
limit gives the Laguerre-I-like system.

For \(q=1\), \eqref{eq:MII-pgf} and coefficient extraction give
\[
 G^{\mathrm{MII}}(z)
 =(1+\tau(1-z))^{-A}
 =\left(\frac{1-c}{1-cz}\right)^A,
 \qquad
 v^{\mathrm{MII}}(k)=\frac{(A)_k}{k!}(1-c)^Ac^k,
 \quad k\in\Nzero,
\]
which is the classical scalar Meixner weight, equivalently the one-weight
reduction of the classical multiple Meixner-II system, while
\(A_d^{\mathrm{MII}}(k)=\pFq{2}{1}{-d,-k}{A}{1-c^{-1}}\) is the
classical Meixner polynomial. For \(q>1\), however, the present
weights differ from the classical multiple Meixner-II system
\cite{ArvesuCoussementVanAssche2003,
BranquinhoDiazFoulquieManasWolfs2024Discrete}.  The terminology
\enquote{Meixner-II-like} refers to their direct Laguerre-I-like limit.

Let \(\tau\to\infty\), put \(k=\lfloor\tau x\rfloor\), and retain
\(Y=y_1\cdots y_{q-1}\). For every \(j\in\{1,\ldots,q\}\), the
normalized Laguerre-I-like limiting weight is
\begin{equation}
 w_j^{\mathrm{LI}}(x)
 =\int_{(0,1)^{q-1}}
 \frac1Y g_{1,A_q}\!\left(\frac{x}{Y}\right)
 \prod_{h=1}^{q-1}b_{A_h,C_{h,j}-A_h}(y_h)\dd\boldsymbol y.
 \label{eq:LI-positive-row}
\end{equation}
The integral in \eqref{eq:LI-positive-row} is nonnegative, has integral
one, and its Laplace transform is
\(\int_0^\infty\e^{-sx}w_j^{\mathrm{LI}}(x)\dd x
=\pFq{q}{q-1}{\A}{\boldsymbol C_j}{-s}\).
We now justify the local limit uniformly away from the origin. For every
\(j\in\{1,\ldots,q\}\), let
\(\dd\mu_j(\boldsymbol y)
\coloneq\prod_{h=1}^{q-1}b_{A_h,C_{h,j}-A_h}(y_h)\dd\boldsymbol y\)
and \(Y=y_1\cdots y_{q-1}\),
with the usual empty-product convention when \(q=1\). Conditional on
\(Y\), coefficient extraction in \eqref{eq:MII-positive-integral} gives
\[
 q_{\tau,Y}(k)
 \coloneq[z^k](1+\tau(1-z)Y)^{-A_q}
 =\frac{(A_q)_k}{k!}
   \frac{(\tau Y)^k}{(1+\tau Y)^{A_q+k}},
 \qquad
 v_j^{\mathrm{MII}}(k)
 =\int_{(0,1)^{q-1}} q_{\tau,Y}(k)\dd\mu_j(\boldsymbol y).
\]
If \(K=[\varepsilon,M]\Subset(0,\infty)\) and
\(k=\lfloor\tau x\rfloor\), the uniform gamma-ratio estimate
\(\frac{\Gamma(k+A_q)}{\Gamma(k+1)}
=k^{A_q-1}\bigl(1+\mathrm O(k^{-1})\bigr)\) as \(k\to\infty\)
shows, for every fixed \(Y>0\), uniformly for \(x\in K\), that
\[
 \tau q_{\tau,Y}(\lfloor\tau x\rfloor)
 \xrightarrow[\tau\to\infty]{}
 \frac{x^{A_q-1}\e^{-x/Y}}{\Gamma(A_q)Y^{A_q}}
 =\frac1Yg_{1,A_q}\!\left(\frac{x}{Y}\right).
\]
The family is uniformly dominated even as \(Y\downarrow0\). Indeed, put
\(x_\tau=\lfloor\tau x\rfloor/\tau\) and
\(r_\tau=Y+\tau^{-1}\). For all sufficiently large \(\tau\),
\(x_\tau\in[\varepsilon/2,M]\), and
\[
 \tau q_{\tau,Y}(\lfloor\tau x\rfloor)
 \le C_{A_q}\,x_\tau^{A_q-1}r_\tau^{-A_q}
 \exp\!\left(-\frac{x_\tau}{r_\tau}\right)
 \le C_Kr_\tau^{-A_q}
 \exp\!\left(-\frac{\varepsilon}{2r_\tau}\right)
 \le C_K'.
\]
Here we used
\((1-1/(\tau r_\tau))^k\le
\exp(-k/(\tau r_\tau))\). The last bound is uniform for
\(0<Y<1\), since \(r^{-A_q}\exp(-\varepsilon/(2r))\) is bounded
near \(r=0\), and the limiting gamma kernel obeys the same bound with
\(r_\tau\) replaced by \(Y\). Since
\(\mu_j((0,1)^{q-1})=1\),
dominated convergence, applied to the supremum over \(x\in K\), gives
\[
 \sup_{x\in K}\left|
 \tau v_j^{\mathrm{MII}}(\lfloor\tau x\rfloor)
 -\int_{(0,1)^{q-1}}\frac1Yg_{1,A_q}\!\left(\frac{x}{Y}\right)
       \dd\mu_j(\boldsymbol y)\right|
 \xrightarrow[\tau\to\infty]{}0.
\]
Consequently,
\begin{equation}
 \tau\,v_j^{\mathrm{MII}}(\lfloor\tau x\rfloor)
 \xrightarrow[\tau\to\infty]{} w_j^{\mathrm{LI}}(x),\qquad x>0,
 \label{eq:MII-LI-row-limit}
\end{equation}
Thus \eqref{eq:MII-LI-row-limit} holds locally uniformly on compact
subsets of \((0,\infty)\). The polynomial
limit is
\begin{equation}
 A_{\mm}^{\mathrm{LI}}(x)
 =\pFq{q}{q}
 {-d,\beta_1+m_1,\ldots,\beta_{q-1}+m_{q-1}}
 {A_1,\ldots,A_q}{x}.
 \label{eq:LI-A}
\end{equation}
Writing \(z=\e^{-s/\tau}\) in \eqref{eq:MII-B-pgf} and letting
\(\tau\to\infty\) gives the Laplace
transform of the normalized Laguerre-I-like type-I linear form:
\begin{equation}
 -\frac{(\A)_n}{(\bbeta_{<q})_{\mm_{<q}+n}}
 s^n\,
 \pFq{q}{q-1}
 {\A+n}{\bbeta_{<q}+\mm_{<q}+n}{-s}.
 \label{eq:LI-B-Laplace}
\end{equation}

\begin{proposition}[Direct Meixner-II-like-to-Laguerre-I-like polynomial and form limits]
\label{prop:MII-LI-polynomial-form-limit}
For a fixed near-diagonal \(\mm\), the polynomial
\eqref{eq:LI-A} has exact degree \(d=\abs\mm\), satisfies
\(A_{\mm}^{\mathrm{LI}}(0)=1\), and has leading coefficient
\begin{equation}
 [x^d]A_{\mm}^{\mathrm{LI}}(x)
 =(-1)^d
 \frac{\displaystyle\prod_{h=1}^{q-1}(\beta_h+m_h)_d}
      {\displaystyle\prod_{h=1}^{q}(A_h)_d}.
 \label{eq:LI-A-leading-coefficient}
\end{equation}
For every active \(j\), it obeys the type-II relations
\begin{equation}
 \int_0^\infty x^rA_{\mm}^{\mathrm{LI}}(x)
 w_j^{\mathrm{LI}}(x)\dd x=0,
 \qquad r\in\{0,\ldots,m_j-1\}.
 \label{eq:LI-A-orthogonality}
\end{equation}
Moreover, as \(\tau\to\infty\),
\begin{equation}
 A_{\mm}^{\mathrm{MII}}(\tau\,\cdot)
 \xrightarrow[\tau\to\infty]{}A_{\mm}^{\mathrm{LI}}
 \quad\hbox{coefficientwise},
 \label{eq:MII-LI-A-limit}
\end{equation}
and hence
\(A_{\mm}^{\mathrm{MII}}(\lfloor\tau x\rfloor)
\to A_{\mm}^{\mathrm{LI}}(x)\) for every fixed \(x>0\).
The corresponding type-I generating functions, evaluated at
\(z=\e^{-s/\tau}\), converge locally uniformly on the open half-plane
\(\{s\in\mathbb C:\operatorname{Re}s>-1\}\) to
\eqref{eq:LI-B-Laplace}.
\end{proposition}

\begin{proof}
In the terminating series \eqref{eq:MII-A}, for every fixed summation
index \(\ell\),
\[
 (-\tau x)_\ell(-\tau^{-1})^\ell\longrightarrow x^\ell.
\]
This proves \eqref{eq:MII-LI-A-limit}. The term with \(\ell=d\) gives
\eqref{eq:LI-A-leading-coefficient}, while the constant term is one.

It remains to justify passage to the type-II relations. By
\eqref{eq:MII-factorial-moments}, for every fixed \(s\),
\[
 \tau^{-s}\sum_{k\ge0}\fall{k}{s}v_j^{\mathrm{MII}}(k)
 =\int_0^\infty x^s w_j^{\mathrm{LI}}(x)\dd x.
\]
The finite Stirling transformations therefore give convergence of every
fixed ordinary moment of the rescaled discrete rows. Divide the
Meixner-II-like orthogonality relation of order \(r\) by \(\tau^r\),
write its polynomial factor as
\(A_{\mm}^{\mathrm{MII}}(\tau(k/\tau))\), and use
\eqref{eq:MII-LI-A-limit}. Only finitely many fixed moments occur, so the
limit is exactly \eqref{eq:LI-A-orthogonality}. Finally,
put \(u_\tau(s)=\tau(1-\e^{-s/\tau})\), so that
\(u_\tau(s)\to s\) locally uniformly. In the shifted Euler representation
of the hypergeometric factor in \eqref{eq:MII-B-pgf}, the integrand is
\((1+u_\tau(s)Y)^{-A_q-n}\) against the probability beta product with
parameters \(A_h+n\) and \(\beta_h+m_h-A_h\), \(h<q\).
If \(\mathcal K\) is a compact subset of
\(\{s\in\mathbb C:\operatorname{Re}s>-1\}\), then, for some
\(\varepsilon_{\mathcal K}>0\) and every sufficiently large \(\tau\),
\[
 \operatorname{Re}(1+Yu_\tau(s))\ge\varepsilon_{\mathcal K},
 \qquad s\in\mathcal K,\quad 0\le Y\le1.
\]
It is therefore bounded in modulus by
\(\varepsilon_{\mathcal K}^{-A_q-n}\), while the beta product has mass one.
Dominated convergence, together with \(u_\tau(s)^n\to s^n\), gives
\eqref{eq:LI-B-Laplace} locally uniformly on the stated open set.
\end{proof}

For \(q=1\), this is the scalar Meixner-to-Laguerre transition contained in
the standard multiple Meixner-II-to-Laguerre-I limit.
These weights and forms constitute the Laguerre-I-like limit described in
\cite[Section~3]{Wolfs2024}. Proposition~\ref{prop:MII-LI-polynomial-form-limit}
proves their direct appearance as limits of the new discrete Hahn-like
family.

\subsection{Explicit Laguerre-I-like type-I components and normality}

Let \(w_j^{\mathrm{LI}}\) be the weights in
\eqref{eq:LI-positive-row}, and define their rescaled versions by
\begin{equation}
 \widetilde w_j^{\mathrm{LI}}=w_j^{\mathrm{LI}}/\beta_j\quad
 (j\in\{1,\ldots,q-1\}),
 \qquad
 \widetilde w_q^{\mathrm{LI}}=w_q^{\mathrm{LI}}.
 \label{eq:LI-canonical-rows}
\end{equation}
Put \(h_{\mathrm{LI}}(r)=(\A)_r/(\bbeta_{<q})_r\). Their moments are
\[
 \int_0^\infty x^r\widetilde w_j^{\mathrm{LI}}(x)\dd x
 =\begin{cases}
 h_{\mathrm{LI}}(r)/(r+\beta_j),&j\in\{1,\ldots,q-1\},\\
 h_{\mathrm{LI}}(r),&j=q.
 \end{cases}
\]
Consequently, multiplication by the monomial \(x^\ell\) gives
\begin{align}
 R_{j,\ell}^{\mathrm{LI}}(r)
 &=\frac{(r+\A)_\ell}
 {(r+\bbeta_{<q})_\ell(r+\beta_j+\ell)},&&
 j\in\{1,\ldots,q-1\},
 \label{eq:LI-component-kernel-row}\\
 R_{q,\ell}^{\mathrm{LI}}(r)
 &=\frac{(r+\A)_\ell}{(r+\bbeta_{<q})_\ell}.
 \label{eq:LI-component-kernel-base}
\end{align}

For \(J\in\{1,\ldots,q-1\}\), \(K\in\Nzero\) with \(K<m_J\), and
\(j\in\{1,\ldots,q\}\), put
\(\xi_{J,K}=\beta_J+K\) and
\(\varepsilon_{j,J}=1-\delta_{j,J}\), and use the vectors \(\ee_j^-\)
defined above. Set
\(\mathcal C_{j,J,K}^{\mathrm{LI}}=0\) when
\(K<\varepsilon_{j,J}\), and otherwise define
\begin{equation}
 \mathcal C_{j,J,K}^{\mathrm{LI}}(x)
 =\Delta_{j,J,K}^{\mathrm{LI}}
 \pFq{q}{q}
 {1,\bbeta_{<q}+(1-\xi_{J,K})\one_{q-1}-\ee_j^-}
 {\A+(1-\xi_{J,K})\one_q}{x},
 \label{eq:LI-finite-pole-block}
\end{equation}
where, for \(j\in\{1,\ldots,q-1\}\),
\begin{equation}
 \Delta_{j,J,K}^{\mathrm{LI}}
 =\frac{\prod_{h=1}^{q}(A_h-\beta_j)
   \prod_{\substack{1\le h\le q-1\\h\ne j}}(\beta_h-\xi_{J,K})}
  {\prod_{\substack{1\le h\le q-1\\h\ne j}}(\beta_h-\beta_j)
   \prod_{h=1}^{q}(A_h-\xi_{J,K})},
 \label{eq:LI-finite-pole-prefactor-row}
\end{equation}
and
\begin{equation}
 \Delta_{q,J,K}^{\mathrm{LI}}
 =-\frac{\prod_{h=1}^{q-1}(\beta_h-\xi_{J,K})}
         {\prod_{h=1}^{q}(A_h-\xi_{J,K})}.
 \label{eq:LI-finite-pole-prefactor-base}
\end{equation}
The hypergeometric polynomial has degree at most \(K\) for \(j=J\) and
degree at most \(K-1\) for \(j\ne J\).

For the polynomial part at infinity, put
\begin{equation}
 \boldsymbol\xi_\ell^{\mathrm{LI}}
 =([A_1]_\ell,\ldots,[A_q]_\ell),
 \qquad
 \boldsymbol\eta_\ell^{\mathrm{LI}}
 =([\beta_1]_\ell,\ldots,[\beta_{q-1}]_\ell),
 \label{eq:LI-infinity-strings}
\end{equation}
and
\begin{equation}
 p_{K,\ell}^{\mathrm{LI}}
 =\mathfrak h_{\ell-K}
 (\boldsymbol\xi_\ell^{\mathrm{LI}};
  \boldsymbol\eta_\ell^{\mathrm{LI}}),
 \qquad K\in\{0,\ldots,\ell\}.
 \label{eq:LI-infinity-transition}
\end{equation}
In particular, \(p_{\ell,\ell}^{\mathrm{LI}}=1\). For
\(\ell\in\mathbb N\), \(J\in\{1,\ldots,q-1\}\), and
\(K\in\Nzero\) with \(K<\min\{m_J,\ell\}\), define
\begin{equation}
 z_{J,K,\ell}^{\mathrm{LI}}
 =\frac{\prod_{g=1}^{q}(A_g-\beta_J-K)_\ell}
 {(-1)^K K!(\ell-K-1)!
  \prod_{\substack{1\le h\le q-1\\h\ne J}}
  (\beta_h-\beta_J-K)_\ell},
 \label{eq:LI-base-residue}
\end{equation}
and
\begin{equation}
 \mathcal E_{j,\ell}^{\mathrm{LI}}(x)
 =\delta_{j,q}x^\ell
 -\sum_{J=1}^{q-1}\sum_{K=0}^{\min(m_J-1,\ell-1)}
 z_{J,K,\ell}^{\mathrm{LI}}
 \mathcal C_{j,J,K}^{\mathrm{LI}}(x).
 \label{eq:LI-corrected-infinity-vector}
\end{equation}
Let \(\gamma_\ell^{\mathrm{LI}}\) be the finite triangular sum
\eqref{eq:unreflected-infinity-chain} with
\(p_{K,\ell}^X\) replaced by \(p_{K,\ell}^{\mathrm{LI}}\).
For every active component \(j\), define the sector at infinity by
\begin{equation}
 \mathscr L_{j,q}^{\infty}(x)
 \coloneq
 \sum_{\ell=0}^{m_q-1}\gamma_\ell^{\mathrm{LI}}
 \mathcal E_{j,\ell}^{\mathrm{LI}}(x),
 \label{eq:LI-infinity-block-definition}
\end{equation}
with the empty sum interpreted as zero.

For \(j\in\{1,\ldots,q\}\), set
\begin{equation}
 D_j^{\mathrm{LI}}(x)=
 \sum_{J=1}^{q-1}\sum_{K=0}^{m_J-1}
 \pi_{J,K}^{-}\mathcal C_{j,J,K}^{\mathrm{LI}}(x)
 +\sum_{\ell=0}^{m_q-1}
 \gamma_\ell^{\mathrm{LI}}
 \mathcal E_{j,\ell}^{\mathrm{LI}}(x).
 \label{eq:LI-explicit-components}
\end{equation}
For each \(J<q\), the finite-pole part of
\eqref{eq:LI-explicit-components} also has one terminating
Kamp\'e de F\'eriet block. Delete the entry
\(k+\beta_J+\varepsilon\) from
\(\boldsymbol a_{j,J}^{-}(k)\) and
\(\boldsymbol d_{j,J}^{-}(k)\); call the resulting strings
\(\widehat{\boldsymbol a}_{j,J}^{-}\) and
\(\widehat{\boldsymbol d}_{j,J}^{-}\). Define
\begin{equation}
 \mathscr L_{j,J}(x)
 \coloneq \pi_{J,\varepsilon}^{-}
 \Delta_{j,J,\varepsilon}^{\mathrm{LI}}
 F_{q+1:q-2;0}^{q:q;1}
 \left(
 \begin{matrix}
  \widehat{\boldsymbol a}_{j,J}^{-}:
  \boldsymbol b_{j,J}^{\mathrm{MII}};1\\
  \boldsymbol c_{j,J}^{\mathrm{MII}}:
  \widehat{\boldsymbol d}_{j,J}^{-};\text{--}
 \end{matrix}
 \middle|1,-x\right),
 \label{eq:LI-compact-finite-pole-block}
\end{equation}
and set it equal to zero if \(m_J<\varepsilon+1\).

\begin{proposition}[Compact Laguerre-I sectors]
\label{prop:LI-compact-finite-poles}
For every \(J\in\{1,\ldots,q-1\}\) and every active component \(j\),
\begin{equation}
 \sum_{K=0}^{m_J-1}\pi_{J,K}^{-}
 \mathcal C_{j,J,K}^{\mathrm{LI}}(x)
 =\mathscr L_{j,J}(x).
 \label{eq:LI-compact-finite-pole-identity}
\end{equation}
If this block is nonzero, then
\begin{equation}
 \deg\mathscr L_{j,J}
 \le m_J-1-\varepsilon_{j,J}.
 \label{eq:LI-finite-sector-degree}
\end{equation}
Hence \eqref{eq:LI-explicit-components} is equivalently the complete
sector decomposition
\begin{equation}
 D_j^{\mathrm{LI}}(x)=
 \sum_{J=1}^{q-1}\mathscr L_{j,J}(x)
 +\mathscr L_{j,q}^{\infty}(x).
 \label{eq:LI-compact-components}
\end{equation}
Writing \(\varepsilon_{j,q}=1-\delta_{j,q}\), the infinity block is zero
when \(m_q<\varepsilon_{j,q}+1\). Otherwise,
\begin{equation}
 \deg\mathscr L_{j,q}^{\infty}\le
 \begin{cases}
  m_q-1,&j=q,\\
  m_q-2,&j<q.
 \end{cases}
 \label{eq:LI-infinity-sector-degree}
\end{equation}
\end{proposition}

\begin{proof}
Put \(k=\tau x\) in
\eqref{eq:MII-compact-finite-pole-block}. For every term of its finite
double sum,
\[
 \frac{(\tau x+\beta_J+\varepsilon)_{r+t}}
 {(\tau x+\beta_J+\varepsilon)_r}
 (-\tau^{-1})^t\longrightarrow(-x)^t.
\]
All other parameter strings are fixed, and the prefactor tends from
\(\Delta_{j,J,\varepsilon}^{\mathrm{MII}}\) to
\(\Delta_{j,J,\varepsilon}^{\mathrm{LI}}\). Termwise convergence is
valid because \(r+t\le m_J-1-\varepsilon\). Thus
\(\mathscr M_{j,J}(\tau x)\to\mathscr L_{j,J}(x)\).
The same termination bound gives
\eqref{eq:LI-finite-sector-degree}.
Taking the same limit in
\eqref{eq:MII-compact-finite-pole-identity}, using the already proved
limit of each \(\mathcal C_{j,J,K}^{\mathrm{MII}}\), gives
\eqref{eq:LI-compact-finite-pole-identity}; summing over \(J\) proves
\eqref{eq:LI-compact-components}. Finally,
\eqref{eq:LI-corrected-infinity-vector} contains the monomial \(x^\ell\)
only in component \(q\), while every correction has degree at most
\(\ell-1\). This proves the vanishing assertion and
\eqref{eq:LI-infinity-sector-degree}.
\end{proof}

For the original weights \(w_j^{\mathrm{LI}}\), the corresponding
components are \(D_j^{\mathrm{LI}}/\beta_j\) for
\(j\in\{1,\ldots,q-1\}\) and
\(D_q^{\mathrm{LI}}\) for \(j=q\).

\begin{theorem}[Laguerre-I-like type-I components and normality]
\label{thm:explicit-LI-B}
Let \(\mm\) be near the diagonal with \(\abs\mm\ge1\), and set
\(n=\abs\mm-1\). Assume that the sets
\(\{-\beta_J-K:K\in\Nzero,\ K<m_J\}\), indexed by
\(J\in\{1,\ldots,q-1\}\), are pairwise disjoint and that
the denominators in \eqref{eq:LI-finite-pole-block}--
\eqref{eq:LI-base-residue} do not vanish. Then
\(\deg D_j^{\mathrm{LI}}<m_j\) for every
\(j\in\{1,\ldots,q\}\) with \(m_j>0\), and
\begin{equation}
 \mathcal B_{\mm}^{\mathrm{LI}}(x)
 =\sum_{j=1}^qD_j^{\mathrm{LI}}(x)
   \widetilde w_j^{\mathrm{LI}}(x).
 \label{eq:LI-component-decomposition}
\end{equation}
Its moments are
\begin{equation}
 \int_0^\infty x^r\mathcal B_{\mm}^{\mathrm{LI}}(x)\dd x
 =-h_{\mathrm{LI}}(r)\frac{(-r)_n}{D_-(r)},
 \qquad r\in\Nzero.
 \label{eq:LI-complete-moments}
\end{equation}
In particular,
\(\int_0^\infty x^n\mathcal B_{\mm}^{\mathrm{LI}}(x)\dd x
=(-1)^{n+1}n!h_{\mathrm{LI}}(n)/D_-(n)\ne0\).
The moments of orders \(r\in\Nzero\) with \(r<n\) vanish. The multi-index
\(\mm\) is normal, and
\((D_1^{\mathrm{LI}},\ldots,D_q^{\mathrm{LI}})\) is the unique tuple of
type-I polynomials with the stated degree bounds relative to the rescaled weights
\eqref{eq:LI-canonical-rows}.
\end{theorem}

\begin{proof}
The quotient of consecutive moments is
\(h_{\mathrm{LI}}(r+\ell)/h_{\mathrm{LI}}(r)
=(r+\A)_\ell/(r+\bbeta_{<q})_\ell\).
For \(j\in\{1,\ldots,q-1\}\), the factor is
\((r+\beta_j+\ell)^{-1}\); the \(q\)-th weight has no such factor. This
proves \eqref{eq:LI-component-kernel-row} and
\eqref{eq:LI-component-kernel-base}.

We prove the simple-pole identity by a coefficientwise
Meixner-II-like-to-Laguerre-I-like limit. In the term of order \(s\) of
\(\mathcal C_{j,J,K}^{\mathrm{MII}}(\tau x)\),
\((1-\tau x-\xi_{J,K})_s(-\tau^{-1})^s
\xrightarrow[\tau\to\infty]{}x^s\).
The terminating parameter in the \(J\)-th position is \(-K\) for
\(j=J\) and \(1-K\) otherwise, so only finitely many values of \(s\)
occur. The prefactor
\eqref{eq:MII-finite-pole-prefactor-row} is independent of \(\tau\),
and
\[
 \Delta_{q,J,K}^{\mathrm{MII}}
 \xrightarrow[\tau\to\infty]{}
 -\frac{\prod_{h=1}^{q-1}(\beta_h-\xi_{J,K})}
       {\prod_{h=1}^{q}(A_h-\xi_{J,K})}
 =\Delta_{q,J,K}^{\mathrm{LI}}.
\]
It follows that
\(\mathcal C_{j,J,K}^{\mathrm{MII}}(\tau x)
\xrightarrow[\tau\to\infty]{}\mathcal C_{j,J,K}^{\mathrm{LI}}(x)\).
For a polynomial of degree at most \(K\), the corresponding
factorial-moment expression uses only moments of
orders \(r,\ldots,r+K\). Dividing those Meixner-II-like moments by
the appropriate powers of \(\tau\) gives the Laguerre-I moments above.
Taking the limit in \eqref{eq:finite-pole-block-pairing} therefore yields
\begin{equation}
 \sum_{j=1}^q\frac1{h_{\mathrm{LI}}(r)}
 \int_0^\infty x^r
 \mathcal C_{j,J,K}^{\mathrm{LI}}(x)
 \widetilde w_j^{\mathrm{LI}}(x)\dd x
 =\frac1{r+\xi_{J,K}}.
 \label{eq:LI-simple-pole-pairing}
\end{equation}

The rational function
\eqref{eq:LI-component-kernel-base} has a pole at
\(r=-\beta_J-K\) precisely when \(K\in\Nzero\) and \(K<\ell\). Since
\(\Res_{r=-\beta_J-K}\frac1{(r+\beta_J)_\ell}
=\frac1{(-1)^K K!(\ell-K-1)!}\),
evaluation of the remaining Pochhammer factors gives
\eqref{eq:LI-base-residue}.  After subtracting all these simple-pole
parts by means of \eqref{eq:LI-simple-pole-pairing}, the normalized moment
function of \(\mathcal E_{\cdot,\ell}^{\mathrm{LI}}\) is a polynomial.

To calculate that polynomial, expand \(R_{q,\ell}^{\mathrm{LI}}(r)\) at
infinity:
\[
 R_{q,\ell}^{\mathrm{LI}}(r)
 =r^\ell
 \frac{\prod_{\xi\in\boldsymbol\xi_\ell^{\mathrm{LI}}}
       (1+\xi/r)}
      {\prod_{\eta\in\boldsymbol\eta_\ell^{\mathrm{LI}}}
       (1+\eta/r)}.
\]
The coefficient of \(r^K\) is
\(\mathfrak h_{\ell-K}
(\boldsymbol\xi_\ell^{\mathrm{LI}};
 \boldsymbol\eta_\ell^{\mathrm{LI}})\), and the subtracted simple
fractions have zero polynomial part.  We have therefore proved
\begin{equation}
 \sum_{j=1}^q\frac1{h_{\mathrm{LI}}(r)}
 \int_0^\infty x^r
 \mathcal E_{j,\ell}^{\mathrm{LI}}(x)
 \widetilde w_j^{\mathrm{LI}}(x)\dd x
 =\sum_{K=0}^{\ell}p_{K,\ell}^{\mathrm{LI}}r^K.
 \label{eq:LI-polynomial-pairing}
\end{equation}
The leading coefficient is one.

The target rational function \(\mathcal R_{\mm}\) has the partial-fraction and
polynomial expansion \eqref{eq:unreflected-target-partial-fractions},
with coefficients \eqref{eq:unreflected-finite-pole-coefficient} and
\eqref{eq:unreflected-target-infinity-coefficients}.  Because
\(p_{\ell,\ell}^{\mathrm{LI}}=1\), descending coefficient comparison in
the polynomials on the right of \eqref{eq:LI-polynomial-pairing} gives
\[
 \gamma_\ell^{\mathrm{LI}}
 =\pi_{\infty,\ell}
 -\sum_{s=\ell+1}^{m_q-1}
  p_{\ell,s}^{\mathrm{LI}}\gamma_s^{\mathrm{LI}}.
\]
Iterating this finite recurrence gives the triangular sum specified before
the theorem. Hence \eqref{eq:LI-simple-pole-pairing} and
\eqref{eq:LI-polynomial-pairing}, inserted in
\eqref{eq:LI-explicit-components}, give
\[
 \sum_{j=1}^q\frac1{h_{\mathrm{LI}}(r)}
 \int_0^\infty x^rD_j^{\mathrm{LI}}(x)
 \widetilde w_j^{\mathrm{LI}}(x)\dd x
 =\mathcal R_{\mm}(r),
\]
which proves \eqref{eq:LI-complete-moments}.

The hypergeometric term associated with \((J,K)\) has degree at most \(K\)
in component \(J\) and at most \(K-1\) in every other component. In the
polynomial part at infinity, the \(q\)-th monomial has
degree \(\ell<m_q\), and its corrections have \(K<\ell\).  If a
correction lies in a component distinct from \(J\), near-diagonality implies
\(K-1\le m_J-2\le m_j-1\).  These inequalities prove
\(\deg D_j^{\mathrm{LI}}<m_j\).

Equation \eqref{eq:LI-complete-moments} also identifies the constructed
linear form with the signed linear form whose Laplace transform is
\eqref{eq:LI-B-Laplace}.  Indeed, the moments vanish for \(r<n\).  For
\(r=n+t\),
\[
 \frac{(-s)^{n+t}}{(n+t)!}
 \left[-h_{\mathrm{LI}}(n+t)
       \frac{(-(n+t))_n}{D_-(n+t)}\right]
 =-\frac{(\A)_n}{(\bbeta_{<q})_{\mm_{<q}+n}}
 s^n
 \frac{(\A+n)_t}{(\bbeta_{<q}+\mm_{<q}+n)_t}
 \frac{(-s)^t}{t!}.
\]
Summing over \(t\ge0\) gives \eqref{eq:LI-B-Laplace}. The representation
\eqref{eq:LI-positive-row} has an exponential moment in
a neighborhood of the origin, and multiplication by a fixed polynomial
preserves that property.  Both Laplace transforms are analytic there;
equality of their Taylor series proves
\eqref{eq:LI-component-decomposition} as an identity of signed densities.

To prove normality, consider the \(n+1\) component vectors
\[
 \{\mathcal C_{\cdot,J,K}^{\mathrm{LI}}:
 J\in\{1,\ldots,q-1\},\ K\in\Nzero,\ K<m_J\}
 \ \cup\
 \{\mathcal E_{\cdot,\ell}^{\mathrm{LI}}:
 \ell\in\Nzero,\ \ell<m_q\}.
\]
Their normalized moment functions, given by
\eqref{eq:LI-simple-pole-pairing} and
\eqref{eq:LI-polynomial-pairing}, are linearly independent: distinct
simple poles separate the first group, and coefficient comparison from
the highest degree down separates the second.  The component space has
dimension \(\sum_jm_j=n+1\), so these vectors form a basis.

Let a component vector of degrees below \(\mm\) have moments zero for
every \(r\in\{0,\ldots,n\}\), and divide its factorial-moment function by
\(h_{\mathrm{LI}}(r)\). Multiplication by \(D_-(r)\) gives a polynomial
of degree at most \(n\): equations
\eqref{eq:LI-component-kernel-row}--
\eqref{eq:LI-component-kernel-base} give growth at most
\(r^{\ell-1}\) for \(j\in\{1,\ldots,q-1\}\) and \(r^\ell\) for
\(j=q\), and
near-diagonality supplies the remaining degree inequalities.  The
polynomial vanishes at \(0,1,\ldots,n\), because
\(D_-(r)h_{\mathrm{LI}}(r)\ne0\) at those points.  It is therefore zero.
Expansion in the preceding component basis shows that the original
vector is zero.  Thus the moment map is injective, proving normality and
the uniqueness assertion.
\end{proof}

\begin{corollary}[Sectorwise Meixner-II-like-to-Laguerre-I-like component limit]
\label{cor:MII-LI-component-limit}
Under the hypotheses of Theorem~\ref{thm:explicit-LI-B}, let \(\tau>0\)
and keep all remaining parameters fixed. For every
\(j\in\{1,\ldots,q\}\) with \(m_j>0\), as \(\tau\to\infty\),
\begin{align}
 \mathscr M_{j,J}(\tau\,\cdot)
 &\xrightarrow[\tau\to\infty]{}\mathscr L_{j,J},
 &&J\in\{1,\ldots,q-1\},
 \label{eq:MII-LI-finite-sector-limit}\\
 \mathscr M_{j,q}^{\infty}(\tau\,\cdot)
 &\xrightarrow[\tau\to\infty]{}\mathscr L_{j,q}^{\infty}.
 \label{eq:MII-LI-infinity-sector-limit}
\end{align}
Consequently,
\begin{equation}
 D_j^{\mathrm{MII}}(\tau\,\cdot)
 \xrightarrow[\tau\to\infty]{}D_j^{\mathrm{LI}}.
 \label{eq:MII-LI-component-limit}
\end{equation}
Equivalently, the convergence is coefficientwise and locally uniform on
compact subsets of \(\mathbb C\). For every
\(j\in\{1,\ldots,q\}\) with \(m_j=0\),
\(D_j^{\mathrm{MII}}=D_j^{\mathrm{LI}}=0\).
\end{corollary}

\begin{proof}
The finite-sector limit
\eqref{eq:MII-LI-finite-sector-limit} was established in
Proposition~\ref{prop:LI-compact-finite-poles}. In
\eqref{eq:unreflected-base-residue}, divide the
Meixner-II-like residue by \(\tau^\ell\). Every term with \(u<\ell\)
tends to zero, and the term \(u=\ell\) gives
\(\tau^{-\ell}z_{J,K,\ell}^{\mathrm{MII}}
\xrightarrow[\tau\to\infty]{}z_{J,K,\ell}^{\mathrm{LI}}\).
Also,
\(\tau^{-\ell}\fall{\tau x}{\ell}
\xrightarrow[\tau\to\infty]{}x^\ell\).  It follows that
\begin{equation}
 \tau^{-\ell}\mathcal E_{j,\ell}^{\mathrm{MII}}(\tau x)
 \xrightarrow[\tau\to\infty]{}\mathcal E_{j,\ell}^{\mathrm{LI}}(x).
 \label{eq:MII-LI-infinity-vector-limit}
\end{equation}

For the polynomial coefficients, divide the Meixner-II-like function
\(R_{q,\ell}^{\mathrm{MII}}\) by
\(\tau^\ell\).  In
\eqref{eq:unreflected-R-base-row}, all summands with \(u<\ell\) vanish,
whereas the term \(u=\ell\) tends to
\((r+\A)_\ell/(r+\bbeta_{<q})_\ell\). Thus
\(\tau^{-\ell}p_{K,\ell}^{\mathrm{MII}}
\xrightarrow[\tau\to\infty]{}p_{K,\ell}^{\mathrm{LI}}\).
Put
\(\widehat\gamma_\ell(\tau)=
 \tau^\ell\gamma_\ell^{\mathrm{MII}}\) and
\(\widehat p_{K,\ell}(\tau)=
 \tau^{-\ell}p_{K,\ell}^{\mathrm{MII}}\).  The coefficient identities
for the target polynomial become
\(\pi_{\infty,K}
=\sum_{\ell=K}^{m_q-1}
\widehat p_{K,\ell}(\tau)\widehat\gamma_\ell(\tau)\).
The diagonal coefficients satisfy
\(\widehat p_{\ell,\ell}(\tau)=(1+\tau)^\ell/\tau^\ell
\xrightarrow[\tau\to\infty]{}1\).
Descending substitution therefore gives
\(\tau^\ell\gamma_\ell^{\mathrm{MII}}\xrightarrow[\tau\to\infty]{}
\gamma_\ell^{\mathrm{LI}}\).
Since the sum over \(\ell\) is finite,
\[
 \mathscr M_{j,q}^{\infty}(\tau x)
 =\sum_{\ell=0}^{m_q-1}
 \bigl(\tau^\ell\gamma_\ell^{\mathrm{MII}}\bigr)
 \bigl(\tau^{-\ell}
       \mathcal E_{j,\ell}^{\mathrm{MII}}(\tau x)\bigr)
 \longrightarrow \mathscr L_{j,q}^{\infty}(x).
\]
This proves \eqref{eq:MII-LI-infinity-sector-limit}. Summing
\eqref{eq:MII-LI-finite-sector-limit} and
\eqref{eq:MII-LI-infinity-sector-limit}, and using
\eqref{eq:MII-all-sector-components} and
\eqref{eq:LI-compact-components}, proves
\eqref{eq:MII-LI-component-limit} coefficient by coefficient.
\end{proof}

\paragraph{Classical Laguerre-I reduction.}
For \(q=1\), the type-II polynomial is
\begin{equation}
 A_{m_1}^{\mathrm{LI}}(x)
 =\pFq{1}{1}{-m_1}{A_1}{x}
 =\frac{m_1!}{(A_1)_{m_1}}L_{m_1}^{(A_1-1)}(x).
 \label{eq:q1-LI-A}
\end{equation}
With \(n=m_1-1\), the type-I component is
\begin{equation}
 D_1^{\mathrm{LI}}(x)
 =-(A_1)_n\pFq{1}{1}{-n}{A_1}{x}
 =-n!L_n^{(A_1-1)}(x).
 \label{eq:q1-LI-B}
\end{equation}

\paragraph{Verification.}
When \(q=1\), there are no finite-pole terms,
\(h_{\mathrm{LI}}(r)=(A_1)_r\), and
\(L_n(x)=-(A_1)_n\pFq{1}{1}{-n}{A_1}{x}\).
Termwise integration against
\(x^{A_1-1}\e^{-x}/\Gamma(A_1)\) gives
\[
 \frac1{\Gamma(A_1)}\int_0^\infty
 x^{A_1+r-1}\e^{-x}L_n(x)\dd x
 =-(A_1)_n(A_1)_r
 \pFq{2}{1}{-n,A_1+r}{A_1}{1}
 =-(A_1)_r(-r)_n.
\]
The last equality is Chu--Vandermonde,
\({}_2F_1(-n,A_1+r;A_1;1)=(-r)_n/(A_1)_n\).
Thus \(L_n\) has the moments in
\eqref{eq:LI-complete-moments}; uniqueness proves
\eqref{eq:q1-LI-B}.

\section{The reflected Meixner-I-like family}
\label{sec:MI}

We next reflect the Hahn lattice and send the opposite endpoint to infinity.
This produces a Meixner-I-like family with an explicit type-II polynomial,
an explicit signed sequence satisfying the type-I moment conditions, and a
direct Laguerre-II-like limit.

\subsection{Right-endpoint scaling and limiting weights}

Here the lattice is reflected before the endpoint is sent to infinity. We
first identify the limiting weights and their generating functions.
Let
\begin{equation}
 A_h=T\rho_h,\qquad
 \beta_h=T\rho_h+\delta_h,\qquad
 T\to\infty,\qquad
 \frac NT\xrightarrow[T\to\infty]{}\nu>0,
 \label{eq:MI-scaling}
\end{equation}
for every \(h\in\{1,\ldots,q\}\), where \(\rho_h,\delta_h>0\), and reflect the lattice by
\(\ell=N-k\). Put
\begin{equation}
 c_h=\frac{\nu}{\nu+\rho_h},
 \qquad
 \lambda_h=\frac{c_h}{1-c_h}=\frac{\nu}{\rho_h},
 \qquad
 R_h(z)=\frac{1-c_h}{1-c_hz},
 \qquad h\in\{1,\ldots,q\}.
 \label{eq:MI-parameters}
\end{equation}
The parameters in \eqref{eq:MI-parameters} satisfy
\(0<c_h<1\) and \(\lambda_h>0\).

For sequences on \(\Nzero\), \(f*g\) denotes the discrete convolution
\((f*g)(k)=\sum_{\ell=0}^kf(\ell)g(k-\ell)\).

After reflection, this scaling produces the following Meixner-I-like
weights.

\begin{theorem}[Meixner-I-like weights]
\label{thm:MI-rows}
As \(T\to\infty\) under \eqref{eq:MI-scaling}, for every
\(j\in\{1,\ldots,q\}\), the reflected normalized weights converge
coefficientwise and in every fixed moment to a normalized nonnegative
sequence \(v_j^{\mathrm{MI}}\) on \(\Nzero\), with generating function
\begin{equation}
 G_j^{\mathrm{MI}}(z)
 =\prod_{h=1}^qR_h(z)^{\delta_h+\delta_{h,j}}.
 \label{eq:MI-row-pgf}
\end{equation}
Equivalently,
\begin{equation}
 v_j^{\mathrm{MI}}
 =r_{\delta_1+\delta_{1,j},c_1}*\cdots
  *r_{\delta_q+\delta_{q,j},c_q}.
 \label{eq:MI-convolution}
\end{equation}
\end{theorem}

\begin{proof}
The generating function of the reflected finite weight is
\[
 \widehat{\mathcal V}_{j,N}(z)
 =\int_{(0,1)^q}
 \bigl[y_1\cdots y_q+(1-y_1\cdots y_q)z\bigr]^N
 \prod_{h=1}^qb_{T\rho_h,\delta_h+\delta_{h,j}}(y_h)
 \dd\boldsymbol y.
\]
Under \(y_h=1-u_h/T\),
\(T^{-1}b_{T\rho_h,d}(1-u/T)
\xrightarrow[T\to\infty]{}g_{\rho_h,d}(u)\),
and, since \(N/T\xrightarrow[T\to\infty]{}\nu\),
\(\bigl[y_1\cdots y_q+(1-y_1\cdots y_q)z\bigr]^N
\xrightarrow[T\to\infty]{} \e^{\nu(z-1)(u_1+\cdots+u_q)}\).
Dominated convergence on compact subsets of the unit disk gives
\[
 \widehat{\mathcal V}_{j,N}(z)\xrightarrow[T\to\infty]{}
 \prod_{h=1}^q
 \left(\frac{\rho_h}{\rho_h+\nu(1-z)}\right)^{
 \delta_h+\delta_{h,j}}
 =G_j^{\mathrm{MI}}(z).
\]
On every disk
\(\abs{z-1}\le\varepsilon<\min_h\rho_h/\nu\), the transformed beta
kernels are bounded by a common integrable exponential majorant.
Dominated convergence is therefore uniform on that disk.  Cauchy's
integral formula permits termwise differentiation at \(z=1\), proving
convergence of all fixed factorial moments. The
coefficient sequence of each factor is
\(r_{\delta_h+\delta_{h,j},c_h}\), which proves
\eqref{eq:MI-convolution}, nonnegativity, and normalization. Cauchy's
coefficient formula gives the asserted coefficient limits.
\end{proof}

\begin{proposition}[Distinctness of the two Meixner-like systems]
\label{prop:Meixner-families-distinct}
For \(q=1\), both constructions give ordinary normalized Meixner weights;
their rows coincide when \(c_1=\tau/(1+\tau)\) and
\(A_1=\delta_1+1>1\). Let \(q>1\), and consider arbitrary families in the
positive domains of Theorems~\ref{thm:MII-rows} and \ref{thm:MI-rows}.
There is no permutation \(\pi\) of \(\{1,\ldots,q\}\) for which
\[
 v_j^{\mathrm{MII}}(k)=v_{\pi(j)}^{\mathrm{MI}}(k),
 \qquad j\in\{1,\ldots,q\},\quad k\in\Nzero.
\]
Consequently, the two normalized multiple orthogonality systems cannot be
identified by a permutation of their weights. Allowing nonzero rowwise
constant factors does not change the conclusion.
\end{proposition}

\begin{proof}
The scalar statement follows by comparing \eqref{eq:MII-pgf} and
\eqref{eq:MI-row-pgf}. Now let \(q>1\), and set
\(\mu_{j,r}^{\mathrm{MII}}
=\sum_{k\ge0}\fall{k}{r}v_j^{\mathrm{MII}}(k)\).
By \eqref{eq:MII-factorial-moments}, for every \(j<q\),
\[
 (\beta_j+r)\mu_{j,r}^{\mathrm{MII}}
 =\beta_j\mu_{q,r}^{\mathrm{MII}},\qquad r\in\Nzero.
\]
Comparing Taylor coefficients at \(z=1\) gives
\begin{equation}
 \bigl((z-1)\partial_z+\beta_j\bigr)G_j^{\mathrm{MII}}(z)
 =\beta_jG_q^{\mathrm{MII}}(z),
 \qquad j\in\{1,\ldots,q-1\}.
 \label{eq:MII-resolvent-relation}
\end{equation}

Suppose such a permutation exists. Normalization first shows that allowing
rowwise factors gives no additional freedom. Fix \(j<q\), and put
\(i=\pi(j)\), \(b=\pi(q)\), and \(\beta=\beta_j\); then \(i\ne b\).
Equality of the weight sequences gives equality of the generating
functions in the unit disk and hence, by analytic continuation along
\([0,1]\), equality of their germs at \(z=1\).
Put \(w=z-1\), \(\lambda_h=c_h/(1-c_h)>0\), and
\(F(w)=\prod_{h=1}^q(1-\lambda_hw)^{-\delta_h}\). Then
\eqref{eq:MI-row-pgf} gives
\[
 G_i^{\mathrm{MI}}(1+w)=\frac{F(w)}{1-\lambda_iw},
 \qquad
 G_b^{\mathrm{MI}}(1+w)=\frac{F(w)}{1-\lambda_bw}.
\]
Substitution in \eqref{eq:MII-resolvent-relation} and division by the first
of these functions gives the rational identity
\begin{equation}
 \beta+
 w\left(
   \sum_{h=1}^q\frac{\delta_h\lambda_h}{1-\lambda_hw}
   +\frac{\lambda_i}{1-\lambda_iw}
 \right)
 =
 \beta\,\frac{1-\lambda_iw}{1-\lambda_bw}.
 \label{eq:MI-rational-obstruction}
\end{equation}

For each distinct value \(\lambda_*\) among the \(\lambda_h\), its
coefficient on the left is
\(\sum_{\lambda_h=\lambda_*}\delta_h+
\mathbf1_{\{\lambda_i=\lambda_*\}}>0\); hence
\(w=\lambda_*^{-1}\) is a nonremovable pole. The right-hand side can have
a pole only at \(w=\lambda_b^{-1}\), so every \(\lambda_h=\lambda_b\).
Then \eqref{eq:MI-rational-obstruction} reduces to the impossible identity
\[
 \beta+
 \frac{w\lambda_b\bigl(1+\sum_{h=1}^q\delta_h\bigr)}
      {1-\lambda_bw},
=\beta.
\]
This contradiction proves the assertion.
\end{proof}

For \(q>1\), these weights differ from the classical multiple Meixner-I
system
\cite{ArvesuCoussementVanAssche2003,
BranquinhoDiazFoulquieManasWolfs2024Discrete}.  The terminology
\enquote{Meixner-I-like} refers to their direct Laguerre-II-like limit;
the weights themselves are the convolutions in \eqref{eq:MI-convolution}.

\subsection{Explicit formulas after reflection}

We now give explicit formulas for the type-II polynomial \(P\) and for a
signed sequence satisfying the type-I moment conditions. We also state when
its type-I polynomials are uniquely determined.

For a near-diagonal \(\mm\), put
\begin{equation}
 D=\abs\mm,\qquad n=D-1,\qquad
 F_{\mm}(s)=\prod_{h=1}^q(s+\rho_h)^{\delta_h+m_h}.
 \label{eq:MI-F}
\end{equation}
The product \eqref{eq:MI-F} is the common factor in the derivative and
type-I formulas below.

The limiting polynomial has three equivalent forms.

\begin{theorem}[Meixner-I-like polynomial]
\label{thm:MI-A}
Let \(\mm\) be near the diagonal and put \(D=\abs\mm\). As
\(T\to\infty\) under
\eqref{eq:MI-scaling}, the exact monic reflection
\begin{equation}
 \widehat P_{\mm,N}(\ell)
 =\fall{N}{D}\frac{(\A)_D}{(\bbeta+\mm)_D}
 A^{\mathrm H}_{\mm,N}(N-\ell)
 \label{eq:MI-finite-monic-reflection}
\end{equation}
converges coefficientwise under \eqref{eq:MI-scaling} to
\begin{align}
 P_{\mm}^{\mathrm{MI}}(\ell)
 &=D![u^D](1+u)^\ell
   \prod_{h=1}^q(1-\lambda_hu)^{\delta_h+m_h}
 \label{eq:MI-A-coefficient}\\
 &=\frac{(-1)^D\nu^D}{F_{\mm}(0)}
 \left.\frac{\dd^D}{\dd s^D}
 \left[F_{\mm}(s)\left(1-\frac{s}{\nu}\right)^\ell\right]
 \right|_{s=0}.
 \label{eq:MI-A-derivative}
\end{align}
It is monic of degree \(D\) and, for every
\(j\in\{1,\ldots,q\}\) with \(m_j>0\), satisfies
\begin{equation}
 \sum_{\ell\ge0}\fall{\ell}{r}P_{\mm}^{\mathrm{MI}}(\ell)
 v_j^{\mathrm{MI}}(\ell)=0,
 \qquad r\in\{0,\ldots,m_j-1\}.
 \label{eq:MI-A-orthogonality}
\end{equation}
Its finite Horn expansion is
\begin{equation}
 P_{\mm}^{\mathrm{MI}}(\ell)
 =D!\!\sum_{\substack{s_0,s_1,\ldots,s_q\in\Nzero\\
                        s_0+s_1+\cdots+s_q=D}}
 \frac{\fall{\ell}{s_0}}{s_0!}
 \prod_{h=1}^q
 \frac{(-\delta_h-m_h)_{s_h}}{s_h!}\lambda_h^{s_h}.
 \label{eq:MI-A-Horn}
\end{equation}
\end{theorem}

\begin{proof}
Put \(a_h=\delta_h+m_h\). From \eqref{eq:Hahn-A-factorial} and
\eqref{eq:MI-finite-monic-reflection},
\begin{equation}
 \widehat P_{\mm,N}(\ell)
 =C_T\sum_{r=0}^D(-1)^r\binom Dr\Psi_T(r;\ell),
 \label{eq:MI-finite-difference-form}
\end{equation}
where \(C_T=\fall{N}{D}\prod_{h=1}^q
(T\rho_h)_D/(T\rho_h+a_h)_D\) and
\(C_T/T^D\xrightarrow[T\to\infty]{}\nu^D\). Also put
\(\Psi_T(r;\ell)=\fall{N-\ell}{r}\prod_{h=1}^q
(T\rho_h+a_h)_r/[\fall{N}{r}(T\rho_h)_r]\).
Extend the latter expression analytically in its index:
\[
 \Psi_T(x;\ell)=
 \frac{\Gamma(N-\ell+1)\Gamma(N-x+1)}
 {\Gamma(N-\ell-x+1)\Gamma(N+1)}
 \prod_{h=1}^q
 \frac{\Gamma(T\rho_h+a_h+x)\Gamma(T\rho_h)}
 {\Gamma(T\rho_h+a_h)\Gamma(T\rho_h+x)}.
\]
Set \(H_T(y;\ell)=\Psi_T(Ty;\ell)\). Uniform gamma-ratio
asymptotics, together with their derivatives through order \(D\), give
\[
 H_T(y;\ell)\xrightarrow[T\to\infty]{}
 H(y;\ell)\coloneq\left(1-\frac y\nu\right)^\ell
 \prod_{h=1}^q\left(1+\frac y{\rho_h}\right)^{a_h}
\]
in \(C^D\) on a neighbourhood of the origin. If
\(\Delta_hf(y)=f(y+h)-f(y)\), then
\[
 T^D\sum_{r=0}^D(-1)^r\binom Dr\Psi_T(r;\ell)
 =(-1)^DT^D\Delta_{1/T}^DH_T(0;\ell)
 \xrightarrow[T\to\infty]{}(-1)^D H^{(D)}(0;\ell).
\]
The last passage follows, for example, from
\(h^{-D}\Delta_h^Df(0)=
\int_{[0,1]^D}f^{(D)}\bigl(h(t_1+\cdots+t_D)\bigr)
\dd t_1\cdots\dd t_D\).
Combining this limit with
\(C_T/T^D\xrightarrow[T\to\infty]{}\nu^D\) yields
\[
 \widehat P_{\mm,N}(\ell)\xrightarrow[T\to\infty]{}
 \nu^D(-1)^DH^{(D)}(0;\ell)
 =D![u^D](1+u)^\ell
   \prod_{h=1}^q(1-\nu u/\rho_h)^{a_h},
\]
which is \eqref{eq:MI-A-coefficient}. The convergence is coefficientwise,
because both sides have degree at most \(D\) and convergence holds at
\(D+1\) fixed values of \(\ell\). The constant term inside the
coefficient extraction is one, so the limit is monic. The substitution
\(s=-\nu u\) gives \eqref{eq:MI-A-derivative}, and multinomial expansion
gives \eqref{eq:MI-A-Horn}. Finally, coefficientwise convergence and the
fixed-moment convergence in Theorem~\ref{thm:MI-rows} permit passage in
the exact reflected orthogonality relations.
\end{proof}

Formula \eqref{eq:MI-A-Horn} is a terminating multiple hypergeometric
polynomial in the standard Horn sense \cite{SrivastavaKarlsson1985}.

The signed sequence satisfying the type-I moment conditions has a product
generating function.
Define the normalized nonnegative coefficient sequence
\begin{equation}
 s_{\mm}^{\mathrm{MI}}(\ell)
 =[z^\ell]\prod_{h=1}^qR_h(z)^{\delta_h+m_h}
 =\bigl(r_{\delta_1+m_1,c_1}*\cdots
 *r_{\delta_q+m_q,c_q}\bigr)(\ell).
 \label{eq:MI-B-seed}
\end{equation}
Set \(s_{\mm}^{\mathrm{MI}}(r)=0\) for \(r<0\).

\begin{theorem}[Unit-normalized Meixner-I-like signed sequence]
\label{thm:MI-B}
Let \(\mm\) be near the diagonal with \(D\coloneq\abs\mm\ge1\), and set
\(n=D-1\). In the normalization in which the moment of order \(n\) equals one, the
signed sequence has generating function
\begin{equation}
 \mathcal H_{\mm}^{\mathrm{MI}}(z)
 =\frac{F_{\mm}(0)}{n!}
   \frac{(z-1)^n}{F_{\mm}(\nu(1-z))}
 =\frac{(z-1)^n}{n!}
   \prod_{h=1}^qR_h(z)^{\delta_h+m_h}.
 \label{eq:MI-B-product-pgf}
\end{equation}
Its coefficients are
\begin{equation}
 \mathcal B_{\mm}^{\mathrm{MI}}(\ell)
 =\frac1{n!}\sum_{a=0}^n(-1)^{n-a}\binom na
 s_{\mm}^{\mathrm{MI}}(\ell-a),
 \qquad \ell\in\Nzero.
 \label{eq:MI-B-difference}
\end{equation}
\end{theorem}

\begin{proof}
Put \(\kappa_{\mm,N}\coloneq\fall{N}{n}\frac{(\A)_n}{(\bbeta)_{\mm+n}}\).
Reflecting the exact finite generating function
\eqref{eq:Hahn-complete-B-pgf} gives
\(z^N\mathcal H^{\mathrm H}_{\mm,N}(z^{-1})
=-\kappa_{\mm,N}(z-1)^nS_{\mm,N}(z)\), where
\(S_{\mm,N}(z)\coloneq z^{N-n}
\pFq{q+1}{q}{n-N,\A+n}{\bbeta+\mm+n}{1-z^{-1}}\).
Although this expression contains negative powers of \(z\) before
cancellation, \(S_{\mm,N}\) is a polynomial. Its Euler representation is
\[
 S_{\mm,N}(z)=\int_{(0,1)^q}
 [y_1\cdots y_q+(1-y_1\cdots y_q)z]^{N-n}
 \prod_{h=1}^q
 b_{A_h+n,\,\beta_h+m_h-A_h}(y_h)\dd\boldsymbol y.
\]
Thus \(S_{\mm,N}\) is the generating function of a normalized
nonnegative finite sequence and \(S_{\mm,N}(1)=1\). In particular, the
factorial moment of order \(n\) of the reflected finite signed sequence is
\(-\kappa_{\mm,N}n!\). After normalization by this moment, its generating
function is
\(-z^N\mathcal H^{\mathrm H}_{\mm,N}(z^{-1})/(\kappa_{\mm,N}n!)
=(z-1)^nS_{\mm,N}(z)/n!\).
Under \eqref{eq:MI-scaling}, the beta factors in the last integral have
parameters
\((A_h+n,\beta_h+m_h-A_h)=(T\rho_h+n,\delta_h+m_h)\), while
\(\frac{N-n}{T}\xrightarrow[T\to\infty]{}\nu\).
After \(y_h=1-u_h/T\), dominated convergence, locally for \(z\) in the
unit disk and in a fixed neighbourhood of \(1\), gives
\[
 S_{\mm,N}(z)\xrightarrow[T\to\infty]{}
 \prod_{h=1}^q
 \left(\frac{\rho_h}{\rho_h+\nu(1-z)}\right)^{\delta_h+m_h}
 =\prod_{h=1}^qR_h(z)^{\delta_h+m_h}.
\]
This proves \eqref{eq:MI-B-product-pgf}, including its normalization: its
derivatives at \(1\) of every order \(r\in\Nzero\) with \(r<n\) vanish, while its \(n\)-th
derivative equals one. Finally, coefficient extraction from
\((z-1)^nS_{\mm,N}(z)/n!\), followed by the coefficientwise limit,
gives \eqref{eq:MI-B-difference}.
\end{proof}

The individual type-I polynomials are characterized by an explicit finite
polynomial system.
Under the hypotheses of Proposition~\ref{prop:MI-normality}, write
\begin{equation}
 \mathcal B_{\mm}^{\mathrm{MI}}(\ell)
 =\sum_{j=1}^qB_j^{\mathrm{MI}}(\ell)v_j^{\mathrm{MI}}(\ell),
 \qquad
 B_j^{\mathrm{MI}}(\ell)=
 \sum_{r=0}^{m_j-1}b_{j,r}\fall{\ell}{r},
 \label{eq:MI-B-components}
\end{equation}
and define
\begin{equation}
 \Phi_{j,r}^{\mathrm{MI}}(z)=
 \frac{z^r(G_j^{\mathrm{MI}})^{(r)}(z)}
 {\prod_hR_h(z)^{\delta_h+m_h}}.
 \label{eq:MI-Phi}
\end{equation}
Here \(j\in\{1,\ldots,q\}\) satisfies \(m_j>0\), and
\(r\in\{0,\ldots,m_j-1\}\).
Leibniz's rule makes every quotient in \eqref{eq:MI-Phi} an explicit
polynomial of degree at most \(n\). The coefficients \(b_{j,r}\) are the
unique solution of the polynomial identity
\begin{equation}
 \sum_{j=1}^q\sum_{r=0}^{m_j-1}
 b_{j,r}\Phi_{j,r}^{\mathrm{MI}}(z)=\frac{(z-1)^n}{n!}.
 \label{eq:MI-component-identity}
\end{equation}
After the common factor \(\prod_hR_h(z)^{\delta_h}\) has been removed,
the remaining interpolation problem is rational in
\(s=\nu(1-z)\); its only singularities are the poles \(-\rho_j\) of
integer orders. Finite partial fractions then recover the coefficients
in \eqref{eq:MI-component-identity}.

Let \(x_1,\ldots,x_q\) be distinct and let \(D=\sum_jm_j\). For every
\(j\in\{1,\ldots,q\}\) and \(r\in\Nzero\) with \(r<m_j\), set
\(E_{j,r}(z)=z^r(1-x_jz)^{m_j-r-1}
\prod_{h\ne j}(1-x_hz)^{m_h}\).

\begin{lemma}[Confluent polynomial determinant]
\label{lem:confluent-polynomial-blocks}
For any fixed ordering of the \(D\) columns \((j,r)\),
\begin{equation}
 \det\bigl([z^\nu]E_{j,r}(z)\bigr)_{
       0\le\nu<D,\,(j,r)}
 =\pm\prod_{j<h}(x_h-x_j)^{m_jm_h}.
 \label{eq:confluent-polynomial-determinant}
\end{equation}
\end{lemma}

\begin{proof}
Order the columns lexicographically by the pairs \((j,r)\) with
\(j\in\{1,\ldots,q\}\), \(r\in\Nzero\), and \(r<m_j\); an arbitrary fixed
ordering changes only the sign. First assume that all \(x_j\ne0\), and put
\(\lambda_{j,s}p=p^{(s)}(x_j^{-1})/s!\) for
\(j\in\{1,\ldots,q\}\) and \(s\in\Nzero\) with \(s<m_j\).
Let \(C=([z^\nu]E_{j,r})\) be the coefficient matrix and let
\(V=(\lambda_{j,s}z^\nu)\), with the pairs \((j,s)\) grouped by \(j\).
The standard confluent Vandermonde evaluation gives
\(\det V=\prod_{j<h}(x_h^{-1}-x_j^{-1})^{m_jm_h}\).
The matrix \(VC=(\lambda_{h,s}E_{j,r})\) is block diagonal, because
\(E_{j,r}\) has a zero of order \(m_h\) at \(x_h^{-1}\) whenever
\(h\ne j\). Set \(G_j=\prod_{h\ne j}(1-x_h/x_j)^{m_h}\) and
\(t_r=m_j-r-1\).
Inside the \(j\)-th diagonal block,
\(\lambda_{j,s}E_{j,r}=0\) for \(s<t_r\), whereas
\(\lambda_{j,t_r}E_{j,r}=(-x_j)^{t_r}x_j^{-r}G_j\).
After reversing the columns this block is triangular. The reversal sign,
the signs \((-1)^{t_r}\), and the powers of \(x_j\) cancel in the
product, because
\(\sum_rt_r=\sum_rr=\binom{m_j}{2}\). Consequently,
\(\det(VC)=\prod_{j=1}^qG_j^{m_j}\).
Dividing by \(\det V\) and grouping the two ordered factors belonging to
each pair \(j<h\) yields
\[
 \det C
 =\prod_{j<h}
 \left(\frac{(1-x_h/x_j)(1-x_j/x_h)}
 {x_h^{-1}-x_j^{-1}}\right)^{m_jm_h}
 =\prod_{j<h}(x_h-x_j)^{m_jm_h}.
\]
Both sides are polynomials in the nodes, so the identity extends to the
case in which one of the distinct nodes is zero. This proves
\eqref{eq:confluent-polynomial-determinant}; any other column ordering
introduces only its permutation sign.
\end{proof}

Let \(\mm\) be near the diagonal, put \(D=\abs\mm\), \(n=D-1\), and
\(K_{\mm}^{\mathrm{MI}}=\prod_{h=1}^q(1-c_h)^{-m_h}\).
For \(j\in\{1,\ldots,q\}\), \(r\in\Nzero\) with \(r<m_j\), and
\(d\in\{0,\ldots,n\}\), define
\(C^{\mathrm{MI}}_{d,(j,r)}=[z^d]\Phi_{j,r}^{\mathrm{MI}}(z)\).

\begin{proposition}[Near-diagonal normality]
\label{prop:MI-normality}
Let \(\mm\) be near the diagonal with \(D\coloneq\abs\mm\ge1\), set
\(n=D-1\), and let
\(C^{\mathrm{MI}}_{d,(j,r)}=[z^d]\Phi_{j,r}^{\mathrm{MI}}(z)\), where
\(j\in\{1,\ldots,q\}\), \(r\in\Nzero\), \(r<m_j\), and
\(d\in\{0,\ldots,n\}\). Assume \(\rho_j,\delta_j>0\) for every
\(j\in\{1,\ldots,q\}\), and that the rates \(\rho_j\) for which
\(m_j>0\) are pairwise
distinct. For any fixed ordering of the columns,
\begin{equation}
 \det C^{\mathrm{MI}}=
 \pm(K_{\mm}^{\mathrm{MI}})^D
 \prod_{j=1}^q\left[(1-c_j)^{m_j}
 c_j^{\binom{m_j}{2}}
 \prod_{r=0}^{m_j-1}(\delta_j+1)_r\right]
 \prod_{j<h}(c_h-c_j)^{m_jm_h}.
 \label{eq:MI-connection-determinant}
\end{equation}
The polynomials \(\Phi_{j,r}^{\mathrm{MI}}\), with
\(j\in\{1,\ldots,q\}\), \(r\in\Nzero\), and \(r<m_j\), form a basis
of the polynomials of degree at
most \(n\). The signed sequence in
Theorem~\ref{thm:MI-B} is the type-I linear form and its moment matrix is
nonsingular.
\end{proposition}

\begin{proof}
Put \(f(z)=\prod_{h=1}^q(1-c_hz)^{-\delta_h}\),
\(Q_{j,r}(z)=\frac{z^r}{(1-c_jz)^{r+1}}\), and
\(P_{j,r}(z)=z^rf(z)^{-1}
\frac{\dd^r}{\dd z^r}\frac{f(z)}{1-c_jz}\).
The principal part of \(P_{j,r}\) at \(z=c_j^{-1}\) has highest-order
term
\((\delta_j+1)_rc_j^rQ_{j,r}\). At every other pole its order is at
most \(r\). Since these rational functions vanish at infinity, subtraction
of the highest-order principal part gives
\[
 P_{j,r}-(\delta_j+1)_rc_j^rQ_{j,r}
 \in\operatorname{span}\{Q_{h,s}:h\in\{1,\ldots,q\},\
 s\in\Nzero,\ s<r\}.
\]
Thus, when the columns are ordered globally by increasing \(r\), the
change of basis is triangular by derivative order, rather than separately
for each fixed \(j\). If
\(D_{\mm}(z)=\prod_h(1-c_hz)^{m_h}\), then
\(D_{\mm}Q_{j,r}=E_{j,r}\) and
\(\Phi_{j,r}^{\mathrm{MI}}
=K_{\mm}^{\mathrm{MI}}(1-c_j)D_{\mm}P_{j,r}\).
Consequently the diagonal factor in column \((j,r)\) is
\(K_{\mm}^{\mathrm{MI}}(1-c_j)c_j^r(\delta_j+1)_r\).
Multiplying these factors and applying
Lemma~\ref{lem:confluent-polynomial-blocks} gives
\eqref{eq:MI-connection-determinant}. Every displayed factor is nonzero
under the hypotheses, proving the basis and normality assertions.
\end{proof}

\begin{corollary}[Hahn-to-Meixner-I-like component limit]
\label{cor:Hahn-MI-component-limit}
Let \(\mm\) be fixed and near the diagonal, put
\(n=\abs\mm-1\in\Nzero\), and assume that the rates \(\rho_j\) for
which \(m_j>0\) are pairwise distinct. Under \eqref{eq:MI-scaling}, set
\(\kappa_{\mm,N}\coloneq\fall{N}{n}(\A)_n/(\bbeta)_{\mm+n}\) and, for
every \(j\in\{1,\ldots,q\}\), set
\(\check v_{j,T}(\ell)\coloneq v_{j,N}(N-\ell)\). Also set
\(\check{\mathcal B}_{\mm,T}(\ell)
\coloneq-\mathcal B_{\mm,N}^{\mathrm H}(N-\ell)/
(\kappa_{\mm,N}n!)\). For all sufficiently large \(T\), there is a
unique tuple \((\check B_{j,T})_{j:m_j>0}\) satisfying
\(\deg\check B_{j,T}<m_j\) for every
\(j\in\{1,\ldots,q\}\) with \(m_j>0\) and
\[
 \check{\mathcal B}_{\mm,T}(\ell)
 =\sum_{\substack{j\in\{1,\ldots,q\}\\m_j>0}}
 \check B_{j,T}(\ell)\check v_{j,T}(\ell).
\]
For every \(j\in\{1,\ldots,q\}\) with \(m_j>0\),
\begin{equation}
 \check B_{j,T}(\ell)\xrightarrow[T\to\infty]{}B_j^{\mathrm{MI}}(\ell)
 \quad\hbox{coefficientwise in }\ell.
 \label{eq:Hahn-MI-component-limit}
\end{equation}
\end{corollary}

\begin{proof}
Expand
\(\check B_{j,T}\) in the falling-factorial basis.  Its coefficient
vector solves the type-I moment system whose entries are finite linear
combinations of moments of the reflected weights \(\check v_{j,T}\) and
whose right-hand side is \((0,\ldots,0,1)^{\mathsf T}\).  The
normalization above is exactly the one computed in the proof of
Theorem~\ref{thm:MI-B}.  By Theorem~\ref{thm:MI-rows}, every entry of
this moment matrix converges to the corresponding Meixner-I-like entry.
The limiting matrix is nonsingular by
\eqref{eq:MI-connection-determinant}; hence the finite matrices are
nonsingular for all sufficiently large \(T\), and continuity of matrix
inversion gives convergence of their coefficient vectors.  The limiting
vector is the unique solution of
\eqref{eq:MI-component-identity}, namely the polynomials
\eqref{eq:MI-B-components}.
\end{proof}

Although Corollary~\ref{cor:Hahn-MI-component-limit} gives componentwise
convergence, the individual Hahn blocks in \eqref{eq:Hahn-B-KdF-sum} cannot be
confluent term by term inside their defining double series: after
reflection, powers of \(T\) cancel between its finitely many terms.
Instead of concealing this cancellation in a formal finite-part
operation, the next subsection performs the limiting triangular inversion
explicitly.  Proposition~\ref{prop:MI-sector-blocks} below gives the
resulting finite Lauricella--Horn sector sums and proves the confluence of each
grouped Hahn block.

For \(q=1\), \eqref{eq:MI-A-coefficient} is the monic classical Meixner
polynomial and the sole type-I polynomial obtained from
\eqref{eq:MI-component-identity} is the corresponding degree-\(n\)
Meixner polynomial. Thus both reflected forms reduce exactly to the scalar
Hahn--Meixner limit.

\subsection{Explicit Meixner-I-like type-I components}
We first fix the terminating Lauricella--Horn coefficients used throughout
the reflected discrete and continuous component formulas. Let
\(\mathcal I\) be a finite index set, let
\(\alpha,x_0\in\mathbb C\), and let
\(\boldsymbol d=(d_g)_{g\in\mathcal I}\) and
\(\boldsymbol x=(x_g)_{g\in\mathcal I}\) belong to
\(\mathbb C^{\mathcal I}\). For \(L\in\mathbb Z\), define
\begin{equation}
 \mathscr L_L(\alpha;\boldsymbol d;x_0,\boldsymbol x)
 \coloneq
 \begin{cases}
 \displaystyle
 \sum_{\substack{t_0\in\Nzero,\ t_g\in\Nzero\ (g\in\mathcal I)\\
                  t_0+\sum_{g\in\mathcal I}t_g=L}}
 \frac{(-\alpha)_{t_0}}{t_0!}x_0^{t_0}
 \prod_{g\in\mathcal I}\frac{(d_g)_{t_g}}{t_g!}x_g^{t_g},
 &L\ge0,\\[3mm]
 0,&L<0,
 \end{cases}
 \label{eq:typeI-reflected-L}
\end{equation}
and
\begin{equation}
 \mathscr H_L(\boldsymbol d;\boldsymbol x)
 \coloneq[w^L]\prod_{g\in\mathcal I}(1-x_gw)^{-d_g}
 =\sum_{\substack{t_g\in\Nzero\ (g\in\mathcal I)\\
                   \sum_{g\in\mathcal I}t_g=L}}
 \prod_{g\in\mathcal I}\frac{(d_g)_{t_g}}{t_g!}x_g^{t_g},
 \qquad \mathscr H_L\coloneq0\quad(L<0).
 \label{eq:typeI-reflected-H}
\end{equation}
Equivalently,
\(\mathscr L_L=[w^L](1-x_0w)^\alpha
\prod_{g\in\mathcal I}(1-x_gw)^{-d_g}\). Both sums are finite.

Put
\[
 g_h(z)\coloneq R_h(z)=\frac{1-c_h}{1-c_hz},
 \qquad h\in\{1,\ldots,q\}.
\]
For every \(j\in\{1,\ldots,q\}\), let
\(G_j^{\mathrm{MI}}(z)=\prod_{h=1}^q
g_h(z)^{\delta_h+\delta_{h,j}}\),
and, for a near-diagonal \(\boldsymbol m\) with
\(\abs{\boldsymbol m}\ge1\), put
\[
 n=\sum_{j=1}^q m_j-1,\qquad
 \mathcal I_{\boldsymbol m}\coloneq\{g\in\{1,\ldots,q\}:m_g>0\},\qquad
 H_{\boldsymbol m}^{\mathrm{MI}}(z)
 =\prod_{h=1}^qg_h(z)^{\delta_h+m_h}.
\]
For \(j\in\{1,\ldots,q\}\) with \(m_j>0\) and
\(r\in\{0,\ldots,m_j-1\}\), multiplication of the \(j\)-th weight by
\(r!\binom{k}{r}\) corresponds to
\[
 \Phi_{j,r}^{\mathrm{MI}}(z)
 \coloneq\frac{z^r(G_j^{\mathrm{MI}})^{(r)}(z)}
 {H_{\boldsymbol m}^{\mathrm{MI}}(z)},
\]
Thus the required coefficients are characterized by
\begin{equation}
 \sum_{j=1}^q\sum_{r=0}^{m_j-1}
 b_{j,r}^{\mathrm{MI}}\Phi_{j,r}^{\mathrm{MI}}(z)
 =\frac{(z-1)^n}{n!}.
 \label{eq:typeI-MI-component-identity}
\end{equation}

Set \(\mathcal D_{\boldsymbol m}(z)=\prod_{g=1}^q(1-c_gz)^{m_g}\) and
\(K_{\mm}^{\mathrm{MI}}=\prod_{g=1}^q(1-c_g)^{-m_g}\).
For \(h\in\{1,\ldots,q\}\) and \(s\in\Nzero\) with \(s<m_h\), define
\begin{equation}
 \Pi_{h,s}^{\mathrm{MI}}
 \coloneq
 \frac{(1-c_h)^n}{n!c_h^n}
 \prod_{\substack{g\in\mathcal I_{\boldsymbol m}\\g\ne h}}
 \left(\frac{c_h-c_g}{c_h}\right)^{-m_g}
 \mathscr L_{m_h-s-1}
 \left(
 n;(m_g)_{\substack{g\in\mathcal I_{\boldsymbol m}\\g\ne h}};
 \frac1{1-c_h},
 \left(-\frac{c_g}{c_h-c_g}\right)_{
       \substack{g\in\mathcal I_{\boldsymbol m}\\g\ne h}}
 \right).
 \label{eq:typeI-MI-Pi}
\end{equation}
These are the coefficients of
\((1-c_hz)^{-s-1}\) in
\((z-1)^n/[n!\mathcal D_{\boldsymbol m}(z)]\).

For \(j\in\{1,\ldots,q\}\) with \(m_j>0\),
\(r\in\{0,\ldots,m_j-1\}\), and
\(\boldsymbol\alpha\in\mathbb N_0^{\mathcal I_{\boldsymbol m}}\) with
\(\sum_{g\in\mathcal I_{\boldsymbol m}}\alpha_g=r\), write
\[
 d_g(\boldsymbol\alpha,j)=\alpha_g+\delta_{g,j},\qquad
 W_{j,r}^{\mathrm{MI}}(\boldsymbol\alpha)
 =r!\prod_{g\in\mathcal I_{\boldsymbol m}}
 \frac{(\delta_g+\delta_{g,j})_{\alpha_g}}{\alpha_g!}
 c_g^{\alpha_g}.
\]
For \(h,j\in\{1,\ldots,q\}\) with \(m_hm_j>0\),
\(s\in\{0,\ldots,m_h-1\}\), and
\(r\in\{0,\ldots,m_j-1\}\), let \(a=(h,s)\), \(b=(j,r)\), and define
\begin{multline}
 Z_{a,b}^{\mathrm{MI}}
 \coloneq
 K_{\mm}^{\mathrm{MI}}(1-c_j)c_h^{-r}
 \sum_{\substack{\boldsymbol\alpha\in\mathbb N_0^{\mathcal I_{\boldsymbol m}}\\
 \sum_{g\in\mathcal I_{\boldsymbol m}}\alpha_g=r\\
 d_h(\boldsymbol\alpha,j)\ge s+1}}
 W_{j,r}^{\mathrm{MI}}(\boldsymbol\alpha)
 \prod_{\substack{g\in\mathcal I_{\boldsymbol m}\\g\ne h}}
 \left(\frac{c_h-c_g}{c_h}\right)^{-d_g(\boldsymbol\alpha,j)}
 \\
 {}\times
 \mathscr L_{d_h(\boldsymbol\alpha,j)-s-1}
 \left(
 r;(d_g(\boldsymbol\alpha,j))_{
      \substack{g\in\mathcal I_{\boldsymbol m}\\g\ne h}};
 1,
 \left(-\frac{c_g}{c_h-c_g}\right)_{
       \substack{g\in\mathcal I_{\boldsymbol m}\\g\ne h}}
 \right).
 \label{eq:typeI-MI-Z}
\end{multline}
In particular,
\begin{equation}
 \Delta_{h,s}^{\mathrm{MI}}
 \coloneq Z_{(h,s),(h,s)}^{\mathrm{MI}}
 =K_{\mm}^{\mathrm{MI}}(1-c_h)(\delta_h+1)_s.
 \label{eq:typeI-MI-Delta}
\end{equation}

Let
\[
 \mathscr P_{\boldsymbol m}^{\mathrm{MI}}
 =\{(h,s):h\in\{1,\ldots,q\},\ s\in\Nzero,\ s<m_h\},
 \qquad
 (h,s)\prec(j,r)\Longleftrightarrow s<r.
\]
For \(a\in\mathscr P_{\boldsymbol m}^{\mathrm{MI}}\), set
\begin{equation}
 \mathfrak C_a^{\mathrm{MI}}
 \coloneq
 \sum_{p\ge0}(-1)^p
 \sum_{\substack{a=a_0\prec a_1\prec\cdots\prec a_p\\
                  a_v\in\mathscr P_{\boldsymbol m}^{\mathrm{MI}}}}
 \frac{\Pi_{a_p}^{\mathrm{MI}}}
 {\Delta_{a_p}^{\mathrm{MI}}}
 \prod_{v=0}^{p-1}
 \frac{Z_{a_v,a_{v+1}}^{\mathrm{MI}}}
 {\Delta_{a_v}^{\mathrm{MI}}}.
 \label{eq:typeI-MI-finite-sum}
\end{equation}
The sum is finite because the second coordinate of the labels increases
strictly.

Set
\begin{equation}
 B_j^{\mathrm{MI}}(k)
 =\sum_{r=0}^{m_j-1}
 \mathfrak C_{(j,r)}^{\mathrm{MI}}r!\binom{k}{r},
 \qquad j\in\{1,\ldots,q\}.
 \label{eq:typeI-MI-components}
\end{equation}

\begin{theorem}[Meixner-I-like type-I components and normality]
\label{thm:explicit-MI-B}
Let \(\boldsymbol m\) be near the diagonal with
\(\abs{\boldsymbol m}\ge1\), and set \(n=\abs{\boldsymbol m}-1\). Assume
\(0<c_j<1\) and \(\delta_j>0\) for every \(j\in\{1,\ldots,q\}\), and assume that
the \(c_j\) with \(m_j>0\) are pairwise distinct. Then
\(\deg B_j^{\mathrm{MI}}<m_j\) for every
\(j\in\{1,\ldots,q\}\) with \(m_j>0\). Moreover, for every
\(r\in\{0,\ldots,n\}\),
\begin{equation}
 \sum_{k\ge0} r!\binom{k}{r}
 \sum_{j=1}^q B_j^{\mathrm{MI}}(k)v_j^{\mathrm{MI}}(k)
 =\begin{cases}0,&0\le r<n,\\1,&r=n.\end{cases}
 \label{eq:typeI-MI-component-moments}
\end{equation}
The multi-index \(\boldsymbol m\) is normal, and
\((B_1^{\mathrm{MI}},\ldots,B_q^{\mathrm{MI}})\) is the unique tuple of
polynomials with the stated degree bounds satisfying
\eqref{eq:typeI-MI-component-moments}.
\end{theorem}

By \eqref{eq:typeI-MI-finite-sum}, every coefficient in
\eqref{eq:typeI-MI-components} is a finite sum of products of terminating
Lauricella polynomials.

\begin{proof}
Put
\[
 f(z)=\prod_{g=1}^q(1-c_gz)^{-\delta_g},\qquad
 P_{j,r}(z)=z^rf(z)^{-1}
 \left(\frac{\mathrm d}{\mathrm dz}\right)^r
 \frac{f(z)}{1-c_jz}.
\]
The multinomial Leibniz formula gives
\[
 P_{j,r}(z)
 =z^r\sum_{\substack{\boldsymbol\alpha\in
                      \mathbb N_0^{\mathcal I_{\boldsymbol m}}\\
                      \sum_{g\in\mathcal I_{\boldsymbol m}}\alpha_g=r}}
 W_{j,r}^{\mathrm{MI}}(\boldsymbol\alpha)
 \prod_{g\in\mathcal I_{\boldsymbol m}}
 (1-c_gz)^{-d_g(\boldsymbol\alpha,j)}
\]
where the indexed sums and products may be restricted to
\(\mathcal I_{\boldsymbol m}\). If an entry of
\(\boldsymbol m\) is zero, near-diagonality gives
\(m_g\in\{0,1\}\) for every \(g\in\{1,\ldots,q\}\), so only \(r=0\)
occurs; otherwise
\(\mathcal I_{\boldsymbol m}=\{1,\ldots,q\}\). Also,
\(\Phi_{j,r}^{\mathrm{MI}}(z)/\mathcal D_{\boldsymbol m}(z)
=K_{\mm}^{\mathrm{MI}}(1-c_j)P_{j,r}(z)\).
At \(z=(1-w)/c_h\), expansion of the factors analytic at \(w=0\)
gives \eqref{eq:typeI-MI-Z}. It also gives
\(Z_{(h,s),(j,r)}^{\mathrm{MI}}=0\) unless
\((h,s)=(j,r)\) or \(s<r\),
and its diagonal term is \eqref{eq:typeI-MI-Delta}.  The same expansion
of \((z-1)^n/[n!\mathcal D_{\boldsymbol m}(z)]\) gives
\eqref{eq:typeI-MI-Pi}.

Separating the \(p=0\) term in \eqref{eq:typeI-MI-finite-sum} and then
the first label after \(a\) gives
\[
 \Delta_a^{\mathrm{MI}}\mathfrak C_a^{\mathrm{MI}}
 +\sum_{a\prec b}Z_{a,b}^{\mathrm{MI}}
 \mathfrak C_b^{\mathrm{MI}}
 =\Pi_a^{\mathrm{MI}}.
\]
Therefore the two rational functions obtained from the two sides of
\eqref{eq:typeI-MI-component-identity} after division by
\(\mathcal D_{\boldsymbol m}\) have the same principal part at every
pole.  Both vanish at infinity, so they coincide.  This proves
\eqref{eq:typeI-MI-component-identity} and
\eqref{eq:typeI-MI-components}.  The nonzero diagonal terms give
uniqueness. The degree bounds follow from
\eqref{eq:typeI-MI-components}, and Theorem~\ref{thm:MI-B} gives
\eqref{eq:typeI-MI-component-moments}.
\end{proof}

\subsection{Direct pole-sector blocks and the Hahn confluence}
\label{subsec:MI-direct-sector-blocks}

The finite inversion above admits a decomposition that preserves the
\(J\)-indexed grouping of the Hahn Kamp\'e de F\'eriet formula. Put
\[
 \mathcal J_{\mm}^{\mathrm{MI}}
 \coloneq\{J\in\{1,\ldots,q\}:m_J>0\}.
\]
For \(a\in\mathscr P_{\boldsymbol m}^{\mathrm{MI}}\) and
\(J\in\mathcal J_{\mm}^{\mathrm{MI}}\), define
\begin{equation}
 \mathfrak C_{a;J}^{\mathrm{MI}}
 \coloneq
 \sum_{p\ge0}(-1)^p
 \sum_{\substack{a=a_0\prec a_1\prec\cdots\prec a_p\\
                  a_v\in\mathscr P_{\boldsymbol m}^{\mathrm{MI}},\
                  a_p=(J,K)\ {\rm for\ some}\ 0\le K<m_J}}
 \frac{\Pi_{a_p}^{\mathrm{MI}}}
      {\Delta_{a_p}^{\mathrm{MI}}}
 \prod_{v=0}^{p-1}
 \frac{Z_{a_v,a_{v+1}}^{\mathrm{MI}}}
      {\Delta_{a_v}^{\mathrm{MI}}}.
 \label{eq:MI-sector-path-coefficients}
\end{equation}
For active \(j,J\), set
\begin{equation}
 \mathscr M_{j,J}^{\mathrm{MI}}(k)
 \coloneq\sum_{r=0}^{m_j-1}
 \mathfrak C_{(j,r);J}^{\mathrm{MI}}r!\binom{k}{r}.
 \label{eq:MI-direct-sector-block}
\end{equation}

\begin{proposition}[Direct Meixner-I pole blocks and blockwise Hahn limit]
\label{prop:MI-sector-blocks}
Under the hypotheses of Theorem~\ref{thm:explicit-MI-B}, every sum in
\eqref{eq:MI-sector-path-coefficients} is finite,
\(\deg\mathscr M_{j,J}^{\mathrm{MI}}<m_j\), and
\begin{equation}
 \sum_{j=1}^q\sum_{r=0}^{m_j-1}
 \mathfrak C_{(j,r);J}^{\mathrm{MI}}
 \Phi_{j,r}^{\mathrm{MI}}(z)
 =\mathcal D_{\boldsymbol m}(z)
 \sum_{s=0}^{m_J-1}
 \frac{\Pi_{J,s}^{\mathrm{MI}}}{(1-c_Jz)^{s+1}}.
 \label{eq:MI-direct-sector-identity}
\end{equation}
Moreover,
\begin{equation}
 B_j^{\mathrm{MI}}(k)
 =\sum_{J\in\mathcal J_{\mm}^{\mathrm{MI}}}
 \mathscr M_{j,J}^{\mathrm{MI}}(k).
 \label{eq:MI-direct-sector-decomposition}
\end{equation}
Moreover, use the Hahn scaling \eqref{eq:MI-scaling} and set a missing
block \(S_{j,J}^{(N)}\) equal to zero when
\(m_J<\varepsilon_{j,J}+1\). For every pair of active indices \(j,J\),
\begin{equation}
 -\frac{S_{j,J}^{(N)}(N-k)}
 {\kappa_{\mm,N}n!\,\beta_j}
 \xrightarrow[T\to\infty]{}
 \mathscr M_{j,J}^{\mathrm{MI}}(k)
 \quad\hbox{coefficientwise in }k.
 \label{eq:Hahn-MI-blockwise-limit}
\end{equation}
Thus a grouped terminating Hahn Kamp\'e de F\'eriet block has an
ordinary, directly evaluated limit. No finite-part prescription is
involved. Its individual double-series summands need not have limits;
their cancellation is performed by the finite triangular inversion
\eqref{eq:MI-sector-path-coefficients}.
\end{proposition}

\begin{proof}
Every path has strictly increasing second coordinate. Hence it is finite,
and separation of the first edge of a path gives
\[
 \Delta_a^{\mathrm{MI}}\mathfrak C_{a;J}^{\mathrm{MI}}
 +\sum_{a\prec b}Z_{a,b}^{\mathrm{MI}}
  \mathfrak C_{b;J}^{\mathrm{MI}}
 =\mathbf1_{\{\text{the first coordinate of }a\text{ is }J\}}
  \Pi_a^{\mathrm{MI}}.
\]
These are exactly the equations obtained by comparing all principal
parts in \eqref{eq:MI-direct-sector-identity} after division by
\(\mathcal D_{\boldsymbol m}\). The difference of the two sides has no
poles and vanishes at infinity, proving that identity. Every path in
\eqref{eq:typeI-MI-finite-sum} has a unique terminal label \((J,K)\).
Partitioning those paths according to \(J\) gives
\(\mathfrak C_a^{\mathrm{MI}}
=\sum_J\mathfrak C_{a;J}^{\mathrm{MI}}\), which proves
\eqref{eq:MI-direct-sector-decomposition}.

For the limit, keep the \(J\)-indexed partial-fraction group separate in
the proof of Corollary~\ref{cor:Hahn-B-KdF}. Put
\[
 Q_T(x)\coloneq
 -T\frac{(-Tx)_n}
 {\prod_{h=1}^q(Tx+T\rho_h+\delta_h)_{m_h}}.
\]
Its Hahn partial-fraction decomposition is
\[
 Q_T(x)=\sum_{J\in\mathcal J_{\mm}^{\mathrm{MI}}}Q_{J,T}(x),
 \qquad
 Q_{J,T}(x)=\sum_{K=0}^{m_J-1}
 \frac{\pi_{J,K}}
 {x+\rho_J+(\delta_J+K)/T},
\]
and, locally away from the limiting poles,
\[
 Q_T(x)\longrightarrow
 Q(x)\coloneq-\frac{(-x)^n}
 {\prod_{h=1}^q(x+\rho_h)^{m_h}}.
\]
Because the active rates are distinct, choose disjoint contours
\(\Gamma_J\) around \(-\rho_J\). For \(x\) outside \(\Gamma_J\),
\[
 Q_{J,T}(x)=\frac1{2\pi\mathrm i}
 \int_{\Gamma_J}\frac{Q_T(\zeta)}{x-\zeta}\dd\zeta .
\]
Uniform convergence on \(\Gamma_J\) proves that each \(Q_{J,T}\)
converges separately to the \(J\)-th principal-part group of \(Q\).
This is the cancellation mechanism among the finitely many \(K\)-terms
inside one Hahn KdF block; no cancellation between different \(J\)'s is
needed.

To identify the limiting source explicitly, put \(s=\nu(1-z)\),
\[
 C_\nu\coloneq
 \frac{\prod_{h=1}^q(\nu+\rho_h)^{m_h}}{n!\nu^n},
 \qquad
 Q(s)=\sum_{J\in\mathcal J_{\mm}^{\mathrm{MI}}}Q_J(s),
 \qquad
 Q_J(s)=\sum_{r=0}^{m_J-1}
 \frac{\Lambda_{J,r}}{(s+\rho_J)^{r+1}},
\]
where the last two identities define the unique partial-fraction
coefficients \(\Lambda_{J,r}\). Explicitly,
\[
 \Lambda_{J,r}=\frac1{(m_J-r-1)!}
 \left.\frac{\mathrm d^{m_J-r-1}}{\mathrm ds^{m_J-r-1}}
 \left[-\frac{(-s)^n}
 {\displaystyle\prod_{h\ne J}(s+\rho_h)^{m_h}}\right]
 \right|_{s=-\rho_J}.
\]
They are the coefficients denoted by \(\Pi_{J,r}\) in
\eqref{eq:LII-component-partial-fractions} below.
Since
\(1-c_Jz=(s+\rho_J)/(\nu+\rho_J)\), one has exactly
\begin{align*}
 \frac{(z-1)^n}{n!\mathcal D_{\boldsymbol m}(z)}
 &=-C_\nu Q(s),\\
 \sum_{r=0}^{m_J-1}
 \frac{\Pi_{J,r}^{\mathrm{MI}}}{(1-c_Jz)^{r+1}}
 &=-C_\nu Q_J(s),
 \qquad
 \Pi_{J,r}^{\mathrm{MI}}
 =-\frac{C_\nu\Lambda_{J,r}}{(\nu+\rho_J)^{r+1}}.
\end{align*}
Thus the contour projection is precisely the target sector on the
right-hand side of \eqref{eq:MI-direct-sector-identity}.

More explicitly, let
\[
 \mathcal B_{J,N}^{\mathrm H}(x)
 \coloneq\sum_{j:m_j>0}S_{j,J}^{(N)}(x)\widetilde v_{j,N}(x),
 \qquad
 \check{\mathcal B}_{J,T}(k)
 \coloneq-\frac{\mathcal B_{J,N}^{\mathrm H}(N-k)}
 {\kappa_{\mm,N}n!},
\]
and write
\(\check S_{j,J,T}(k)=-S_{j,J}^{(N)}(N-k)/
(\kappa_{\mm,N}n!\beta_j)\) in the falling-factorial basis. Its
coefficient vector solves
\[
 M_T\check{\boldsymbol b}^{(J,T)}=\boldsymbol y_{J,T},
 \qquad
 (M_T)_{p,(j,r)}
 =\sum_{k=0}^N\fall{k}{p}\fall{k}{r}\check v_{j,T}(k),
 \quad 0\le p\le n,
\]
where
\((y_{J,T})_p=\sum_{k=0}^N\fall{k}{p}
\check{\mathcal B}_{J,T}(k)\). Thus \(M_T\) is the ordinary reflected
moment matrix, without an implicit change of basis. By
\eqref{eq:falling-linearization}, each entry is a fixed finite linear
combination of factorial moments of \(\check v_{j,T}\); hence
Theorem~\ref{thm:MI-rows} gives \(M_T\to M_{\mathrm{MI}}\) entrywise.
The limiting matrix is nonsingular by
Proposition~\ref{prop:MI-normality}.

For completeness, let \(a_{p,u}^{(N)}\) be the unique coefficients in
\[
 \fall{N-x}{p}=\sum_{u=0}^p a_{p,u}^{(N)}\fall{x}{u},
 \qquad
 a_{p,u}^{(N)}=(-1)^u\binom pu\fall{N-u}{p-u}.
\]
The explicit value follows by applying \(u\) forward differences at
\(x=0\).
The construction of the \(J\)-th Hahn block in
Corollary~\ref{cor:Hahn-B-KdF} gives the exact sector moments
\[
 \sum_{x=0}^N\fall{x}{u}\mathcal B_{J,N}^{\mathrm H}(x)
 =h_N(u)\sum_{K=0}^{m_J-1}
 \frac{\pi_{J,K}}{u+\beta_J+K}
 =\frac{h_N(u)}{T}Q_{J,T}(u/T).
\]
Consequently
\[
 (y_{J,T})_p=-\frac1{\kappa_{\mm,N}n!T}
 \sum_{u=0}^p a_{p,u}^{(N)}h_N(u)Q_{J,T}(u/T).
\]
Only the fixed indices \(u\le p\le n\) occur. The contour convergence
\(Q_{J,T}\to Q_J\), together with the exact formulas for \(h_N\) and
\(\kappa_{\mm,N}\) and the identities above with \(s=\nu(1-z)\), gives
\[
 (y_{J,T})_p\longrightarrow
 \left.\frac{\mathrm d^p}{\mathrm dz^p}\right|_{z=1}
 \left[H_{\boldsymbol m}^{\mathrm{MI}}(z)
 \mathcal D_{\boldsymbol m}(z)
 \sum_{r=0}^{m_J-1}
 \frac{\Pi_{J,r}^{\mathrm{MI}}}{(1-c_Jz)^{r+1}}\right].
\]
This is exactly the source vector of
\eqref{eq:MI-direct-sector-identity}. Continuity of the inverse of the
full matrix \(M_T\) now proves \eqref{eq:Hahn-MI-blockwise-limit}, and
uniqueness identifies its limit with
\eqref{eq:MI-sector-path-coefficients}. Finally, the factor
\(1/\beta_j\) is forced by the exact identity
\(\widetilde v_{j,N}=v_{j,N}/\beta_j\).
\end{proof}

\subsection{The direct Laguerre-II limit}

Dilating the reflected lattice gives its continuous limit. We track the
weights, the polynomial \(P\), and the signed sequence satisfying the
type-I moment conditions.

For the same fixed near-diagonal \(\mm\), let \(\nu\to\infty\) and put
\(\ell=\lfloor\nu x\rfloor\). For \(d>0\), Stirling's formula gives
\begin{equation}
 \nu\,r_{d,\,\nu/(\nu+\rho)}(\lfloor\nu x\rfloor)
 \xrightarrow[\nu\to\infty]{}g_{\rho,d}(x),\qquad x>0,
 \label{eq:Meixner-kernel-gamma-limit}
\end{equation}
locally uniformly on compact subsets of \((0,\infty)\). Passing from the
one-factor limit \eqref{eq:Meixner-kernel-gamma-limit} to a convolution
requires a uniform estimate near the boundary. Let \(d_h,\rho_h>0\) for
every \(h\in\{1,\ldots,q\}\), put
\(c_{h,\nu}=\nu/(\nu+\rho_h)\) and
\(d_*\coloneq\sum_{h=1}^q d_h\), and set
\(s_\nu=r_{d_1,c_{1,\nu}}*\cdots*r_{d_q,c_{q,\nu}}\) and
\(w=g_{\rho_1,d_1}*\cdots*g_{\rho_q,d_q}\),
where \(w\) is extended by zero to \((-\infty,0]\).

\begin{lemma}[Local limit for lattice gamma convolutions]
\label{lem:lattice-gamma-convolution}
If \(d_*>1\), then
\(\nu s_\nu(\lfloor\nu x\rfloor)\xrightarrow[\nu\to\infty]{}w(x)\)
uniformly for \(x\in\R\).
\end{lemma}

\begin{proof}
Extend by zero outside \([-\pi\nu,\pi\nu]\) the function
\[
 Q_\nu(t)=
 \prod_{h=1}^q
 \left(\frac{\rho_h}
 {\rho_h+\nu(1-\e^{it/\nu})}\right)^{d_h},
 \qquad
 Q(t)=\prod_{h=1}^q
 \left(\frac{\rho_h}{\rho_h-it}\right)^{d_h}.
\]
For \(\abs t\le\pi\nu\),
\[
 \left|\rho_h+\nu(1-\e^{it/\nu})\right|^2
 =\rho_h^2+2\nu(\rho_h+\nu)(1-\cos(t/\nu))
 \ge \rho_h^2+\frac{4t^2}{\pi^2}.
\]
Consequently \(\abs{Q_\nu(t)}\le
C(1+\abs t)^{-d_*}\). Since \(d_*>1\), dominated convergence gives
\(Q_\nu\xrightarrow[\nu\to\infty]{}Q\) in \(L^1(\R)\). Fourier inversion gives
\[
 \nu s_\nu(\lfloor\nu x\rfloor)
 =\frac1{2\pi}\int_{\R}
 \e^{-it\lfloor\nu x\rfloor/\nu}Q_\nu(t)\dd t,
\]
whereas the inverse Fourier transform of \(Q\) is \(w\). The
\(L^1\)-convergence makes the inverse transforms converge uniformly. Since
\(Q\in L^1(\R)\), \(w\) is uniformly continuous, and
\(\abs{\lfloor\nu x\rfloor/\nu-x}\le\nu^{-1}\); the stated uniform limit
follows.
\end{proof}

For \(j\in\{1,\ldots,q\}\), take
\(d_h=\delta_h+\delta_{h,j}\). Then
\(d_*=1+\sum_h\delta_h>1\), so the lemma gives the Laguerre-II-like
convolution
\begin{equation}
 w_j^{\mathrm{LII}}
 =g_{\rho_1,\delta_1}*\cdots *
 g_{\rho_j,\delta_j+1}*\cdots *
 g_{\rho_q,\delta_q}.
 \label{eq:LII-rows}
\end{equation}
Here the densities are extended by zero to the negative half-line, so
\((f*g)(x)=\int_0^xf(x-y)g(y)\dd y\).
More explicitly,
\begin{equation}
 \nu\,v_j^{\mathrm{MI}}(\lfloor\nu x\rfloor)
 \xrightarrow[\nu\to\infty]{}w_j^{\mathrm{LII}}(x),
 \label{eq:MI-LII-row-limit}
\end{equation}
The weight limit \eqref{eq:MI-LII-row-limit} is locally uniform, in fact
uniform on \(\R\). The lemma also applies to
the normalized coefficient sequence used below: for
\(d_h=\delta_h+m_h\), one has
\(d_*=\sum_h\delta_h+D>1\). The corresponding
Laplace-transform calculation is
\[
 \sum_{k\ge0}r_{d,\nu/(\nu+\rho)}(k)\e^{-sk/\nu}
 =\left(
 \frac{\rho}{\rho+\nu(1-\e^{-s/\nu})}
 \right)^d
 \xrightarrow[\nu\to\infty]{}\left(\frac{\rho}{\rho+s}\right)^d.
\]
The convergence is locally uniform for complex \(s\) in a fixed
neighbourhood of zero; differentiation there gives convergence of every
fixed moment of the convolutions.
For \(h\in\{1,\ldots,q\}\), put
\(\Lambda_h=\delta_h+m_h\). The monic polynomials
\(x\mapsto\nu^{-D}P_{\mm}^{\mathrm{MI}}(\nu x)\) converge
coefficientwise to the terminating \(q\)-variable polynomial
\(\mathfrak K_q\) defined in \eqref{eq:multiple-KdF-simple}:
\begin{equation}
 P_{\mm}^{\mathrm{LII}}(x)
 =x^D\mathfrak K_q
 \left(
 -D;-\Lambda_1,\ldots,-\Lambda_q;
 -\frac1{\rho_1x},\ldots,-\frac1{\rho_qx}
 \right).
 \label{eq:LII-A}
\end{equation}
The expression is a polynomial despite the displayed negative powers of
\(x\). An equivalent formula without apparent singularities is
\begin{equation}
 P_{\mm}^{\mathrm{LII}}(x)
 =D!\sum_{\substack{\ell,r_1,\ldots,r_q\in\Nzero\\
                     \ell+r_1+\cdots+r_q=D}}
 \frac{x^\ell}{\ell!}
 \prod_{h=1}^q
 \frac{(-\Lambda_h)_{r_h}}{r_h!\rho_h^{r_h}}.
 \label{eq:LII-A-finite-sum}
\end{equation}
For later limit arguments, the same polynomial has the Rodrigues-type
identity
\begin{equation}
 P_{\mm}^{\mathrm{LII}}(x)
 =\frac{(-1)^D}{F_{\mm}(0)}
 \left.\frac{\dd^D}{\dd s^D}
 \bigl[F_{\mm}(s)\e^{-sx}\bigr]\right|_{s=0}.
 \label{eq:LII-A-Rodrigues}
\end{equation}
Indeed, expanding
\(F_{\mm}(s)/F_{\mm}(0)=
\prod_h(1+s/\rho_h)^{\Lambda_h}\) and \(\e^{-sx}\) in
\eqref{eq:LII-A-Rodrigues} gives
\eqref{eq:LII-A-finite-sum}; factoring out \(x^D\) gives
\eqref{eq:LII-A}.
Consequently,
\(\nu^{-D}P_{\mm}^{\mathrm{MI}}(\lfloor\nu x\rfloor)
\xrightarrow[\nu\to\infty]{}P_{\mm}^{\mathrm{LII}}(x)\) pointwise.
The limiting polynomial is monic of degree \(D\) and, for every
\(j\in\{1,\ldots,q\}\) with \(m_j>0\), satisfies
\begin{equation}
 \int_0^\infty x^rP_{\mm}^{\mathrm{LII}}(x)
 w_j^{\mathrm{LII}}(x)\dd x=0,
 \qquad r\in\Nzero,\quad r<m_j.
 \label{eq:LII-A-orthogonality}
\end{equation}
Passing to the fixed moments in the Meixner-I-like orthogonality relations
proves \eqref{eq:LII-A-orthogonality}.
If \(s_{\mm}^{\mathrm{LII}}\) is the convolution obtained from
\eqref{eq:LII-rows} by replacing each shape \(\delta_h+\delta_{h,j}\)
with \(\Lambda_h=\delta_h+m_h\), then the normalized Laguerre-II-like
type-I linear form is the single confluent Lauricella expression
\begin{equation}
 \mathcal B_{\mm}^{\mathrm{LII}}(x)
 =\frac{(-1)^n\prod_{h=1}^q\rho_h^{\Lambda_h}}
 {n!\,\Gamma(\Lambda-n)}
 x^{\Lambda-n-1}
 \Phi_2^{(q)}
 \left(
 \Lambda_1,\ldots,\Lambda_q;\Lambda-n;
 -\rho_1x,\ldots,-\rho_qx
 \right),
 \qquad \Lambda\coloneq\sum_{h=1}^q\Lambda_h.
 \label{eq:LII-B}
\end{equation}
Equivalently, it satisfies the Rodrigues formula
\begin{equation}
 \mathcal B_{\mm}^{\mathrm{LII}}(x)
 =\frac{(-1)^n}{n!}\frac{\dd^n}{\dd x^n}
 s_{\mm}^{\mathrm{LII}}(x).
 \label{eq:LII-B-Rodrigues}
\end{equation}
The Dirichlet-simplex formula first writes the convolution as
\[
 s_{\mm}^{\mathrm{LII}}(x)
 =\frac{\prod_h\rho_h^{\Lambda_h}}
 {\Gamma(\Lambda)}x^{\Lambda-1}
 \Phi_2^{(q)}
 (\Lambda_1,\ldots,\Lambda_q;\Lambda;
 -\rho_1x,\ldots,-\rho_qx).
\]
Termwise differentiation of this entire series proves the equivalence of
\eqref{eq:LII-B} and \eqref{eq:LII-B-Rodrigues}.
Since \(s_{\mm}^{\mathrm{LII}}\) is a probability density, integration by
parts in \eqref{eq:LII-B-Rodrigues} gives the normalized moments
\begin{equation}
 \int_0^\infty x^r\mathcal B_{\mm}^{\mathrm{LII}}(x)\dd x
 =\delta_{r,n},\qquad r\in\{0,\ldots,n\}.
 \label{eq:LII-B-moments}
\end{equation}
The normalization in \eqref{eq:LII-B-moments} requires a scaling factor
under the lattice dilation. More precisely, put
\begin{equation}
 \mu_{\nu}^{\mathrm{MI}}
 \coloneq\nu^n\sum_{\ell\ge0}
 \mathcal B_{\mm}^{\mathrm{MI}}(\ell)\,\delta_{\ell/\nu}.
 \label{eq:MI-LII-signed-measures}
\end{equation}
Then the measures \eqref{eq:MI-LII-signed-measures}, as \(\nu\to\infty\),
\begin{equation}
 \mu_{\nu}^{\mathrm{MI}}\xrightarrow[\nu\to\infty]{}
 \mathcal B_{\mm}^{\mathrm{LII}}(x)\dd x
 \label{eq:MI-LII-B-limit}
\end{equation}
against every test function in \(C_c^\infty(\R)\), and in every fixed
moment.

Indeed, let \(s_{\mm,\nu}^{\mathrm{MI}}\) denote the normalized nonnegative
coefficient sequence in \eqref{eq:MI-B-seed}.
Lemma~\ref{lem:lattice-gamma-convolution}, applied with
\(d_h=\delta_h+m_h\), gives
\(\nu s_{\mm,\nu}^{\mathrm{MI}}(\lfloor\nu x\rfloor)
\xrightarrow[\nu\to\infty]{}s_{\mm}^{\mathrm{LII}}(x)\) uniformly; in particular, the associated
dilated measures converge to
\(s_{\mm}^{\mathrm{LII}}(x)\dd x\). For every
\(\varphi\in C_c^\infty(\mathbb R)\), summation by parts gives
\[
 \left\langle\mu_{\nu}^{\mathrm{MI}},\varphi\right\rangle
 =\frac1{n!}\sum_{\ell\ge0}s_{\mm,\nu}^{\mathrm{MI}}(\ell)
 \nu^n\Delta_{1/\nu}^n\varphi(\ell/\nu).
\]
\(\nu^n\Delta_{1/\nu}^n\varphi
\xrightarrow[\nu\to\infty]{}\varphi^{(n)}\) uniformly. This proves
\eqref{eq:MI-LII-B-limit}; convergence of fixed moments follows from the
same identity applied to polynomials and the exact factorial moments of
\(r_{d,c}\). The factor \(\nu^n\) is also forced by the normalization, since
the discrete signed sequence has moment of order \(n\) equal to one.

The individual Laguerre-II-like type-I polynomials also have a finite
hypergeometric construction. The coefficients below use the terminating
Lauricella--Horn sums \eqref{eq:typeI-reflected-L}--
\eqref{eq:typeI-reflected-H}, and the corresponding weighted linear forms
are confluent Lauricella functions.

Fix a near-diagonal multi-index \(\mm\), put
\(D\coloneq\abs\mm\ge1\), \(n\coloneq D-1\), and
\(\mathcal I_{\mm}\coloneq\{h\in\{1,\ldots,q\}:m_h>0\}\), and suppose that the rates
\(\rho_h\), \(h\in\mathcal I_{\mm}\), are
pairwise distinct. Define the coefficients
\(\Pi_{J,K}\) by the partial-fraction identity
\begin{equation}
 -\frac{(-z)^n}{\prod_{h=1}^q(z+\rho_h)^{m_h}}
 =\sum_{J\in\mathcal I_{\mm}}
  \sum_{K=0}^{m_J-1}\frac{\Pi_{J,K}}{(z+\rho_J)^{K+1}}.
 \label{eq:LII-component-partial-fractions}
\end{equation}
Equivalently,
\begin{equation}
 \Pi_{J,K}
 =-\frac{\rho_J^n}
 {\prod_{\substack{h\in\mathcal I_{\mm}\\h\ne J}}
  (\rho_h-\rho_J)^{m_h}}
 \mathscr L_{m_J-K-1}
 \left(
 n;(m_h)_{\substack{h\in\mathcal I_{\mm}\\h\ne J}};\rho_J^{-1},
 \bigl(-(\rho_h-\rho_J)^{-1}\bigr)_{
       \substack{h\in\mathcal I_{\mm}\\h\ne J}}
 \right).
 \label{eq:LII-component-Pi}
\end{equation}
For \(j\in\mathcal I_{\mm}\) and \(r\in\{0,\ldots,m_j-1\}\), define
the rational function
\begin{equation}
\Psi_{j,r}^{\mathrm{LII}}(z)
 \coloneq r!\sum_{\substack{\boldsymbol s\in
                     \mathbb N_0^{\mathcal I_{\mm}}\\
                     \sum_{g\in\mathcal I_{\mm}}s_g=r}}
 \prod_{g\in\mathcal I_{\mm}}
 \frac{(\delta_g+\delta_{j,g})_{s_g}}{s_g!}
 (z+\rho_g)^{-s_g-\delta_{j,g}}.
\label{eq:LII-component-Psi}
\end{equation}
For every \(J\in\mathcal I_{\mm}\) and
\(K\in\{0,\ldots,m_J-1\}\), we seek coefficients satisfying
\begin{equation}
 \sum_{j\in\mathcal I_{\mm}}
 \sum_{r=0}^{\min\{K,m_j-1\}}
 \mathfrak u_{j,r}^{J,K}\Psi_{j,r}^{\mathrm{LII}}(z)
 =\frac1{(z+\rho_J)^{K+1}}.
 \label{eq:LII-component-target}
\end{equation}
The following finite formulas give the triangular inverse explicitly.
Let \(\mathscr P_{\mm}^{\mathrm{LII}}
\coloneq\{(j,r):j\in\mathcal I_{\mm},\
r\in\{0,\ldots,m_j-1\}\}\).
For \(a=(h,\ell),b=(j,r)\in\mathscr P_{\mm}^{\mathrm{LII}}\), put
\(a\prec b\) when \(\ell<r\), and set
\begin{equation}
 \Delta_{h,\ell}^{\mathrm{LII}}\coloneq(\delta_h+1)_\ell .
 \label{eq:LII-component-pivots}
\end{equation}
With \(d_g\coloneq s_g+\delta_{j,g}\) in the summand below, define
\begin{equation}
\begin{split}
 Z_{h,\ell,j,r}^{\mathrm{LII}}
 \coloneq{}&r!\!\!
 \sum_{\substack{\boldsymbol s\in
                  \mathbb N_0^{\mathcal I_{\mm}}\\
                  \sum_{g\in\mathcal I_{\mm}}s_g=r,\ d_h\ge\ell+1}}
 \prod_{g\in\mathcal I_{\mm}}
 \frac{(\delta_g+\delta_{j,g})_{s_g}}{s_g!}
 \prod_{\substack{g\in\mathcal I_{\mm}\\g\ne h}}
 (\rho_g-\rho_h)^{-d_g}
 \\
 &\times
 \mathscr H_{d_h-\ell-1}
 \left(
 (d_g)_{\substack{g\in\mathcal I_{\mm}\\g\ne h}};
 \bigl(-(\rho_g-\rho_h)^{-1}\bigr)_{
       \substack{g\in\mathcal I_{\mm}\\g\ne h}}
 \right).
\end{split}
 \label{eq:LII-component-connection}
\end{equation}
Empty products and empty parameter strings are equal to one. Moreover,
\(Z_{h,\ell,j,\ell}^{\mathrm{LII}}
=\delta_{h,j}\Delta_{h,\ell}^{\mathrm{LII}}\).

For every \(J\in\mathcal I_{\mm}\) and
\(K\in\{0,\ldots,m_J-1\}\), let
\[
 \mathscr P_{\boldsymbol m,K}^{\mathrm{LII}}
 \coloneq\{(j,r):j\in\{1,\ldots,q\},\ r\in\Nzero,\
 r\le\min\{K,m_j-1\}\}.
\]
For \(a=(j,r)\in\mathscr P_{\boldsymbol m,K}^{\mathrm{LII}}\), define
\begin{equation}
 \mathfrak u_{j,r}^{J,K}
 \coloneq\sum_{p=0}^{K-r}(-1)^p
 \sum_{\substack{
 a=a_0\prec a_1\prec\cdots\prec a_p=(J,K)\\
 a_v\in\mathscr P_{\boldsymbol m,K}^{\mathrm{LII}}}}
 \frac1{\Delta_{J,K}^{\mathrm{LII}}}
 \prod_{v=0}^{p-1}
 \frac{Z_{a_v,a_{v+1}}^{\mathrm{LII}}}
 {\Delta_{a_v}^{\mathrm{LII}}},
 \label{eq:LII-component-chain}
\end{equation}
where, for \(a=(h,\ell)\) and \(b=(j,r)\),
\(Z_{a,b}^{\mathrm{LII}}\coloneq Z_{h,\ell,j,r}^{\mathrm{LII}}\) and
\(\Delta_a^{\mathrm{LII}}\coloneq\Delta_{h,\ell}^{\mathrm{LII}}\).
If no such index sequence exists, the sum is zero. For every
\(J\in\mathcal I_{\mm}\), \(K\in\Nzero\) with \(K<m_J\), and
\(j\in\{1,\ldots,q\}\), define
\begin{equation}
 \mathcal C_{j,J,K}^{\mathrm{LII}}(x)
 \coloneq\sum_{r=0}^{\min\{K,m_j-1\}}
 \mathfrak u_{j,r}^{J,K}x^r.
 \label{eq:LII-component-blocks}
\end{equation}
The expressions in \eqref{eq:LII-component-blocks} are finite
polynomials. For every \(j\in\{1,\ldots,q\}\), define
\[
 \widetilde w_j\coloneq
 \frac{w_j^{\mathrm{LII}}}{\rho_j\prod_{h=1}^q\rho_h^{\delta_h}}.
\]
With
\(\delta_{\mathrm{tot}}=\sum_h\delta_h\), put
\(\gamma(z)\coloneq\prod_h(z+\rho_h)^{-\delta_h}\) and
\(Q_{\mm}^{\mathrm{LII}}(z)\coloneq\prod_h(z+\rho_h)^{m_h}\).

\begin{proposition}[Explicit Laguerre-II-like type-I polynomials]
\label{prop:LII-explicit-components}
Let \(\mm\) be a near-diagonal multi-index with
\(D=\abs\mm\ge1\) and \(n=D-1\). Assume
\(\rho_h,\delta_h>0\) for every \(h\in\{1,\ldots,q\}\), that the rates
\(\rho_h\), \(h\in\mathcal I_{\mm}\), are pairwise distinct, and use the
quantities defined in \eqref{eq:LII-component-partial-fractions}--
\eqref{eq:LII-component-blocks}.
\begin{enumerate}[label=\textnormal{(\roman*)}]
\item For every \(J\in\mathcal I_{\mm}\) and
\(K\in\{0,\ldots,m_J-1\}\), the polynomials
\(\mathcal C_{j,J,K}^{\mathrm{LII}}\) satisfy
\begin{equation}
 \sum_{j=1}^q\mathcal C_{j,J,K}^{\mathrm{LII}}(x)\widetilde w_j(x)
 =\frac{x^{\delta_{\mathrm{tot}}+K}}
 {\Gamma(\delta_{\mathrm{tot}}+K+1)}
 \Phi_2^{(q)}
 \left(
 \boldsymbol\delta+(K+1)\ee_J;\delta_{\mathrm{tot}}+K+1;
 -\rho_1x,\ldots,-\rho_qx
 \right).
 \label{eq:LII-component-Phi2-block}
\end{equation}
\item For every \(j\in\{1,\ldots,q\}\) with \(m_j>0\), the normalized type-I polynomial
\(B_j^{\mathrm{LII}}\) in
\(\mathcal B_{\mm}^{\mathrm{LII}}
=\sum_{h=1}^qB_h^{\mathrm{LII}}w_h^{\mathrm{LII}}\) is
\begin{equation}
 B_j^{\mathrm{LII}}(x)
 =-\frac{\prod_{h=1}^q\rho_h^{m_h}}{n!\,\rho_j}
 \sum_{J\in\mathcal I_{\mm}}
 \sum_{K=0}^{m_J-1}\Pi_{J,K}
 \mathcal C_{j,J,K}^{\mathrm{LII}}(x).
 \label{eq:LII-explicit-components}
\end{equation}
For every \(J\in\mathcal I_{\mm}\) and
\(K\in\{0,\ldots,m_J-1\}\),
\(\deg \mathcal C_{J,J,K}^{\mathrm{LII}}=K\), while
\(\deg \mathcal C_{j,J,K}^{\mathrm{LII}}\le K-1\) for
\(j\in\{1,\ldots,q\}\setminus\{J\}\). Moreover,
\(\deg B_j^{\mathrm{LII}}<m_j\) for every
\(j\in\mathcal I_{\mm}\), and the representation is unique.

\item With columns ordered lexicographically by \((j,r)\), the coefficient
matrix in \(1,z,\ldots,z^n\) of the \(D\) polynomials
\[
 Q_{\mm}^{\mathrm{LII}}(z)\gamma(z)^{-1}(-\partial_z)^r
 \left[\frac{\gamma(z)}{z+\rho_j}\right],
 \qquad j\in\mathcal I_{\mm},\quad
 r\in\{0,\ldots,m_j-1\},
\]
has determinant
\begin{equation}
 (-1)^{\sum_j\binom{m_j}{2}}
 \left[\prod_{j=1}^q\prod_{r=0}^{m_j-1}(\delta_j+1)_r\right]
 \prod_{1\le j<h\le q}
 (\rho_h-\rho_j)^{m_jm_h},
 \label{eq:LII-component-determinant}
\end{equation}
which is nonzero under the stated hypothesis.
\end{enumerate}
\end{proposition}

\begin{proof}
The Laplace transform of \(\widetilde w_j\) is
\(\int_0^\infty\e^{-zx}\widetilde w_j(x)\dd x
=\frac1{z+\rho_j}\prod_{h=1}^q(z+\rho_h)^{-\delta_h}\).
Put \(\gamma(z)=\prod_h(z+\rho_h)^{-\delta_h}\).  The finite Leibniz
expansion gives
\(\frac{(-\partial_z)^r[\gamma(z)/(z+\rho_j)]}{\gamma(z)}
=\Psi_{j,r}^{\mathrm{LII}}(z)\).
Indeed, if some \(m_h=0\), near-diagonality gives
\(m_g\in\{0,1\}\) for every \(g\in\{1,\ldots,q\}\), so only
\(r=0\) occurs; otherwise
\(\mathcal I_{\mm}=\{1,\ldots,q\}\).
At \(z=-\rho_h\), expansion of the factors with \(g\ne h\) shows that
the coefficient of \((z+\rho_h)^{-\ell-1}\) is exactly the finite sum
\(Z_{h,\ell,j,r}^{\mathrm{LII}}\) in
\eqref{eq:LII-component-connection}.  At order \(r=\ell\), only
\(j=h\) contributes, with diagonal coefficient
\(\Delta_{h,\ell}^{\mathrm{LII}}\).

Consequently \eqref{eq:LII-component-target} is triangular by increasing
pole order.
Equation~\eqref{eq:LII-component-chain} is the resulting finite
back-substitution formula and proves the identity.
Taking inverse Laplace transforms and using the Dirichlet-simplex
formula gives \eqref{eq:LII-component-Phi2-block}.

Expanding the left-hand side of
\eqref{eq:LII-component-partial-fractions} at
\(z=-\rho_J\) gives \eqref{eq:LII-component-Pi}: the numerator produces
\((1-w/\rho_J)^n\), while each remaining denominator produces
\((1+w/(\rho_h-\rho_J))^{-m_h}\).  Multiplying
\eqref{eq:LII-component-Phi2-block} by \(\Pi_{J,K}\), summing, and using
\eqref{eq:LII-component-partial-fractions} yields the type-I linear form
for the unnormalized weights \(\widetilde w_j\).

The form in \eqref{eq:LII-B} is
\(-\prod_h\rho_h^{\delta_h+m_h}/n!\) times that form, whereas
\(w_j^{\mathrm{LII}}
=\rho_j\prod_h\rho_h^{\delta_h}\,\widetilde w_j\).
This gives exactly the factor in \eqref{eq:LII-explicit-components}.
Every index sequence starts with \(r\le K\); if \(j\ne J\), the strict
increase of the second index excludes \(r=K\). This proves the stated degree
bounds.

It remains to verify the determinant, including its sign. For
\(j\in\mathcal I_{\mm}\) and \(r\in\Nzero\) with \(r<m_j\), put
\(E_{j,r}(z)\coloneq Q_{\mm}^{\mathrm{LII}}(z)/(z+\rho_j)^{r+1}\).
Reduction at the poles expresses the polynomial in column \((j,r)\) of
\eqref{eq:LII-component-determinant} as
\((\delta_j+1)_rE_{j,r}\) plus a linear combination of the \(E_{h,s}\)
with \(s<r\).
Thus this change of basis has determinant
\(\prod_j\prod_{r=0}^{m_j-1}(\delta_j+1)_r\), independently of the
temporary ordering by increasing \(r\).

For \(j\in\mathcal I_{\mm}\) and \(s\in\Nzero\) with \(s<m_j\),
apply to the remaining coefficient determinant the
normalized evaluation functional
\(\mathcal L_{j,s}p\coloneq p^{(s)}(-\rho_j)/s!\).
If \(h\ne j\), then \(\mathcal L_{j,s}E_{h,r}=0\), because
\(E_{h,r}\) contains the factor \((z+\rho_j)^{m_j}\).  In the \(j\)-th
diagonal block,
\(E_{j,r}(z)=(z+\rho_j)^{m_j-r-1}
\prod_{g\ne j}(z+\rho_g)^{m_g}\).
Its first nonzero normalized derivative at \(-\rho_j\) occurs at
\(s=m_j-r-1\) and equals
\(G_j\coloneq\prod_{g\ne j}(\rho_g-\rho_j)^{m_g}\).
Each diagonal block is antitriangular, so
\[
 \det\bigl(\mathcal L_{j,s}E_{h,r}\bigr)
 =(-1)^{\sum_j\binom{m_j}{2}}\prod_jG_j^{m_j}.
\]
On the monomial basis \(1,z,\ldots,z^{D-1}\), the same functionals form
the normalized confluent Vandermonde matrix, whose determinant is
\(\prod_{j<h}[(-\rho_h)-(-\rho_j)]^{m_jm_h}
=\prod_{j<h}(\rho_j-\rho_h)^{m_jm_h}\).
For each pair \(j<h\), its contribution to \(G_j^{m_j}G_h^{m_h}\) is
\((\rho_h-\rho_j)^{m_jm_h}(\rho_j-\rho_h)^{m_jm_h}\).
Dividing by the confluent Vandermonde determinant leaves
\[
 \det\bigl([z^\nu]E_{j,r}(z)\bigr)
 =(-1)^{\sum_j\binom{m_j}{2}}
 \prod_{j<h}(\rho_h-\rho_j)^{m_jm_h}.
\]
Multiplication by the preceding triangular-change determinant proves
\eqref{eq:LII-component-determinant}.
\end{proof}

For each \(\nu>0\) and \(j\in\{1,\ldots,q\}\), set
\(c_{j,\nu}\coloneq\nu/(\nu+\rho_j)\), and let
\(B_{j,\nu}^{\mathrm{MI}}\) denote the unit-normalized components in
\eqref{eq:MI-B-components} with \(c_j=c_{j,\nu}\).

\begin{corollary}[Meixner-I-like-to-Laguerre-II-like component limit]
\label{cor:MI-LII-component-limit}
Under the distinctness hypothesis of
Proposition~\ref{prop:LII-explicit-components}, for every
\(j\in\{1,\ldots,q\}\) with \(m_j>0\), as \(\nu\to\infty\),
\begin{equation}
 \nu^n B_{j,\nu}^{\mathrm{MI}}(\nu x)
 \xrightarrow[\nu\to\infty]{}B_j^{\mathrm{LII}}(x)
 \quad\hbox{coefficientwise in }x .
 \label{eq:MI-LII-component-limit}
\end{equation}
\end{corollary}

\begin{proof}
Let
\[
 \mu_{j,\nu}\coloneq\sum_{\ell\ge0}
 v_{j,\nu}^{\mathrm{MI}}(\ell)\,\delta_{\ell/\nu},
 \qquad
 \widetilde B_{j,\nu}(x)\coloneq
 \nu^nB_{j,\nu}^{\mathrm{MI}}(\nu x).
\]
Then \eqref{eq:MI-LII-signed-measures} can be written
\(\mu_{\nu}^{\mathrm{MI}}
=\sum_{j=1}^q\widetilde B_{j,\nu}\,\mu_{j,\nu}\).
The factorial normalization in Theorem~\ref{thm:MI-B} is also the
ordinary-power normalization through order \(n\); hence the coefficient
vectors of the polynomials \(\widetilde B_{j,\nu}\), in the power basis,
solve
\[
 \sum_{j=1}^q\int x^r\widetilde B_{j,\nu}(x)
 \dd\mu_{j,\nu}(x)=\delta_{r,n},
 \qquad r\in\{0,\ldots,n\}.
\]
The Laplace-transform calculation preceding
\eqref{eq:LII-A} gives convergence of every fixed moment of
\(\mu_{j,\nu}\) to that of
\(w_j^{\mathrm{LII}}(x)\dd x\).  Thus the finite coefficient matrices
in this system converge entrywise to the Laguerre-II type-I moment
matrix. Multiplication of all Laplace-transform columns by the common
analytic factor \(Q_{\mm}^{\mathrm{LII}}/\gamma\), which is nonzero at
the origin, changes their jets there by an invertible triangular matrix;
passing between jets and monomial coefficients is also invertible.
Consequently this moment matrix is nonsingular precisely when the
polynomial coefficient matrix in
\eqref{eq:LII-component-determinant} is nonsingular. Continuity of matrix inversion
therefore gives convergence of the coefficient vectors to the unique
Laguerre-II solution, which is
\eqref{eq:LII-explicit-components}.
\end{proof}

\paragraph{An exact sectorwise Meixner-I-to-Laguerre-II deformation.}
The component limit in Corollary~\ref{cor:MI-LII-component-limit} can be
sharpened: no regularization is needed. Put
\[
 \gamma(s)=\prod_{h=1}^q(s+\rho_h)^{-\delta_h},\qquad
 Q_{\mm}(s)=\prod_{h=1}^q(s+\rho_h)^{m_h},
\]
and, consistently with \eqref{eq:LII-component-Psi}, write
\[
 \Psi_{j,r}^{\mathrm{LII}}(s)
 =\frac{(-\partial_s)^r[\gamma(s)/(s+\rho_j)]}{\gamma(s)}.
\]
For \(\nu>0\), define
\begin{equation}
 \Theta_{j,r}^{(\nu)}(s)
 \coloneq\left(1-\frac{s}{\nu}\right)^r
 \Psi_{j,r}^{\mathrm{LII}}(s).
 \label{eq:MI-LII-deformed-basis}
\end{equation}
If \(p\ge0\), let \(Z_{h,p,j,r}^{\mathrm{LII}}\) denote the coefficient
of \((s+\rho_h)^{-p-1}\) in the principal part of
\(\Psi_{j,r}^{\mathrm{LII}}\); it is
\eqref{eq:LII-component-connection} when it can be nonzero and is zero
for \(p\ge m_h\). Indeed, if \(h=j\), then
\(p\le r\le m_h-1\); if \(h\ne j\), then near-diagonality gives
\(p\le r-1\le m_j-2\le m_h-1\). The coefficient of
\((s+\rho_h)^{-\ell-1}\) in \(\Theta_{j,r}^{(\nu)}\) is
\begin{equation}
 Z_{h,\ell,j,r}^{(\nu)}
 =\sum_{a=0}^{r}\binom ra
 \left(1+\frac{\rho_h}{\nu}\right)^{r-a}
 \left(-\frac1\nu\right)^a
 Z_{h,\ell+a,j,r}^{\mathrm{LII}}.
 \label{eq:MI-LII-deformed-connection}
\end{equation}
With the order relation used in
\eqref{eq:LII-component-chain}, this coefficient is zero unless
\((h,\ell)=(j,r)\) or \(\ell<r\), and its diagonal value is
\begin{equation}
 \Delta_{h,\ell}^{(\nu)}
 =\left(1+\frac{\rho_h}{\nu}\right)^\ell
   (\delta_h+1)_\ell .
 \label{eq:MI-LII-deformed-pivots}
\end{equation}
For labels \(a=(h,\ell)\) and \(b=(j,r)\), write
\(Z_{a,b}^{(\nu)}=Z_{h,\ell,j,r}^{(\nu)}\) and
\(\Delta_a^{(\nu)}=\Delta_{h,\ell}^{(\nu)}\).

For \(J,j\in\mathcal I_{\mm}\), \(0\le K<m_J\), and
\(0\le r<m_j\), define
\begin{equation}
 \mathfrak u_{j,r}^{J,K}(\nu)
 \coloneq\sum_{p=0}^{K-r}(-1)^p
 \sum_{\substack{
 (j,r)=a_0\prec a_1\prec\cdots\prec a_p=(J,K)\\
 a_v\in\mathscr P_{\boldsymbol m,K}^{\mathrm{LII}}}}
 \frac1{\Delta_{J,K}^{(\nu)}}
 \prod_{v=0}^{p-1}
 \frac{Z_{a_v,a_{v+1}}^{(\nu)}}
      {\Delta_{a_v}^{(\nu)}} ,
 \label{eq:MI-LII-deformed-chain}
\end{equation}
where the value is zero if \(r>K\). Finite triangular inversion gives
\begin{equation}
 \sum_{j\in\mathcal I_{\mm}}
 \sum_{r=0}^{\min\{K,m_j-1\}}
 \mathfrak u_{j,r}^{J,K}(\nu)
 \Theta_{j,r}^{(\nu)}(s)
 =\frac1{(s+\rho_J)^{K+1}}.
 \label{eq:MI-LII-deformed-target}
\end{equation}
Set
\begin{align}
 \mathcal C_{j,J,K}^{(\nu)}(x)
 &\coloneq
 \sum_{r=0}^{\min\{K,m_j-1\}}
 \mathfrak u_{j,r}^{J,K}(\nu)
 \prod_{a=0}^{r-1}\left(x-\frac a\nu\right),
 \label{eq:MI-LII-deformed-component-block}\\
 \mathcal S_{j,J}^{(\nu)}(x)
 &\coloneq
 -\frac{\prod_{h=1}^q\rho_h^{m_h}}{n!\,\rho_j}
 \sum_{K=0}^{m_J-1}\Pi_{J,K}
 \mathcal C_{j,J,K}^{(\nu)}(x).
 \label{eq:MI-LII-deformed-sector}
\end{align}

\begin{proposition}[Exact sectorwise Meixner-I--Laguerre-II deformation]
\label{prop:MI-LII-exact-sectors}
Let the hypotheses of Proposition~\ref{prop:LII-explicit-components}
hold and set \(c_{h,\nu}=\nu/(\nu+\rho_h)\). Denote by
\(\mathscr M_{j,J;\nu}^{\mathrm{MI}}\) the direct sector
\eqref{eq:MI-direct-sector-block} evaluated at these parameters. Then,
for every pair of active indices \(j,J\),
\begin{equation}
 \nu^n\mathscr M_{j,J;\nu}^{\mathrm{MI}}(\nu x)
 =\mathcal S_{j,J}^{(\nu)}(x).
 \label{eq:MI-LII-exact-sector-identification}
\end{equation}
Consequently, for every active \(j\),
\begin{equation}
 \nu^n B_{j,\nu}^{\mathrm{MI}}(\nu x)
 =\sum_{J\in\mathcal I_{\mm}}\mathcal S_{j,J}^{(\nu)}(x)
 \label{eq:MI-LII-exact-sector-sum}
\end{equation}
identically in \(x\). Moreover, every pole sector has the ordinary
coefficientwise limit
\begin{align}
 \mathcal S_{j,J}^{(\nu)}(x)
 &\xrightarrow[\nu\to\infty]{}
 \mathcal S_{j,J}^{\mathrm{LII}}(x),
 \label{eq:MI-LII-sector-limit}\\
 \mathcal S_{j,J}^{\mathrm{LII}}(x)
 &\coloneq-\frac{\prod_{h=1}^q\rho_h^{m_h}}{n!\,\rho_j}
 \sum_{K=0}^{m_J-1}\Pi_{J,K}
 \mathcal C_{j,J,K}^{\mathrm{LII}}(x).
 \label{eq:LII-explicit-pole-sector}
\end{align}
Consequently
\(B_j^{\mathrm{LII}}=\sum_J\mathcal S_{j,J}^{\mathrm{LII}}\), exactly as
in \eqref{eq:LII-explicit-components}. The confluence is sectorwise and
contains no finite-part operation.
\end{proposition}

\begin{proof}
The principal-part expansion of
\((1-s/\nu)^r\Psi_{j,r}^{\mathrm{LII}}(s)\) at
\(s=-\rho_h\) gives \eqref{eq:MI-LII-deformed-connection}.
Its diagonal term is \eqref{eq:MI-LII-deformed-pivots}, which is nonzero
for every \(\nu>0\). Back-substitution therefore proves
that both sides of \eqref{eq:MI-LII-deformed-target} have the same
principal part at every pole. Moreover,
\(\Psi_{j,r}^{\mathrm{LII}}(s)=\mathrm O(s^{-r-1})\) as
\(s\to\infty\), and hence
\(\Theta_{j,r}^{(\nu)}(s)=\mathrm O(s^{-1})\). Their difference is
therefore a rational function without poles which vanishes at infinity;
it is zero. This proves \eqref{eq:MI-LII-deformed-target}.

Now put
\[
 C_j=\rho_j\prod_{h=1}^q\rho_h^{\delta_h},\qquad
 C_{\mm}=\prod_{h=1}^q\rho_h^{\delta_h+m_h}.
\]
In the probability-generating variable make the exact change
\(z=1-s/\nu\). Since \(c_{h,\nu}=\nu/(\nu+\rho_h)\),
\[
 R_h(1-s/\nu)=\frac{\rho_h}{\rho_h+s}.
\]
The target partial-fraction coefficients transform exactly as
\begin{equation}
 \Pi_{J,K}^{\mathrm{MI}}(\nu)
 =-\frac{\prod_{h=1}^q(\nu+\rho_h)^{m_h}}
 {n!\nu^n(\nu+\rho_J)^{K+1}}\,\Pi_{J,K}.
 \label{eq:MI-LII-exact-target-coefficients}
\end{equation}
Indeed,
\(1-c_{J,\nu}(1-s/\nu)=(s+\rho_J)/(\nu+\rho_J)\);
substitution in the target
\((z-1)^n/[n!\mathcal D_{\boldsymbol m}(z)]\) and use of
\eqref{eq:LII-component-partial-fractions} proves the formula.
It follows from the definitions and \(\partial_z=-\nu\partial_s\) that
\begin{align}
 G_j^{\mathrm{MI}}(1-s/\nu)
 &=C_j\frac{\gamma(s)}{s+\rho_j},&
 H_{\mm}^{\mathrm{MI}}(1-s/\nu)
 &=C_{\mm}\frac{\gamma(s)}{Q_{\mm}(s)},\nonumber\\
 \Phi_{j,r}^{\mathrm{MI}}(1-s/\nu)
 &=\nu^r\frac{C_j}{C_{\mm}}Q_{\mm}(s)
 \Theta_{j,r}^{(\nu)}(s).
 \label{eq:MI-LII-exact-basis-transform}
\end{align}
Write
\(B_{j,\nu}^{\mathrm{MI}}(k)
=\sum_{r=0}^{m_j-1}b_{j,r}^{(\nu)}\fall{k}{r}\) and set
\(a_{j,r}^{(\nu)}=\nu^{n+r}b_{j,r}^{(\nu)}\).
Substitution in \eqref{eq:typeI-MI-component-identity} gives the exact
rational interpolation
\begin{equation}
 \sum_{j\in\mathcal I_{\mm}}\sum_{r=0}^{m_j-1}
 a_{j,r}^{(\nu)}C_j\Theta_{j,r}^{(\nu)}(s)
 =\frac{C_{\mm}(-s)^n}{n!Q_{\mm}(s)}.
 \label{eq:MI-LII-exact-rational-interpolation}
\end{equation}
By \eqref{eq:LII-component-partial-fractions}, the right-hand side is
\[
 -\frac{C_{\mm}}{n!}
 \sum_{J\in\mathcal I_{\mm}}\sum_{K=0}^{m_J-1}
 \frac{\Pi_{J,K}}{(s+\rho_J)^{K+1}}.
\]
More precisely, applying \eqref{eq:MI-LII-exact-basis-transform} and
\eqref{eq:MI-LII-exact-target-coefficients} directly to
\eqref{eq:MI-direct-sector-identity} gives, for each active \(J\),
\begin{equation}
 \sum_{j\in\mathcal I_{\mm}}\sum_{r=0}^{m_j-1}
 \nu^{n+r}\mathfrak C_{(j,r);J}^{\mathrm{MI}}(\nu)
 C_j\Theta_{j,r}^{(\nu)}(s)
 =-\frac{C_{\mm}}{n!}\sum_{K=0}^{m_J-1}
 \frac{\Pi_{J,K}}{(s+\rho_J)^{K+1}}.
 \label{eq:MI-LII-exact-sector-rational-identity}
\end{equation}
Here \(\mathfrak C_{(j,r);J}^{\mathrm{MI}}(\nu)\) denotes the coefficient
\eqref{eq:MI-sector-path-coefficients} evaluated at
\(c_h=c_{h,\nu}\). Comparing
\eqref{eq:MI-LII-exact-sector-rational-identity} with
\eqref{eq:MI-LII-deformed-target} and using uniqueness of the triangular
inverse yields
\[
 \nu^{n+r}\mathfrak C_{(j,r);J}^{\mathrm{MI}}(\nu)
 =-\frac{C_{\mm}}{n!C_j}
 \sum_{K=0}^{m_J-1}\Pi_{J,K}
 \mathfrak u_{j,r}^{J,K}(\nu).
\]
Since
\(C_{\mm}/C_j=\prod_h\rho_h^{m_h}/\rho_j\) and
\[
 \nu^n\fall{\nu x}{r}
 =\nu^{n+r}\prod_{a=0}^{r-1}\left(x-\frac a\nu\right),
\]
this proves \eqref{eq:MI-LII-exact-sector-identification}. Summing over
\(J\) gives \eqref{eq:MI-LII-exact-sector-sum}.

Finally all sums in \eqref{eq:MI-LII-deformed-connection} and
\eqref{eq:MI-LII-deformed-chain} are finite and their denominators tend
to the nonzero pivots \((\delta_h+1)_\ell\). Hence
\(Z^{(\nu)}\to Z^{\mathrm{LII}}\),
\(\mathfrak u_{j,r}^{J,K}(\nu)\to\mathfrak u_{j,r}^{J,K}\), and
\(\prod_{a=0}^{r-1}(x-a/\nu)\to x^r\), coefficientwise. This proves
\eqref{eq:MI-LII-sector-limit} and
\eqref{eq:LII-explicit-pole-sector}.
\end{proof}

For \(q=1\), with \(n=m_1-1\), formula
\eqref{eq:LII-explicit-components} reduces to
\(B_1^{\mathrm{LII}}(x)=(-1)^n\rho_1^n
{}_1F_1(-n;\delta_1+1;\rho_1x)/n!\), the classical Laguerre polynomial
normalized so that the moment of order \(n\) is one. Thus the direct limit
agrees with the scalar Meixner-to-Laguerre limit; component convergence is
given by \eqref{eq:MI-LII-component-limit}.

\section{The two Charlier-like families and their Hermite limits}
\label{sec:Charlier-Hermite}

The Meixner-II-like and Meixner-I-like families lead, respectively, to the
Charlier-II-like and Charlier-I-like systems when \(q>1\). We derive both
families and show that their centered continuous limits give Hermite-like
systems related by reflection.

\subsection{The Charlier-II-like family}

This limit of the Meixner-II-like family gives the Charlier-II-like system.
Its weights, polynomial \(A\), and signed sequence satisfying the type-I
moment conditions remain explicit.

Start from the Meixner-II-like system and impose
\begin{equation}
 A_q=R,\qquad \tau_R=\frac aR,
 \qquad c_R=\frac{a}{R+a},\qquad R\to\infty,
 \label{eq:CII-scaling}
\end{equation}
with \(a>0\), while \(A_h,\beta_h\),
\(h\in\{1,\ldots,q-1\}\), remain fixed. Under
\eqref{eq:CII-scaling}, the factor in the Euler integral
\eqref{eq:MII-positive-integral}
\((1+a(1-z)Y/R)^{-R}\) tends to \(\e^{aY(z-1)}\) as \(R\to\infty\).

This limit gives explicit formulas for the Charlier-II-like system and
exhibits its beta-product representation.

\begin{theorem}[Charlier-II-like system]
\label{thm:CII-system}
As \(R\to\infty\) under \eqref{eq:CII-scaling}, the following assertions hold.
\begin{enumerate}[label=\textnormal{(\roman*)}]
\item For every
\(j\in\{1,\ldots,q\}\), the limiting weight has generating function,
coefficients, and factorial moments
\begin{align}
 G_j^{\mathrm{CII}}(z)
 &=\pFq{q-1}{q-1}{\A_{<q}}{\boldsymbol C_j}{a(z-1)},
 \label{eq:CII-pgf}\\
 v_j^{\mathrm{CII}}(k)
 &=\frac{a^k}{k!}
 \frac{(\A_{<q})_k}{(\boldsymbol C_j)_k}
 \pFq{q-1}{q-1}{\A_{<q}+k}{\boldsymbol C_j+k}{-a},
 \label{eq:CII-mass}\\
 \sum_{k\ge0}\fall{k}{r}v_j^{\mathrm{CII}}(k)
 &=a^r\frac{(\A_{<q})_r}{(\boldsymbol C_j)_r},
 \label{eq:CII-moments}
\end{align}
for \(k,r\in\Nzero\), where
\(\boldsymbol C_j=(\beta_h+\delta_{h,j})_{h=1}^{q-1}\), with no shift
when \(j=q\). With \(Y=y_1\cdots y_{q-1}\), they also satisfy
\begin{align}
 G_j^{\mathrm{CII}}(z)
 &=\int_{(0,1)^{q-1}}\e^{aY(z-1)}
 \prod_{h=1}^{q-1}b_{A_h,C_{h,j}-A_h}(y_h)\dd\boldsymbol y,
 \label{eq:CII-Euler-integral}\\
 v_j^{\mathrm{CII}}(k)
 &=\frac1{k!}\int_{(0,1)^{q-1}}\e^{-aY}(aY)^k
 \prod_{h=1}^{q-1}b_{A_h,C_{h,j}-A_h}(y_h)\dd\boldsymbol y.
 \label{eq:CII-positive-coefficients}
\end{align}
These weights are strictly positive and normalized.

\item For a near-diagonal multi-index \(\mm\) with \(D=\abs\mm\ge1\),
the type-II polynomial is
\begin{equation}
 A_{\mm}^{\mathrm{CII}}(k)
 =\pFq{q+1}{q-1}
 {-D,-k,\bbeta_{<q}+\mm_{<q}}{\A_{<q}}{-a^{-1}}.
 \label{eq:CII-A}
\end{equation}
The polynomial \(A_{\mm}^{\mathrm{CII}}\) has degree \(D\), satisfies
\(A_{\mm}^{\mathrm{CII}}(0)=1\), and, for every
\(j\in\{1,\ldots,q\}\) with \(m_j>0\), obeys
\begin{equation}
 \sum_{k\ge0}\fall{k}{s}A_{\mm}^{\mathrm{CII}}(k)
 v_j^{\mathrm{CII}}(k)=0,
 \qquad s\in\{0,\ldots,m_j-1\}.
 \label{eq:CII-A-orthogonality}
\end{equation}

\item For the same \(\mm\), put \(n=D-1\). The signed sequence has
generating function
\begin{equation}
 \mathcal H_{\mm}^{\mathrm{CII}}(z)
 \coloneq\sum_{k\ge0}\mathcal B_{\mm}^{\mathrm{CII}}(k)z^k
 =-a^n\frac{(\A_{<q})_n}{(\bbeta_{<q})_{\mm_{<q}+n}}(1-z)^n
 \,\pFq{q-1}{q-1}
 {\A_{<q}+n}{\bbeta_{<q}+\mm_{<q}+n}{a(z-1)}.
 \label{eq:CII-B}
\end{equation}
The factorial moments of \(\mathcal B_{\mm}^{\mathrm{CII}}\) satisfy,
for every \(r\in\Nzero\),
\begin{equation}
 \sum_{k\ge0}\fall{k}{r}\mathcal B_{\mm}^{\mathrm{CII}}(k)
 =
 -a^r(-r)_n\frac{(\A_{<q})_r}
 {\prod_{h=1}^{q-1}(\beta_h)_{m_h+r}}.
 \label{eq:CII-B-moments}
\end{equation}
In particular, they vanish for \(r\in\Nzero\) with \(r<n\), whereas
\begin{equation}
 \sum_{k\ge0}\fall{k}{n}\mathcal B_{\mm}^{\mathrm{CII}}(k)
 =(-1)^{n+1}n!a^n
 \frac{(\A_{<q})_n}{\prod_{h=1}^{q-1}(\beta_h)_{m_h+n}}\ne0.
 \label{eq:CII-B-leading-moment}
\end{equation}
\end{enumerate}
\end{theorem}

\begin{proof}
For fixed \(s\), \((R)_s(a/R)^s\to a^s\); hence the generating function
and signed sequence follow termwise from Theorem~\ref{thm:MII-forms}.
For the type-II polynomial, \((-R/a)^s/(R)_s\to(-a^{-1})^s\), which gives
\eqref{eq:CII-A}. Its degree and normalization at zero follow
directly from the terminating expression. As \(R\to\infty\), local convergence of the weight
generating functions near \(z=1\) gives convergence of every fixed
factorial moment; passage to the limit in the Meixner-II-like orthogonality
relations proves \eqref{eq:CII-A-orthogonality}. Expanding the exponential
in \eqref{eq:CII-Euler-integral} and evaluating each beta integral gives
\eqref{eq:CII-pgf}; coefficient extraction gives
\eqref{eq:CII-positive-coefficients}, and differentiation of
\eqref{eq:CII-pgf} at \(z=1\) gives \eqref{eq:CII-moments}. Differentiating
\eqref{eq:CII-B} at \(z=1\) gives \eqref{eq:CII-B-moments}.
\end{proof}

The limiting generating functions satisfy
\begin{equation}
 \bigl((z-1)\partial_z+\beta_j\bigr)G_j^{\mathrm{CII}}(z)
 =\beta_jG_q^{\mathrm{CII}}(z),\qquad
 j\in\{1,\ldots,q-1\}.
 \label{eq:CII-resolvent}
\end{equation}
Identity \eqref{eq:CII-resolvent} is the Charlier-II-like
counterpart of the finite differential identity
\eqref{eq:finite-resolvent-relation}.
Theorem~\ref{thm:explicit-unreflected-B} gives every individual
Charlier-II-like type-I polynomial directly as a finite sum of
terminating \({}_{q+1}F_{q-1}\) polynomials and terminating Horn
coefficients.
It also proves uniqueness under its explicit separation and nonvanishing
assumptions.  The polynomial \eqref{eq:CII-A} and the signed sequence
\eqref{eq:CII-B} require no such assumptions.

\begin{corollary}[Charlier-II-like components and their sectorwise Meixner limit]
\label{cor:CII-sectorwise-MII-limit}
Under the hypotheses of Theorem~\ref{thm:explicit-unreflected-B}, the
Charlier-II-like components are the specialization $X=\mathrm{CII}$ of
\eqref{eq:unreflected-explicit-components}. They have degrees $<m_j$,
are unique, and, for every active component \(j\), their block at infinity is
\begin{equation}
 \mathscr C_{j,q}^{\infty}(k)
 \coloneq
 \sum_{\ell=0}^{m_q-1}\gamma_\ell^{\mathrm{CII}}
 \mathcal E_{j,\ell}^{\mathrm{CII}}(k).
 \label{eq:CII-infinity-block-definition}
\end{equation}
Thus the component formula with all sectors displayed is
\begin{equation}
 D_j^{\mathrm{CII}}(k)
 =\sum_{J=1}^{q-1}\mathscr C_{j,J}(k)
  +\mathscr C_{j,q}^{\infty}(k).
 \label{eq:CII-all-sector-components}
\end{equation}
If \(\varepsilon_{j,q}=1-\delta_{j,q}\) and
\(m_q<\varepsilon_{j,q}+1\), the block
\(\mathscr C_{j,q}^{\infty}\) is zero. Otherwise,
\(\deg\mathscr C_{j,q}^{\infty}\le m_q-1\) for \(j=q\), and
\(\deg\mathscr C_{j,q}^{\infty}\le m_q-2\) for \(j<q\).

Under the scaling \eqref{eq:CII-scaling}, the block at infinity from
Corollary~\ref{cor:MII-moving-KdF-block}, together with every finite-pole
sector, has a separate coefficientwise limit:
\begin{align}
 \left.\mathscr M_{j,J}(k)\right|_{A_q=R,\,\tau=a/R}
 &\xrightarrow[R\to\infty]{}\mathscr C_{j,J}(k),
 &&J\in\{1,\ldots,q-1\},
 \label{eq:MII-CII-finite-sector-limit}\\
 \left.\mathscr M_{j,q}^{\infty}(k)\right|_{A_q=R,\,\tau=a/R}
 &\xrightarrow[R\to\infty]{}\mathscr C_{j,q}^{\infty}(k).
 \label{eq:MII-CII-infinity-sector-limit}
\end{align}
Their sum is the component confluence
\eqref{eq:MII-CII-component-limit}. For the original positive weights,
every displayed canonical sector is divided by \(\beta_j\) when \(j<q\),
and is unchanged when \(j=q\), as in
\eqref{eq:unreflected-probability-components}.
\end{corollary}

\begin{proof}
Apply Theorem~\ref{thm:explicit-unreflected-B} with
$X=\mathrm{CII}$ and use
Proposition~\ref{prop:compact-unreflected-finite-poles}; this gives
\eqref{eq:CII-all-sector-components}. The last part of the proof of that
proposition takes the terminating Kamp\'e de F\'eriet limit term by term and
gives \eqref{eq:MII-CII-finite-sector-limit}.

For the block at infinity, the proof of
Corollary~\ref{cor:unreflected-component-limits} establishes, for every
fixed indicated index,
\[
 z_{J,K,\ell}^{\mathrm{MII}}\longrightarrow
 z_{J,K,\ell}^{\mathrm{CII}},\qquad
 p_{K,\ell}^{\mathrm{MII}}\longrightarrow
 p_{K,\ell}^{\mathrm{CII}},\qquad
 \gamma_\ell^{\mathrm{MII}}\longrightarrow
 \gamma_\ell^{\mathrm{CII}}.
\]
Together with the finite-pole limit, the first convergence and
\eqref{eq:unreflected-corrected-infinity-vector} give
\(\mathcal E_{j,\ell}^{\mathrm{MII}}\to
\mathcal E_{j,\ell}^{\mathrm{CII}}\). The sum in
\eqref{eq:CII-infinity-block-definition} is finite, so the last convergence
proves \eqref{eq:MII-CII-infinity-sector-limit}. The degree bounds follow
from \eqref{eq:unreflected-corrected-infinity-vector}: its \(q\)-th
component contains \(\fall{k}{\ell}\), whereas every correction has
degree at most \(\ell-1\). Finally, summing the sector limits gives
\eqref{eq:MII-CII-component-limit}; the conversion to the original weights
is \eqref{eq:unreflected-probability-components}.
\end{proof}

At \(q=1\), empty products in
\eqref{eq:CII-pgf}--\eqref{eq:CII-B}
give \(G(z)=\e^{a(z-1)}\) and
\(A_D(k)=\pFq{2}{0}{-D,-k}{-}{-a^{-1}}\), and the sole type-I
polynomial is the degree-\(n\) Charlier polynomial.
The classical multiple-Charlier limit belongs to the multiple Askey scheme
\cite{BranquinhoDiazFoulquieManasWolfs2024Discrete} and is
mentioned only for comparison with the scalar case.

\subsection{The Charlier-I-like family}

The reflected Meixner-I-like family has the Charlier-I-like limit. We derive
its weights, polynomial \(P\), and signed sequence satisfying the type-I
moment conditions.

Now start from the reflected Meixner-I-like system. Keep
\((\delta_h,c_h)\), \(h\in\{1,\ldots,q-1\}\), fixed, write
\(\lambda_h=c_h/(1-c_h)\) for these indices, and let
\begin{equation}
 \delta_q=R,\qquad c_q=\frac{a}{R+a},
 \qquad \lambda_q=\frac aR,\qquad R\to\infty.
 \label{eq:CI-scaling}
\end{equation}
Under the scaling \eqref{eq:CI-scaling}, put
\begin{equation}
 F(z)=\e^{a(z-1)}\prod_{h=1}^{q-1}R_h(z)^{\delta_h}.
 \label{eq:CI-seed-pgf}
\end{equation}
The common factor \eqref{eq:CI-seed-pgf} is a normalized probability
generating function.
For a near-diagonal multi-index \(\mm\), set
\(D=\abs\mm\ge1\), \(n=D-1\), and
\(d_h=\delta_h+m_h\) for \(h\in\{1,\ldots,q-1\}\), and define
\(\widehat C_r(\ell;a)\coloneq(-a)^r{}_2F_0(-r,-\ell;-;-a^{-1})\)
and \(\kappa_{\mm}\coloneq\e^{-a}\prod_{h=1}^{q-1}(1-c_h)^{d_h}\).

This limit gives explicit formulas for the Charlier-I-like system and
exhibits its convolutional representation.

\begin{theorem}[Charlier-I-like system]
\label{thm:CI-system}
Let \(\mm\) be near the diagonal with \(D\coloneq\abs\mm\ge1\), set
\(n=D-1\), \(d_h=\delta_h+m_h\) for
\(h\in\{1,\ldots,q-1\}\), and
\(\kappa_{\mm}=\e^{-a}\prod_{h=1}^{q-1}(1-c_h)^{d_h}\).
Assume \(a>0\), and assume \(0<c_h<1\) and \(\delta_h>0\) for
every \(h\in\{1,\ldots,q-1\}\). As
\(R\to\infty\) under \eqref{eq:CI-scaling}, the following assertions hold.
\begin{enumerate}[label=\textnormal{(\roman*)}]
\item The \(q\) limiting positive weights have
probability generating functions
\begin{equation}
 G_q^{\mathrm{CI}}(z)=F(z),
 \qquad
 G_j^{\mathrm{CI}}(z)=F(z)R_j(z),\quad
 j\in\{1,\ldots,q-1\}.
 \label{eq:CI-row-pgfs}
\end{equation}
Their coefficient sequences are explicitly
\begin{equation}
 v_q^{\mathrm{CI}}
 =p_a*r_{\delta_1,c_1}*\cdots
 *r_{\delta_{q-1},c_{q-1}},
 \qquad
 v_j^{\mathrm{CI}}=v_q^{\mathrm{CI}}*r_{1,c_j},\quad
 j\in\{1,\ldots,q-1\}.
 \label{eq:CI-row-convolutions}
\end{equation}

\item The monic type-II polynomial is the terminating \(q\)-variable polynomial
\(\mathfrak K_q\) defined in \eqref{eq:multiple-KdF-simple}:
\begin{equation}
 P_{\mm}^{\mathrm{CI}}(\ell)
 =(-a)^D\mathfrak K_q
 \left(
 -D;-\ell,-d_1,\ldots,-d_{q-1};
 -a^{-1},\lambda_1/a,\ldots,\lambda_{q-1}/a
 \right).
 \label{eq:CI-A}
\end{equation}
It also satisfies
\begin{equation}
 P_{\mm}^{\mathrm{CI}}(\ell)
 =\sum_{\substack{s_1,\ldots,s_{q-1}\ge0\\
                  s_1+\cdots+s_{q-1}\le D}}
 \frac{D!}{(D-s_1-\cdots-s_{q-1})!}
 \prod_{h=1}^{q-1}\frac{(-d_h)_{s_h}}{s_h!}\lambda_h^{s_h}
 \widehat C_{D-s_1-\cdots-s_{q-1}}(\ell;a).
 \label{eq:CI-A-Charlier-expansion}
\end{equation}
The equivalent generating identity is
\begin{equation}
 P_{\mm}^{\mathrm{CI}}(\ell)
 =D![u^D](1+u)^\ell\e^{-au}
 \prod_{h=1}^{q-1}(1-\lambda_hu)^{d_h}.
\label{eq:CI-A-generating}
\end{equation}
The polynomial \(P_{\mm}^{\mathrm{CI}}\) is monic of degree \(D\) and,
for every \(j\in\{1,\ldots,q\}\) with \(m_j>0\), satisfies
\begin{equation}
 \sum_{\ell\ge0}\fall{\ell}{s}P_{\mm}^{\mathrm{CI}}(\ell)
 v_j^{\mathrm{CI}}(\ell)=0,
 \qquad s\in\{0,\ldots,m_j-1\}.
 \label{eq:CI-A-orthogonality}
\end{equation}

\item The signed sequence normalized by its factorial moment of order
\(n\) has generating function
\begin{equation}
 \mathcal H_{\mm}^{\mathrm{CI}}(z)
 =\frac{(z-1)^n}{n!}\e^{a(z-1)}
 \prod_{h=1}^{q-1}R_h(z)^{d_h}.
 \label{eq:CI-B}
\end{equation}
Its coefficients are
\begin{equation}
 \mathcal B_{\mm}^{\mathrm{CI}}(\ell)
 =\frac{(-1)^n\kappa_{\mm}a^\ell}{n!\,\ell!}
 \mathfrak K_q
 \left(
 -\ell;-n,d_1,\ldots,d_{q-1};
 -a^{-1},-c_1/a,\ldots,-c_{q-1}/a
 \right),
 \qquad \ell\in\Nzero.
 \label{eq:CI-B-coefficients}
\end{equation}
Its factorial moments are
\begin{equation}
 \sum_{\ell\ge0}\fall{\ell}{r}\mathcal B_{\mm}^{\mathrm{CI}}(\ell)
 =
 \begin{cases}
  0,&r\in\Nzero,\ r<n,\\
  1,&r=n.
 \end{cases}
 \label{eq:CI-B-moments}
\end{equation}
\end{enumerate}
\end{theorem}

\begin{proof}
Under \eqref{eq:CI-scaling},
\(R_q(z)=(1-c_q)/(1-c_qz)
=(1+a(1-z)/R)^{-1}\). Consequently,
\(R_q(z)^{R+\delta_{q,j}}
=(1+a(1-z)/R)^{-R-\delta_{q,j}}
\xrightarrow[R\to\infty]{}\e^{a(z-1)}\) locally uniformly in
\(z\). The remaining factors \(R_h(z)\), indexed by
\(h\in\{1,\ldots,q-1\}\),
stay fixed.  Substitution in \eqref{eq:MI-row-pgf} proves
\eqref{eq:CI-row-pgfs}.  Since \(\e^{a(z-1)}\) and \(R_h(z)^d\) are the
probability generating functions of \(p_a\) and \(r_{d,c_h}\),
respectively, multiplication gives the convolution identities
\eqref{eq:CI-row-convolutions}. Thus every limiting weight is nonnegative
and has total mass one.

In the coefficient formula \eqref{eq:MI-A-coefficient} of
Theorem~\ref{thm:MI-A}, the last factor satisfies
\((1-\lambda_qu)^{\delta_q+m_q}
=(1-au/R)^{R+m_q}\xrightarrow[R\to\infty]{}\e^{-au}\).
The extracted coefficient has the fixed order \(D\), so, as \(R\to\infty\), the limit can be
taken inside \([u^D]\), giving \eqref{eq:CI-A-generating}. Now expand
\((1+u)^\ell=\sum_{r\ge0}\fall{\ell}{r}u^r/r!\),
\(\e^{-au}=\sum_{t\ge0}(-a)^tu^t/t!\), and
\((1-\lambda_hu)^{d_h}
=\sum_{s\ge0}(-d_h)_s(\lambda_hu)^s/s!\).
Eliminating
\(t=D-r-s_1-\cdots-s_{q-1}\) yields \eqref{eq:CI-A}.  If instead
\(S=s_1+\cdots+s_{q-1}\) is fixed first, the remaining coefficient is
\((D-S)![u^{D-S}](1+u)^\ell\e^{-au}
=\widehat C_{D-S}(\ell;a)\),
which proves \eqref{eq:CI-A-Charlier-expansion}.  Taking \(u^D\) from
\((1+u)^\ell\) and constant terms from all other factors produces
\(\fall{\ell}{D}\); every other contribution has smaller degree in
\(\ell\).  Hence \(P_{\mm}^{\mathrm{CI}}\) is monic of degree \(D\).

As \(R\to\infty\), the convergence of the weight generating functions is uniform on a fixed
neighborhood of \(z=1\), so their derivatives of every fixed order
converge there.  Together with the coefficientwise convergence of the
degree-\(D\) polynomials, this permits passage to the limit in each
Meixner-I-like relation
\(\sum_{\ell\ge0}\fall{\ell}{s}
P_{\mm}^{\mathrm{MI}}(\ell)v_j^{\mathrm{MI}}(\ell)=0\),
for \(s\in\{0,\ldots,m_j-1\}\),
and proves \eqref{eq:CI-A-orthogonality}.

In \eqref{eq:MI-B-product-pgf}, the last negative-binomial factor tends
to \(\exp\{a(z-1)\}\), while every factor indexed by
\(h\in\{1,\ldots,q-1\}\) is fixed.
Therefore the product tends to
\((z-1)^n\e^{a(z-1)}\prod_{h=1}^{q-1}R_h(z)^{d_h}/n!\), which is
\eqref{eq:CI-B}. To compute its coefficient, write
\[
 \mathcal H_{\mm}^{\mathrm{CI}}(z)
 =\frac{(-1)^n\kappa_{\mm}}{n!}
 (1-z)^n\e^{az}\prod_{h=1}^{q-1}(1-c_hz)^{-d_h}.
\]
Expanding all factors and putting
\(T=r+s_1+\cdots+s_{q-1}\) gives
\[
 [z^\ell]\mathcal H_{\mm}^{\mathrm{CI}}(z)
 =\frac{(-1)^n\kappa_{\mm}a^\ell}{n!\,\ell!}
 \sum_{r,\boldsymbol s\ge0}
 (-\ell)_T\frac{(-n)_r}{r!}(-a^{-1})^r
 \prod_{h=1}^{q-1}\frac{(d_h)_{s_h}}{s_h!}
 \left(-\frac{c_h}{a}\right)^{s_h}.
\]
This is \eqref{eq:CI-B-coefficients} by
\eqref{eq:multiple-KdF-simple}.  Finally, the factor multiplying
\((z-1)^n/n!\) in \eqref{eq:CI-B} equals one at \(z=1\).
Leibniz's rule therefore gives zero for derivatives of every order
\(r\in\Nzero\) with \(r<n\) and one for the derivative of order \(n\), which is
\eqref{eq:CI-B-moments}.
\end{proof}

The polynomial \(\widehat C_r(\ell;a)\) used in
\eqref{eq:CI-A-Charlier-expansion} is the monic classical Charlier
polynomial.

Determining the type-I polynomials remains a finite polynomial problem.
Put
\begin{equation}
 H_{\mm}^{\mathrm{CI}}(z)=
 \e^{a(z-1)}\prod_{h=1}^{q-1}R_h(z)^{\delta_h+m_h},
 \qquad
 \Phi_{j,r}^{\mathrm{CI}}(z)=
 \frac{z^r(G_j^{\mathrm{CI}})^{(r)}(z)}
 {H_{\mm}^{\mathrm{CI}}(z)}.
 \label{eq:CI-Phi}
\end{equation}
Here \(j\in\{1,\ldots,q\}\) satisfies \(m_j>0\), and
\(r\in\{0,\ldots,m_j-1\}\).
Set \(K_{\mm}^{\mathrm{CI}}\coloneq\prod_{h=1}^{q-1}(1-c_h)^{-m_h}\).
Let \(\varkappa_j=1-c_j\) for \(j\in\{1,\ldots,q-1\}\), while
\(\varkappa_q=1\). Leibniz's rule removes the derivative from
\eqref{eq:CI-Phi} and gives the finite polynomial
\begin{equation}
 \Phi_{j,r}^{\mathrm{CI}}(z)
 =K_{\mm}^{\mathrm{CI}}\varkappa_j r!z^r
 \sum_{\substack{s_1,\ldots,s_{q-1}\ge0\\
                  s_1+\cdots+s_{q-1}\le r}}
 \frac{a^{r-s_1-\cdots-s_{q-1}}}
 {(r-s_1-\cdots-s_{q-1})!}
 \prod_{h=1}^{q-1}
 \frac{(\delta_h+\delta_{j,h})_{s_h}c_h^{s_h}}{s_h!}
 (1-c_hz)^{m_h-\delta_{j,h}-s_h}.
 \label{eq:CI-Phi-finite}
\end{equation}
Near-diagonality makes every displayed exponent nonnegative for
\(r\in\{0,\ldots,m_j-1\}\).
Then the coefficients of
\(B_j^{\mathrm{CI}}(\ell)=\sum_{r=0}^{m_j-1}b_{j,r}\fall{\ell}{r}\)
satisfy
\begin{equation}
 \sum_{j=1}^q\sum_{r=0}^{m_j-1}
 b_{j,r}\Phi_{j,r}^{\mathrm{CI}}(z)=\frac{(z-1)^n}{n!}.
 \label{eq:CI-component-identity}
\end{equation}
The next result proves uniqueness and normality when the parameters
corresponding to nonzero entries of \(\mm\) are distinct.

Let \(\mm\) be near the diagonal, let \(D=\abs\mm=n+1\), and
fix any ordering of the columns \((j,r)\), where
\(j\in\{1,\ldots,q\}\), \(r\in\Nzero\), and \(r<m_j\). Define the
entries of \(C^{\mathrm{CI}}\) to be zero for \(\nu<r\), and, for
\(0\le\nu\le n\) and \(\nu\ge r\), by the explicit finite sum
\begin{equation}
 C^{\mathrm{CI}}_{\nu,(j,r)}
 =K_{\mm}^{\mathrm{CI}}\varkappa_j\,r!
 \sum_{\substack{s_1,\ldots,s_{q-1}\ge0\\\sum_{h=1}^{q-1}s_h\le r}}
 \frac{a^{r-\sum_{h=1}^{q-1}s_h}}{(r-\sum_{h=1}^{q-1}s_h)!}
 \prod_{h=1}^{q-1}
 \frac{(\delta_h+\delta_{j,h})_{s_h}c_h^{s_h}}{s_h!}
 \sum_{\substack{t_1,\ldots,t_{q-1}\ge0\\
                  \sum_{h=1}^{q-1}t_h=\nu-r}}
 \prod_{h=1}^{q-1}
 \frac{(-(m_h-\delta_{j,h}-s_h))_{t_h}}{t_h!}c_h^{t_h}.
 \label{eq:CI-connection-entries}
\end{equation}

\begin{proposition}[Charlier-I-like near-diagonal normality]
\label{prop:CI-normality}
Let \(\mm\) be near the diagonal with \(D\coloneq\abs\mm\ge1\), put
\(n=D-1\), and let \(C^{\mathrm{CI}}\) be the matrix defined by
\eqref{eq:CI-connection-entries}. Then
\begin{equation}
 \det C^{\mathrm{CI}}
 =\pm (K_{\mm}^{\mathrm{CI}})^{D}
 \prod_{j=1}^{q-1}\left[(1-c_j)^{m_j}
 c_j^{\binom{m_j}{2}}
 \prod_{r=0}^{m_j-1}(\delta_j+1)_r\right]
 a^{\binom{m_q}{2}}
 \prod_{1\le j<h\le q-1}(c_h-c_j)^{m_jm_h}
 \prod_{j=1}^{q-1}(-c_j)^{m_jm_q}.
 \label{eq:CI-connection-determinant}
\end{equation}
If \(a>0\), \(0<c_h<1\) and \(\delta_h>0\) for every
\(h\in\{1,\ldots,q-1\}\), and the parameters \(c_h\) for which
\(h\in\{1,\ldots,q-1\}\) and \(m_h>0\) are pairwise distinct, the fixed
near-diagonal index \(\mm\) is normal and \eqref{eq:CI-B} is the normalized
type-I linear form.
\end{proposition}

\begin{proof}
Expanding the powers in \eqref{eq:CI-Phi-finite} gives
\eqref{eq:CI-connection-entries}; thus \(C^{\mathrm{CI}}\) is the
coefficient matrix of the polynomials \(\Phi_{j,r}^{\mathrm{CI}}\) in
the basis \(1,z,\ldots,z^n\).
Put
\[
 f(z)=\e^{a(z-1)}\prod_{h=1}^{q-1}(1-c_hz)^{-\delta_h},
 \quad Q_{j,r}(z)=\frac{z^r}{(1-c_jz)^{r+1}}\
 \quad(j\in\{1,\ldots,q-1\}),
 \quad Q_{q,r}(z)=z^r.
\]
For \(j\in\{1,\ldots,q-1\}\), define
\(P_{j,r}=z^rf^{-1}(\dd/\dd z)^r[f/(1-c_jz)]\), while
\(P_{q,r}=z^rf^{-1}f^{(r)}\). The same principal-part argument as above
gives
\[
 P_{j,r}-(\delta_j+1)_rc_j^rQ_{j,r}
 \in\operatorname{span}\{Q_{h,s}:h\in\{1,\ldots,q\},\
 s\in\Nzero,\ s<r\},
 \qquad j\in\{1,\ldots,q-1\}.
\]
Here and below, \(h\in\{1,\ldots,q\}\). The polynomial
remainder has degree at most \(r-1\). For the column with \(j=q\), the exponential
factor gives
\[
 P_{q,r}-a^rQ_{q,r}
 \in\operatorname{span}\{Q_{h,s}:h\in\{1,\ldots,q\},\
 s\in\Nzero,\ s<r\}.
\]
With \(c_q=0\) and
\(D_{\mm}(z)=\prod_{h=1}^{q-1}(1-c_hz)^{m_h}\), one has
\(D_{\mm}Q_{j,r}=E_{j,r}\). Ordering the columns globally by increasing
derivative order therefore gives a triangular change of basis. For
\(j\in\{1,\ldots,q-1\}\), its diagonal factor in
\(\Phi_{j,r}^{\mathrm{CI}}\) is
\(K_{\mm}^{\mathrm{CI}}(1-c_j)c_j^r(\delta_j+1)_r\); for \(j=q\), it is
\(K_{\mm}^{\mathrm{CI}}a^r\). Their product, together with
Lemma~\ref{lem:confluent-polynomial-blocks}, is exactly
\eqref{eq:CI-connection-determinant}: the pairs involving the zero node
give \(\prod_{j=1}^{q-1}(-c_j)^{m_jm_q}\). Thus the formula is proved up to
the sign set by the column order. Its nonvanishing makes
\eqref{eq:CI-component-identity} uniquely solvable.
\end{proof}

The polynomial identity above admits the explicit finite inversion in terms
of the coefficients \(\mathscr L_L\) defined in
\eqref{eq:typeI-reflected-L}.

\subsection{Explicit Charlier-I-like type-I components}
We use the terminating Lauricella coefficients \(\mathscr L_L\) and
\(\mathscr H_L\) defined in
\eqref{eq:typeI-reflected-L}--\eqref{eq:typeI-reflected-H}.

Fix a near-diagonal \(\boldsymbol m\) with
\(\abs{\boldsymbol m}\ge1\), and set
\(n=\sum_{j=1}^q m_j-1\),
\(\mathcal I_{\boldsymbol m}^-\coloneq
\{g\in\{1,\ldots,q-1\}:m_g>0\}\), and put
\[
 g_h(z)\coloneq R_h(z)=\frac{1-c_h}{1-c_hz},
 \qquad h\in\{1,\ldots,q-1\}.
\]
Also set
\[
 F(z)=\mathrm e^{a(z-1)}
 \prod_{h=1}^{q-1}g_h(z)^{\delta_h},\qquad
 G_q^{\mathrm{CI}}(z)=F(z),\qquad
 G_j^{\mathrm{CI}}(z)=F(z)g_j(z)\quad
 (j\in\{1,\ldots,q-1\}),
\]
and
\[
 H_{\boldsymbol m}^{\mathrm{CI}}(z)
 =\mathrm e^{a(z-1)}
 \prod_{h=1}^{q-1}g_h(z)^{\delta_h+m_h},\qquad
 \Phi_{j,r}^{\mathrm{CI}}(z)
 =\frac{z^r(G_j^{\mathrm{CI}})^{(r)}(z)}
 {H_{\boldsymbol m}^{\mathrm{CI}}(z)}.
\]
Here \(j\in\{1,\ldots,q\}\) satisfies \(m_j>0\), and
\(r\in\{0,\ldots,m_j-1\}\).
The component identity is
\begin{equation}
 \sum_{j=1}^q\sum_{r=0}^{m_j-1}
 b_{j,r}^{\mathrm{CI}}\Phi_{j,r}^{\mathrm{CI}}(z)
 =\frac{(z-1)^n}{n!}.
 \label{eq:typeI-CI-component-identity}
\end{equation}
Set \(\mathcal D_-(z)=\prod_{g=1}^{q-1}(1-c_gz)^{m_g}\),
\(K_-=\prod_{g=1}^{q-1}(1-c_g)^{-m_g}\), and
\(d_-=\sum_{g=1}^{q-1}m_g\).
The label set is
\[
 \mathscr P_{\boldsymbol m}^{\mathrm{CI}}
 =\{(h,s):h\in\{1,\ldots,q-1\},\ s\in\Nzero,\ s<m_h\}
 \cup\{(q,s):s\in\Nzero,\ s<m_q\},
\]
ordered by strict increase of the second coordinate. For
\(h\in\mathcal I_{\boldsymbol m}^-\) and
\(s\in\{0,\ldots,m_h-1\}\), put
\begin{equation}
 \Pi_{h,s}^{\mathrm{CI}}
 \coloneq
 \frac{(1-c_h)^n}{n!c_h^n}
 \prod_{\substack{g\in\mathcal I_{\boldsymbol m}^-\\g\ne h}}
 \left(\frac{c_h-c_g}{c_h}\right)^{-m_g}
 \mathscr L_{m_h-s-1}
 \left(
 n;(m_g)_{\substack{g\in\mathcal I_{\boldsymbol m}^-\\g\ne h}};
 \frac1{1-c_h},
 \left(-\frac{c_g}{c_h-c_g}\right)_{
       \substack{g\in\mathcal I_{\boldsymbol m}^-\\g\ne h}}
 \right).
 \label{eq:typeI-CI-Pi-finite}
\end{equation}
For \(s\in\Nzero\) with \(s<m_q\), put
\begin{equation}
 \Pi_{q,s}^{\mathrm{CI}}
 \coloneq
 \frac{(-1)^{d_-}}
 {n!\displaystyle\prod_{g\in\mathcal I_{\boldsymbol m}^-}c_g^{m_g}}
 \mathscr L_{m_q-1-s}
 \left(
 n;(m_g)_{g\in\mathcal I_{\boldsymbol m}^-};
 1,(c_g^{-1})_{g\in\mathcal I_{\boldsymbol m}^-}
 \right).
 \label{eq:typeI-CI-Pi-infinity}
\end{equation}

For \(j\in\{1,\ldots,q\}\), define
\[
 \epsilon_j=\begin{cases}1,&j\in\{1,\ldots,q-1\},\\0,&j=q,\end{cases}
 \qquad
 \kappa_j=\begin{cases}
 K_-(1-c_j),&j\in\{1,\ldots,q-1\},\\
 K_-,&j=q.
 \end{cases}
\]
For \(j\in\{1,\ldots,q\}\) with \(m_j>0\),
\(r\in\{0,\ldots,m_j-1\}\), \(t\in\Nzero\), and
\(\boldsymbol\alpha\in
\mathbb N_0^{\mathcal I_{\boldsymbol m}^-}\) satisfying
\(t+\sum_{g\in\mathcal I_{\boldsymbol m}^-}\alpha_g=r\), put
\(d_g(\boldsymbol\alpha,j)=\alpha_g+\delta_{g,j}\) for
\(g\in\mathcal I_{\boldsymbol m}^-\),
\begin{equation}
 W_{j,r}^{\mathrm{CI}}(t,\boldsymbol\alpha)
 =\frac{r!a^t}{t!}
 \prod_{g\in\mathcal I_{\boldsymbol m}^-}
 \frac{(\delta_g+\delta_{g,j})_{\alpha_g}}{\alpha_g!}
 c_g^{\alpha_g}.
 \label{eq:typeI-CI-W}
\end{equation}
For labels \(\lambda=(h,s)\) and \(\mu=(j,r)\),
\(h\in\mathcal I_{\boldsymbol m}^-\), \(s\in\Nzero\), \(s<m_h\),
\(j\in\{1,\ldots,q\}\) with \(m_j>0\), and
\(r\in\{0,\ldots,m_j-1\}\), define
\begin{multline}
 Z_{\lambda,\mu}^{\mathrm{CI}}
 \coloneq
 \kappa_jc_h^{-r}
 \sum_{\substack{t\in\Nzero,\ 
                  \boldsymbol\alpha\in\mathbb N_0^{\mathcal I_{\boldsymbol m}^-}\\
                  t+\sum_{g\in\mathcal I_{\boldsymbol m}^-}\alpha_g=r\\
 d_h(\boldsymbol\alpha,j)\ge s+1}}
 W_{j,r}^{\mathrm{CI}}(t,\boldsymbol\alpha)
 \prod_{\substack{g\in\mathcal I_{\boldsymbol m}^-\\g\ne h}}
 \left(\frac{c_h-c_g}{c_h}\right)^{-d_g(\boldsymbol\alpha,j)}
 \\
 {}\times
 \mathscr L_{d_h(\boldsymbol\alpha,j)-s-1}
 \left(
 r;(d_g(\boldsymbol\alpha,j))_{
      \substack{g\in\mathcal I_{\boldsymbol m}^-\\g\ne h}};
 1,
 \left(-\frac{c_g}{c_h-c_g}\right)_{
       \substack{g\in\mathcal I_{\boldsymbol m}^-\\g\ne h}}
 \right).
 \label{eq:typeI-CI-Z-finite}
\end{multline}
For labels \(\lambda=(q,s)\) and \(\mu=(j,r)\), where
\(s\in\Nzero\) and \(s<m_q\),
\(j\in\{1,\ldots,q\}\) satisfies \(m_j>0\), and
\(r\in\{0,\ldots,m_j-1\}\), define
\begin{equation}
 Z_{\lambda,\mu}^{\mathrm{CI}}
 \coloneq
 \kappa_j
 \sum_{\substack{t\in\Nzero,\ 
                  \boldsymbol\alpha\in\mathbb N_0^{\mathcal I_{\boldsymbol m}^-}\\
                  t+\sum_{g\in\mathcal I_{\boldsymbol m}^-}\alpha_g=r\\
                  t-\epsilon_j\ge s}}
 W_{j,r}^{\mathrm{CI}}(t,\boldsymbol\alpha)
 (-1)^{\sum_{g\in\mathcal I_{\boldsymbol m}^-}d_g(\boldsymbol\alpha,j)}
 \prod_{g\in\mathcal I_{\boldsymbol m}^-}
 c_g^{-d_g(\boldsymbol\alpha,j)}
 \mathscr H_{t-\epsilon_j-s}
 \left(
 (d_g(\boldsymbol\alpha,j))_{g\in\mathcal I_{\boldsymbol m}^-};
 (c_g^{-1})_{g\in\mathcal I_{\boldsymbol m}^-}
 \right).
 \label{eq:typeI-CI-Z-infinity}
\end{equation}
For \(h\in\mathcal I_{\boldsymbol m}^-\) and
\(s\in\Nzero\) with \(s<m_h\), and for \(s\in\Nzero\) with
\(s<m_q\), respectively, the diagonal terms are
\begin{equation}
 \Delta_{h,s}^{\mathrm{CI}}
 =K_-(1-c_h)(\delta_h+1)_s,
 \qquad
 \Delta_{q,s}^{\mathrm{CI}}=K_-a^s.
 \label{eq:typeI-CI-Delta}
\end{equation}
For \(\lambda\in\mathscr P_{\boldsymbol m}^{\mathrm{CI}}\), define
\begin{equation}
 \mathfrak C_\lambda^{\mathrm{CI}}
 \coloneq
 \sum_{p\ge0}(-1)^p
 \sum_{\substack{\lambda=\lambda_0\prec\lambda_1\prec\cdots\prec\lambda_p\\
                  \lambda_v\in\mathscr P_{\boldsymbol m}^{\mathrm{CI}}}}
 \frac{\Pi_{\lambda_p}^{\mathrm{CI}}}
 {\Delta_{\lambda_p}^{\mathrm{CI}}}
 \prod_{v=0}^{p-1}
 \frac{Z_{\lambda_v,\lambda_{v+1}}^{\mathrm{CI}}}
 {\Delta_{\lambda_v}^{\mathrm{CI}}}.
 \label{eq:typeI-CI-finite-sum}
\end{equation}

Set
\begin{align}
 B_j^{\mathrm{CI}}(k)
 &=\sum_{r=0}^{m_j-1}
 \mathfrak C_{(j,r)}^{\mathrm{CI}}r!\binom{k}{r},
 &&j\in\{1,\ldots,q-1\},\label{eq:explicit-CI-components-finite}\\
 B_q^{\mathrm{CI}}(k)
 &=\sum_{r=0}^{m_q-1}
 \mathfrak C_{(q,r)}^{\mathrm{CI}}r!\binom{k}{r}.
 \label{eq:explicit-CI-component-base}
\end{align}

\begin{theorem}[Charlier-I-like type-I components and normality]
\label{thm:explicit-CI-B}
Let \(\boldsymbol m\) be near the diagonal with
\(\abs{\boldsymbol m}\ge1\), and set \(n=\abs{\boldsymbol m}-1\).
Assume \(a>0\), and assume
\(0<c_h<1\) and \(\delta_h>0\) for
\(h\in\{1,\ldots,q-1\}\). Also assume that the \(c_h\) for which
\(h\in\{1,\ldots,q-1\}\) and \(m_h>0\) are pairwise distinct. Then
\(\deg B_j^{\mathrm{CI}}<m_j\) for every
\(j\in\{1,\ldots,q\}\) with \(m_j>0\). Moreover, for every
\(r\in\{0,\ldots,n\}\),
\begin{equation}
 \sum_{k\ge0} r!\binom{k}{r}
 \sum_{j=1}^q B_j^{\mathrm{CI}}(k)v_j^{\mathrm{CI}}(k)
 =\begin{cases}0,&0\le r<n,\\1,&r=n.\end{cases}
 \label{eq:typeI-CI-component-moments}
\end{equation}
The multi-index \(\boldsymbol m\) is normal, and
\((B_1^{\mathrm{CI}},\ldots,B_q^{\mathrm{CI}})\) is the unique tuple of
polynomials with the stated degree bounds satisfying
\eqref{eq:typeI-CI-component-moments}.
\end{theorem}

By \eqref{eq:typeI-CI-finite-sum}, every coefficient in
\eqref{eq:explicit-CI-components-finite}--
\eqref{eq:explicit-CI-component-base} is a finite sum of products of
terminating Lauricella polynomials.

\begin{proof}
Let \(f(z)=\mathrm e^{a(z-1)}
\prod_{g=1}^{q-1}(1-c_gz)^{-\delta_g}\).
Set
\[
 P_{j,r}(z)=z^rf(z)^{-1}
 \left(\frac{\mathrm d}{\mathrm dz}\right)^r
 \begin{cases}
  f(z)/(1-c_jz),&j\in\{1,\ldots,q-1\},\\
  f(z),&j=q.
 \end{cases}
\]
In the multinomial Leibniz formula, let \(t\) derivatives fall on the
exponential and \(\alpha_g\) derivatives fall on the factor indexed by
\(g\).  Division by \(f\) then gives
\[
 P_{j,r}(z)=z^r
 \sum_{\substack{t\in\Nzero,\ 
                  \boldsymbol\alpha\in\mathbb N_0^{\mathcal I_{\boldsymbol m}^-}\\
                  t+\sum_{g\in\mathcal I_{\boldsymbol m}^-}\alpha_g=r}}
 W_{j,r}^{\mathrm{CI}}(t,\boldsymbol\alpha)
 \prod_{g\in\mathcal I_{\boldsymbol m}^-}
 (1-c_gz)^{-d_g(\boldsymbol\alpha,j)}.
\]
If a finite entry of \(\boldsymbol m\) is zero, near-diagonality gives
\(m_g\in\{0,1\}\) for every \(g\in\{1,\ldots,q-1\}\), so only \(r=0\)
occurs; otherwise \(m_g>0\) for every
\(g\in\{1,\ldots,q-1\}\). Hence the indexed sums
may be restricted to \(\mathcal I_{\boldsymbol m}^-\).  This proves
\eqref{eq:typeI-CI-W}. The constants in the probability generating
functions give
\(\Phi_{j,r}^{\mathrm{CI}}(z)/\mathcal D_-(z)
=\kappa_jP_{j,r}(z)\).

Fix a finite pole \(z=c_h^{-1}\) and write \(z=(1-w)/c_h\).  Expanding
the factors with \(g\ne h\) as power series in \(w\), the coefficient of
\(w^{-s-1}\) is precisely \eqref{eq:typeI-CI-Z-finite}. At infinity, the
summand indexed by \((t,\boldsymbol\alpha)\) starts with the power
\(z^{r-\sum_g d_g(\boldsymbol\alpha,j)}=z^{t-\epsilon_j}\); expansion in \(z^{-1}\)
therefore gives
\eqref{eq:typeI-CI-Z-infinity}. A term with \(s=r\) can occur only
on the diagonal: for a finite label it has \(h=j\), \(t=0\), and
\(\alpha_j=r\), whereas at infinity it has \(j=q\), \(t=r\), and
\(\boldsymbol\alpha=0\).  These two contributions are respectively
\(K_-(1-c_h)(\delta_h+1)_s\) and \(K_-a^s\), proving
\eqref{eq:typeI-CI-Delta}; every other nonzero entry indexed by
\(\lambda=(h,s)\) and \(\mu=(j,r)\) satisfies \(s<r\).

Apply the same local expansions to
\((z-1)^n/[n!\mathcal D_-(z)]\).  At \(c_h^{-1}\), its principal-part
coefficients are \eqref{eq:typeI-CI-Pi-finite}; at infinity, its
polynomial coefficients are \eqref{eq:typeI-CI-Pi-infinity}.  Thus these
are exactly the target coefficients \(\Pi_\lambda^{\mathrm{CI}}\), with the
same ordering as the coefficients \(Z_{\lambda,\mu}^{\mathrm{CI}}\).

As in the Meixner-I-like case, separating the first label in the finite sum
\eqref{eq:typeI-CI-finite-sum} yields
\[
 \Delta_\lambda^{\mathrm{CI}}\mathfrak C_\lambda^{\mathrm{CI}}
 +\sum_{\lambda\prec\mu}Z_{\lambda,\mu}^{\mathrm{CI}}
 \mathfrak C_\mu^{\mathrm{CI}}
 =\Pi_\lambda^{\mathrm{CI}}.
\]
Thus every finite principal part and every coefficient of the polynomial
part at infinity agrees on the two sides of
\eqref{eq:typeI-CI-component-identity}.  This proves the identity and
the component formulas.  The diagonal terms in
\eqref{eq:typeI-CI-Delta} are nonzero, which proves uniqueness. The degree
bounds follow from \eqref{eq:explicit-CI-components-finite}--
\eqref{eq:explicit-CI-component-base}, and Theorem~\ref{thm:CI-system}
gives \eqref{eq:typeI-CI-component-moments}.
\end{proof}

\begin{corollary}[Explicit Charlier-I-like type-I polynomials]
\label{cor:CI-explicit-components}
Under the hypotheses of Proposition~\ref{prop:CI-normality}, the
polynomials in \eqref{eq:explicit-CI-components-finite} and
\eqref{eq:explicit-CI-component-base} are the unique components of the
signed sequence \eqref{eq:CI-B}. For \(q=1\), the formula
reduces to
\begin{equation}
 B_1^{\mathrm{CI}}(\ell)
 =\frac{(-1)^n}{n!}{}_2F_0(-n,-\ell;-;-a^{-1}).
 \label{eq:CI-scalar-component}
\end{equation}
\end{corollary}

\begin{proof}
Apply Theorem~\ref{thm:explicit-CI-B}.  When \(q=1\), only its polynomial
part at infinity remains, and the terminating sum is
\eqref{eq:CI-scalar-component}.
\end{proof}

\begin{corollary}[Meixner-I-like-to-Charlier-I-like component limit]
\label{cor:MI-CI-component-limit}
Under the hypotheses of Theorem~\ref{thm:explicit-CI-B}, let
\(B_{j,R}^{\mathrm{MI}}\) be the unit-normalized Meixner-I-like
components with the parameters in \eqref{eq:CI-scaling}. Then, for every
active \(j\),
\begin{equation}
 B_{j,R}^{\mathrm{MI}}(k)
 \xrightarrow[R\to\infty]{}B_j^{\mathrm{CI}}(k)
 \quad\hbox{coefficientwise in }k.
 \label{eq:MI-CI-component-limit}
\end{equation}
\end{corollary}

\begin{proof}
Expand all components in the falling-factorial basis. Their coefficients
solve the square type-I moment system through order \(n\). The generating
functions in Theorem~\ref{thm:CI-system} converge locally near \(z=1\),
and hence so do all their derivatives of every fixed order at \(z=1\).
The entry in moment row \(s\in\{0,\ldots,n\}\) and polynomial column
\((j,r)\), \(0\le r<m_j\), is
\[
 \sum_{k\ge0}\fall{k}{s}\fall{k}{r}v_{j,R}^{\mathrm{MI}}(k).
\]
By the falling-factorial linearization
\eqref{eq:falling-linearization}, it is a fixed linear combination of
derivatives at \(z=1\) of orders at most
\(s+r\le n+\max_j(m_j-1)\). Consequently every entry of the
Meixner-I moment matrix converges to the corresponding Charlier-I
entry. The right-hand side is the fixed vector
\((0,\ldots,0,1)^{\mathsf T}\). The limiting matrix is nonsingular by
Theorem~\ref{thm:explicit-CI-B}. Continuity of inversion proves
\eqref{eq:MI-CI-component-limit} and identifies its limit with the tuple
in Corollary~\ref{cor:CI-explicit-components}.
\end{proof}

\paragraph{Direct Charlier-I finite-pole and infinity sectors.}
The component limit in Corollary~\ref{cor:MI-CI-component-limit} admits a
sectorwise refinement. The explicit inversion above separates according to the pole family
from which the target datum originates. Put
\[
 \mathcal J_{\mm}^{\mathrm{CI}}
 \coloneq\{J\in\{1,\ldots,q\}:m_J>0\}.
\]
For \(\lambda\in\mathscr P_{\boldsymbol m}^{\mathrm{CI}}\) and
\(J\in\mathcal J_{\mm}^{\mathrm{CI}}\), define
\begin{equation}
 \mathfrak C_{\lambda;J}^{\mathrm{CI}}
 \coloneq
 \sum_{p\ge0}(-1)^p
 \sum_{\substack{\lambda=\lambda_0\prec\lambda_1\prec\cdots\prec\lambda_p\\
                  \lambda_v\in\mathscr P_{\boldsymbol m}^{\mathrm{CI}},\
                  \lambda_p=(J,s)\ {\rm for\ some}\ 0\le s<m_J}}
 \frac{\Pi_{\lambda_p}^{\mathrm{CI}}}
      {\Delta_{\lambda_p}^{\mathrm{CI}}}
 \prod_{v=0}^{p-1}
 \frac{Z_{\lambda_v,\lambda_{v+1}}^{\mathrm{CI}}}
      {\Delta_{\lambda_v}^{\mathrm{CI}}}.
 \label{eq:CI-sector-path-coefficients}
\end{equation}
For active \(j,J\), set
\begin{equation}
 \mathscr C_{j,J}^{\mathrm{CI}}(k)
 \coloneq\sum_{r=0}^{m_j-1}
 \mathfrak C_{(j,r);J}^{\mathrm{CI}}r!\binom{k}{r}.
 \label{eq:CI-direct-sector-block}
\end{equation}
For \(J<q\), put
\[
 R_J^{\mathrm{CI}}(z)\coloneq
 \sum_{s=0}^{m_J-1}
 \frac{\Pi_{J,s}^{\mathrm{CI}}}{(1-c_Jz)^{s+1}},
\]
whereas
\[
 R_q^{\mathrm{CI}}(z)\coloneq
 \sum_{s=0}^{m_q-1}\Pi_{q,s}^{\mathrm{CI}}z^s .
\]

\begin{proposition}[Direct Charlier-I blocks and blockwise Meixner limit]
\label{prop:CI-sector-blocks}
Under the hypotheses of Theorem~\ref{thm:explicit-CI-B}, put
\(\varepsilon_{j,J}=1-\delta_{j,J}\). The sums
\eqref{eq:CI-sector-path-coefficients} are finite. If
\(m_J<\varepsilon_{j,J}+1\), then
\(\mathscr C_{j,J}^{\mathrm{CI}}=0\); otherwise,
\begin{equation}
 \deg\mathscr C_{j,J}^{\mathrm{CI}}
 \le\min\{m_j-1,m_J-1-\varepsilon_{j,J}\}.
 \label{eq:CI-direct-sector-degree}
\end{equation}
In particular, \(\deg\mathscr C_{j,J}^{\mathrm{CI}}<m_j\), and
\begin{equation}
 \sum_{j=1}^q\sum_{r=0}^{m_j-1}
 \mathfrak C_{(j,r);J}^{\mathrm{CI}}
 \Phi_{j,r}^{\mathrm{CI}}(z)
 =\mathcal D_-(z)R_J^{\mathrm{CI}}(z).
 \label{eq:CI-direct-sector-identity}
\end{equation}
Moreover,
\begin{equation}
 B_j^{\mathrm{CI}}(k)
 =\sum_{J\in\mathcal J_{\mm}^{\mathrm{CI}}}
 \mathscr C_{j,J}^{\mathrm{CI}}(k).
 \label{eq:CI-direct-sector-decomposition}
\end{equation}
Let \(\mathscr M_{j,J;R}^{\mathrm{MI}}\) denote
\eqref{eq:MI-direct-sector-block} with the scaling
\eqref{eq:CI-scaling}. Then, for every pair of active indices,
\begin{equation}
 \mathscr M_{j,J;R}^{\mathrm{MI}}(k)
 \xrightarrow[R\to\infty]{}
 \mathscr C_{j,J}^{\mathrm{CI}}(k)
 \quad\hbox{coefficientwise in }k.
 \label{eq:MI-CI-blockwise-limit}
\end{equation}
For \(J<q\) this is the confluence of a fixed finite-pole sector. The
sector \(J=q\), whose pole \(c_q^{-1}\) tends to infinity, converges to
the polynomial sector at infinity in
\(\mathscr P_{\boldsymbol m}^{\mathrm{CI}}\). Thus every block is
evaluated by the finite Lauricella sums
\eqref{eq:typeI-CI-Pi-finite}--\eqref{eq:typeI-CI-Z-infinity};
no finite-part operation remains.
\end{proposition}

\begin{proof}
The finiteness, degree bound, and
\eqref{eq:CI-direct-sector-identity} follow by separating the first
edge in \eqref{eq:CI-sector-path-coefficients}. This gives the triangular
system with target \(\Pi_\lambda^{\mathrm{CI}}\) when the first
coordinate of \(\lambda\) is \(J\), and zero target otherwise.
More explicitly, a path contributing to the coefficient indexed by
\((j,r)\) ends at a label \((J,s)\). If \(j\ne J\), the path has positive
length and the strict order gives \(r<s\); if \(j=J\), it gives \(r\le s\).
Thus \(r\le s-\varepsilon_{j,J}\le
m_J-1-\varepsilon_{j,J}\), which proves the vanishing assertion and
\eqref{eq:CI-direct-sector-degree}.
Comparison of every finite principal part and of the polynomial part at
infinity proves the rational identity. Partitioning the finite paths in
\eqref{eq:typeI-CI-finite-sum} according to the unique first coordinate
\(J\) of their terminal label proves
\eqref{eq:CI-direct-sector-decomposition}.

For \eqref{eq:MI-CI-blockwise-limit}, retain the same terminal-sector
partition in the Meixner-I triangular system. Under
\(\delta_q=R\) and \(c_q=a/(R+a)\), every fixed-sector connection
coefficient \(Z_{a,b}^{\mathrm{MI}}\), pivot
\(\Delta_a^{\mathrm{MI}}\), and target coefficient
\(\Pi_a^{\mathrm{MI}}\) is a finite sum of products of Pochhammer
symbols of fixed order. For \(J<q\), direct substitution gives the
finite-pole data
\eqref{eq:typeI-CI-Pi-finite}, \eqref{eq:typeI-CI-Z-finite}, and the
corresponding pivot in \eqref{eq:typeI-CI-Delta}. For \(J=q\), write the
principal part at the moving pole as
\[
 P_q^{(R)}(z)=\sum_{u=0}^{m_q-1}
 \frac{p_u^{(R)}}{(1-c_qz)^{u+1}}
\]
and replace its principal-part coordinates by
\begin{equation}
 \widehat p_s^{(R)}
 \coloneq[z^s]P_q^{(R)}(z)
 =\frac{c_q^s}{s!}\sum_{u=0}^{m_q-1}(u+1)_s p_u^{(R)},
 \qquad 0\le s<m_q.
 \label{eq:MI-CI-moving-pole-row-transform}
\end{equation}
For \(c_q>0\) this is an invertible row transformation: after removal
of the factors \(c_q^s\), its matrix has entries
\(\binom{u+s}{s}\) and determinant one. Apply it both to every column
of the triangular system and to the separated target sector. The
principal parts at all fixed poles have already converged. After
subtracting them, the remaining rational functions converge locally to
their polynomial parts, so the transformed moving-pole rows converge to
the coefficients at infinity. In the columnwise Leibniz sums, the
elementary limit
\((R+\epsilon)_t c_q^t/t!\to a^t/t!\) identifies them exactly with
\eqref{eq:typeI-CI-Pi-infinity}, \eqref{eq:typeI-CI-Z-infinity}, and
\(\Delta_{q,s}^{\mathrm{CI}}=K_-a^s\).
The row transformation may mix pole orders for finite \(R\), so the
transformed system need not remain triangular. Nevertheless, its full
finite matrix converges entrywise to the Charlier-I matrix, which is
nonsingular by Theorem~\ref{thm:explicit-CI-B}. Continuity of matrix
inversion, applied separately to every \(J\)-source, proves
\eqref{eq:MI-CI-blockwise-limit}.
\end{proof}

\begin{remark}[Why the moving pole must remain grouped]
\label{rem:MI-CI-moving-pole-cancellation}
The individual principal-part coefficients of the \(J=q\) sector need
not converge. Already for \(q=1\) and \(m_1=2\),
\[
 \frac{z-1}{(1-cz)^2}
 =\frac{1-c}{c(1-cz)^2}-\frac1{c(1-cz)}.
\]
Both coefficients on the right diverge as \(c\to0\), although the
grouped block converges to \(z-1\). Transformation
\eqref{eq:MI-CI-moving-pole-row-transform} gives
\(\widehat p_0=-1\) and \(\widehat p_1=1-2c\to1\), the coefficients of
the limiting polynomial. Thus the ordinary limit exists for the complete
\(J=q\) block, not term by term in its moving-pole expansion.
\end{remark}

The coefficients in
\eqref{eq:explicit-CI-components-finite}--\eqref{eq:explicit-CI-component-base}
are finite sums of products of terminating Lauricella polynomials.

Formula \eqref{eq:CI-scalar-component} is the ordinary Charlier polynomial
in the normalization of the type-I moment condition.

\begin{proposition}[Distinctness of the two Charlier-like systems]
\label{prop:Charlier-families-distinct}
Under the admissibility assumptions of Theorems~\ref{thm:CII-system} and
\ref{thm:CI-system}, in particular \(a>0\), \(0<c_h<1\), and
\(\delta_h>0\) on the Charlier-I-like side, the two families define the
same scalar orthogonality system when \(q=1\), up to the normalizations of
their type-I and type-II objects; indeed,
\(v_1^{\mathrm{CI}}=v_1^{\mathrm{CII}}=p_a\). If \(q>1\), no permutation
of the weights identifies the two systems: one has
\(v_j^{\mathrm{CI}}\ne v_h^{\mathrm{CII}}\) for every
\(j,h\in\{1,\ldots,q\}\).
\end{proposition}

\begin{proof}
For a Charlier-II-like weight, \eqref{eq:CII-pgf} gives
\(G_j^{\mathrm{CII}}(z)
=\pFq{q-1}{q-1}{\A_{<q}}{\boldsymbol C_j}{a(z-1)}\).
As the summation index tends to infinity, the ratio of consecutive terms tends
to zero uniformly for \(z\) in
compact sets, so this generating function is entire.  In contrast,
\eqref{eq:CI-row-pgfs} gives
\(G_j^{\mathrm{CI}}(z)=\e^{a(z-1)}
\prod_{h=1}^{q-1}\left(\frac{1-c_h}{1-c_hz}\right)^{
\delta_h+\delta_{h,j}}\).
Suppose \(q>1\), and put
\(c_*=\max\{c_h:h\in\{1,\ldots,q-1\}\}\). All the exponents are
positive.  Hence coincident maximal values of \(c_h\) merely add positive
exponents, and \(z=c_*^{-1}\) is a nonremovable singularity.  The Taylor
radius at the origin is therefore exactly \(c_*^{-1}<\infty\), whereas the
Charlier-II-like radius is infinite.  Equality of two weight sequences
would give equal Taylor coefficients and equal radii, which is impossible.
Equivalently, Cauchy--Hadamard yields
\[
 \limsup_{k\to\infty}\bigl(v_j^{\mathrm{CII}}(k)\bigr)^{1/k}=0,
 \qquad
 \limsup_{k\to\infty}\bigl(v_j^{\mathrm{CI}}(k)\bigr)^{1/k}=c_*.
\]
When \(q=1\), the products are empty and
\({}_0F_0(a(z-1))=\e^{a(z-1)}\), so both systems reduce to the ordinary
scalar Charlier system.
\end{proof}

\subsection{The Charlier-I-like-to-Hermite limit}

After centering and rescaling, the Charlier-I-like family converges to a
Hermite-like family. The Charlier-II-like family gives its reflection.
Let
\begin{equation}
 a\to\infty,\qquad s=\sqrt{2a}\to\infty,\qquad
 c_h=\frac{s}{s+\rho_h}\quad
 (h\in\{1,\ldots,q-1\}),\qquad
 \ell=\lfloor a+st\rfloor.
 \label{eq:Hermite-scaling}
\end{equation}
Write \(\phi(t)=\pi^{-1/2}\e^{-t^2}\). Stirling's formula gives the
explicit kernel limits
\begin{equation}
 s\,p_a(\lfloor a+st\rfloor)\xrightarrow[s\to\infty]{}\phi(t),
 \qquad
 s\,r_{d,\,s/(s+\rho)}(\lfloor sy\rfloor)
 \xrightarrow[s\to\infty]{}g_{\rho,d}(y),
 \label{eq:CI-Hermite-kernel-limits}
\end{equation}
locally uniformly for \(t\in\R\) and \(y>0\), respectively. The two limits
in \eqref{eq:CI-Hermite-kernel-limits} permit a direct convolution argument
even when some gamma densities
are unbounded at the origin. Fix \(j\in\{1,\ldots,q\}\), put
\(d_{h,j}=\delta_h+\delta_{h,j}\) for
\(h\in\{1,\ldots,q-1\}\), and define the normalized discrete measure
\(\eta_{j,s}\coloneq\sum_{m\ge0}q_{j,s}(m)\delta_{m/s}\), where
\(q_{j,s}\coloneq r_{d_{1,j},c_1}*\cdots*r_{d_{q-1,j},c_{q-1}}\).
When the product is empty, \(\eta_{j,s}=\delta_0\). For \(\xi\ge0\),
\[
 \int_0^\infty\e^{-\xi y}\dd\eta_{j,s}(y)
 =\prod_{h=1}^{q-1}
 \left(\frac{\rho_h}
 {\rho_h+s(1-\e^{-\xi/s})}\right)^{d_{h,j}}
 \xrightarrow[s\to\infty]{}
 \prod_{h=1}^{q-1}
 \left(\frac{\rho_h}{\rho_h+\xi}\right)^{d_{h,j}}.
\]
It follows that, as \(s\to\infty\), \(\eta_{j,s}\) converges weakly to the measure with density
given by the corresponding convolution of the gamma kernels. When \(q>1\),
for every
\(0<\theta<\min\{\rho_h:h\in\{1,\ldots,q-1\}\}\), the same formula with
\(\xi=-\theta\) shows that
\(\sup_{s\ge s_0}\int_0^\infty\e^{\theta y}\dd\eta_{j,s}(y)<\infty\),
while for \(q=1\) this bound is immediate from \(\eta_{1,s}=\delta_0\).

Set \(P_s(x)=s\,p_a(\lfloor a+sx\rfloor)\), where \(p_a(k)=0\) for
\(k<0\). Stirling's formula, uniformly on compact
sets, and the standard bound for the largest Poisson coefficient give
\(P_s(x)\xrightarrow[s\to\infty]{}\phi(x)\) and
\(\sup_{s\ge s_0}\sup_{x\in\R}P_s(x)<\infty\).
The convolution identity \eqref{eq:CI-row-convolutions} gives exactly
\(s\,v_j^{\mathrm{CI}}(\lfloor a+st\rfloor)
=\int_0^\infty P_s(t-y)\dd\eta_{j,s}(y)\).
For \(t\) in a fixed compact set, first restrict the integral to
\(0\le y\le M\). There
\(P_s(t-y)\xrightarrow[s\to\infty]{}\phi(t-y)\) uniformly. The uniform
exponential-moment bound controls the discarded tail, while weak
convergence and uniform continuity of \(\phi\) give uniform convergence in
\(t\) of the integral with \(P_s\) replaced by \(\phi\). Letting
\(M\to\infty\) proves, for every \(j\in\{1,\ldots,q\}\),
\begin{equation}
 s\,v_j^{\mathrm{CI}}(\lfloor a+st\rfloor)
 \xrightarrow[s\to\infty]{}\mathcal W_j^{\mathrm H,+}(t),
 \qquad
 \mathcal W_j^{\mathrm H,+}
 \coloneq\phi*g_{\rho_1,\delta_1+\delta_{1,j}}*\cdots *
 g_{\rho_{q-1},\delta_{q-1}+\delta_{q-1,j}},
 \label{eq:Hermite-base-row}
\end{equation}
locally uniformly for \(t\) in compact subsets of \(\R\). Here the
convolution in \eqref{eq:Hermite-base-row} is
\((f*g)(t)=\int_0^\infty f(t-y)g(y)\dd y\).
Let \(\sigma_{\mm,s}^{\mathrm{CI}}\) denote the nonnegative sequence
whose generating function is the factor multiplying
\((z-1)^n/n!\) in \eqref{eq:CI-B}.  Replace \(d_{h,j}\) by
\(\delta_h+m_h\) in the measures \(\eta_{j,s}\).  As \(s\to\infty\), their weak convergence
and uniform exponential-moment bound remain valid, now with the limiting
convolution of the densities \(g_{\rho_h,\delta_h+m_h}\), indexed by
\(h\in\{1,\ldots,q-1\}\).
Repeating the displayed
convolution identity, splitting its integral at \(M\), and then letting
\(M\to\infty\) gives locally uniformly the centered limit
\[
 s\,\sigma_{\mm,s}^{\mathrm{CI}}(\lfloor a+st\rfloor)
 \xrightarrow[s\to\infty]{}
 \bigl(\phi*g_{\rho_1,\delta_1+m_1}*\cdots *
 g_{\rho_{q-1},\delta_{q-1}+m_{q-1}}\bigr)(t).
\]

For \(r\in\Nzero\), we use the monic Hermite normalization
\begin{equation}
 \mathsf H_r(t)
 \coloneq r!\sum_{v=0}^{\lfloor r/2\rfloor}
 \frac{(-1/4)^v}{v!(r-2v)!}t^{r-2v}
 =2^{-r}H_r(t),
 \label{eq:monic-Hermite-definition}
\end{equation}
where \(H_r\) is the physicists' Hermite polynomial.

For a near-diagonal \(\mm\), put \(D=\abs\mm\ge1\) and \(n=D-1\), and
define
\begin{equation}
 S_{\mm}^{\mathrm H}
 \coloneq\phi*g_{\rho_1,\delta_1+m_1}*\cdots *
 g_{\rho_{q-1},\delta_{q-1}+m_{q-1}},
 \qquad
 \mathcal B_{\mm}^{\mathrm H,+}(t)
 \coloneq\frac{(-1)^n}{n!}\frac{\dd^n}{\dd t^n}S_{\mm}^{\mathrm H}(t).
 \label{eq:Hermite-seed-B}
\end{equation}
Equivalently,
\begin{equation}
 \mathcal B_{\mm}^{\mathrm H,+}(t)
 =\frac{2^n}{n!}
 \int_{(0,\infty)^{q-1}}
 \mathsf H_n(t-y_1-\cdots-y_{q-1})
 \phi(t-y_1-\cdots-y_{q-1})
 \prod_{h=1}^{q-1}g_{\rho_h,\delta_h+m_h}(y_h)
 \dd\boldsymbol y.
 \label{eq:Hermite-B-kernel-plus}
\end{equation}
For \(q=1\), this reads
\(\mathcal B_{m_1}^{\mathrm H,+}(t)
=2^n\mathsf H_n(t)\phi(t)/n!\).
The centered signed measures are
\begin{equation}
 \mu_s^{\mathrm{CI}}
 \coloneq s^n\sum_{\ell\ge0}\mathcal B_{\mm}^{\mathrm{CI}}(\ell)
 \delta_{(\ell-a)/s}.
 \label{eq:Hermite-B-measures}
\end{equation}

\begin{theorem}[Hermite-like type-II polynomial and type-I linear form]
\label{thm:Hermite-limit}
Let \(\mm\) be near the diagonal, put \(D=\abs\mm\ge1\) and \(n=D-1\),
assume \(\rho_h,\delta_h>0\) for
\(h\in\{1,\ldots,q-1\}\), and impose
\eqref{eq:Hermite-scaling}. The following statements hold as
\(s\to\infty\).
\begin{enumerate}[label=\textnormal{(\roman*)}]
\item Locally uniformly for \(t\)
in compact subsets of \(\R\),
\(s^{-D}P_{\mm}^{\mathrm{CI}}(\lfloor a+st\rfloor)
\xrightarrow[s\to\infty]{}P_{\mm}^{\mathrm H,+}(t)\), where
\begin{equation}
 P_{\mm}^{\mathrm H,+}(t)
 =\sum_{\substack{r_1,\ldots,r_{q-1}\ge0\\
                    r_1+\cdots+r_{q-1}\le D}}
 (-D)_{r_1+\cdots+r_{q-1}}
 \prod_{h=1}^{q-1}
 \frac{(-\delta_h-m_h)_{r_h}}{r_h!}
 \left(-\frac1{\rho_h}\right)^{r_h}
 \mathsf H_{D-r_1-\cdots-r_{q-1}}(t).
 \label{eq:Hermite-A}
\end{equation}
The limiting polynomial satisfies
\begin{equation}
 P_{\mm}^{\mathrm H,+}(t)
 =D![u^D]\e^{tu-u^2/4}
 \prod_{h=1}^{q-1}(1-u/\rho_h)^{\delta_h+m_h}.
 \label{eq:Hermite-A-generating}
\end{equation}

\item Against every test function in \(C_c^\infty(\R)\), and in every fixed
moment,
\begin{equation}
 \mu_s^{\mathrm{CI}}\xrightarrow[s\to\infty]{}
 \mathcal B_{\mm}^{\mathrm H,+}(t)\dd t.
 \label{eq:Hermite-B}
\end{equation}

\item For every \(j\in\{1,\ldots,q\}\) with \(m_j>0\), the limiting
polynomial satisfies
\begin{equation}
 \int_{\mathbb R}t^rP_{\mm}^{\mathrm H,+}(t)
 \mathcal W_j^{\mathrm H,+}(t)\dd t=0,
 \qquad r\in\{0,\ldots,m_j-1\},
 \label{eq:Hermite-A-orthogonality}
\end{equation}
and the signed density satisfies
\begin{equation}
 \int_{\mathbb R}t^r\mathcal B_{\mm}^{\mathrm H,+}(t)\dd t
 =\delta_{r,n},\qquad r\in\{0,\ldots,n\}.
 \label{eq:Hermite-B-moments}
\end{equation}
\end{enumerate}
\end{theorem}

\begin{proof}
The preceding convolution argument proves the limits of the weights. Equivalently,
locally for \(w\) near zero,
\[
 \e^{-aw/s}G_j^{\mathrm{CI}}(\e^{w/s})
 \xrightarrow[s\to\infty]{}
 \e^{w^2/4}
 \prod_{h=1}^{q-1}\left(\frac{\rho_h}{\rho_h-w}\right)^{
 \delta_h+\delta_{h,j}}.
\]
Differentiating this locally uniform limit gives convergence of every
fixed moment.
In \eqref{eq:CI-A-generating}, set \(u\mapsto u/s\) and use
\((1+u/s)^{a+st}\e^{-au/s}
\xrightarrow[s\to\infty]{}\e^{tu-u^2/4}\).
This proves \eqref{eq:Hermite-A-generating}; expanding the finite product
and using \eqref{eq:monic-Hermite-definition} gives
\eqref{eq:Hermite-A}.
For the limit of the signed function, take the bilateral Laplace transform of
\eqref{eq:Hermite-B-measures}. If \(H_{\mm,s}^{\mathrm{CI}}\) denotes
the positive generating-function factor in \eqref{eq:CI-B} without
\((z-1)^n/n!\), then
\[
 \int\e^{wt}\dd\mu_s^{\mathrm{CI}}(t)
 =s^n\e^{-aw/s}\mathcal H_{\mm}^{\mathrm{CI}}(\e^{w/s})
 =\frac{[s(\e^{w/s}-1)]^n}{n!}
 \e^{-aw/s}H_{\mm,s}^{\mathrm{CI}}(\e^{w/s})
 \xrightarrow[s\to\infty]{}\frac{w^n}{n!}\Lambda_{\mm}(w),
\]
locally near \(w=0\), where \(\Lambda_{\mm}\) is the bilateral Laplace
transform of \(S_{\mm}^{\mathrm H}\). This is precisely the transform
of the signed measure in \eqref{eq:Hermite-B}. For a direct test-function
proof, define the normalized nonnegative coefficient sequence
\(p_{\mm,s}^{\mathrm{CI}}(\ell)
\coloneq[z^\ell]H_{\mm,s}^{\mathrm{CI}}(z)\).
Coefficient extraction gives
\(\mathcal B_{\mm}^{\mathrm{CI}}(\ell)
=\frac1{n!}\sum_{r=0}^n(-1)^{n-r}\binom nr
p_{\mm,s}^{\mathrm{CI}}(\ell-r)\).
Hence, for every \(\varphi\in C_c^\infty(\R)\),
\[
 \left\langle\mu_s^{\mathrm{CI}},\varphi\right\rangle
 =\frac1{n!}\sum_{\ell\ge0}p_{\mm,s}^{\mathrm{CI}}(\ell)
 s^n\Delta_{1/s}^n\varphi\!\left(\frac{\ell-a}{s}\right).
\]
Since \(s^n\Delta_{1/s}^n\varphi
\xrightarrow[s\to\infty]{}\varphi^{(n)}\) uniformly and the same
convolution argument, with shapes \(\delta_h+m_h\), gives, as \(s\to\infty\), convergence of
the associated centered coefficient measures to
\(S_{\mm}^{\mathrm H}(t)\dd t\), the right-hand side tends to
\[
 \frac1{n!}\int_{\R}S_{\mm}^{\mathrm H}(t)
 \varphi^{(n)}(t)\dd t
 =\left\langle\mathcal B_{\mm}^{\mathrm H,+},\varphi\right\rangle.
\]
Differentiating the transform identity at the origin, or equivalently
using the exact factorial moments of the coefficient sequences, proves
convergence of
every fixed moment. For \(0\le r<m_j\), first express
\(((\ell-a)/s)^r\) as the corresponding finite linear combination of
\(\fall{\ell}{0},\ldots,\fall{\ell}{r}\). The Charlier-I-like
orthogonality makes the resulting centered relation identically zero for
every \(s\); the coefficientwise polynomial limit and the centered
weight-moment convergence proved above then give
\eqref{eq:Hermite-A-orthogonality}; integration by parts in
\eqref{eq:Hermite-seed-B} proves \eqref{eq:Hermite-B-moments}.
\end{proof}

For a near-diagonal multi-index \(\mm\), put
\(D=\abs\mm\ge1\) and \(n=D-1\).  For the reflected Hermite-like
system, set
\begin{equation}
 S_{\mm}^{\mathrm H,-}(t)
 \coloneq\int_{(0,\infty)^{q-1}}
 \phi(t+y_1+\cdots+y_{q-1})
 \prod_{h=1}^{q-1}g_{\rho_h,\delta_h+m_h}(y_h)\dd\boldsymbol y,
 \qquad
 \mathcal B_{\mm}^{\mathrm H,-}(t)
 \coloneq\frac{(-1)^n}{n!}\frac{\dd^n}{\dd t^n}
 S_{\mm}^{\mathrm H,-}(t).
 \label{eq:Hermite-minus-seed-and-density}
\end{equation}
Since
\(\phi^{(n)}(t)=(-1)^n2^n\mathsf H_n(t)\phi(t)\), this definition is
equivalent to
\begin{equation}
 \mathcal B_{\mm}^{\mathrm H,-}(t)
 =\frac{2^n}{n!}
 \int_{(0,\infty)^{q-1}}
 \mathsf H_n(t+y_1+\cdots+y_{q-1})
 \phi(t+y_1+\cdots+y_{q-1})
 \prod_{h=1}^{q-1}g_{\rho_h,\delta_h+m_h}(y_h)
 \dd\boldsymbol y.
 \label{eq:Hermite-B-kernel-minus}
\end{equation}

We next fix the finite type-I coefficient system.  Let \(\mm\) be near the
diagonal, put \(D=\abs\mm=n+1\ge1\), and assume
\(\rho_h,\delta_h>0\) for \(h\in\{1,\ldots,q-1\}\). Set
\(L_-(z)\coloneq\e^{z^2/4}\prod_{h=1}^{q-1}
\left(\rho_h/(\rho_h+z)\right)^{\delta_h}\).
For \(j\in\{1,\ldots,q\}\), let
\(\mathcal W_j^{\mathrm H,-}(t)\coloneq\mathcal W_j^{\mathrm H,+}(-t)\); its
bilateral Laplace transform is
\(\Lambda_h^-(z)=\rho_hL_-(z)/(\rho_h+z)\) for
\(h\in\{1,\ldots,q-1\}\), while \(\Lambda_q^-(z)=L_-(z)\).
Set \(Q_{\mm}^{\mathrm H}(z)\coloneq\prod_{h=1}^{q-1}(\rho_h+z)^{m_h}\),
and let \(\varkappa_j=\rho_j\) for \(j\in\{1,\ldots,q-1\}\), while
\(\varkappa_q=1\). For every \(j\in\{1,\ldots,q\}\) and
\(k\in\Nzero\), put
\(\Phi_{j,k}^{\mathrm H,-}(z)
\coloneq Q_{\mm}^{\mathrm H}(z)L_-(z)^{-1}
\mathsf H_k(\partial_z)\Lambda_j^-(z)\).
For \(j\in\{1,\ldots,q\}\) with \(m_j>0\),
\(k\in\{0,\ldots,m_j-1\}\), and \(\nu\in\{0,\ldots,n\}\), set
\(C^{\mathrm H,-}_{\nu,(j,k)}\coloneq[z^\nu]
\Phi_{j,k}^{\mathrm H,-}(z)\).

\begin{lemma}[Hermite type-I coefficient system]
\label{lem:Hermite-connection-determinant}
\begin{enumerate}[label=\textnormal{(\roman*)}]
\item Fix \(j\in\{1,\ldots,q\}\). The formal generating-function
identity
\[
 \sum_{k\ge0}\frac{u^k}{k!}\Phi_{j,k}^{\mathrm H,-}(z)
 =\varkappa_j Q_{\mm}^{\mathrm H}(z)
 \prod_{h=1}^{q-1}(z+\rho_h)^{-\delta_{j,h}}
 \e^{zu/2}\prod_{h=1}^{q-1}
 \left(1+\frac{u}{z+\rho_h}\right)^{-
 (\delta_h+\delta_{j,h})}
\]
holds. Equivalently, for every \(k\in\Nzero\),
\begin{equation}
 \Phi_{j,k}^{\mathrm H,-}(z)
 =\varkappa_j k!
 \sum_{\substack{r_0,s_1,\ldots,s_{q-1}\ge0\\
                   r_0+s_1+\cdots+s_{q-1}=k}}
 \frac{(z/2)^{r_0}}{r_0!}
 \times\prod_{h=1}^{q-1}
 \frac{(-1)^{s_h}(\delta_h+\delta_{j,h})_{s_h}}{s_h!}
 (z+\rho_h)^{m_h-\delta_{j,h}-s_h}.
 \label{eq:Hermite-Phi-finite}
\end{equation}
If \(m_j>0\) and \(k\in\{0,\ldots,m_j-1\}\), the polynomial
\(\Phi_{j,k}^{\mathrm H,-}\) has degree at most \(n\).

\item For every \(j\in\{1,\ldots,q\}\) with \(m_j>0\),
\(k\in\{0,\ldots,m_j-1\}\), and \(\nu\in\{0,\ldots,n\}\), the entries
of the coefficient matrix are
\begin{multline}
 C^{\mathrm H,-}_{\nu,(j,k)}
 =\varkappa_j k!
 \sum_{\substack{r_0,s_1,\ldots,s_{q-1}\ge0\\
                   r_0+s_1+\cdots+s_{q-1}=k}}
 \frac{2^{-r_0}}{r_0!}
 \prod_{h=1}^{q-1}
 \frac{(-1)^{s_h}(\delta_h+\delta_{j,h})_{s_h}}{s_h!}
 \\\times\sum_{\substack{0\le t_h\le
 m_h-\delta_{j,h}-s_h\ (h\in\{1,\ldots,q-1\})\\
 t_1+\cdots+t_{q-1}=\nu-r_0}}
 \prod_{h=1}^{q-1}
 \binom{m_h-\delta_{j,h}-s_h}{t_h}
 \rho_h^{m_h-\delta_{j,h}-s_h-t_h}.
 \label{eq:Hermite-connection-entries}
\end{multline}
The inner sum is zero when \(\nu<r_0\). When \(q=1\), an empty product
equals one and the inner constrained sum equals \(\delta_{\nu,r_0}\).

\item For every fixed ordering of the columns,
\begin{equation}
 \det C^{\mathrm H,-}
 =\pm 2^{-\binom{m_q}{2}}
 \prod_{h=1}^{q-1}\left[
 \rho_h^{m_h}\prod_{k=0}^{m_h-1}(\delta_h+1)_k\right]
 \prod_{1\le j<h\le q-1}(\rho_h-\rho_j)^{m_jm_h}.
 \label{eq:Hermite-connection-determinant}
\end{equation}
If the rates \(\rho_h\) for which
\(h\in\{1,\ldots,q-1\}\) and \(m_h>0\) are pairwise
distinct, both \(C^{\mathrm H,-}\) and the coefficient matrix obtained
from it by reflecting the underlying polynomials and weights under
\(t\mapsto-t\) are nonsingular. We call the latter the \(H^+\)
coefficient system.
\end{enumerate}
\end{lemma}

\begin{proof}
The generating function
\(\sum_{k\ge0}\mathsf H_k(\partial_z)u^k/k!
=\exp(u\partial_z-u^2/4)\) acts by translation. Applying it to
\(\Lambda_j^-\), dividing by \(L_-\), and simplifying the Gaussian
factor gives the displayed generating identity. Its multinomial expansion
gives \eqref{eq:Hermite-Phi-finite}; expanding the remaining powers of
\(z+\rho_h\) gives \eqref{eq:Hermite-connection-entries}.

Put
\[
 E_{h,k}(z)=\frac{Q_{\mm}^{\mathrm H}(z)}
 {(\rho_h+z)^{k+1}}
 \quad(h\in\{1,\ldots,q-1\}),\qquad
 E_{q,k}(z)=z^kQ_{\mm}^{\mathrm H}(z).
\]
The near-diagonal hypothesis ensures that all denominators created by
\(\mathsf H_k(\partial_z)\) are cleared by \(Q_{\mm}^{\mathrm H}\), and the behavior
at infinity gives degree at most \(n\). The coefficient determinant of
these \(D\) polynomials is the confluent Vandermonde
\[
 \det\bigl([z^\nu]E_{j,k}(z)\bigr)_{
 0\le\nu<D,\,(j,k)}
 =\pm\prod_{1\le j<h\le q-1}(\rho_h-\rho_j)^{m_jm_h}.
\]
Indeed, the columns \(E_{q,k}=z^kQ_{\mm}^{\mathrm H}\), indexed by
\(k\in\Nzero\) with \(k<m_q\), have leading degree
\(\sum_{h=1}^{q-1}m_h+k\). After removing this possibly empty family, the
remaining determinant is the standard confluent resultant of the factors
\((z+\rho_h)^{m_h}\).

For \(k\in\Nzero\), denote by \(\mathcal E_{<k}\) the span of the columns
\(E_{j,r}\) with \(j\in\{1,\ldots,q\}\), \(m_j>0\), and \(r\in\Nzero\),
\(r<\min\{k,m_j\}\). Order all columns globally by increasing \(k\). At the pole
\(z=-\rho_h\), the highest principal part gives
\[
 \Phi_{h,k}^{\mathrm H,-}
 -(-1)^k\rho_h(\delta_h+1)_kE_{h,k}
 \in\mathcal E_{<k},\qquad h\in\{1,\ldots,q-1\}.
\]
At infinity, since \(\mathsf H_k(\partial_z)\) is monic in
\(\partial_z\) and
\((\partial_z^k\e^{z^2/4})/\e^{z^2/4}
=2^{-k}z^k+\mathrm O(z^{k-2})\) as \(|z|\to\infty\), one has
\(\Phi_{q,k}^{\mathrm H,-}-2^{-k}E_{q,k}\in\mathcal E_{<k}\).
Thus the change of basis is triangular. Multiplying its diagonal entries
and the confluent Vandermonde gives
\eqref{eq:Hermite-connection-determinant}; the signs \((-1)^k\) are
absorbed in \(\pm\). Reflection \(t\mapsto-t\) proves the last assertion.
\end{proof}

The components are obtained by matching the finite principal parts and
the polynomial part at infinity of the transformed type-I identity.  We
first fix the coefficient data used in this matching.  Let \(\mm\) be
near the diagonal, put
\[
 D\coloneq\abs\mm\ge1,\qquad n\coloneq D-1,\qquad
 \mathcal I_{\mm}^-\coloneq\{h\in\{1,\ldots,q-1\}:m_h>0\},\qquad
 C_{\mm}\coloneq\prod_{h\in\mathcal I_{\mm}^-}\rho_h^{m_h},
\]
and assume \(\rho_h,\delta_h>0\) for
\(h\in\mathcal I_{\mm}^-\), and assume that these rates are pairwise
distinct. We seek the coefficients \(c_{j,k}\) in
\begin{equation}
 \sum_{j=1}^q\sum_{k=0}^{m_j-1}
 c_{j,k}\Phi_{j,k}^{\mathrm H,-}(z)
 =\frac{C_{\mm}}{n!}z^n.
 \label{eq:Hermite-explicit-transformed-identity}
\end{equation}
Finite principal parts account for the indices
\(h\in\{1,\ldots,q-1\}\), while the
polynomial part at infinity is indexed by \(h=q\). Accordingly, set
\[
 \mathscr P_{\mm}^{\mathrm H}
 \coloneq\{(h,K):h\in\mathcal I_{\mm}^-,\
 K\in\Nzero,\ K<m_h\}
 \cup\{(q,K):K\in\Nzero,\ K<m_q\}.
\]
For \(a=(h,K)\) and \(b=(j,r)\) in this set, write \(a\prec b\) when
\(K<r\). The formulas below give the coefficients in the triangular
system
\begin{equation}
 \Delta_a^{\mathrm H}c_a+
 \sum_{a\prec b}Z_{a,b}^{\mathrm H}c_b
 =\Omega_a^{\mathrm H},
 \qquad a\in\mathscr P_{\mm}^{\mathrm H}.
 \label{eq:Hermite-explicit-triangular-system}
\end{equation}

For every \((h,K)\in\mathscr P_{\mm}^{\mathrm H}\) with
\(h\in\mathcal I_{\mm}^-\), define
\begin{equation}
 \Omega_{h,K}^{\mathrm H}
 \coloneq\frac{C_{\mm}(-\rho_h)^n}
 {n!\displaystyle\prod_{\substack{g\in\mathcal I_{\mm}^-\\g\ne h}}
       (\rho_g-\rho_h)^{m_g}}
 \mathscr L_{m_h-K-1}
 \left(
 n;(m_g)_{\substack{g\in\mathcal I_{\mm}^-\\g\ne h}};
 \rho_h^{-1},
 \bigl(-(\rho_g-\rho_h)^{-1}\bigr)_{
       \substack{g\in\mathcal I_{\mm}^-\\g\ne h}}
 \right).
 \label{eq:Hermite-explicit-target-finite}
\end{equation}
For every \(K\in\Nzero\) with \(K<m_q\), define
\begin{equation}
 \Omega_{q,K}^{\mathrm H}
 \coloneq\frac{C_{\mm}}{n!}
 \mathscr H_{m_q-1-K}
 \left((m_g)_{g\in\mathcal I_{\mm}^-};
       (-\rho_g)_{g\in\mathcal I_{\mm}^-}\right).
 \label{eq:Hermite-explicit-target-infinity}
\end{equation}
For every \(a=(h,K)\in\mathscr P_{\mm}^{\mathrm H}\), write
\(\Omega_a^{\mathrm H}\coloneq\Omega_{h,K}^{\mathrm H}\).

For \(j\in\{1,\ldots,q\}\) with \(m_j>0\), put
\(\varkappa_j=\rho_j\) for \(j\in\{1,\ldots,q-1\}\) and
\(\varkappa_q=1\).
For \(j\in\{1,\ldots,q\}\) with \(m_j>0\),
\(g\in\mathcal I_{\mm}^-\), and
\(\boldsymbol s\in\mathbb N_0^{\mathcal I_{\mm}^-}\), put
\(d_g(\boldsymbol s,j)\coloneq s_g+\delta_{j,g}\). For a fixed
\(h\in\mathcal I_{\mm}^-\), abbreviate
\[
 \boldsymbol d_{\widehat h}(\boldsymbol s,j)
 =\bigl(d_g(\boldsymbol s,j)\bigr)_{
   g\in\mathcal I_{\mm}^-\setminus\{h\}},\qquad
 \boldsymbol\eta_{\widehat h}
 =\bigl(-(\rho_g-\rho_h)^{-1}\bigr)_{
   g\in\mathcal I_{\mm}^-\setminus\{h\}}.
\]
For
\(j\in\{1,\ldots,q\}\) with \(m_j>0\),
\(k\in\{0,\ldots,m_j-1\}\), and
\((h,K)\in\mathscr P_{\mm}^{\mathrm H}\) with
\(h\in\mathcal I_{\mm}^-\), define
\begin{multline}
 Z_{h,K,j,k}^{\mathrm H}
 \coloneq\varkappa_j k!
 \sum_{\substack{r_0\in\Nzero,\ 
                  \boldsymbol s\in\mathbb N_0^{\mathcal I_{\mm}^-}\\
                  r_0+\sum_{g\in\mathcal I_{\mm}^-}s_g=k,\ 
                  d_h(\boldsymbol s,j)\ge K+1}}
 \frac{(-\rho_h/2)^{r_0}}{r_0!}
 \prod_{g\in\mathcal I_{\mm}^-}
 \frac{(-1)^{s_g}(\delta_g+\delta_{j,g})_{s_g}}{s_g!}
 \prod_{\substack{g\in\mathcal I_{\mm}^-\\g\ne h}}
 (\rho_g-\rho_h)^{-d_g(\boldsymbol s,j)}
 \\
 {}\times
 \mathscr L_{d_h(\boldsymbol s,j)-K-1}
 \left(
 r_0;\boldsymbol d_{\widehat h}(\boldsymbol s,j);
 \rho_h^{-1},\boldsymbol\eta_{\widehat h}
 \right).
 \label{eq:Hermite-explicit-connection-finite}
\end{multline}
For \(j\in\{1,\ldots,q\}\) with \(m_j>0\), \(r_0\in\Nzero\), and
\(\boldsymbol s\in\mathbb N_0^{\mathcal I_{\mm}^-}\), set
\(e(r_0,\boldsymbol s,j)\coloneq
r_0-\sum_{g\in\mathcal I_{\mm}^-}d_g(\boldsymbol s,j)\).
For \(j\in\{1,\ldots,q\}\) with \(m_j>0\),
\(k\in\{0,\ldots,m_j-1\}\), and
\(K\in\Nzero\) with \(K<m_q\), put
\begin{multline}
 Z_{q,K,j,k}^{\mathrm H}
 \coloneq\varkappa_j k!
 \sum_{\substack{r_0\in\Nzero,\ 
                  \boldsymbol s\in\mathbb N_0^{\mathcal I_{\mm}^-}\\
                  r_0+\sum_{g\in\mathcal I_{\mm}^-}s_g=k\\
                  e(r_0,\boldsymbol s,j)\ge K}}
 \frac{2^{-r_0}}{r_0!}
 \prod_{g\in\mathcal I_{\mm}^-}
 \frac{(-1)^{s_g}(\delta_g+\delta_{j,g})_{s_g}}{s_g!}
 \mathscr H_{e(r_0,\boldsymbol s,j)-K}
 \left(
 \bigl(d_g(\boldsymbol s,j)\bigr)_{g\in\mathcal I_{\mm}^-};
 (-\rho_g)_{g\in\mathcal I_{\mm}^-}
 \right).
 \label{eq:Hermite-explicit-connection-infinity}
\end{multline}
For \(a=(h,K),b=(j,k)\in\mathscr P_{\mm}^{\mathrm H}\), write
\(Z_{a,b}^{\mathrm H}\coloneq Z_{h,K,j,k}^{\mathrm H}\).
For \(h\in\mathcal I_{\mm}^-\) and
\(K\in\Nzero\) with \(K<m_h\), and for \(K\in\Nzero\) with
\(K<m_q\), respectively, the diagonal coefficients are
\begin{equation}
 \Delta_{(h,K)}^{\mathrm H}
 \coloneq(-1)^K\rho_h(\delta_h+1)_K,
 \qquad
 \Delta_{(q,K)}^{\mathrm H}
 \coloneq2^{-K}.
 \label{eq:Hermite-explicit-pivots}
\end{equation}
For every \(a\in\mathscr P_{\mm}^{\mathrm H}\), define
\begin{equation}
 \mathfrak C_a^{\mathrm H}
 \coloneq\sum_{p\ge0}(-1)^p
 \sum_{\substack{a=a_0\prec a_1\prec\cdots\prec a_p\\
                 a_v\in\mathscr P_{\mm}^{\mathrm H}}}
 \frac{\Omega_{a_p}^{\mathrm H}}{\Delta_{a_p}^{\mathrm H}}
 \prod_{v=0}^{p-1}
 \frac{Z_{a_v,a_{v+1}}^{\mathrm H}}
      {\Delta_{a_v}^{\mathrm H}}.
 \label{eq:Hermite-explicit-chain}
\end{equation}

A sum over an empty index set is understood to be zero.

\begin{theorem}[Explicit Hermite-like type-I components]
\label{thm:explicit-Hermite-B}
Let \(\mm\) be near the diagonal with \(\abs\mm\ge1\). Assume
\(\rho_h,\delta_h>0\) for every
\(h\in\{1,\ldots,q-1\}\), and assume that the active rates \(\rho_h\)
are pairwise distinct. With the data defined above, for
\(\alpha,\beta\in\mathscr P_{\mm}^{\mathrm H}\),
\(Z_{\alpha,\beta}^{\mathrm H}=0\) unless
\(\alpha=\beta\) or \(\alpha\prec\beta\), and
\(Z_{\alpha,\alpha}^{\mathrm H}=\Delta_\alpha^{\mathrm H}\). The sum in
\eqref{eq:Hermite-explicit-chain} is finite. The unique type-I
polynomials are
\begin{equation}
 B_j^{\mathrm H,-}(t)
 =\sum_{k=0}^{m_j-1}
 \mathfrak C_{(j,k)}^{\mathrm H}\mathsf H_k(t),
 \qquad j\in\{1,\ldots,q\},
 \label{eq:explicit-Hermite-minus-components}
\end{equation}
\end{theorem}

\begin{proof}
Divide \eqref{eq:Hermite-Phi-finite} by \(Q_{\mm}^{\mathrm H}(z)\). If
some entry of \(\mm\) is zero, near-diagonality gives
\(m_h\in\{0,1\}\) for every \(h\in\{1,\ldots,q\}\), so only \(k=0\)
occurs; otherwise \(m_h>0\) for every
\(h\in\{1,\ldots,q-1\}\).
Consequently the finite sum may be written over
\(\mathcal I_{\mm}^-\) in every case.

At a finite pole \(z=-\rho_h\), write \(w=z+\rho_h\). In each summand,
\((z/2)^{r_0}=(-\rho_h/2)^{r_0}(1-w/\rho_h)^{r_0}\); expand the
remaining factors in powers of \(w\).
The coefficient of \(w^{-K-1}\) is exactly
\eqref{eq:Hermite-explicit-connection-finite}.  At infinity,
\[
 (z/2)^{r_0}\prod_{g\in\mathcal I_{\mm}^-}
 (z+\rho_g)^{-d_g(\boldsymbol s,j)}
 =2^{-r_0}z^{e(r_0,\boldsymbol s,j)}
  \prod_{g\in\mathcal I_{\mm}^-}
  (1+\rho_g/z)^{-d_g(\boldsymbol s,j)},
\]
so the coefficient of \(z^K\) in the polynomial part is precisely
\eqref{eq:Hermite-explicit-connection-infinity}.  The highest possible
pole order and polynomial degree give the diagonal values
\eqref{eq:Hermite-explicit-pivots} and show that all other entries with
the same second index vanish.

After dividing \eqref{eq:Hermite-explicit-transformed-identity} by
\(Q_{\mm}^{\mathrm H}\), direct finite expansions of the right-hand side
at \(z=-\rho_h\) and at infinity give
\eqref{eq:Hermite-explicit-target-finite} and
\eqref{eq:Hermite-explicit-target-infinity}, respectively, and give
\eqref{eq:Hermite-explicit-triangular-system}. Grouping the paths in
\eqref{eq:Hermite-explicit-chain} according to their first step proves
directly that
\(c_a=\mathfrak C_a^{\mathrm H}\) solves this system.  Equality of all
finite principal parts and of the polynomial part proves the rational
identity, hence \eqref{eq:Hermite-explicit-transformed-identity}.
Undoing the transform gives
\(\mathcal B_{\mm}^{\mathrm H,-}
=\sum_{j=1}^qB_j^{\mathrm H,-}\mathcal W_j^{\mathrm H,-}\).
The nonzero diagonal coefficients also give uniqueness, consistently with
Lemma~\ref{lem:Hermite-connection-determinant}.
\end{proof}

\begin{corollary}[Finite Hermite expansion of the type-I polynomials]
\label{cor:Hermite-explicit-components}
Under the hypotheses of Theorem~\ref{thm:explicit-Hermite-B},
formula \eqref{eq:explicit-Hermite-minus-components} gives every
\(B_j^{\mathrm H,-}\) as a finite sum of monic Hermite polynomials whose
coefficients are finite sums built from \(\mathscr L_L\) and
\(\mathscr H_L\). The reflected type-I polynomials satisfy, for every
\(j\in\{1,\ldots,q\}\) with \(m_j>0\),
\begin{equation}
 B_j^{\mathrm H,+}(t)=(-1)^nB_j^{\mathrm H,-}(-t).
 \label{eq:Hermite-B-component-reflection}
\end{equation}
When \(q=1\), both signs give
\begin{equation}
 B_1^{\mathrm H,\pm}(t)=\frac{2^n}{n!}\mathsf H_n(t).
 \label{eq:Hermite-scalar-component}
\end{equation}
\end{corollary}

\begin{proof}
Apply Theorem~\ref{thm:explicit-Hermite-B}.  Under \(t\mapsto-t\), the
normalizing moment of order \(n\) acquires the factor \((-1)^n\), which
proves \eqref{eq:Hermite-B-component-reflection}.  For \(q=1\),
\(\mathcal B=(-1)^n\phi^{(n)}/n!=2^n\mathsf H_n\phi/n!\), proving
\eqref{eq:Hermite-scalar-component}.
\end{proof}

The explicit polynomials of Corollary~\ref{cor:Hermite-explicit-components}
also arise as limits. Under the same distinctness condition on the rates \(\rho_h\) for which
\(h\in\{1,\ldots,q-1\}\) and \(m_h>0\), the type-I polynomials converge
as well.

\begin{corollary}[Limit of the Charlier-I-like type-I polynomials]
\label{cor:CI-Hermite-components}
Under the hypotheses of Theorem~\ref{thm:Hermite-limit}, assume that the
rates \(\rho_h\) for which \(h\in\{1,\ldots,q-1\}\) and \(m_h>0\) are
pairwise distinct.
Then there is a unique tuple
\(\bigl(B_j^{\mathrm H,+}\bigr)_{
 j\in\{1,\ldots,q\},\,m_j>0}\) satisfying
\(\deg B_j^{\mathrm H,+}<m_j\) for every
\(j\in\{1,\ldots,q\}\) with \(m_j>0\), and
\[
 \mathcal B_{\mm}^{\mathrm H,+}(t)
 =\sum_{\substack{j\in\{1,\ldots,q\}\\m_j>0}}
 B_j^{\mathrm H,+}(t)
 \mathcal W_j^{\mathrm H,+}(t).
\]
If
\(B_{j,s}^{\mathrm{CI}}\) denotes the Charlier-I-like type-I polynomial along
\eqref{eq:Hermite-scaling}, then, for every
\(j\in\{1,\ldots,q\}\) with \(m_j>0\), the polynomial
\(t\mapsto s^nB_{j,s}^{\mathrm{CI}}(a+st)\) converges coefficientwise, as
\(s\to\infty\), to
the Hermite-like type-I polynomial
\(B_j^{\mathrm H,+}(t)\).
\end{corollary}

\begin{proof}
Integration by parts and \eqref{eq:Hermite-seed-B} give
\[
 \int_{\mathbb R}\e^{zt}\mathcal B_{\mm}^{\mathrm H,+}(t)\dd t
 =\frac{z^n}{n!}\e^{z^2/4}
 \prod_{h=1}^{q-1}\left(\frac{\rho_h}{\rho_h-z}\right)^{\delta_h+m_h}.
\]
After reflection, multiplication by
\(\prod_{h=1}^{q-1}(\rho_h-z)^{m_h}\) reduces the type-I equation to the
polynomial right-hand side \(z^n\prod_{h=1}^{q-1}\rho_h^{m_h}/n!\).
Lemma~\ref{lem:Hermite-connection-determinant} therefore gives unique
\(H^+\) type-I polynomials.

Expand the centered type-I polynomial in the Hermite basis from that lemma as
\(s^nB_{j,s}^{\mathrm{CI}}(a+st)
=\sum_{k=0}^{m_j-1}\gamma_{j,k,s}\mathsf H_k(t)\).
As \(s\to\infty\), the centered transforms of the weights
\(\widetilde G_{j,s}(w)
\coloneq\e^{-aw/s}G_{j,s}^{\mathrm{CI}}(\e^{w/s})\)
converge locally, together with their derivatives through order
\(n+\max_{h:m_h>0}(m_h-1)\), to
the corresponding Hermite-like weight transforms. Hence the finite
coefficient system for \((\gamma_{j,k,s})\) converges entrywise to the
invertible continuous coefficient system. Continuity of matrix inversion
gives coefficientwise convergence.
\end{proof}

\paragraph{Direct finite Hermite--Lauricella blocks.}
The rational interpolation underlying
Theorem~\ref{thm:explicit-Hermite-B} gives a block decomposition without
any limiting or finite-part prescription. Its active sectors are
\[
 \mathcal J_{\mm}^{\mathrm H}
 \coloneq\{J\in\{1,\ldots,q\}:m_J>0\}.
\]
The sectors \(J<q\) correspond to the finite poles
\(z=-\rho_J\), whereas \(J=q\), when active, corresponds to the
polynomial part at infinity. For
\(a\in\mathscr P_{\mm}^{\mathrm H}\) and
\(J\in\mathcal J_{\mm}^{\mathrm H}\), define
\begin{equation}
 \mathfrak C_{a;J}^{\mathrm H}
 \coloneq
 \sum_{p\ge0}(-1)^p
 \sum_{\substack{a=a_0\prec a_1\prec\cdots\prec a_p\\
                  a_v\in\mathscr P_{\mm}^{\mathrm H},\
                  a_p=(J,K)\ {\rm for\ some}\ 0\le K<m_J}}
 \frac{\Omega_{a_p}^{\mathrm H}}{\Delta_{a_p}^{\mathrm H}}
 \prod_{v=0}^{p-1}
 \frac{Z_{a_v,a_{v+1}}^{\mathrm H}}
      {\Delta_{a_v}^{\mathrm H}}.
 \label{eq:Hermite-sector-path-coefficients}
\end{equation}
For every \(j,J\in\mathcal J_{\mm}^{\mathrm H}\), put
\begin{align}
 \mathscr H_{j,J}^{\mathrm H,-}(t)
 &\coloneq
 \sum_{k=0}^{m_j-1}
 \mathfrak C_{(j,k);J}^{\mathrm H}\mathsf H_k(t),
 \label{eq:Hermite-minus-direct-block}\\
 \mathscr H_{j,J}^{\mathrm H,+}(t)
 &\coloneq(-1)^n\mathscr H_{j,J}^{\mathrm H,-}(-t)
 =\sum_{k=0}^{m_j-1}(-1)^{n+k}
 \mathfrak C_{(j,k);J}^{\mathrm H}\mathsf H_k(t).
 \label{eq:Hermite-plus-direct-block}
\end{align}

\begin{proposition}[Direct block decomposition of the Hermite-like components]
\label{prop:Hermite-direct-block-decomposition}
Under the hypotheses of
Theorem~\ref{thm:explicit-Hermite-B}, every sum in
\eqref{eq:Hermite-sector-path-coefficients} is finite and, for every
\(j,J\in\mathcal J_{\mm}^{\mathrm H}\), put
\(\varepsilon_{j,J}=1-\delta_{j,J}\). If
\(m_J<\varepsilon_{j,J}+1\), then
\(\mathscr H_{j,J}^{\mathrm H,\pm}=0\); otherwise,
\begin{equation}
 \deg\mathscr H_{j,J}^{\mathrm H,\pm}
 \le\min\{m_j-1,m_J-1-\varepsilon_{j,J}\}.
 \label{eq:Hermite-direct-sector-degree}
\end{equation}
In particular, every block has degree \(<m_j\). Moreover, for every
\(j\in\mathcal J_{\mm}^{\mathrm H}\),
\begin{equation}
 B_j^{\mathrm H,-}(t)
 =\sum_{J\in\mathcal J_{\mm}^{\mathrm H}}
   \mathscr H_{j,J}^{\mathrm H,-}(t),
 \qquad
 B_j^{\mathrm H,+}(t)
 =\sum_{J\in\mathcal J_{\mm}^{\mathrm H}}
   \mathscr H_{j,J}^{\mathrm H,+}(t).
 \label{eq:Hermite-direct-block-decomposition}
\end{equation}
Thus each component is a sum of at most \(q\) directly evaluated
terminating blocks. Every block is a finite Hermite expansion whose
coefficients are finite sums of products of the terminating
Lauricella--Horn coefficients \(\mathscr L_L\) and \(\mathscr H_L\)
displayed in
\eqref{eq:Hermite-explicit-target-finite}--
\eqref{eq:Hermite-explicit-connection-infinity}.
\end{proposition}

\begin{proof}
At every step of a path in
\eqref{eq:Hermite-sector-path-coefficients}, the second coordinate
increases strictly, so all path sums are finite. Moreover, every path
occurring in \eqref{eq:Hermite-explicit-chain} has one and only one
terminal label \(a_p=(J,K)\), with
\(J\in\mathcal J_{\mm}^{\mathrm H}\). Partitioning that finite path sum
according to \(J\) gives
\[
 \mathfrak C_a^{\mathrm H}
 =\sum_{J\in\mathcal J_{\mm}^{\mathrm H}}
 \mathfrak C_{a;J}^{\mathrm H}.
\]
For a coefficient indexed by \(a=(j,k)\), a path ending at
\((J,K)\) has \(k\le K\) when \(j=J\), whereas it has positive length
and hence \(k<K\) when \(j\ne J\). Since \(K<m_J\), this proves the
vanishing assertion and \eqref{eq:Hermite-direct-sector-degree}.
Substitution in \eqref{eq:explicit-Hermite-minus-components} proves the
first identity in \eqref{eq:Hermite-direct-block-decomposition}. The
second follows from \eqref{eq:Hermite-B-component-reflection} and
\(\mathsf H_k(-t)=(-1)^k\mathsf H_k(t)\). Finally,
\eqref{eq:Hermite-explicit-target-finite}--
\eqref{eq:Hermite-explicit-connection-infinity} express every
\(\Omega\) and \(Z\) by finite \(\mathscr L_L\)- and
\(\mathscr H_L\)-sums, proving the stated explicitness without an
asymptotic extraction.
\end{proof}

\begin{corollary}[Sectorwise Charlier-I-like-to-Hermite confluence]
\label{cor:CI-Hermite-sector-limit}
Under the hypotheses of Corollary~\ref{cor:CI-Hermite-components}, let
\(\mathscr C_{j,J;s}^{\mathrm{CI}}\) denote the direct sector
\eqref{eq:CI-direct-sector-block}, evaluated at
\[
 a=\frac{s^2}{2},
 \qquad c_{h,s}=\frac{s}{s+\rho_h},
 \qquad h\in\{1,\ldots,q-1\}.
\]
For every pair of active indices \(j,J\),
\begin{equation}
 s^n\mathscr C_{j,J;s}^{\mathrm{CI}}(a+st)
 \xrightarrow[s\to\infty]{}
 \mathscr H_{j,J}^{\mathrm H,+}(t)
 \quad\hbox{coefficientwise in }t .
 \label{eq:CI-Hermite-sector-limit}
\end{equation}
\end{corollary}

\begin{proof}
Put
\[
 M\coloneq\sum_{h=1}^{q-1}m_h,\qquad
 C_{\mm}\coloneq\prod_{h=1}^{q-1}\rho_h^{m_h},\qquad
 Q_+(w)\coloneq\prod_{h=1}^{q-1}(\rho_h-w)^{m_h},
\]
so that \(n=M+m_q-1\), and set
\[
 L_+(w)\coloneq
 \e^{w^2/4}\prod_{h=1}^{q-1}
 \left(\frac{\rho_h}{\rho_h-w}\right)^{\delta_h},
 \qquad
 \mathcal T^+(w)\coloneq
 \frac{C_{\mm}w^n}{n!\,Q_+(w)}.
\]
Choose pairwise disjoint positively oriented circles
\(\Gamma_J^+\) around the active points \(\rho_J\), \(J<q\), and no
other active point \(\rho_h\).  If \(f\) has poles only at those points and a
polynomial part at infinity, write
\[
 (\mathcal P_J^+f)(w)\coloneq
 \frac{1}{2\pi\mathrm i}\int_{\Gamma_J^+}
 \frac{f(\zeta)}{w-\zeta}\,\dd\zeta
 \quad(J<q),\qquad
 \mathcal P_q^+f\coloneq
 f-\sum_{\substack{J<q\\m_J>0}}\mathcal P_J^+f .
\]
Thus \(\mathcal P_J^+f\) is the complete principal part at
\(w=\rho_J\), while \(\mathcal P_q^+f\) is the polynomial part at
infinity.  For every active \(J\), put
\(\mathcal T_J^+\coloneq\mathcal P_J^+\mathcal T^+\).
The projections are disjoint and
\(\mathcal T^+=\sum_{J\in\mathcal J_{\mm}^{\mathrm H}}
\mathcal T_J^+\).

Let
\[
 \mathcal D_{-,s}(z)\coloneq
 \prod_{h=1}^{q-1}(1-c_{h,s}z)^{m_h},\qquad
 \alpha_s\coloneq
 \prod_{h=1}^{q-1}\left(1+\frac{\rho_h}{s}\right)^{m_h}.
\]
The rational function split into the sectors
\(R_{J,s}^{\mathrm{CI}}\) in the paragraph preceding
Proposition~\ref{prop:CI-sector-blocks} is
\[
 \sum_{J\in\mathcal J_{\mm}^{\mathrm{CI}}}
 R_{J,s}^{\mathrm{CI}}(z)
 =\frac{(z-1)^n}{n!\,\mathcal D_{-,s}(z)}.
\]
The affine substitution \(z=1+w/s\) sends the pole
\(c_{J,s}^{-1}\) exactly to \(w=\rho_J\) and preserves the polynomial
degree at infinity.  Since
\[
 \mathcal D_{-,s}(1+w/s)
 =\frac{Q_+(w)}{s^M\alpha_s},
\]
uniqueness of partial fractions, equivalently the Cauchy projections
above, gives the exact sector identity
\begin{equation}
 C_{\mm}s^{m_q-1}
 R_{J,s}^{\mathrm{CI}}(1+w/s)
 =\alpha_s\mathcal T_J^+(w)
 \qquad(J\in\mathcal J_{\mm}^{\mathrm{CI}}).
 \label{eq:CI-Hermite-exact-sector-source}
\end{equation}
This also covers \(J=q\), because both sides then mean the polynomial
part at infinity.

Let
\[
 \mathcal H_{J,s}^{\mathrm{CI}}(z)\coloneq
 \sum_{\ell\ge0}\sum_{j\in\mathcal J_{\mm}^{\mathrm H}}
 \mathscr C_{j,J;s}^{\mathrm{CI}}(\ell)
 v_{j,s}^{\mathrm{CI}}(\ell)z^\ell
\]
be the signed generating function of the \(J\)-th block.  From
\eqref{eq:CI-direct-sector-identity} and the definitions of
\(\Phi_{j,r}^{\mathrm{CI}}\) and \(H_{\mm}^{\mathrm{CI}}\),
\[
 \mathcal H_{J,s}^{\mathrm{CI}}(z)
 =F_s^{\mathrm{CI}}(z)
 \prod_{h=1}^{q-1}(1-c_{h,s})^{m_h}
 R_{J,s}^{\mathrm{CI}}(z),
 \qquad
 F_s^{\mathrm{CI}}(z)
 =\e^{a(z-1)}
 \prod_{h=1}^{q-1}
 \left(\frac{1-c_{h,s}}{1-c_{h,s}z}\right)^{\delta_h}.
\]
Put \(\omega_s(w)\coloneq s(\e^{w/s}-1)\).  Since
\[
 \prod_{h=1}^{q-1}(1-c_{h,s})^{m_h}
 =\frac{C_{\mm}}{s^M\alpha_s},
\]
\eqref{eq:CI-Hermite-exact-sector-source} yields the exact centered
source formula
\begin{equation}
 s^n\e^{-aw/s}
 \mathcal H_{J,s}^{\mathrm{CI}}(\e^{w/s})
 =
 \widetilde F_s^{\mathrm{CI}}(w)
 \mathcal T_J^+(\omega_s(w)),
 \qquad
 \widetilde F_s^{\mathrm{CI}}(w)
 \coloneq\e^{-aw/s}F_s^{\mathrm{CI}}(\e^{w/s}).
 \label{eq:CI-Hermite-centered-sector-source}
\end{equation}
Here no asymptotic cancellation has been suppressed: the power of \(s\)
on the left is exactly \(n=M+m_q-1\).  Locally for
\(\abs w<\min_{h<q}\rho_h\) when \(q>1\), and locally on
\(\mathbb C\) when \(q=1\),
\[
 \widetilde F_s^{\mathrm{CI}}(w)
 =
 \exp\!\left\{\frac{s^2}{2}
 \left(\e^{w/s}-1-\frac ws\right)\right\}
 \prod_{h=1}^{q-1}
 \left(\frac{\rho_h}{\rho_h-\omega_s(w)}\right)^{\delta_h}
 \longrightarrow L_+(w)
\]
together with every fixed derivative, and
\(\omega_s(w)\to w\).  Therefore the right-hand side of
\eqref{eq:CI-Hermite-centered-sector-source} converges locally, with
all derivatives required by the fixed finite system, to
\(L_+(w)\mathcal T_J^+(w)\).

For completeness, the latter is exactly the transform of the direct
Hermite block:
\begin{equation}
 \int_{\mathbb R}\e^{wt}
 \sum_{j\in\mathcal J_{\mm}^{\mathrm H}}
 \mathscr H_{j,J}^{\mathrm H,+}(t)
 \mathcal W_j^{\mathrm H,+}(t)\,\dd t
 =L_+(w)\mathcal T_J^+(w).
 \label{eq:Hermite-plus-sector-transform}
\end{equation}
Indeed, divide \eqref{eq:Hermite-explicit-transformed-identity} by
\(Q_{\mm}^{\mathrm H}\) and apply the Cauchy projector at
\(-\rho_J\) (or the complementary polynomial projector when \(J=q\)).
The first-step verification of
Proposition~\ref{prop:Hermite-direct-block-decomposition} then gives
the minus-sector source
\(\mathcal T_J^-(z)\coloneq(-1)^n\mathcal T_J^+(-z)\).
Using \eqref{eq:Hermite-plus-direct-block} and
\(L_+(w)=L_-(-w)\) gives
\eqref{eq:Hermite-plus-sector-transform}.  In particular, this
identification uses the projected rational source, not the desired
component limit.

Finally expand
\[
 s^n\mathscr C_{j,J;s}^{\mathrm{CI}}(a+st)
 =\sum_{k=0}^{m_j-1}\gamma_{j,k;J,s}\mathsf H_k(t).
\]
For the corresponding centered basis columns put
\[
 \Psi_{j,k;s}^{\mathrm{CI}}(w)\coloneq
 \e^{-aw/s}\sum_{\ell\ge0}
 \mathsf H_k\left(\frac{\ell-a}{s}\right)
 v_{j,s}^{\mathrm{CI}}(\ell)\e^{w\ell/s}.
\]
The generating function of the monic Hermite polynomials gives the
exact identity
\[
 \sum_{k\ge0}\Psi_{j,k;s}^{\mathrm{CI}}(w)\frac{u^k}{k!}
 =\e^{-u^2/4}\widetilde G_{j,s}(w+u),
 \qquad
 \widetilde G_{j,s}(w)\coloneq
 \e^{-aw/s}G_{j,s}^{\mathrm{CI}}(\e^{w/s}).
\]
Writing
\(\Lambda_j^+(w)\coloneq
\int_{\mathbb R}\e^{wt}\mathcal W_j^{\mathrm H,+}(t)\dd t\),
formula \eqref{eq:CI-row-pgfs} therefore gives local convergence, with
every fixed derivative, to
\[
 \e^{-u^2/4}\Lambda_j^+(w+u)
 =\sum_{k\ge0}
 \mathsf H_k(\partial_w)\Lambda_j^+(w)\frac{u^k}{k!}.
\]
Thus the centered basis columns converge explicitly to the \(H^+\)
columns.  By
\eqref{eq:CI-Hermite-centered-sector-source}, their right-hand side
converges to the independently identified source
\eqref{eq:Hermite-plus-sector-transform}.  The limiting square matrix
of Taylor jets through order \(n\) is nonsingular: multiplication of
all columns by \(Q_+/L_+\) is an invertible lower-triangular operation
on these jets and gives the reflected matrix of
Lemma~\ref{lem:Hermite-connection-determinant}.
Continuity of matrix inversion now gives
\(\gamma_{j,k;J,s}\to(-1)^{n+k}
\mathfrak C_{(j,k);J}^{\mathrm H}\), which is precisely
\eqref{eq:CI-Hermite-sector-limit}.
\end{proof}

The Charlier-II-like family also has a Hermite-like limit.  Fix a
near-diagonal multi-index \(\mm\), put \(D=\abs\mm\ge1\) and
\(n=D-1\), and assume \(\rho_h,\delta_h>0\) for
\(h\in\{1,\ldots,q-1\}\). For \(a>0\), set \(s=\sqrt{2a}\),
\(A_h=s\rho_h/2\), and \(\beta_h=A_h+\delta_h\) for
\(h\in\{1,\ldots,q-1\}\).
For \(j\in\{1,\ldots,q\}\), the limiting weights are
\begin{equation}
 \mathcal W_j^{\mathrm H,-}(t)
 =\int_{(0,\infty)^{q-1}}
 \phi(t+y_1+\cdots+y_{q-1})
 \prod_{h=1}^{q-1}g_{\rho_h,\delta_h+\delta_{h,j}}(y_h)
 \dd\boldsymbol y.
\end{equation}
Normalize the type-II polynomial by
\[
 \widehat A_{\mm}^{\mathrm{CII}}(k)
 =(-a)^D\frac{(\A_{<q})_D}
 {(\bbeta_{<q}+\mm_{<q})_D}A_{\mm}^{\mathrm{CII}}(k).
\]
The unit-normalized signed sequence and its centered measure are
\[
 \mathcal B_{\mm}^{\mathrm{CII,unit}}(k)
 \coloneq
 \frac{(-1)^{n+1}\prod_{h=1}^{q-1}(\beta_h)_{m_h+n}}
 {n!a^n(\A_{<q})_n}
 \mathcal B_{\mm}^{\mathrm{CII}}(k),
 \qquad
 \mu_s^{\mathrm{CII}}
 \coloneq s^n\sum_{k\ge0}\mathcal B_{\mm}^{\mathrm{CII,unit}}(k)
 \delta_{(k-a)/s}.
\]
\begin{proposition}[Charlier-II-like-to-Hermite limit]
\label{prop:CII-Hermite}
Let \(\mm\) be a near-diagonal multi-index with \(D=\abs\mm\ge1\) and
\(n=D-1\), and assume \(\rho_h,\delta_h>0\) for
\(h\in\{1,\ldots,q-1\}\). For \(a>0\), set
\(s=\sqrt{2a}\), \(A_h=s\rho_h/2\), and
\(\beta_h=A_h+\delta_h\) for \(h\in\{1,\ldots,q-1\}\). Then the following limits hold as
\(s\to\infty\).

\begin{enumerate}[label=\textnormal{(\roman*)}]
\item For every \(j\in\{1,\ldots,q\}\), locally uniformly for \(t\) in
compact subsets of \(\R\),
\begin{equation}
 s\,v_j^{\mathrm{CII}}(\lfloor a+st\rfloor)
 \xrightarrow[s\to\infty]{}\mathcal W_j^{\mathrm H,-}(t).
 \label{eq:CII-Hermite-rows}
\end{equation}

\item Locally uniformly for \(t\) in compact subsets of \(\R\),
\(s^{-D}\widehat A_{\mm}^{\mathrm{CII}}(\lfloor a+st\rfloor)
\xrightarrow[s\to\infty]{}P_{\mm}^{\mathrm H,-}(t)\), where
\begin{equation}
 P_{\mm}^{\mathrm H,-}(t)
 \coloneq\sum_{\substack{r_1,\ldots,r_{q-1}\ge0\\
                    r_1+\cdots+r_{q-1}\le D}}
 (-D)_{r_1+\cdots+r_{q-1}}
 \prod_{h=1}^{q-1}
 \frac{(-\delta_h-m_h)_{r_h}}{r_h!}
 \rho_h^{-r_h}
 \mathsf H_{D-r_1-\cdots-r_{q-1}}(t).
 \label{eq:CII-Hermite-A}
\end{equation}
The limiting polynomial also satisfies
\begin{equation}
 P_{\mm}^{\mathrm H,-}(t)
 =D![u^D]\e^{tu-u^2/4}
 \prod_{h=1}^{q-1}(1+u/\rho_h)^{\delta_h+m_h}.
 \label{eq:CII-Hermite-A-generating}
\end{equation}

\item The signed measures converge against every test function in
\(C_c^\infty(\R)\) and in every fixed moment:
\begin{equation}
 \mu_s^{\mathrm{CII}}\xrightarrow[s\to\infty]{}
 \mathcal B_{\mm}^{\mathrm H,-}(t)\dd t.
 \label{eq:CII-Hermite-B}
\end{equation}
\end{enumerate}
\end{proposition}

\begin{proof}
Put \(d_{h,j}\coloneq\delta_h+\delta_{h,j}>0\) and extend by zero outside
\((0,s/2)\) the transformed beta densities
\[
 \widetilde b_{s,h,j}(u)
 \coloneq\frac2s\,b_{s\rho_h/2,d_{h,j}}\!\left(1-\frac{2u}{s}\right)
 =\frac{(2/s)^{d_{h,j}}(s\rho_h/2)_{d_{h,j}}}{\Gamma(d_{h,j})}
 u^{d_{h,j}-1}
 \left(1-\frac{2u}{s}\right)^{s\rho_h/2-1}.
\]
For every fixed \(u>0\),
\(\widetilde b_{s,h,j}(u)
\xrightarrow[s\to\infty]{}g_{\rho_h,d_{h,j}}(u)\). Let
\(k_s(t)\coloneq\lfloor a+st\rfloor\) and
\(Y_s(\boldsymbol u)\coloneq\prod_{h=1}^{q-1}(1-2u_h/s)\).
Since \(a=s^2/2\), for every compact
\(\mathcal K\subset[0,\infty)^{q-1}\) one has
\(aY_s(\boldsymbol u)
=a-s\sum_{h=1}^{q-1}u_h+\mathrm O_{\mathcal K}(1)\) as \(s\to\infty\),
uniformly for \(\boldsymbol u\in\mathcal K\). Consequently,
Stirling's formula gives
\(s\,p_{aY_s(\boldsymbol u)}(k_s(t))
\xrightarrow[s\to\infty]{}
\phi\left(t+\sum_{h=1}^{q-1}u_h\right)\),
locally uniformly in
\((t,\boldsymbol u)\in\R\times(0,\infty)^{q-1}\).

We next control the expanding integration domain.  There are constants
\(M_{h,j}>0\) such that, for all sufficiently large \(s\), uniformly for
\(0<u<s/2\),
\(\widetilde b_{s,h,j}(u)
\le M_{h,j}u^{d_{h,j}-1}\e^{-\rho_hu/2}\).
Indeed, \((2/s)^{d_{h,j}}(s\rho_h/2)_{d_{h,j}}\) is bounded and
\((1-2u/s)^{s\rho_h/2-1}\le\e^{-\rho_hu/2}\)
for all sufficiently large \(s\). Moreover, for every integer \(k\ge1\),
\(\sup_{\lambda\ge0}p_\lambda(k)
=\e^{-k}k^k/k!\le C/\sqrt{k}\).
Hence, for \(t\) in a fixed compact set and all sufficiently large \(s\),
\(\sup_{\lambda\ge0}s\,p_\lambda(k_s(t))\le C_K\). After the change of
variables above, \eqref{eq:CII-positive-coefficients} becomes
\[
 s\,v_j^{\mathrm{CII}}(k_s(t))
 =\int_{(0,\infty)^{q-1}}
 s\,p_{aY_s(\boldsymbol u)}(k_s(t))
 \prod_{h=1}^{q-1}\widetilde b_{s,h,j}(u_h)\dd\boldsymbol u.
\]
Its integrand is bounded, uniformly for \(t\in K\), by the integrable
function \(C_K\prod_{h=1}^{q-1}u_h^{d_{h,j}-1}\e^{-\rho_hu_h/2}\).
Dominated convergence, applied also to the supremum of the difference over
\(t\in K\), proves the locally uniform limit of the weights
\eqref{eq:CII-Hermite-rows}.
Using \eqref{eq:CII-Euler-integral}, expanding
\(\prod_{h=1}^{q-1}(1-2u_h/s)\), and putting \(z=\e^{w/s}\), one obtains
locally near \(w=0\)
\[
 \e^{-aw/s}G_j^{\mathrm{CII}}(\e^{w/s})
 \xrightarrow[s\to\infty]{}
 \e^{w^2/4}
 \prod_{h=1}^{q-1}\left(\frac{\rho_h}{\rho_h+w}\right)^{
 \delta_h+\delta_{h,j}}.
\]
The right-hand side is the bilateral transform of
\eqref{eq:CII-Hermite-rows}.

The polynomial limit requires the cancellation among all \(D+1\) terms of
the terminating series. Put \(\eta_h=\delta_h+m_h\) for
\(h\in\{1,\ldots,q-1\}\). For a fixed
sufficiently small \(\epsilon>0\) and all large \(s\), the following gamma
quotient is regular on \(\abs{x}\le\epsilon s\):
\[
 \Psi_s(x;t)
 \coloneq\frac{\Gamma(k_s(t)+1)}{\Gamma(k_s(t)-x+1)}a^{-x}
 \prod_{h=1}^{q-1}
 \frac{\Gamma(A_h+\eta_h+x)\Gamma(A_h)}
 {\Gamma(A_h+\eta_h)\Gamma(A_h+x)}.
\]
Writing \(\Delta_hf(y)=f(y+h)-f(y)\), the terminating formula
\eqref{eq:CII-A} gives the exact identity
\[
 \widehat A_{\mm}^{\mathrm{CII}}(k_s(t))
 =a^D\prod_{h=1}^{q-1}\frac{(A_h)_D}{(A_h+\eta_h)_D}
 \Delta_1^D\Psi_s(0;t).
\]
Set \(H_s(y;t)=\Psi_s(sy;t)\). Uniform gamma-ratio expansions, together
with their first \(D\) derivatives in \(y\), yield locally uniformly in
\((y,t)\) the limit
\(H_s(y;t)\xrightarrow[s\to\infty]{}
H(y;t)\coloneq\e^{2ty-y^2}
\prod_{h=1}^{q-1}(1+2y/\rho_h)^{\eta_h}\),
together with
\(\partial_y^rH_s\xrightarrow[s\to\infty]{}\partial_y^rH\) for
\(r\in\{0,\ldots,D\}\). Therefore
\begin{equation*}
	 s^{-D}\widehat A_{\mm}^{\mathrm{CII}}(k_s(t))
 =\frac{a^D}{s^D}
 \prod_{h=1}^{q-1}\frac{(A_h)_D}{(A_h+\eta_h)_D}
 \Delta_{1/s}^DH_s(0;t)
 \xrightarrow[s\to\infty]{}2^{-D}\partial_y^DH(0;t)
 =D![u^D]\e^{tu-u^2/4}
 \prod_{h=1}^{q-1}(1+u/\rho_h)^{\eta_h},
\end{equation*}
which proves \eqref{eq:CII-Hermite-A-generating} without separating its
divergent summands. Expanding the finite product and using
\eqref{eq:monic-Hermite-definition} gives \eqref{eq:CII-Hermite-A}.
Applying the transform calculation above to the generating
function normalized by its factorial moment of order \(n\) gives
\(w^n\int_{\R}\e^{wt}S_{\mm}^{\mathrm H,-}(t)\dd t/n!\).
For a direct test-function proof, define
\(p_{\mm,s}^{\mathrm{CII}}(k)\coloneq[z^k]H_{\mm,s}^{\mathrm{CII}}(z)\),
where, under the parameters of the proposition,
\[
 H_{\mm,s}^{\mathrm{CII}}(z)
 \coloneq\pFq{q-1}{q-1}
 {\A_{<q}+n}{\bbeta_{<q}+\mm_{<q}+n}{a(z-1)}
 =\int_{(0,1)^{q-1}}\e^{a(z-1)y_1\cdots y_{q-1}}
 \prod_{h=1}^{q-1}b_{A_h+n,\delta_h+m_h}(y_h)\dd\boldsymbol y.
\]
This is the factor remaining after \((z-1)^n/n!\) is removed. Its
coefficients are nonnegative and sum to one by the displayed Euler
integral. Thus
\(\mathcal B_{\mm}^{\mathrm{CII,unit}}(k)
=n!^{-1}\sum_{r=0}^n(-1)^{n-r}\binom nr
p_{\mm,s}^{\mathrm{CII}}(k-r)\).
For \(\varphi\in C_c^\infty(\R)\), summation by parts gives
\[
 \left\langle\mu_s^{\mathrm{CII}},\varphi\right\rangle
 =\frac1{n!}\sum_{k\ge0}p_{\mm,s}^{\mathrm{CII}}(k)
 s^n\Delta_{1/s}^n\varphi\!\left(\frac{k-a}{s}\right).
\]
Uniform convergence of the finite difference to \(\varphi^{(n)}\),
together with convergence of the centered coefficient measures against
compactly supported continuous test functions to
\(S_{\mm}^{\mathrm H,-}(t)\dd t\), proves
\eqref{eq:CII-Hermite-B}. The transform identity and the exact coefficient
moments prove convergence of every fixed moment.
Here the required coefficient-measure convergence follows from the same
beta--Poisson dominated-convergence argument used for
\eqref{eq:CII-Hermite-rows}, now with first beta parameters \(A_h+n\)
and fixed second shapes \(\delta_h+m_h\); its limiting density is exactly
\(S_{\mm}^{\mathrm H,-}\).
\end{proof}

The two limiting systems are related by reflection:
\begin{equation}
 P_{\mm}^{\mathrm H,+}(t)=(-1)^D
 P_{\mm}^{\mathrm H,-}(-t),\qquad
 \mathcal B_{\mm}^{\mathrm H,+}(t)=(-1)^n
 \mathcal B_{\mm}^{\mathrm H,-}(-t).
 \label{eq:Hermite-reflection}
\end{equation}
The identities \eqref{eq:Hermite-reflection} show that the Charlier-I-like
and Charlier-II-like families converge to Hermite-like systems related by
reflection, although they are different before the continuous limit.
In particular, \(P_{\mm}^{\mathrm H,-}\) is monic of degree \(D\), and
reflection of \eqref{eq:Hermite-A-orthogonality} and
\eqref{eq:Hermite-B-moments} gives explicitly
\begin{align*}
 \int_{\mathbb R}t^rP_{\mm}^{\mathrm H,-}(t)
 \mathcal W_j^{\mathrm H,-}(t)\dd t
 &=0,
 &&0\le r<m_j,
&
 \int_{\mathbb R}t^r\mathcal B_{\mm}^{\mathrm H,-}(t)\dd t
 &=\delta_{r,n},
 &&0\le r\le n.
\end{align*}

The individual type-I polynomials also converge in the Charlier-II-like
family.

\begin{corollary}[Limit of the Charlier-II-like type-I polynomials]
\label{cor:CII-Hermite-components}
Under the hypotheses of Proposition~\ref{prop:CII-Hermite}, assume that
the rates \(\rho_h\) for which \(h\in\{1,\ldots,q-1\}\) and \(m_h>0\)
are pairwise distinct.
Let \(v_{j,s}^{\mathrm{CII}}\) denote the weight with the \(s\)-dependent
parameters of Proposition~\ref{prop:CII-Hermite}, for every
\(j\in\{1,\ldots,q\}\). Then:
\begin{enumerate}[label=\textnormal{(\roman*)}]
\item For all sufficiently large \(s\), there is a unique tuple
\(\bigl(B_{j,s}^{\mathrm{CII,unit}}\bigr)_{
 j\in\{1,\ldots,q\},\,m_j>0}\) satisfying
\(\deg B_{j,s}^{\mathrm{CII,unit}}<m_j\) for every
\(j\in\{1,\ldots,q\}\) with \(m_j>0\), and
\[
 \mathcal B_{\mm}^{\mathrm{CII,unit}}(k)
 =\sum_{\substack{j\in\{1,\ldots,q\}\\m_j>0}}
 B_{j,s}^{\mathrm{CII,unit}}(k)
 v_{j,s}^{\mathrm{CII}}(k).
\]
\item Let \(\bigl(B_j^{\mathrm H,-}\bigr)_{
 j\in\{1,\ldots,q\},\,m_j>0}\) be the unique tuple satisfying
\(\deg B_j^{\mathrm H,-}<m_j\) for every
\(j\in\{1,\ldots,q\}\) with \(m_j>0\), and
\[
 \mathcal B_{\mm}^{\mathrm H,-}(t)
 =\sum_{\substack{j\in\{1,\ldots,q\}\\m_j>0}}
 B_j^{\mathrm H,-}(t)
 \mathcal W_j^{\mathrm H,-}(t).
\]
For every \(j\in\{1,\ldots,q\}\) with \(m_j>0\), as \(s\to\infty\), the polynomial
\(t\mapsto s^nB_{j,s}^{\mathrm{CII,unit}}(a+st)\) converges
coefficientwise to \(B_j^{\mathrm H,-}(t)\).
\end{enumerate}
\end{corollary}

\begin{proof}
The same integration-by-parts calculation as in the proof of
Corollary~\ref{cor:CI-Hermite-components} gives
\[
 Q_{\mm}^{\mathrm H}(z)L_-(z)^{-1}
 \int_{\mathbb R}\e^{zt}\mathcal B_{\mm}^{\mathrm H,-}(t)\dd t
 =\frac{z^n}{n!}\prod_{h=1}^{q-1}\rho_h^{m_h}.
\]
Thus Lemma~\ref{lem:Hermite-connection-determinant} gives the unique
type-I polynomials \(B_j^{\mathrm H,-}\). The centered transforms of the
weights in
Proposition~\ref{prop:CII-Hermite}, together with their derivatives
through order \(n+\max_{h:m_h>0}(m_h-1)\), converge locally, as
\(s\to\infty\), to the corresponding \(H^-\)
transforms. After the displayed moment and weight normalizations, the
discrete coefficient matrices therefore converge entrywise to that
Hermite matrix. Their determinants are nonzero for all sufficiently large
\(s\). Expanding
\(s^nB_{j,s}^{\mathrm{CII,unit}}(a+st)
=\sum_{k=0}^{m_j-1}\gamma_{j,k,s}\mathsf H_k(t)\) and applying continuity of
matrix inversion proves the asserted
coefficientwise limit. The representation with the original weights is
obtained from the finite reconstruction with the canonical rows
\(w_j^{\mathrm{CII}}\) by
\eqref{eq:unreflected-probability-components}.
\end{proof}

\paragraph{Unit-normalized Charlier-II-like sectors.}
Retain the notation and scaling of
Proposition~\ref{prop:CII-Hermite}; in particular,
\[
 a=\frac{s^2}{2},\qquad
 A_{h,s}=\frac{s\rho_h}{2},\qquad
 \beta_{h,s}=A_{h,s}+\delta_h
 \quad(h\in\{1,\ldots,q-1\}).
\]
Put
\begin{equation}
 U_s\coloneq
 \frac{(-1)^{n+1}
 \prod_{h=1}^{q-1}
 (\beta_{h,s})_{m_h+n}}
 {n!\,a^n
 \prod_{h=1}^{q-1}(A_{h,s})_n},
 \qquad
 \chi_{j,s}\coloneq
 \begin{cases}
  \beta_{j,s}^{-1},&j<q,\\
  1,&j=q.
 \end{cases}
 \label{eq:CII-Hermite-sector-unit-factor}
\end{equation}
For active \(j,J\), define the canonical sectors
\begin{equation}
 \mathscr D_{j,J;s}^{\mathrm{CII}}(k)\coloneq
 \begin{cases}
  \mathscr C_{j,J;s}(k),&J<q,\\[2mm]
  \mathscr C_{j,q;s}^{\infty}(k),&J=q,
 \end{cases}
 \label{eq:CII-Hermite-canonical-sectors}
\end{equation}
where \(\mathscr C_{j,J;s}\) and
\(\mathscr C_{j,q;s}^{\infty}\) are respectively
\eqref{eq:CII-compact-finite-pole-block} and
\eqref{eq:CII-infinity-block-definition} evaluated at the displayed
\(s\)-dependent parameters. Thus the first line is the whole grouped
finite-pole block
\(\sum_{K=0}^{m_J-1}\pi_{J,K;s}^{-}
\mathcal C_{j,J,K;s}^{\mathrm{CII}}\), not one of its summands.
The second line likewise keeps each corrected vector in the infinity block
intact. The sectors relative
to the original positive weights \(v_{j,s}^{\mathrm{CII}}\), with unit
type-I normalization, are
\begin{equation}
 \mathscr U_{j,J;s}^{\mathrm{CII}}(k)
 \coloneq U_s\chi_{j,s}
 \mathscr D_{j,J;s}^{\mathrm{CII}}(k).
 \label{eq:CII-Hermite-unit-sectors}
\end{equation}
The factor \(\chi_{j,s}\) is essential: the polynomials
\(D_j^{\mathrm{CII}}\) are relative to the canonical rows
\(w_j^{\mathrm{CII}}=v_j^{\mathrm{CII}}/\beta_j\) for \(j<q\).
Equations \eqref{eq:unreflected-explicit-components},
\eqref{eq:unreflected-probability-components}, and
\eqref{eq:CII-compact-finite-pole-identity} give the exact decomposition
\begin{equation}
 B_{j,s}^{\mathrm{CII,unit}}(k)
 =\sum_{J\in\mathcal J_{\mm}^{\mathrm H}}
 \mathscr U_{j,J;s}^{\mathrm{CII}}(k).
 \label{eq:CII-Hermite-unit-sector-decomposition}
\end{equation}

\begin{corollary}[Sectorwise Charlier-II-like-to-Hermite confluence]
\label{cor:CII-Hermite-sector-limit}
Under the hypotheses of Corollary~\ref{cor:CII-Hermite-components}, for
all sufficiently large \(s\) the sectors
\eqref{eq:CII-Hermite-unit-sectors} are well defined.  For every pair of
active indices \(j,J\),
\begin{equation}
 s^n\mathscr U_{j,J;s}^{\mathrm{CII}}(a+st)
 \xrightarrow[s\to\infty]{}
 \mathscr H_{j,J}^{\mathrm H,-}(t)
 \quad\hbox{coefficientwise in }t .
 \label{eq:CII-Hermite-sector-limit}
\end{equation}
\end{corollary}

\begin{proof}
We first identify the source of each sector independently of the desired
component limit.  Put
\[
 h_s(r)\coloneq
 a^r\prod_{h=1}^{q-1}
 \frac{(A_{h,s})_r}{(\beta_{h,s})_r},
 \qquad
 \mathcal R_{\mm,s}(r)\coloneq
 -\frac{(-r)_n}
 {\prod_{h=1}^{q-1}(r+\beta_{h,s})_{m_h}}.
\]
Split \(\mathcal R_{\mm,s}\) exactly as in
\eqref{eq:unreflected-target-partial-fractions}:
\begin{equation}
 \mathcal R_{J,s}(r)\coloneq
 \begin{cases}
  \displaystyle\sum_{K=0}^{m_J-1}
  \frac{\pi_{J,K;s}^{-}}{r+\beta_{J,s}+K},&J<q,\\[3mm]
  \displaystyle\sum_{K=0}^{m_q-1}
  \pi_{\infty,K;s}r^K,&J=q.
 \end{cases}
 \label{eq:CII-Hermite-rational-sectors}
\end{equation}
The finite-pole reconstruction and the corrected polynomial block give,
for every \(r\in\Nzero\),
\begin{equation}
 \sum_{k\ge0}\fall{k}{r}
 \sum_{j\in\mathcal J_{\mm}^{\mathrm H}}
 \mathscr U_{j,J;s}^{\mathrm{CII}}(k)
 v_{j,s}^{\mathrm{CII}}(k)
 =U_sh_s(r)\mathcal R_{J,s}(r).
 \label{eq:CII-Hermite-exact-sector-moments}
\end{equation}

We now take the source limit.  Set
\[
 Q_-(z)\coloneq\prod_{h=1}^{q-1}(z+\rho_h)^{m_h},\qquad
 \mathcal T^-(z)\coloneq
 \frac{C_{\mm}z^n}{n!\,Q_-(z)},\qquad
 C_{\mm}\coloneq\prod_{h=1}^{q-1}\rho_h^{m_h},
\]
and
\[
 p_s(z)\coloneq\prod_{u=0}^{n-1}\left(z-\frac{2u}{s}\right),
 \qquad
 Q_s(z)\coloneq
 \prod_{h=1}^{q-1}\prod_{K=0}^{m_h-1}
 \left(z+\rho_h+\frac{2(\delta_h+K)}s\right).
\]
Direct factorization, with no limiting step, gives
\begin{equation}
 \left(\frac2s\right)^{m_q-1}
 \mathcal R_{\mm,s}\left(\frac{sz}{2}\right)
 =(-1)^{n+1}\frac{p_s(z)}{Q_s(z)}.
 \label{eq:CII-Hermite-scaled-rational-source}
\end{equation}
For every active \(J<q\), choose a positively oriented circle
\(\Gamma_J^-\) around \(-\rho_J\), with the circles pairwise disjoint.
For large \(s\), \(\Gamma_J^-\) contains precisely the \(m_J\) poles
\[
 -\rho_J-\frac{2(\delta_J+K)}s,
 \qquad K\in\{0,\ldots,m_J-1\}.
\]
Define
\[
 (\mathcal P_J^-f)(z)\coloneq
 \frac{1}{2\pi\mathrm i}\int_{\Gamma_J^-}
 \frac{f(\zeta)}{z-\zeta}\,\dd\zeta
 \quad(J<q),\qquad
 \mathcal P_q^-f\coloneq
 f-\sum_{\substack{J<q\\m_J>0}}\mathcal P_J^-f .
\]
Cauchy's formula applied to
\eqref{eq:CII-Hermite-scaled-rational-source} gives, locally away from
the limiting poles,
\begin{equation}
 \left(\frac2s\right)^{m_q-1}
 \mathcal R_{J,s}\left(\frac{sz}{2}\right)
 \longrightarrow
 (-1)^{n+1}\mathcal P_J^-
 \left(\frac{z^n}{Q_-(z)}\right).
 \label{eq:CII-Hermite-sector-source-before-unit}
\end{equation}
For \(J=q\), the same conclusion follows by subtracting all finite-pole
projections from \eqref{eq:CII-Hermite-scaled-rational-source}; it is
therefore convergence of the whole polynomial part, not of separately
divergent coefficients.

The remaining scalar normalization is
\begin{equation}
 s^nU_s\left(\frac{s}{2}\right)^{m_q-1}
 \longrightarrow
 \frac{(-1)^{n+1}C_{\mm}}{n!}.
 \label{eq:CII-Hermite-unit-factor-limit}
\end{equation}
Indeed,
\[
 \frac{(\beta_{h,s})_{m_h+n}}{(A_{h,s})_n}
 =A_{h,s}^{m_h}\bigl(1+\mathrm O(s^{-1})\bigr),
\]
and \(a=s^2/2\), while \(n=\sum_{h<q}m_h+m_q-1\).
The two factors \((-1)^{n+1}\) in
\eqref{eq:CII-Hermite-sector-source-before-unit} and
\eqref{eq:CII-Hermite-unit-factor-limit} cancel.  Hence, with
\(\mathcal T_J^-\coloneq\mathcal P_J^-\mathcal T^-\),
\begin{equation}
 s^nU_s\mathcal R_{J,s}\left(\frac{sz}{2}\right)
 \longrightarrow\mathcal T_J^-(z).
 \label{eq:CII-Hermite-unit-sector-source-limit}
\end{equation}

It remains to show that this source convergence controls the individual
polynomials.  Expand uniquely
\[
 s^n\mathscr U_{j,J;s}^{\mathrm{CII}}(a+st)
 =\sum_{k=0}^{m_j-1}\gamma_{j,k;J,s}\mathsf H_k(t)
\]
and, initially for \(r\in\Nzero\), define the normalized centered
columns
\[
 \mathscr R_{j,k;s}^{\,c}(r)\coloneq
 \frac{1}{h_s(r)}
 \sum_{\ell\ge0}\fall{\ell}{r}
 \mathsf H_k\left(\frac{\ell-a}{s}\right)
 v_{j,s}^{\mathrm{CII}}(\ell).
\]
Their meromorphic continuation in \(r\) is obtained directly by
expanding the fixed polynomial
\(\mathsf H_k((\ell-a)/s)\) in falling factorials and using
\eqref{eq:CII-moments}.  Substitution in
\eqref{eq:CII-Hermite-exact-sector-moments} gives the exact finite
system
\begin{equation}
 \sum_{\substack{j\in\mathcal J_{\mm}^{\mathrm H}\\0\le k<m_j}}
 \gamma_{j,k;J,s}
 \mathscr R_{j,k;s}^{\,c}(r)
 =s^nU_s\mathcal R_{J,s}(r).
 \label{eq:CII-Hermite-centered-sector-system}
\end{equation}

We record the column limit explicitly.  With
\(C_{h,j;s}\coloneq\beta_{h,s}+\delta_{h,j}\), the Hermite generating
function and the derivative formula for
\eqref{eq:CII-pgf} give, coefficientwise in \(u\),
\begin{equation}
 \sum_{k\ge0}\mathscr R_{j,k;s}^{\,c}(r)\frac{u^k}{k!}
 =
 \exp\left\{-\frac{u^2}{4}+\frac{(r-a)u}{s}\right\}
 \prod_{h=1}^{q-1}\frac{(\beta_{h,s})_r}{(C_{h,j;s})_r}
 \pFq{q-1}{q-1}
 {\bigl(A_{h,s}+r\bigr)_{h=1}^{q-1}}
 {\bigl(C_{h,j;s}+r\bigr)_{h=1}^{q-1}}
 {a(\e^{u/s}-1)}.
 \label{eq:CII-Hermite-centered-column-generating}
\end{equation}
Set \(r=sz/2\).  The Euler integral for the last hypergeometric
function now contains beta variables with first parameters
\(s(\rho_h+z)/2\) and fixed second parameters
\(\delta_h+\delta_{h,j}\).  The change
\(y_h=1-2x_h/s\), followed by dominated convergence on compact
subsets of \(\Re z>-\min_{h<q}\rho_h\) when \(q>1\), and on compact
\(z\)-sets when \(q=1\), with \(u\) near zero, gives
\begin{equation}
 \sum_{k\ge0}
 \mathscr R_{j,k;s}^{\,c}\left(\frac{sz}{2}\right)\frac{u^k}{k!}
 \longrightarrow
 \e^{zu/2}
 \begin{cases}
  \displaystyle\frac{\rho_j}{\rho_j+z},&j<q,\\
  1,&j=q,
 \end{cases}
 \prod_{h=1}^{q-1}
 \left(\frac{\rho_h+z}{\rho_h+z+u}\right)^{
 \delta_h+\delta_{h,j}}
=\e^{-u^2/4}\frac{\Lambda_j^-(z+u)}{L_-(z)}
 =\sum_{k\ge0}\frac{\Phi_{j,k}^{\mathrm H,-}(z)}
 {Q_-(z)}\frac{u^k}{k!}.
 \label{eq:CII-Hermite-centered-column-limit}
\end{equation}
This calculation also fixes the sign: the beta variables approach one,
so their fluctuations add to \(t\) and produce the \(H^-\), rather than
the \(H^+\), columns.

Evaluate \eqref{eq:CII-Hermite-centered-sector-system} at \(r=sz/2\)
and multiply it by \(Q_s(z)\).  By
\eqref{eq:unreflected-R-pairing} and near-diagonality, every resulting
column is a polynomial of degree at most \(n\).  Equations
\eqref{eq:CII-Hermite-centered-column-limit} and
\eqref{eq:CII-Hermite-unit-sector-source-limit} show coefficientwise
convergence of this square polynomial system to
\[
 \sum_{\substack{j\in\mathcal J_{\mm}^{\mathrm H}\\0\le k<m_j}}
 \gamma_{j,k;J}\Phi_{j,k}^{\mathrm H,-}(z)
 =Q_-(z)\mathcal T_J^-(z).
\]
The limiting matrix is nonsingular by
Lemma~\ref{lem:Hermite-connection-determinant}.  Cauchy projection of
\eqref{eq:Hermite-explicit-transformed-identity}, or equivalently the
first-step verification in
Proposition~\ref{prop:Hermite-direct-block-decomposition}, identifies
its unique solution as
\(\gamma_{j,k;J}=\mathfrak C_{(j,k);J}^{\mathrm H}\).
Continuity of matrix inversion proves
\eqref{eq:CII-Hermite-sector-limit}.

When \(q=1\), all products and finite-pole contours are empty,
\(m_1=n+1\), and
\[
 U_s=\frac{(-1)^{n+1}}{n!a^n},\qquad
 \left(\frac2s\right)^n
 \bigl[-(-sz/2)_n\bigr]\longrightarrow(-1)^{n+1}z^n.
\]
The signs cancel exactly as above, and the sole block tends to
\(2^n\mathsf H_n(t)/n!=\mathscr H_{1,1}^{\mathrm H,-}(t)\).
\end{proof}

For \(q=1\), the products in \eqref{eq:Hermite-A} and
\eqref{eq:Hermite-seed-B} are empty, and one recovers the monic Hermite
polynomial and the Gaussian derivative. The standard multiple-Charlier
to multiple-Hermite confluence belongs to the classical multiple Askey scheme
\cite{BranquinhoDiazFoulquieManasWolfs2024Discrete}.

\section{Conclusions}
\label{sec:conclusions}

We have placed the Jacobi-like and Laguerre-like systems for ordinary
type-I/type-II multiple orthogonality considered by Wolfs in a single
Askey-type confluence scheme. Its common finite-lattice ancestor is the
Hahn-like system obtained by applying the Bernstein transform to the
Jacobi-like weights. This construction gives positive weights, explicit
factorial moments, a terminating type-II polynomial, a normalized type-I
form, and an exact inverse Bernstein identity. For near-diagonal indices,
the separation and nonvanishing hypotheses of
Theorem~\ref{thm:Hahn-B-components} give finite hypergeometric formulas
and uniqueness for every type-I component.

The confluence analysis places this common ancestor in the displayed
diagram, which contains the Hahn-like, Kravchuk-like, two Meixner-like,
and two Charlier-like discrete systems, together with the Jacobi-like, two
Laguerre-like, and Hermite-like continuous systems. Along every directed
arrow we determine the type-II polynomial, the normalized type-I form
and, under the stated hypotheses, the type-I components and their
confluence. Thus the explicit
hypergeometric formulas realize the diagram itself rather than serving as
ancillary formulas attached to its nodes.

At the Hahn-like node, the unreflected finite-pole contributions regroup
into terminating Kamp\'e de F\'eriet blocks.
Proposition~\ref{prop:Hahn-Jacobi-KdF-confluence} gives their locally
uniform componentwise Hahn-like-to-Jacobi-like limit. Along the
Kravchuk-like-to-Charlier-II-like,
Meixner-II-like-to-Charlier-II-like, and
Meixner-II-like-to-Laguerre-I-like paths, these blocks and the
reconstructed block at infinity converge sectorwise
(Corollaries~\ref{cor:K-to-CII}, \ref{cor:CII-sectorwise-MII-limit}, and
\ref{cor:MII-LI-component-limit}). Along the
reflected Hahn-like--Meixner-I-like--Charlier-I-like path
(including Theorem~\ref{thm:K-to-CI}) and the
Meixner-I-like--Laguerre-II-like path, grouped blocks have finite
Lauricella--Horn sector limits; the latter is governed exactly by
Proposition~\ref{prop:MI-LII-exact-sectors}. Both
Charlier-like--Hermite-like arrows admit finite Hermite--Lauricella sector
decompositions related by reflection
(Corollaries~\ref{cor:CI-Hermite-sector-limit} and
\ref{cor:CII-Hermite-sector-limit}).

For one weight, all these constructions reduce to the corresponding
classical families. For several weights, no permutation of the normalized
rows identifies the two Meixner-like systems or the two Charlier-like
systems, by Propositions~\ref{prop:Meixner-families-distinct} and
\ref{prop:Charlier-families-distinct}, respectively. The present paper
establishes the orthogonality and explicit hypergeometric layer of the
Askey-type scheme. Recurrence coefficients and their factorizations form
a separate layer and are not treated here. Normality beyond the stated
nonvanishing domains, coincident poles, and possible AT proofs of
perfectness remain open. The mixed-type Hahn-like setting with two
independent numbers of weights is deliberately kept as a separate
development.

\section*{Funding}

This work was supported by research project PID2024-155133NB-I00,
\emph{Ortogonalidad, aproximaci\'on e integrabilidad: aplicaciones en procesos
	estoc\'asticos cl\'asicos y cu\'anticos}, funded by
MICIU/AEI/10.13039/501100011033 and by ERDF/EU.

\renewcommand*{\bibfont}{\small}
\enlargethispage{2\baselineskip}
\printbibliography

\end{document}